\documentclass[11pt,reqno]{amsart}

\usepackage{
  mathabx,
  amsmath,
  amsfonts,
  amssymb,
  amsthm,
  amscd,
  graphicx,
  mathtools,    
  enumitem,
  mathdots,     
  txfonts,
  stackrel
}
\usepackage[all]{xy}
\usepackage[usenames,dvipsnames]{xcolor}

\newcommand{\arxiv}[1]{{\tt arXiv:#1}}

\makeatletter
\@namedef{subjclassname@2020}{%
  \textup{2020} Mathematics Subject Classification}
\makeatother

\usepackage[vcentermath]{youngtab}
\def\boxempty{\,{\displaystyle\boxvoid}\,}
\def\boxed#1{\,{\displaystyle\boxvoid}\hspace{-2.2mm}{\scriptstyle #1}\:\,}

\usepackage{bbm}                     

\DeclareFontFamily{OT1}{pzc}{}
\DeclareFontShape{OT1}{pzc}{m}{it}{<-> s * [1.10] pzcmi7t}{}
\DeclareMathAlphabet{\mathpzc}{OT1}{pzc}{m}{it}

\usepackage[colorlinks=true, linkcolor=blue, citecolor=purple, urlcolor=blue, breaklinks=true]{hyperref}

\usepackage{tikz}
\usetikzlibrary{decorations.markings}
\usetikzlibrary{cd}
\usetikzlibrary{patterns}
\usetikzlibrary{shapes}

\tikzset{anchorbase/.style={baseline={([yshift=-0.5ex]current bounding box.center)}}}
\tikzset{wipe/.style={white,line width=4pt}}

\newcommand{\braidto}{to[out=up,in=down]}
\newcommand{\tobraid}{to[out=down,in=up]}

\tikzset{->-/.style={decoration={
  markings,
  mark=at position #1 with {\arrow{>}}},postaction={decorate}}}
\tikzset{-<-/.style={decoration={
  markings,
  mark=at position #1 with {\arrow{<}}},postaction={decorate}}}

\tikzset{darkg/.style={green!70!black}}

\def\dt{{\color{white}\bullet}\!\!\!\circ}
\newcommand\opendot[1]{\node at (#1) {$\dt$}}
\newcommand\closeddot[1]{\node at (#1) {$\bullet$}}

\tikzset{->-/.style={decoration={
  markings,
  mark=at position #1 with {\arrow{>}}},postaction={decorate}}}
\tikzset{-<-/.style={decoration={
  markings,
  mark=at position #1 with {\arrow{<}}},postaction={decorate}}}

\def\clock{\begin{tikzpicture}[baseline=-.9mm]
\filldraw[white] (0,0) circle (1.72mm);
\draw[-] (0,-0.18) to[out=180,in=-90] (-.18,0);
\draw[->] (-0.18,0) to[out=90,in=180] (0,0.18);
\draw[-] (-0.02,0.178) to[out=12,in=90] (0.18,0);
\draw[-] (0.18,0) to[out=-90,in=0] (0,-0.18);
\end{tikzpicture}\,}

\def\anticlock{\begin{tikzpicture}[baseline=-.9mm]
\filldraw[white] (0,0) circle (1.72mm);
\draw[-] (0,-0.18) to[out=180,in=-102] (-.178,0.02);
\draw[-] (-0.18,0) to[out=90,in=180] (0,0.18);
\draw[-] (0.18,0) to[out=-90,in=0] (0,-0.18);
\draw[<-] (0,0.18) to[out=0,in=90] (0.18,0);
\end{tikzpicture}\,}

\usepackage{cleveref}

\crefname{definition}{Definition}{Definitions}
\crefname{example}{Example}{Examples}
\crefname{lemma}{Lemma}{Lemmas}
\crefname{corollary}{Corollary}{Corollaries}
\crefname{theorem}{Theorem}{Theorems}
\crefname{remark}{Remark}{Remarks}
\crefname{equation}{}{}
\crefname{enumi}{}{}
\crefname{section}{${\S\!\!}$}{${\S\S\!\!}$}

\def\mid{\pmb{|}}
\def\Specm{\operatorname{Spec}}

\newcommand\Z{\mathbb{Z}}

\newcommand\N{\mathbb{N}}
\newcommand\kk{\Bbbk}
\newcommand\KK{\mathbb{K}}
\def\op{{\operatorname{op}}}
\def\rev{{\operatorname{rev}}}
\newcommand\one{\mathbbm{1}}
\def\P{\mathcal{P}}
\def\add{\operatorname{add}}
\def\rem{\operatorname{rem}}
\def\stringlabel#1{{\scriptscriptstyle\color{blue}#1}}
\def\stringnumber#1{{\scriptstyle\color{blue}#1}}
\def\dual{\circledast}
\def\ostar{{\medcirc\hspace{-2.2mm}\star\hspace{.5mm}}}
\def\costar{\underline{\ostar}}
\def\Nat{U_t}
\def\PR{\operatorname{pr}}

\def\TT{{\mathrm{T}}}
\def\gr{\operatorname{gr}}
\def\ob{\operatorname{ob}}
\def\OO{\mathbb{O}}
\newcommand\cA{\mathpzc{A}}
\newcommand\cB{\mathpzc{B}}
\newcommand\cC{\mathpzc{C}}
\newcommand\cU{\mathpzc{U}}

\newcommand\cD{\mathpzc{D}}

\newcommand\Par{\mathpzc{Par}}
\newcommand\Sym{\mathpzc{Sym}}
\newcommand\sym{{S\!ym}}
\newcommand\Asym{{AS\!ym}}
\newcommand\ASym{\mathpzc{ASym}}
\newcommand\Heis{\mathpzc{Heis}}
\newcommand\APar{\mathpzc{APar}}
\renewcommand\Vec{\mathpzc{Vec}}
\newcommand\Vecfd{\mathpzc{Vec}_{\operatorname{fd}}}

\newcommand\Mod{{\operatorname{-Mod}}}
\newcommand\doM{{\operatorname{Mod-}}}
\newcommand\dfldoM{{\operatorname{Mod_{lfd}-}}}
\newcommand\Proj{{\operatorname{-Proj}}}
\newcommand\Modlfd{{\operatorname{-Mod_{lfd}}}}
\newcommand\Moddelta{{\operatorname{-Mod_\Delta}}}
\newcommand\Modfd{{\operatorname{-Mod_{fd}}}}

\def\eps{\varepsilon}
\def\cont{\operatorname{cont}}
\def\wt{\operatorname{wt}}
\def\HC{\operatorname{HC}}
\def\END{\mathpzc{End}}
\def\ev{\operatorname{ev}}
\def\Ev{\operatorname{Ev}}
\def\Act{\operatorname{Act}}
\def\coev{\operatorname{coev}}

\def\sigmadual{{\scriptscriptstyle\bigcirc\hspace{-1.7mm}\sigma\hspace{.1mm}}}
\def\REP{\underline{\text{Re}}\text{p}}

\DeclareMathOperator{\End}{End}
\DeclareMathOperator{\Ext}{Ext}
\DeclareMathOperator{\Tor}{Tor}
\DeclareMathOperator{\Hom}{Hom}
\DeclareMathOperator{\HOM}{\mathpzc{Hom}}
\DeclareMathOperator{\Id}{Id}       
\DeclareMathOperator{\Kar}{Kar}     
\def\Add{\operatorname{Add}}
\DeclareMathOperator{\res}{res}
\DeclareMathOperator{\ind}{ind}
\DeclareMathOperator{\coind}{coind}

\DeclareMathOperator{\rad}{rad}
\DeclareMathOperator{\soc}{soc}

\DeclareMathOperator{\triv}{triv}
\DeclareMathOperator{\infl}{infl}

\newtheorem{theorem}{Theorem}[section]
\newtheorem{lemma}[theorem]{Lemma}
\newtheorem*{lemma*}{Lemma}
\newtheorem{corollary}[theorem]{Corollary}

\theoremstyle{definition}
\newtheorem{definition}[theorem]{Definition}
\newtheorem{remark}[theorem]{Remark}
\newtheorem{example}[theorem]{Example}

\numberwithin{equation}{section}
\allowdisplaybreaks

\setenumerate[1]{label=(\alph*)}  
\setenumerate[2]{label=(\roman*)}

\begin{document}

\title[Representations of the partition category]{A new approach to the representation theory of the partition category}
\thanks{
Authors supported in part by the National Science
Foundation grant DMS-2101783.
}
\author{Jonathan Brundan}
\address[J.B.]{
  Department of Mathematics \\
  University of Oregon \\
  Eugene, OR, USA
}
\email{brundan@uoregon.edu}

\author{Max Vargas}
\address[M.V.]{
  Department of Mathematics \\
  University of Oregon \\
  Eugene, OR, USA
}
\email{mvargas2@uoregon.edu}

\begin{abstract}
We explain a new approach to the representation theory of the
partition category based on a reformulation of the definition of the Jucys-Murphy
elements introduced originally by Halverson and Ram and developed further by Enyang. Our reformulation involves a new graphical monoidal category, the
{\em affine partition category}, which is defined here as a certain monoidal subcategory of Khovanov's Heisenberg category.
We use the Jucys-Murphy elements to construct
some special projective functors, then apply these functors
to give self-contained proofs of results of Comes and
Ostrik on blocks of Deligne’s category $\REP(S_t)$.
\end{abstract}


\maketitle
\thispagestyle{empty}

\section{Introduction}\label{sec1}

Let $\kk$ be an algebraically closed field of characteristic zero
and $t \in \kk$ be a parameter.
The {\em partition category} $\Par_t$ is the free strict $\kk$-linear symmetric
monoidal category
generated by a special commutative Frobenius object
of categorical dimension $t$. Its additive Karoubi envelope is
the category $\REP(S_t)$ introduced by Deligne \cite{Del}, which interpolates the categories of representations of the symmetric groups $S_t\:(t \in \N)$ to non-integer values of $t$.
When $t \notin \N$, Deligne's category is a semisimple tensor category which is not of sub-exponential growth, hence, it does not admit a fiber functor; see \cite[Sec.~9.12]{EGNO} for further background here.
When $t \in \N$, the category $\REP(S_t)$ is not semisimple,
and its semisimplification is the usual tensor category
$\kk S_t\Modfd$ of representations of the symmetric group.

The objects of the partition category are indexed by the natural numbers.
For $n \in \N$,
the endomorphism algebra $\End_{\Par_t}(n)$ is the {\em partition algebra}
$P_n(t)$ introduced by Martin \cite{Martinfirst} and Jones \cite{Jones}.
The representation theory of this finite-dimensional algebra
has been well studied. In \cite{Martin}, Martin showed that
$P_n(t)\Modfd$ is a highest weight category except when $t=0$, and he
determined the precise structure of the standard modules; see also \cite{DW}.
When $t=0$, $P_n(t)$ still has the structure of a cellular algebra, as established in \cite{DW,X}, and its representation theory is also well understood.
The partition algebras form a tower $P_0(t) < P_1(t) < \cdots$, but
the cell modules do not restrict along this tower in a multiplicity-free way,
so that standard techniques like the Jones basic construction cannot be applied directly. To address this, Martin \cite{Martinlast} and Halverson and Ram \cite{HR}
consider
an intermediate family of ``half partition algebras'' fitting into
a tower $$
P_0(t) < P_{\frac{1}{2}}(t) < P_1(t) < P_{\frac{3}{2}}(t) < \cdots.
$$
Halverson and Ram also defined analogs
$L_0, L_{\frac{1}{2}}, L_1, L_{\frac{3}{2}},\dots$
of Jucys-Murphy elements in these partition algebras, which were studied further by Enyang \cite{En, En2}. Enyang
worked out a recursive definition for the Jucys-Murphy elements and
used them to construct an analog of Young's orthogonal form for the
irreducible $P_n(t)$-modules. His definition
involves a complicated
five term recurrence relation, making the Jucys-Murphy elements for partition algebras considerably harder to work with than the classical Jucys-Murphy elements of the symmetric groups.
Recently, Creedon \cite{Creedon} has revisited Enyang's work,
showing that supersymmetric polynomials in a renormalization of the Jucys-Murphy elements give a family of central elements which is large
enough to separate blocks.

In this article, we give a new treatment of the representation theory of
$\Par_t$. Let
$$
Par_t := \bigoplus_{m,n \in \N} \Hom_{\Par_t}(n,m)
$$
be the path algebra of this
$\kk$-linear category, denoting the idempotents arising from the identity endomorphisms of the objects of $\Par_t$ by $\{1_n\:|\:n \in \N\}$.
Since the partition algebra $P_n(t)$
is the idempotent truncation $1_n Par_t 1_n$, most of the known results about the representation theory of the algebras $P_n(t)$ can be deduced
from that of the partition category in a standard way.
In fact, as well as producing more general results, we are convinced that it is
easier to study the representation theory of the partition
category $\Par_t$, instead of working with the tower of partition algebras.
To start with, $\Par_t$ has an efficient monoidal presentation encoding its
universal property, with generating morphisms
$\begin{tikzpicture}[anchorbase,scale=.8]
      \draw[-,thick](0,0) -- (0.55,0.6);
      \draw[-,thick](0.55,0) -- (0,0.6);
\end{tikzpicture}$ (``crossing''),
$\begin{tikzpicture}[anchorbase,scale=.8]
      \draw[-,thick](0,0) -- (0.3,0.3);
      \draw[-,thick](0.6,0) -- (0.3,0.3);
      \draw[-,thick](0.3,0.3) -- (0.3,0.6);
\end{tikzpicture}$ (``merge''),
$\begin{tikzpicture}[anchorbase,scale=.8]
      \draw[-,thick](0.3,0) -- (0.3,0.3);
      \draw[-,thick](0.6,0.6) -- (0.3,0.3);
      \draw[-,thick](0.3,0.3) -- (0,0.6);
\end{tikzpicture}$ (``split''),
$\begin{tikzpicture}[anchorbase,scale=.6]
\draw[-,thick] (0,0)to(0,0.1)to[out=up,in=up,looseness=2](0.8,0.1)to(0.8,0);
\end{tikzpicture}$ (``cap''),
$\begin{tikzpicture}[anchorbase,scale=.6]
\draw[-,thick] (0,0.5)to(0,0.4)to[out=down,in=down,looseness=2](0.8,0.4)to(0.8,0.5);
\end{tikzpicture}$ (``cup''),
$\begin{tikzpicture}[anchorbase,scale=.8]
      \draw[-,thick](0.3,0) -- (0.3,0.25);
      \node at (0.3,0.3) {$\dt$};
\end{tikzpicture}$ (``downward leaf'') and
$\begin{tikzpicture}[anchorbase,scale=.8]
      \draw[-,thick](0.3,0.6) -- (0.3,0.37);
      \node at (0.3,0.3) {$\dt$};
\end{tikzpicture}$ (``upward leaf''); see \cref{Par}.
This means that one can make calculations in $\Par_t$ using
the string calculus for strict
monoidal categories, which seems more flexible
than the traditional algebraic expressions used when working in $P_n(t)$.
But the key reason we prefer to work with $\Par_t$ is that its
path algebra has a {\em triangular decomposition} in the sense of \cite[Def.~5.31]{BS}, hence, the category $Par_t\Modlfd$ of locally finite-dimensional $Par_t$-modules
is an {\em upper finite highest weight category} as in \cite[Def.~3.34]{BS}.
The Cartan subalgebra in this triangular decomposition is the locally unital algebra
$$
\sym := \bigoplus_{n \geq 0} \kk S_n,
$$
with its irreducible modules being the Specht modules $\{S(\lambda)\:|\:\lambda \in \P\}$ indexed by the set $\P$ of all partitions.
The {\em standard modules} $\{\Delta(\lambda)\:|\:\lambda \in \P\}$
for $Par_t$ are the modules defined by parabolically  inducing the
Specht modules.
Then we obtain a full set of pairwise inequivalent irreducible $Par_t$-modules
$\{L(\lambda)\:|\:\lambda \in \P\}$ from the irreducible heads of the standard modules. This gives a quick proof of the classification of irreducible $Par_t$-modules, which was established originally by Deligne \cite{Del}
and Comes and Ostrik \cite{CO}.

The highest weight
approach to the representation theory of combinatorial monoidal categories such as
the partition category as just outlined
has been developed systematically
by Sam and Snowden \cite{SaS}. In their language, $\Par_t$ is a
{\em monoidal triangular category}. There are many
other interesting examples of this structure,
including several that are actually monoidal subcategories of $\Par_t$:
the Brauer category (cups, caps and crossings but no splits and merges),
the Temperley-Lieb category (just cups and caps), and
the category
studied by Khovanov and Sazdanovic in \cite{KS} (just leaves).
In their earlier work \cite{SSold}, Sam and Snowden had already
exposed the importance of the structure of the Borel subcategories of these
and other such categories, although at that time
they did not work out the details fully in the case of the partition category.
In \cref{sec3}, we fill this gap by giving an exposition
of some of their ideas in this case, exploiting the structure of the {\em upper partition category}, i.e., the positive Borel subcategory,
to determine the Grothendieck ring $K_0(Par_t)$ of the
category of finitely generated projective $Par_t$-modules.
In fact, as a ring, this is identified with the ring $\Lambda$ of symmetric functions, but the isomorphism classes of the standard modules $\Delta(\lambda)$
produce an interesting inhomogeneous
basis $\{\tilde s_\lambda\:|\:\lambda \in \P\}$ for
$\Lambda$ of {\em deformed Schur functions}.
These also appeared implicitly in \cite{Del, CO} and again in \cite{SSold}, and
were rediscovered from a slightly different perspective by
Orellana and Zabrocki \cite{OZ}. They are interesting because the structure constants for multiplication in $\Lambda$ with respect to this basis are the {\em reduced Kronecker coefficients}.

Although
there are many important examples of monoidal triangular categories,
and the work of Sam and Snowden has revealed many common features, this still seems to be a subject where the more intricate combinatorics needs to be studied
separately in each case. For example, one wants to understand the {\em center} of the underlying category, and the induced
decomposition of the irreducible representations into blocks which comes from considering central characters. This can be framed as a question about
an analog of {\em Harish-Chandra homomorphism} for monoidal triangular categories; see \cref{HCH}.
However, to answer it,
one needs some way to construct sufficiently many central elements, and we do not know any uniform way to approach this.
After understanding the block decomposition, the next step is to
consider the combinatorics of
{\em special
  projective functors}, which are functors on the module category induced by tensoring with generating objects of the underlying monoidal category.

The crucial new ingredient in our approach to $\Par_t$
is the definition of another graphical monoidal category, the {\em affine partition category} $\APar$.
This is obtained from the partition category by adjoining two new generating morphisms
                                    $\begin{tikzpicture}[anchorbase,scale=1.2]
\draw[-,thick](0,0)--(0,0.4);
\draw[-,thick](0,0.2)--(-0.15,0.2);
\node  at (-0.15,0.2) {$\bullet$};
\end{tikzpicture}$ (``left dot'')
and
$\begin{tikzpicture}[anchorbase,scale=1.2]
\draw[-,thick](0,0)--(0,0.4);
\draw[-,thick](0,0.2)--(0.15,0.2);
\node  at (0.15,0.2) {$\bullet$};
\end{tikzpicture}$ (``right dot'') in such a way that
$\Par_t$ can be recovered as the quotient of $\APar$ by a
 certain left tensor ideal, with the left and right dots mapping to
 renormalized versions of the Jucys-Murphy elements; see \cref{nextmain,agree}.
 However, it is not easy to do this without making additional choices.
 The actual definition of $\APar$
 given in \cref{apardef} below adopts a quite different point of view
 based on an observation due to Likeng and Savage \cite{LSR}: we construct
 $\APar$ initially as a
 monoidal subcategory of Khovanov's Heisenberg category $\Heis$ from \cite{K}.
This allows complicated relations in $\APar$ to be derived
rather quickly by elementary calculations using
the string calculus for $\Heis$; e.g., see \cref{amazinggrace}
which recovers
Enyang's five term recurrence relation for the Jucys-Murphy elements.

In the affine partition category, there is an obvious way to
construct a large family of central elements; see \cref{ourcentralelements}.
These map to central elements in $\Par_t$ which turn out to be closely related to Creedon's central elements of the partition algebras from \cite{Creedon}.
After that, we consider the self-adjoint projective functor
$$
D:Par_t\Mod \rightarrow Par_t\Mod
$$
induced by tensoring with the generating object $1$ of $\Par_t$.
This plays an analogous role in our approach to induction and restriction
along the tower of partition algebras in the work of Martin and others discussed earlier.
We use the action of the left and right dots from  $\APar$ to decompose $D$
into summands
$D = \bigoplus_{a,b \in \kk} D_{b|a}$; see \cref{keyresult}. There is
a close analogy here to the way Jucys-Murphy elements were used to give a
new approach to the representation theory of the symmetric groups in \cite{OV}.
In fact, the Jucys-Murphy elements of $Par_t$ generate a large commutative subalgebra, and the resulting ``Gelfand-Tsetlin characters'' of the standard modules
$\Delta(\lambda)$ can be computed explicitly using the branching rules from
\cref{keyresult}, although we do not pursue this further here.
Finally, we use the combinatorial properties of the special projective functors
$D_{b|a}$ to reprove the main structural result
about the representation theory of $\Par_t$ for $t \in \N$.
This was established originally by Comes and Ostrik \cite{CO}.

\vspace{2mm}
\noindent
{\bf Theorem.}
{\em When $t \in \N$, i.e., $Par_t$ is not semisimple,
  the non-simple blocks of $Par_t$
  are in bijection with isomorphism classes of irreducibles in
the semisimplification
  $\kk S_t\Modfd$.
All of the non-simple blocks are Morita equivalent. These blocks have
infinitely many isomorphism classes of irreducible modules
parametrized by $\N$, and the structure of the corresponding indecomposable
projectives is as follows:
\begin{align*}
P(0)&=\!\!
\begin{tikzpicture}[anchorbase]
\node at (0,0){{\xymatrix@C=0.9em@R=1.1em{
0\ar@{~}[d]\\
1
}}};\end{tikzpicture}\!,
&
\!P(1) &=\!\!
\begin{tikzpicture}[anchorbase]
\node at (0,0){{\xymatrix@C=0.9em@R=1.1em{
&1\ar@{-}[dl]\ar@{~}[dr]\\
0\ar@{~}[dr]&&2\ar@{-}[dl]\\
&1
}}};\end{tikzpicture}\!,
&
\!P(2) &=\!\!
\begin{tikzpicture}[anchorbase]
\node at (0,0){{\xymatrix@C=0.9em@R=1.1em{
&2\ar@{-}[dl]\ar@{~}[dr]\\
1\ar@{~}[dr]&&3\ar@{-}[dl]\\
&2
}}};\end{tikzpicture}\!,
&
\!P(3) &=\!\!
\begin{tikzpicture}[anchorbase]
\node at (0,0){{\xymatrix@C=0.9em@R=1.1em{
&3\ar@{-}[dl]\ar@{~}[dr]\\
2\ar@{~}[dr]&&4\ar@{-}[dl]\\
&3
}}};\end{tikzpicture}\!,\,
\dots
\end{align*}
}

For a more formal statement, see \cref{MAIN}.
It is a straightforward exercise to deduce from this
that each non-simple block is Morita equivalent
to the path algebra of the infinite quiver
$$
\xymatrix{
0
\ar@/^/[r]^{x_0}
&
\ar@/^/[l]^{y_0}{1}
\ar@/^/[r]^{x_1}
&
\ar@/^/[r]^{x_2}
\ar@/^/[l]^{y_1}{2}
&
\ar@/^/[l]^{y_2}
3\cdots}
\quad
\!\!\!\!\!\text{with relations }y_0 x_0 =0, x_{i+1}x_i = y_i y_{i+1} = x_i y_i - y_{i+1}
x_{i+1}=0.
$$
This quiver
is well known in representation theory, for example,
it also describes the non-trivial block of the Temperley-Lieb category, as noted in \cite[Rem.~6.5]{CO}.

It is interesting to compare the general strategy
developed here with the original arguments of Comes and Ostrik. There are many parallels. For example, they also construct a large family of central elements, although different from ours, and they also use summands of the functor $D$ to construct equivalences between blocks; see \cref{omegas} and \cref{blockreduction}.
Another technique which is crucial in \cite{CO}
is the idea of lifting projectives
to the (semisimple) generic partition category. In our approach, this
is replaced everywhere with arguments involving standard modules
and the BGG reciprocity coming from the highest weight structure. In fact,
largely due to the fact that they did not think in these module-theoreic terms, Comes and Ostrik were forced in the end to refer to some of Martin's results from \cite{Martin} to obtain the precise submodule structure of projectives in the above theorem, whereas our proof is independent of {\em loc. cit.}, indeed, Martin's results can now be deduced from here.
One more key idea used by Comes and Ostrik involves
an explicit formula for categorical dimensions derived ultimately from the hook
formula, although we have avoided such considerations entirely
by exploiting the functors $D_{b|a}$ for $a=b$.
The definition of these diagonal components of $D$ cannot be formulated without using
Jucys-Murphy elements, so
no counterpart for this part of our argument appears in \cite{CO}.

\vspace{2mm}
\noindent
    {\em Acknowledgements.}
    We thank Alistair Savage for discussions which influenced
     the final form of the definition of the affine partition category
    given in \cref{apardef}.

\section{Monoidal categories and representations}\label{sec2}

In the opening section, we explain our general conventions for
representations of $\kk$-linear (monoidal) categories.
Always in this article $\kk$ will be an algebraically closed field of characteristic zero, although all of the generalities recorded
make sense more generally. Then we briefly recall some
classical results about $\Sym$, the free strict $\kk$-linear symmetric monoidal category on one object, which categorifies the ring of symmetric functions.

\subsection{Path algebras and modules}
Let $\cA$ be a $\kk$-linear category.
Its {\em path algebra} is
the associative algebra
$$
A := \bigoplus_{X,Y \in \ob \cA} \Hom_{\cA}(X,Y)
$$
with multiplication induced by composition in $\cA$, so that $g f = g \circ f$
for $f:X \rightarrow Y,\ g: Y \rightarrow Z$.
Note that $A$ is not necessarily unital, but it is always 
a {\em locally unital algebra}, i.e., there is a distinguished family
$\{1_X\:|\:X \in \OO_A\}$ of mutually orthogonal idempotents such that
$A = \bigoplus_{X,Y \in \OO_A} 1_Y A 1_X$. In this case, $\OO_A$ is
the object set $\ob\cA$ of the
category $\cA$, with $1_X$ being the identity endomorphism of $X$.
If $\cA$ is a {\em finite-dimensional category}, i.e., its morphism spaces
are finite-dimensional, then the path algebra is {\em locally finite-dimensional}
in the sense that $\dim 1_Y A 1_X < \infty$
for all $X, Y \in \OO_A$.

The category $A\Mod$ of left $A$-modules is the category
$\HOM_\kk(\cA, \Vec)$ of $\kk$-linear functors from $\cA$ to the category $\Vec$ of vector spaces, morphisms being natural transformations.
Equivalently, using the language we systematically adopt below,
a left $A$-module $V$ is a left module in the usual sense of associative algebras such that $V = \bigoplus_{X \in \OO_A} 1_X V$;
this corresponds to the $\kk$-linear functor $V:\cA \rightarrow \Vec$
taking object $X$ to the vector space $V(X) = 1_X V$
and morphism $f \in \Hom_\cA(X,Y)$ to the linear map
$V(f): 1_X V \rightarrow 1_Y V$ defined by left multiplication by $f \in 1_Y A 1_X$.
There is also the category $\doM A$ of right $A$-modules, which is just the same as the
category $\HOM_\kk(\cA^\op, \Vec)$.

We say that $V \in A\Mod$ 
is {\em locally finite-dimensional} if $\dim 1_X V < \infty$
for all $X \in \OO_A$; equivalently, the associated functor goes from 
$\cA$ to the category $\Vecfd$ of finite-dimensional vector spaces.
Let $A\Modlfd$ be the full subcategory of $A\Mod$ consisting of the locally finite-dimensional $A$-modules. 
For more background material about the structure of the category $A\Modlfd$ 
in the case that $A$ is locally finite-dimensional, we refer to
\cite[$\S$2.2--$\S$2.3]{BS}, where Abelian categories of this form are called
{\em Schurian categories}.

We also let $A\Modfd$ be the full subcategory of $A\Mod$ 
consisting of the
globally finite-dimensional modules, i.e., the $V$ with $\dim V < \infty$,
and $A\Proj$ be the full subcategory of $A\Mod$ consisting of the finitely generated projective modules. If $A$ is a locally finite-dimensional locally unital algebra 
then $A 1_X$ is a locally finite-dimensional module for each $X \in \OO_A$, 
hence, $A\Proj$ is a subcategory of $A\Modlfd$.
The category $A\Proj$ can also be obtained in equivalent form directly from the 
$\kk$-linear category $\cA$ since
the Yoneda embedding $h^*:\cA \rightarrow \HOM_\kk(\cA, \Vec)$
induces a contravariant $\kk$-linear equivalence between
$\Kar(\cA)$ and
$A\Proj$. Here, $\Kar(\cA)$ denotes the
{\em additive Karoubi envelope} of $\cA$, that is, the idempotent completion of its additive envelope $\Add(\cA)$.

We let $K_0(A)$ be the split Grothendieck group of the category $A\Proj$. Assuming that $A$ is
locally finite-dimensional, every finitely generated module
has a projective cover in $A\Proj$. Moreover,
$K_0(A)$ is a free Abelian group with
     canonical
     basis coming from the projective covers of the irreducible $A$-modules.

\subsection{Pull-back and push-forward}
Suppose that $\cA$ and $\cB$ are two $\kk$-linear categories.
Let 
$$
A = \bigoplus_{X,Y \in \OO_A} 1_Y A 1_X,
\qquad
B=\bigoplus_{X,Y \in \OO_B} 1_Y B 1_X
$$ 
be their path algebras.
To a $\kk$-linear functor $F:\cA \rightarrow \cB$, 
we associate an exact
functor 
\begin{equation}\label{Frestriction}
\res_F:
B\Mod \rightarrow A\Mod 
\end{equation}
which we call {\em restriction along $F$}.
It is just the functor $\HOM_\kk(\cB,\Vec) \rightarrow \HOM_\kk(\cA,\Vec)$
defined by composing on the right with $F$.
In elementary terms, and introducing a shorthand which will be ubiquitous later on, the functor
$\res_F$ takes $V \in B\Mod$ to 
\begin{equation}\label{Frestrictionleft}
1_{F} V := \bigoplus_{X \in \OO_A} 1_{FX} V \in A\Mod,
\end{equation}
with the left module structure defined so that $f\in 1_Y A 1_X$ acts on the 
$X$-th summand $1_{FX} V$ as
the linear map $Ff:1_{FX} V \rightarrow 1_{FY} V$, and it acts as zero on all other summands. 
It takes a $B$-module homomorphism $\phi:V \rightarrow W$ to the
$A$-module homomorphism $\res_F(\phi):1_{F} V \rightarrow 1_{F} W$
defined by $\phi_{FX}:1_{FX} V \rightarrow 1_{FX} W$
for each $X \in \OO_A$.
Similarly, there is the exact restriction functor we denote by
\begin{equation}\label{Frestrictionop}
  {_F\!}\res:\doM
  B \rightarrow \doM A
  \end{equation}
between the categories of right modules taking $V \in \doM B$ to 
\begin{equation}\label{Frestrictionright}
V 1_{F} := \bigoplus_{X \in \OO_A} V 1_{FX} \in \doM A.
\end{equation}
The functors $\res_F$ and ${_F\!}\res$ may also be denoted
$F^*$ and $(F^\op)^*$; e.g., see
\cite[(2.1.4)]{SSold}, \cite[$\S$3.6]{SaS}.

The restriction $B 1_{F} =
\bigoplus_{X \in \OO_A} B 1_{FX}$
is a $(B,A)$-bimodule.
The functor $\res_F:B\Mod\rightarrow A\Mod$ is isomorphic to 
$\bigoplus_{X \in \OO_A} \Hom_B(B 1_{FX},?)$. Then adjointness of tensor and hom in the locally unital setting (e.g., see \cite[Lem.~2.7]{BS}) implies that the functor
\begin{equation}\label{Finduction}
\ind_F := B 1_{F} \otimes_A :A\Mod \rightarrow B\Mod
\end{equation}
is left adjoint to $\res_F$. We call this {\em induction along
  $F$}.
Since it is left adjoint to an exact functor, $\ind_F$ is right exact and takes projectives to
projectives. In fact, we have that
\begin{equation}\label{indonprojectives}
\ind_F \;A 1_X = B 1_F \otimes_A A 1_X \cong B 1_{FX},
\end{equation}
i.e., $\ind_F$ can be viewed as an extension of $F$ to
arbitrary modules.
From this, it is clear that $\ind_F$ preserves finite generation.
Likewise, the restriction $1_F B=\bigoplus_{X\in\OO_A} 1_{FX} B$ is an $(A,B)$-bimodule.
The functor $\res_F$ is also isomorphic to $1_{F} B \otimes_B $, hence, it has a right adjoint given by the functor
\begin{equation}\label{Fcoinduction}
\coind_F := \bigoplus_{Y \in \OO_B} \Hom_A(1_{F} B 1_Y, ?):A\Mod\rightarrow B\Mod.
\end{equation}
We call this {\em coinduction along $F$}.
Since it is right adjoint to an exact functor, $\coind_F$ is left
exact and takes injectives to injectives.
The functors $\ind_F$ and $\coind_F$ are also called
left and right Kan extensions and may be denoted $F_!$ and $F_*$, respectively;
e.g., see \cite[(2.1.4)]{SSold}, \cite[$\S$3.6]{SaS}. 
There are also analogs of $\ind_F$ and $\coind_F$ with
left modules replaced by right modules, which we denote by ${_F\!}\ind$ and
${_F\!}\coind$; in \cite{SaS}, these are denoted $(F^\op)_!$ and $(F^\op)_*$.

\begin{lemma}\label{cornell}
  Let $F:\cA \rightarrow \cB$ be a $\kk$-linear functor as above.
  \begin{itemize}
  \item[(1)] If $B 1_F$ is a projective right $A$-module then $\ind_F$ and ${_F\!}\coind$ are exact functors.
  \item[(2)] If $1_F B$ is a projective left $A$-module then ${_F\!}\ind$ and $\coind_F$ are exact functors.
    \end{itemize}
\end{lemma}

\begin{proof}
  This is obvious from the definitions of these functors.
\end{proof}

Suppose that $F, G:\cA \rightarrow \cB$ are $\kk$-linear functors.
A natural transformation $\alpha:F\Rightarrow G$
induces natural
transformations $\res_\alpha:\res_F \Rightarrow \res_G$,
                                           $\ind_\alpha:\ind_G\Rightarrow\ind_F$ and
                                           $\coind_\alpha:\coind_G\Rightarrow\coind_F$.
We  leave the detailed definitions of these to the reader, just noting that
$\ind_\alpha$ and $\coind_\alpha$ are the left and right mates of $\res_\alpha$.
Similarly, $\alpha$ induces natural transformations
  ${_\alpha\!}\res:{_G\!}\res \Rightarrow {_F\!}\res$,
                                           ${_\alpha\!}\ind:{_F\!}\ind
  \Rightarrow{_G\!}\ind$ and ${_\alpha\!}\coind:{_F\!}\coind \Rightarrow{_G\!}\coind$.
  Assuming for simplicity\footnote{To formulate analogs of
  \cref{parmesan,parmesan2} without this assumption, one needs to
  work in the strict $2$-category of $\kk$-linear categories.}
   that $\cA = \cB$, so that $F$ and $G$ are $\kk$-linear endofunctors of $\cA$, these constructions define $\kk$-linear monoidal  functors
\begin{align}\label{parmesan}
  \res_*&: \END_{\kk}(\cA)
  \rightarrow \END_{\kk}(A\Mod)^{\rev},
  &\ind_*, \coind_*&:\END_{\kk}(\cA)^{\op} \rightarrow
\END_{\kk}(A\Mod),\\
{_*\!}\res&: \END_{\kk}(\cA)^{\op} \rightarrow \END_{\kk}(\doM A)^{\rev},&
{_*\!}\ind, {_*\!}\coind&:\END_{\kk}(\cA) \rightarrow
            \END_{\kk}(\doM A).\label{parmesan2}
\end{align}
Here, $\END_{\kk}(\cA)$ denotes the strict monoidal $\kk$-linear
  category of $\kk$-linear endofunctors and natural transformations,
  ``op" means the opposite category with the same monoidal product,
  and ``rev" means the same category with the reversed monoidal product.

  \subsection{Duality}
  Continue with $A$ and $B$ be the path algebras of $\cA$ and $\cB$, respectively.
   There is a contravariant functor
\begin{equation}\label{weirdo1}
?^\dual:
A\Mod \rightarrow \doM A
\end{equation}
taking $V = \bigoplus_{X \in \OO_A} 1_X V$ to $V^\dual :=
\bigoplus_{X \in \OO_A} (1_X V)^*$, the direct sum of the linear
duals of the ``weight spaces'' $1_X V$.
The restriction of this to locally finite-dimensional modules is an
equivalence, with quasi-inverse given by the restriction of the
analogously-defined duality functor \begin{equation}\label{weirdo2}
  {^\dual}?:\doM A \rightarrow A\Mod
\end{equation}
in the other direction.
To obtain a duality ($=$ contravariant auto-equivalence) on
$A\Modlfd$ from \cref{weirdo1,weirdo2}, one also needs a $\kk$-linear equivalence $\sigma:\cA
\rightarrow \cA^\op$. Restriction along $\sigma$ gives equivalences
$\res_\sigma:\dfldoM A \rightarrow A\Modlfd$ and
${_\sigma\!}\res:A \Modlfd \rightarrow \dfldoM A$,
hence, we obtain the
duality functor
\begin{equation}\label{duality}
?^\sigmadual := \res_\sigma \circ  ?^\dual = {^\dual}? \circ {_\sigma\!}\res
:A\Modlfd \rightarrow A\Modlfd.
\end{equation}

Given a $\kk$-linear functor $F:\cA \rightarrow \cB$, we obviously
have that
\begin{align}\label{predual}
?^\dual \circ \res_F &\cong {_F\!}\res \circ ?^\dual
\end{align}
as functors from $B\Mod$ to $\doM A$. 
We deduce that
\begin{align}\label{almostwhatIneed}
  {^\dual}? \circ {_F\!}\ind &\cong \coind_F \circ {^\dual}?,&
  {^\dual}? \circ {_F\!}\coind &\cong \ind_F \circ {^\dual}?
\end{align}
as functors from $\doM A$ to $B\Mod$.

\subsection{Induction product}\label{ip}
The $\kk$-linear categories of interest later on will usually
have some additional monoidal structure. In fact, they will be strict $\kk$-linear monoidal categories defined by generators and relations.
We use the symbol $\star$ for the monoidal product in such categories, reserving $\otimes$ for the tensor product $\otimes_\kk$ of vector spaces over $\kk$.
We adopt the usual string calculus for morphisms in strict monoidal
categories, our convention being that $f \circ g$, the composition of
$f$ and $g$, is drawn as $f$ on top of $g$ and $f \star g$, the monoidal
product of $f$ and $g$, is drawn as $f$ to the left of $g$.

Let $\cC$ be a strict $\kk$-linear monoidal category with 
path algebra $C = \bigoplus_{X, Y \in \OO_C} 1_Y C 1_X$.
The monoidal product $\star$ on $\cC$ extends canonically 
to $\Kar(\cC)$. There is also a monoidal product $\ostar$ 
making $C\Proj$ into a (no longer strict) $\kk$-linear monoidal category such that the 
contravariant Yoneda equivalence from $\Kar(\cC)$ to $C\Proj$ is 
monoidal. This functor
is the restriction of a tensor product functor 
$\ostar$ on the Abelian category $C\Mod$.
We call this the {\em induction product}.
Category theorists refer to this instead
as {\em Day convolution} and define it via the coend expression:
$$
V_1 \ostar V_2 = \int^{X_1,X_2 \in \cC} \Hom_{\cC}(X_1
\star X_2,?) \otimes V_1(X_1) \otimes V_2(X_2).
$$
We give the algebraist's formulation of the definition
in the next paragraph; see also \cite[(2.1.14)]{SSold}, \cite[$\S$3.10]{SaS}.
Using $\ostar$, we can make the split Grothendieck group $K_0(C)$ into
a ring with
multiplication \begin{equation}\label{backinhideaway}
  [P][Q] := [P \ostar
       Q].\end{equation}
     Its identity element is the isomorphism class of
     the distinguished projective module $C 1_\one$, where $\one\in\OO_C$ is
     the unit object.

Here is the detailed definition of $\ostar$. Let $\cC \boxtimes \cC$ be the $\kk$-linearization of the Cartesian product
$\cC \times \cC$. The objects in $\cC \boxtimes \cC$ are pairs $(X_1,X_2)
\in \OO_C \times \OO_C$, and the morphism space from $(X_1,X_2)$ to $(Y_1,Y_2)$
is $\Hom_\cC(X_1,Y_1) \otimes \Hom_{\cC}(X_2,Y_2)$.
We denote its path algebra by 
$$
C \boxtimes C = \bigoplus_{X_1,X_2,Y_1,Y_2 \in \OO_C}
1_{Y_1} C 1_{X_1} \otimes 1_{Y_2} C 1_{X_2}.
$$
Multiplication in $C \boxtimes C$ is the obvious ``tensor-wise" product just like for a tensor product of algebras.
If $C$ is locally finite-dimensional, so too is $C \boxtimes C$.
Given $V_1,V_2 \in C\Mod$, let $$
V_1 \boxtimes V_2
= \bigoplus_{X_1,X_2 \in \OO_C} 1_{X_1} V_1 \otimes 1_{X_2} V_2
$$
be their tensor product over $\kk$
viewed as a left $C\boxtimes C$-module in the obvious way.
In fact, this defines a functor 
$\boxtimes: C\Mod \boxtimes C\Mod \rightarrow C \boxtimes C \Mod$.
The monoidal product on $\cC$ is a $\kk$-linear functor 
$\star:\cC \boxtimes \cC\rightarrow \cC$.
Let $$
C 1_\star
= \bigoplus_{X_1,X_2 \in \OO_C} C 1_{X_1 \star X_2}
$$
be the $(C,C\boxtimes C)$-bimodule obtained by restricting the right $C$-module $C$
along this functor. 
Induction along $\star$, that is, the  
functor $\ind_\star = C 1_\star \otimes_{C \boxtimes C} :C\boxtimes
C\Mod \rightarrow C\Mod$ from \cref{Finduction}, is left adjoint to the
restriction functor $\res_\star$ from \cref{Frestriction}. 
Then the induction product is the composition
\begin{equation}\label{ourway}
\ostar := \ind_\star \circ\; \boxtimes:
C\Mod \boxtimes C\Mod \rightarrow C\Mod.
\end{equation}
Thus, for $V_1, V_2 \in C\Mod$, we have that
$V_1 \ostar V_2 = C 1_\star \otimes_{C\boxtimes C} (V_1 \boxtimes V_2)$.
Associativity of $\ostar$
(up to natural isomorphism) follows from ``transitivity of induction",
i.e., the associativity of tensor products of modules over locally
unital algebras.
We obviously have that
\begin{equation}\label{onprojectives}
C 1_{X_1} \ostar C 1_{X_2} \cong C 1_{X_1\star X_2}
\end{equation}
for $X_1,X_2 \in \OO_C$.
This justifies our earlier assertion that $\ostar$ extends the monoidal
product $\star$ on
$\Kar(\cC)$. It also follows that $V_1 \ostar V_2$ is finitely generated if
both $V_1$ and $V_2$ are finitely generated.

The induction product $\ostar$ is right exact in both arguments, 
but in general it is not left exact.
We denote the $i$th left derived functor of $\ostar$ on $C$-modules
$V, W$ by $\Tor_i^{C}(V,W)$. This can be computed from a projective
resolution of either $V$ or $W$.

\begin{lemma}\label{indprodex}
  If $C 1_\star$ is a projective right $C \boxtimes C$-module then
  the induction product $\ostar$ is biexact.
\end{lemma}

\begin{proof}
  This follows from \cref{cornell}.
\end{proof}

Finally suppose that $\cC$ is a strict $\kk$-linear {\em symmetric} monoidal category, so that there is given a symmetric braiding $R:\star \stackrel{\sim}{\Rightarrow} \star^\rev$.
From this, we obtain a braiding $\ind_R \circ \boxtimes:\ostar^\rev \stackrel{\sim}{\Rightarrow} \ostar$ making $C\Mod$ into a $\kk$-linear
symmetric monoidal category too.

\begin{remark}
  There is a second convolution product $\costar$ which we call the {\em coinduction product}. This is defined by
replacing $\ind_\star$ with $\coind_\star$ in \cref{ourway}. It is easy to understand on
injective rather than projective modules. It will not often be used
subsequently, but note that
the induction and coinduction products are interchanged by duality.
\end{remark}

\subsection{Projective functors}\label{secpf}
Suppose that $\cC$ is a strict $\kk$-linear
monoidal category and $\cA$ is a $\kk$-linear category, denoting their
path algebras by $C$ and $A$ as usual.
We say that $\cA$ is a {\em strict $\cC$-module category} if there is
a strictly associative and unital $\kk$-linear monoidal functor
$\star:\cC \boxtimes \cA \rightarrow \cA$. Equivalently, this is the
data of a strict $\kk$-linear monoidal functor
$\Psi:\cC \rightarrow \END_{\kk}(\cA)$.
For $f \in \Hom_{\cC}(X,X')$, we sometimes denote the
evaluation of the natural transformation
$\Psi(f)$ on $Y \in \OO_A$ simply by $f_Y:X\star Y \rightarrow X'
\star Y$.

The definition of the induction product $\star$ from \cref{ourway}
extends naturally to this setting, thereby defining a $\kk$-linear functor
\begin{equation}\label{ourway2}
\ostar:= \ind_\star \;\circ\; \boxtimes :C\Mod \boxtimes A\Mod
\rightarrow A\Mod
\end{equation}
which makes $A\Mod$ into a (no longer strict) $C\Mod$-module category.
For objects $X \in \OO_C$ and $Y \in \OO_A$, we have
that \begin{equation}\label{backinhideaway2}
  C 1_X \ostar A 1_Y \cong A 1_{X\star Y},
\end{equation}
i.e.,
$\ostar$ extends $\star:\cC \boxtimes \cA \rightarrow
\cA$. Using $\ostar$ to define the action as in
\cref{backinhideaway}, the split Grothendieck group $K_0(A)$ becomes a left
module over the split Grothendieck ring $K_0(C)$.

Now fix $X \in \OO_C$ and consider the functor $X\star:\cA \rightarrow
\cA$.
There is an adjoint pair of endofunctors
$(\ind_{X\star}, \res_{X\star})$ of $A\Mod$ defined by
induction and restriction along $X\star$:
\begin{align}\label{if1}
\ind_{X\star} := A 1_{X\star} \otimes_A \qquad\text{where}\qquad
A 1_{X\star}:= \bigoplus_{Y \in \OO_A} A 1_{X\star Y},\\
\res_{X\star} := 1_{X\star} A \otimes_A \qquad\text{where}\qquad
1_{X\star} A := \bigoplus_{Y \in \OO_A} 1_{X\star Y} A.\label{if2}
\end{align}
The general properties discussed earlier give that $\res_{X\star}$ is
exact, and $\ind_{X\star}$ is right exact and sends 
(finitely generated) projectives to (finitely generated) projectives.
Thus, $\ind_{X\star}$ restricts to a well-defined functor
$\ind_{X\star}:A\Proj\rightarrow A \Proj$.
Note also that
\begin{equation}
\ind_{X\star} (A 1_Y) \cong A 1_{X\star Y}
\end{equation}
for all $Y \in \OO_A$.
One can also interpret $\ind_{X\star}$ as a special induction product, thanks to the following lemma.

\begin{lemma}\label{dillylilly}
For any $X \in \OO_C$, we have that $\ind_{X\star} \cong C 1_X \ostar $.
\end{lemma}

\begin{proof}
This follows from the chain of isomorphisms
\begin{align*}
A 1_{X\star} \otimes_A V &\cong
(A 1_\star \otimes_{C \boxtimes A} (C 1_X \boxtimes A))
\otimes_A V\\
&\cong
A 1_\star \otimes_{C \boxtimes A} ((C 1_X \boxtimes A)
\otimes_A V)\\
&\cong
A 1_\star \otimes_{C \boxtimes A} (C 1_X \boxtimes V)
= C 1_X \ostar V
\end{align*}
for $V\in A\Mod$.
\end{proof}

Let $X, Y$ be objects of $\cC$.
Recall that $Y$ is a {\em left dual} of $X$ (equivalently, $X$ is a {\em right dual} of $Y$) if there are 
evaluation and coevaluation morphisms
$\ev:Y\star X \rightarrow \one$ and $\coev:\one \rightarrow X \star Y$
satisfying the zig-zag identities.
In string diagrams, we denote $\ev$ and $\coev$ by the cap
$\begin{tikzpicture}[anchorbase,scale=.6]
\draw[-] (0,0)to(0,0.1)to[out=up,in=up,looseness=2](0.8,0.1)to(0.8,0);
      \node at (0,-0.1) {$\stringlabel{Y}$};
      \node at (0.8,-0.1) {$\stringlabel{X}$};
\end{tikzpicture}$
and the cup
$\begin{tikzpicture}[anchorbase,scale=.6]
\draw[-] (0,0.8)to(0,0.7)to[out=down,in=down,looseness=2](0.8,0.7)to(0.8,0.8);
      \node at (0,.9) {$\stringlabel{X}$};
      \node at (0.8,.9) {$\stringlabel{Y}$};
\end{tikzpicture}$,
respectively, so that the zig-zag identities become
\begin{equation}\label{zigzag}
\begin{tikzpicture}[anchorbase]
\draw[-] (0,-.7)to(0,0.1)to[out=up,in=up,looseness=2](0.6,0.1)to(0.6,0);
\draw[-] (.6,0)to(.6,-0.1)to[out=down,in=down,looseness=2](1.2,-.1)to(1.2,.7);
      \node at (0,-.78) {$\stringlabel{Y}$};
\end{tikzpicture}\ =\ 
\begin{tikzpicture}[anchorbase]
\draw[-] (0,-.7)to(0,0.7);
      \node at (0,-.78) {$\stringlabel{Y}$};
\end{tikzpicture}\ ,
\qquad
\begin{tikzpicture}[anchorbase]
\draw[-] (0,.7)to(0,-0.1)to[out=down,in=down,looseness=2](0.6,-0.1)to(0.6,0);
\draw[-] (.6,0)to(.6,0.1)to[out=up,in=up,looseness=2](1.2,.1)to(1.2,-.7);
      \node at (1.2,-.78) {$\stringlabel{X}$};
\end{tikzpicture}\ =\ 
\begin{tikzpicture}[anchorbase]
\draw[-] (0,-.7)to(0,0.7);
      \node at (0,-.78) {$\stringlabel{X}$};
\end{tikzpicture}\ .
\end{equation}

\begin{lemma}\label{dualitylem}
If $X$ has a left dual $Y$ in $\cC$
then there is an isomorphism
$\phi:1_{X\star} A \stackrel{\sim}{\rightarrow} A 1_{Y\star}$ of $(A,A)$-bimodules given
explicitly by
\begin{equation}\label{dualitypic}
\varphi\left(\begin{tikzpicture}[anchorbase,scale=1.5]
\draw[-] (-0.5,-0.2)--(-0.5,0.2)--(0.5,0.2)--(0.5,-0.2)--(-0.5,-0.2);
\node at (-0,-0) {$\scriptstyle{f}$};
\draw[-] (-0.4,-0.2)--(-0.4,-0.5);
\node at (-.47,.58) {$\stringlabel{X}$};
\draw[-] (0.4,-0.2)--(0.4,-0.5);
\draw[-] (-0.4,0.2)--(-0.4,0.5);
\draw[-] (-0.3,0.2)--(-0.3,0.5);
\draw[-] (0.4,0.2)--(0.4,0.5);
\node at (-0.1,0.35) {$\cdot$};
\node at (0.05,0.35) {$\cdot$};
\node at (0.2,0.35) {$\cdot$};
\node at (-0.2,-0.4) {$\cdot$};
\node at (-0,-0.4) {$\cdot$};
\node at (0.2,-0.4) {$\cdot$}; 
\end{tikzpicture}\right) = \begin{tikzpicture}[anchorbase,scale=1.5]
\draw[-] (-0.5,-0.2)--(-0.5,0.2)--(0.5,0.2)--(0.5,-0.2)--(-0.5,-0.2);
\node at (-0,-0) {$\scriptstyle{f}$};
\draw[-] (-0.4,0.2)to[out=up,in=right](-0.6,0.5)to[out=left,in=up](-0.8,0.3)to(-0.8,-0.5);
\draw[-] (-0.4,-0.2)--(-0.4,-0.5);
\node at (-.8,-.58) {$\stringlabel{Y}$};
\draw[-] (0.4,-0.2)--(0.4,-0.5);
\draw[-] (-0.3,0.2)--(-0.3,0.5);
\draw[-] (0.4,0.2)--(0.4,0.5);
\node at (-0.1,0.35) {$\cdot$};
\node at (0.05,0.35) {$\cdot$};
\node at (0.2,0.35) {$\cdot$};
\node at (-0.2,-0.4) {$\cdot$};
\node at (-0,-0.4) {$\cdot$};
\node at (0.2,-0.4) {$\cdot$}; 
\end{tikzpicture}
\end{equation}
Hence, the functors
$\res_{X\star}$ and $\ind_{Y\star}$ are isomorphic.
\end{lemma}

\begin{proof}
  It is easily checked that $\phi$ is a bimodule homomorphism.
  It is an isomorphism because it has a two-sided inverse $\psi$ defined by
  \begin{align*}
\psi\left(\begin{tikzpicture}[anchorbase,scale=1.5]
\draw[-] (-0.5,-0.2)--(-0.5,0.2)--(0.5,0.2)--(0.5,-0.2)--(-0.5,-0.2);
\node at (-0,-0) {$\scriptstyle{g}$};
\draw[-] (-0.4,0.2)--(-0.4,0.5);
\draw[-] (0.4,0.2)--(0.4,0.5);
\draw[-] (-0.4,-0.2)--(-0.4,-0.5);
\draw[-] (-0.3,-0.2)--(-0.3,-0.5);
\draw[-] (0.4,-0.2)--(0.4,-0.5);
\node at (-0.1,-0.4) {$\cdot$};
\node at (0.05,-0.4) {$\cdot$};
\node at (0.2,-0.4) {$\cdot$};
\node at (-0.2,0.35) {$\cdot$};
\node at (-0,0.35) {$\cdot$};
\node at (0.2,0.35) {$\cdot$}; 
\node at (-.4,-.58) {$\stringlabel{Y}$};
\end{tikzpicture}\right) = \begin{tikzpicture}[anchorbase,scale=1.5]
\draw[-] (-0.5,-0.2)--(-0.5,0.2)--(0.5,0.2)--(0.5,-0.2)--(-0.5,-0.2);
\node at (-0,-0) {$\scriptstyle{g}$};
\draw[-] (-0.4,-0.2)to[out=down,in=right](-0.6,-0.5)to[out=left,in=down](-0.8,-0.3)to(-0.8,0.5);
\draw[-] (-0.4,0.2)--(-0.4,0.5);
\draw[-] (0.4,0.2)--(0.4,0.5);
\node at (-.8,.58) {$\stringlabel{X}$};
\draw[-] (-0.3,-0.2)--(-0.3,-0.5);
\draw[-] (0.4,-0.2)--(0.4,-0.5);
\node at (-0.1,-0.4) {$\cdot$};
\node at (0.05,-0.4) {$\cdot$};
\node at (0.2,-0.4) {$\cdot$};
\node at (-0.2,0.35) {$\cdot$};
\node at (-0,0.35) {$\cdot$};
\node at (0.2,0.35) {$\cdot$}; 
\end{tikzpicture}\: .
\end{align*}
\end{proof}

\begin{corollary}\label{dualitycor}
If $X$ has a left dual $Y$ in $\cC$ then $(\ind_{X\star}, \ind_{Y\star})$ and $(\res_{X\star},\res_{Y\star})$
 are adjoint pairs of functors.
 \end{corollary}

From the corollary, we deduce that if $X$ is {\em rigid}, i.e., it has both a left and a right dual, then both of the functors
$\ind_{X\star}$ and $\res_{X\star}$ have both a right and a left adjoint. 
Moreover, as discussed earlier, both of these functors are exact and
they preserve finitely generated projectives.
We will refer to finite direct sums of direct summands of
endofunctors of $A\Mod$ of this sort as {\em projective functors}.

\subsection{The symmetric category}\label{ssym}
For a basic example,
we have the {\em symmetric category} $\Sym$, which is the free strict $\kk$-linear symmetric monoidal category on one object. In string diagrams, we denote this 
generating 
object simply by $\mid$; then an arbitrary object is the monoidal
product $\mid^{\star n}$ for some $n \geq 0$.
Morphisms in $\Sym$ are generated by
a single morphism depicted by the crossing
\begin{equation}
\begin{tikzpicture}[anchorbase]
      \draw[-,thick](0,0) -- (0.4,0.6);
      \draw[-,thick](0.4,0) -- (0,0.6);
\end{tikzpicture}:\mid\star\mid\to\mid\star\mid
\end{equation}
subject to the relations
\begin{align}
  \begin{tikzpicture}[anchorbase]
    \draw[-,thick] (0,0) \braidto (0.4,0.4) \braidto (0,0.8);
    \draw[-,thick] (0.4,0) \braidto (0,0.4) \braidto (0.4,0.8);
  \end{tikzpicture}\ &=\ \begin{tikzpicture}[anchorbase]
    \draw[-,thick] (0,0) -- (0,0.8);
    \draw[-,thick] (0.3,0) -- (0.3,0.8);
  \end{tikzpicture}\ ,
  & \begin{tikzpicture}[anchorbase]
  \draw[-,thick] (0,0)--(0.6,0.8);
  \draw[-,thick] (0.6,0)--(0,0.8);
  \draw[-,thick] (0.3,0) to[out=120,in=-120,looseness=1.5] (0.3,0.8);
  \end{tikzpicture}\ &= \begin{tikzpicture}[anchorbase]
  \draw[-,thick] (0,0)--(0.6,0.8);
  \draw[-,thick] (0.6,0)--(0,0.8);
  \draw[-,thick] (0.3,0)to[out=60,in=-60,looseness=1.5](0.3,0.8);
\end{tikzpicture}\ .\label{relssym0}
                       \end{align}
Sometimes it is convenient to identify objects in $\Sym$ with natural numbers, so that the object set $\{\mid^{\star n}\:|\:n \in \N\}$
of $\Sym$ is identified with $\N$.
 For $m,n \geq 0$, the morphism space $\Hom_\Sym(n,m)$
is $\{0\}$ if $m \neq n$, while if $m = n$ it consists of $\kk$-linear
combinations of string diagrams representing permutations in the
symmetric group $S_n$, i.e., we have that $\End_{\Sym}(n)
= \kk S_n.$
Note our general convention here is to number strings by $1,\dots,n$
from {\em right to left}, so that the transposition $(1\:2) \in S_n$
is represented by the string diagram
$$
\begin{tikzpicture}[anchorbase,scale=1.2]
  \draw[-,thick] (-.4,0) to (-.4,.8);
  \draw[-,thick] (.5,0) to (.5,.8);
  \draw[-,thick](.8,0) to (1.1,.8);
  \draw[-,thick](1.1,0) to (.8,.8);
\node at (0.1,.4) {$\cdots$};
\node at (1.1,-.1) {$\stringlabel{1}$};
\node at (0.5,-.1) {$\stringlabel{3}$};
\node at (0.8,-.1) {$\stringlabel{2}$};
\node at (-0.4,-.1) {$\stringlabel{n}$};
  \end{tikzpicture}\ .
  $$

Let $\sym$ be the path algebra of $\Sym$. 
Thus, we have that
\begin{equation}\label{semisimpledec}
\sym = \bigoplus_{n \geq 0} \kk S_n.
\end{equation}
Since $\kk$ is of characteristic zero, we deduce from Maschke's theorem that
$\sym$ is a {\em semisimple} locally unital algebra.
In this case, the induction product $\ostar$ making $\sym\Mod$ into a monoidal category 
is nothing more than 
the usual induction product 
on representations of the symmetric groups:
we have that
$$
V \ostar W = 
\ind_{S_n \times S_m}^{S_{n+m}} (V \boxtimes W)
$$
for $V \in \kk S_n\Mod$ and $W \in \kk S_m\Mod$.
In fact, the induction product $\ostar$ and the
coinduction product $\costar$ on $\sym\Mod$ are isomorphic as $\ind_{S_n
  \times S_m}^{S_{n+m}} \cong \coind_{S_n \times S_m}^{S_{n+m}}$ (as
always for finite groups).

Recall that the irreducible $\kk S_n$-modules are the
{\em Specht modules} $S(\lambda)$ parametrized by the set $\P_n$
of partitions $\lambda = (\lambda_1,\lambda_2,\dots)$ of $n$.
Hence, the irreducible $\sym$-modules are
the Specht modules $S(\lambda)$ parametrized by all partitions
$\lambda \in \P = \bigsqcup_{n \geq 0} \P_n$.
We sometimes write $|\lambda|$ for the size $\lambda_1+\lambda_2+\cdots$ of 
a partition $\lambda \in \P$, and $\ell(\lambda)$ for its length, that is, the number of non-zero parts.
We will often identify $\lambda \in \P$ with its Young diagram.
For example, the partition $(5, 3^2, 2)$ is identified with
\begin{align*}
\yng(5,3,3,2)
\end{align*}

The Grothendieck ring $K_0(\sym)$ of the symmetric category is positively graded 
with degree $n$ component being $K_0(\kk S_n)$. It is well known that
$K_0(\sym)$ is canonically isomorphic as a graded ring 
to the {\em ring of symmetric functions} $\Lambda = \bigoplus_{n \geq 0} \Lambda_n$,
with the class $[S(\lambda)]$ of the Specht module corresponding under the isomorphism to the 
{\em Schur function} $s_\lambda \in \Lambda$.
In $\Lambda$, we have that
\begin{equation}\label{LR}
s_\mu s_\nu = \sum_{\lambda \in \P} 
LR_{\mu,\nu}^\lambda s_\lambda
\end{equation}
where
$LR_{\mu,\nu}^\lambda$ is the Littlewood-Richardson coefficient.
Since $\sym$ is semisimple, this is equivalent to the existence of an isomorphism
\begin{equation}\label{LR2}
  S(\mu)\ostar S(\nu) \cong \bigoplus_{\lambda \in \P}
S(\lambda)^{\oplus LR_{\mu,\nu}^\lambda}
\end{equation}
at the level of modules.
Later on, we will also need the ``triple'' Littlewood-Richardson coefficient
\begin{align}\label{tripleLR}
     LR^{\kappa}_{\lambda,\mu,\nu}&:=\sum_{\gamma \in \P} LR^\gamma_{\lambda,\mu} LR^\kappa_{\gamma,\nu} = [S(\lambda)\ostar S(\mu)
   \ostar S(\nu):S(\kappa)].
            \end{align}

The {\em content} of the node in row $i$ and column $j$ of a Young diagram is the integer $c=j-i$. Let $\add(\lambda)$ be the set consisting of the 
contents of the {\em addable nodes}
of $\lambda$, that is, the places in the Young diagram where a node can be added to 
the diagram to obtain a new Young diagram.
Similarly, let $\rem(\lambda)$ be the set of contents of the {\em removable nodes}
of $\lambda$, that is, the places in the Young diagram where a node can be removed from the diagram to obtain a new Young diagram.
Note that all of the addable and removable nodes of a Young diagram
are of different contents (another of the benefits of working in
characteristic zero).
For $a \in \add(\lambda)$, let $\lambda+\boxed{a}$ be the partition obtained by adding 
the unique addable node of content $a$ to the diagram.
For $b \in \rem(\lambda)$, let $\lambda - \boxed{b}$ be the partition obtained by removing the unique removable node of content $b$ from the diagram.

The combinatorial notions just introduced arise naturally on
considering {\em branching rules} for the symmetric group. In our
setup, the sums over all $n \geq 0$ of the usual
restriction and
induction functors
$\res_{S_n}^{S_{n+1}}$ and
$\ind_{S_n}^{S_{n+1}}=  \kk S_{n+1}\otimes_{\kk S_n}$
are isomorphic to the functors
\begin{align}\label{EFdef}
F &:= \res_{\mid\,\star}:\sym\Modfd \rightarrow \sym\Modfd,&
E &:= \ind_{\mid\,\star}:\sym\Modfd \rightarrow \sym\Modfd,
\end{align}
notation as in \cref{if1,if2}.
This follows because
the functor
\begin{align}
  \mid\, \star:&\Sym \rightarrow \Sym,\qquad
            \begin{tikzpicture}[anchorbase]
\draw[-] (-0.5,-0.2)--(-0.5,0.2)--(0.5,0.2)--(0.5,-0.2)--(-0.5,-0.2);
\node at (-0,-0) {$\scriptstyle{g}$};
\draw[-,thick] (-0.4,-0.2)--(-0.4,-0.5);
\draw[-,thick] (0.4,-0.2)--(0.4,-0.5);
\draw[-,thick] (-0.4,0.2)--(-0.4,0.5);
\draw[-,thick] (0.4,0.2)--(0.4,0.5);
\node at (-0.2,0.35) {$\cdot$};
\node at (-0,0.35) {$\cdot$};
\node at (0.2,0.35) {$\cdot$};
\node at (-0.2,-0.4) {$\cdot$};
\node at (-0,-0.4) {$\cdot$};
\node at (0.2,-0.4) {$\cdot$};
\end{tikzpicture} \mapsto \begin{tikzpicture}[anchorbase]
\draw[-,thick] (-0.65,-0.5)--(-0.65,0.5);
\draw[-] (-0.5,-0.2)--(-0.5,0.2)--(0.5,0.2)--(0.5,-0.2)--(-0.5,-0.2);
\node at (-0,-0) {$\scriptstyle{g}$};
\draw[-,thick] (-0.4,-0.2)--(-0.4,-0.5);
\draw[-,thick] (0.4,-0.2)--(0.4,-0.5);
\draw[-,thick] (-0.4,0.2)--(-0.4,0.5);
\draw[-,thick] (0.4,0.2)--(0.4,0.5);
\node at (-0.2,0.35) {$\cdot$};
\node at (-0,0.35) {$\cdot$};
\node at (0.2,0.35) {$\cdot$};
\node at (-0.2,-0.4) {$\cdot$};
\node at (-0,-0.4) {$\cdot$};
\node at (0.2,-0.4) {$\cdot$};
\end{tikzpicture}\,.
\end{align}
coincides with the natural 
inclusion $S_n\hookrightarrow S_{n+1}$ on permutations $g \in
S_n \subset \End_{\Sym}(n)$.
The canonical adjunction makes $(E,F)$
into an adjoint pair of functors. In fact, these functors are
are {\em biadjoint}, i.e., there is also an adjunction making $(F,E)$ into an
adjoint pair.
The effect of the functors $F$ and $E$ on the Specht module
$S(\lambda)$ is well known: we have that
\begin{align}\label{branchey}
  F S(\lambda)&\cong\bigoplus_{b \in \rem(\lambda)}
                S(\lambda-\boxed{b}),&
  E S(\lambda) &\cong \bigoplus_{a \in \add(\lambda)}
                 S(\lambda+\boxed{a}).
                \end{align}

                We finally recall a bit about the Jucys-Murphy
                elements in $\Sym$. One natural way to obtain
                these is to start from the {\em affine symmetric
                  category} $\ASym$, which is the strict $\kk$-linear monoidal category obtained from $\Sym$ by adjoining an extra generator $\begin{tikzpicture}[anchorbase]\draw[-,thick](0,0)to(0,0.4);\opendot{0,0.2};\end{tikzpicture}$ subject to the equivalent relations 
\begin{align}\label{daharel}
\mathord{
\begin{tikzpicture}[baseline = -1mm]
	\draw[-,thick] (0.25,.3) to (-0.25,-.3);
	\draw[-,thick] (0.25,-.3) to (-0.25,.3);
 \node at (-0.12,-0.145) {$\dt$};
\end{tikzpicture}
}
&=
\mathord{
\begin{tikzpicture}[baseline = -1mm]
	\draw[-,thick] (0.25,.3) to (-0.25,-.3);
	\draw[-,thick] (0.25,-.3) to (-0.25,.3);
     \node at (0.12,0.135) {$\dt$};
\end{tikzpicture}}
+\:\mathord{
\begin{tikzpicture}[baseline = -1mm]
 	\draw[-,thick] (0.08,-.3) to (0.08,.3);
	\draw[-,thick] (-0.28,-.3) to (-0.28,.3);
\end{tikzpicture}
},&
\mathord{
\begin{tikzpicture}[baseline = -1mm]
	\draw[-,thick] (0.25,.3) to (-0.25,-.3);
	\draw[-,thick] (0.25,-.3) to (-0.25,.3);
 \node at (-0.12,0.135) {$\dt$};
\end{tikzpicture}
}
&=
\mathord{
\begin{tikzpicture}[baseline = -1mm]
	\draw[-,thick] (0.25,.3) to (-0.25,-.3);
	\draw[-,thick] (0.25,-.3) to (-0.25,.3);
     \node at (0.12,-0.145) {$\dt$};
\end{tikzpicture}}
+\:\mathord{
\begin{tikzpicture}[baseline = -1mm]
 	\draw[-,thick] (0.08,-.3) to (0.08,.3);
	\draw[-,thick] (-0.28,-.3) to (-0.28,.3);
\end{tikzpicture}
}\:.
\end{align}
The path algebra $\Asym$ is isomorphic to $\bigoplus_{n \geq 0} AH_n$
where $AH_n$ is the $n$th {\em degenerate affine Hecke algebra}.
There is an obvious faithful strict $\kk$-linear monoidal
functor $i:\Sym \rightarrow \ASym$. There is also a unique (non-monoidal)
full $\kk$-linear functor
\begin{equation}\label{symp}
  p:\ASym \rightarrow \Sym
\end{equation}
such that $p \circ i = \Id_{\Sym}$ and \begin{equation}
p\left(\begin{tikzpicture}[anchorbase]
\draw[-,thick] (0,0) to (0,0.8);
\node at (0.1,0.4) {$\cdot$};
\node at (0.25,0.4) {$\cdot$};
\node at (0.4,0.4) {$\cdot$};
\draw[-,thick] (0.5,0) to (0.5,0.8);
\draw[-,thick] (0.85,0) to (0.85,0.8);
\node at (0.85,0.4) {$\dt$};
\node at (.85,-.1) {$\stringlabel{1}$};
\node at (.5,-.1) {$\stringlabel{2}$};
\node at (0,-.1) {$\stringlabel{n}$};
\end{tikzpicture}\right)\ =\ 0
\end{equation}
for all $n \geq 1$.
                For $1 \leq j \leq n$, the $j$th Jucys-Murphy element of the
                symmetric group $S_n$
is
                \begin{equation}\label{jmdef}
                x_j = p\left(\begin{tikzpicture}[anchorbase]
\draw[-,thick] (-.1,0) to (-.1,0.8);
\node at (0.05,0.4) {$\cdot$};
\node at (0.2,0.4) {$\cdot$};
\node at (0.35,0.4) {$\cdot$};
\draw[-,thick] (0.6,0) to (0.6,0.8);
\node at (0.6,0.4) {$\dt$};
\node at (0.85,0.4) {$\cdot$};
\node at (1,0.4) {$\cdot$};
\node at (1.15,0.4) {$\cdot$};
\node at (1.3,-.1) {$\stringlabel{1}$};
\node at (-.1,-.1) {$\stringlabel{n}$};
\node at (.6,-.1) {$\stringlabel{j}$};
\draw[-,thick] (1.3,0) to (1.3,0.8);
\end{tikzpicture}\right) =\sum_{i=1}^{j-1} (i\:j)\in \kk S_n,
                \end{equation}
                i.e., it is the sum of the transpositions ``ending'' in $j$.
                Whenever we use this notation, it should be clear from context exactly which symmetric group we have in mind.
                Note $x_{1} = 0$ always. We may also occasionally write $x_0$, which should be interpreted as zero by convention.
                
                The Jucys-Murphy elements $x_{1}, \dots,
                x_{n}$ generate a commutative subalgebra of $\kk
                S_n$ known as the {\em Gelfand-Tsetlin subalgebra}. As
                concisely explained by \cite{OV}, for $\lambda \in
                \P_n$, each Jucys-Murphy element acts
                diagonalizably on the Specht module $S(\lambda)$, and
                the {\em Gelfand-Tsetlin character} of $S(\lambda)$
                recording the dimensions of the simultaneous
                generalized eigenspaces of $x_1,\dots,x_n$ may be obtained from the
                contents of standard $\lambda$-tableaux. 
Indeed, Young's orthonormal basis $\{v_\TT\}$ for $S(\lambda)$ indexed
by standard $\lambda$-tableaux $\TT$ is a basis of simultaneous
eigenvectors for $x_{1},\dots,x_{n}$, with $x_j$ acting on
$v_\TT$ as the content $\cont_j(\TT)$ of the node labelled by $j$ in $T$.
We will
                assume the reader is familiar with these ideas without
                giving any further explanation.

                The functor $p$ induces an isomorphism
                $\ASym / \mathcal{I}\stackrel{\sim}{\rightarrow} \Sym$
                where $\mathcal{I}$ is the
                left tensor ideal\footnote{A left tensor ideal $\mathcal I$ of a $\kk$-linear monoidal category $\cC$
                is the
                data of subspaces $\mathcal{I}(X,Y) \leq \Hom_{\cC}(X,Y)$ for
                all $X,Y \in \ob\cC$, such that these subspaces
                are closed in the obvious sense
                under vertical composition either on top of bottom and under horizontal composition on the left with any morphism.
                Then $\cC / \mathcal{I}$ is the $\cC$-module category with the same objects as $\cC$ and morphisms that are the quotient
                spaces $\Hom_{\cC}(X,Y) / \mathcal{I}(X,Y)$.}
                of $\ASym$ generated by the morphism
$\begin{tikzpicture}[anchorbase]\draw[-,thick](0,0)to(0,0.4);\opendot{0,0.2};\end{tikzpicture}\;$. It
follows that $\Sym$ is a strict $\ASym$-module category.
The functors $E$ and $F$ from \cref{EFdef} are also
the induction and restriction functors $\ind_{\mid\,\star}$ and $\res_{\mid\,\star}$
defined using
this categorical
action of $\ASym$ on $\Sym$. The advantage of passing from $\Sym$ to $\ASym$ here is that
the object $\mid$ of $\ASym$ has the endomorphism defined by the dot,
giving us a natural transformation
$$
\alpha:=\begin{tikzpicture}[anchorbase]\draw[-,thick](0,0)to(0,0.4);\opendot{0,0.2};\node at (.3,.2) {$\star$};\end{tikzpicture}
: \mid\,\star \Rightarrow \mid\,\star.
$$
Applying the general construction from \cref{parmesan} to this, we obtain
endomorphisms
\begin{align}
x&:=
\res_\alpha:F \Rightarrow F,&
x^\vee&:=\ind_\alpha:E\Rightarrow E.
\end{align}
Explicitly, on a $\kk S_n$-module $V$,
$x_V$ is the endomorphism of $FV = \res^{S_n}_{S_{n-1}} V$
defined by multiplying on the left by $x_n \in \kk S_n$,
while
$x^\vee_V$ is the endomorphism of $EV = \kk S_{n+1} \otimes_{\kk S_n} V$ defined
by multiplying $\kk S_{n+1}$ on the right by $x_{n+1} \in \kk S_{n+1}$.
For $c \in \kk$, let $F_c$ and $E_c$ be the $c$ eigenspaces of
$x:F\Rightarrow F$ and $x^\vee:E\Rightarrow E$, respectively.
Since $x^\vee$ is the mate of $x$ and $E$ and $F$ are biadjoint,
it follows that $E_c$ and $F_c$ are biadjoint endofunctors of $\sym\Modfd$
for each $c \in \kk$.
The description of Gelfand-Tsetlin characters of Specht
modules from the previous paragraph is equivalent to the
assertion
that the functors $E_a$ and $F_b$ take the Specht module
$S(\lambda)$ to exactly the summands
                $S(\lambda+\boxed{a})$ and $S(\lambda-\boxed{b})$ in
\cref{branchey}, or to zero if $a \notin \add(\lambda)$ or $b \notin \rem(\lambda)$, respectively.
                It follows that
\begin{align}\label{endofterm}
  F &= \bigoplus_{b \in
                                                 \Z} F_b,&E &= \bigoplus_{a \in \Z} E_a.
\end{align}
            
\section{The partition category and its triangular decomposition}\label{sec3}

Next we introduce the partition category $\Par_t$, which we define by
generators and relations. We then
make some basic observations about its representation theory.
Most of the results here are due to Sam and Snowden \cite[Sec.~6]{SaS}, but we
have tried to give a self-contained account since our general notation and other
conventions are often different. The most important point is
that the path algebra $Par_t$ of the category $\Par_t$
has a triangular decomposition, hence, the category of locally finite-dimensional
$Par_t$-modules is an upper finite highest weight category in the
sense of \cite[$\S$3.3]{BS}.
In fact, $\Par_t$ is a monoidal triangular category in the sense of
Sam and Snowden.

\subsection{The partition category}
Let $t \in \kk$ be a parameter. According to the following definition,
the partition category $\Par_t$ is the free strict $\kk$-linear 
symmetric monoidal category generated by a commutative Frobenius object which is special of categorical dimension $t$.

\begin{definition}\label{Par} The {\em partition category} $\Par_t$ is the strict $\kk$-linear monoidal category generated by one object $\mid$ and the morphisms
\begin{align}\label{gens}
\begin{tikzpicture}[anchorbase]
      \draw[-,thick](0,0) -- (0.4,0.6);
      \draw[-,thick](0.4,0) -- (0,0.6);
\end{tikzpicture}&:\mid\star\mid\to\mid\star\mid\ ,
&
\begin{tikzpicture}[anchorbase]
      \draw[-,thick](0,0) -- (0.3,0.3);
      \draw[-,thick](0.6,0) -- (0.3,0.3);
      \draw[-,thick](0.3,0.3) -- (0.3,0.6);
\end{tikzpicture}&:\mid\star\mid\to\mid\ ,
&
\begin{tikzpicture}[anchorbase]
      \draw[-,thick](0.3,0) -- (0.3,0.3);
      \draw[-,thick](0.6,0.6) -- (0.3,0.3);
      \draw[-,thick](0.3,0.3) -- (0,0.6);
\end{tikzpicture}&:\mid\to\mid\star\mid\ ,
&\begin{tikzpicture}[anchorbase]
      \draw[-,thick](0.3,0) -- (0.3,0.25);
      \node at (0.3,0.3) {$\dt$};
\end{tikzpicture}&:\mid\to\one\ ,
&
\begin{tikzpicture}[anchorbase]
      \draw[-,thick](0.3,0.6) -- (0.3,0.37);
      \node at (0.3,0.3) {$\dt$};
\end{tikzpicture}&:\one\to \mid
\end{align}
subject to the following relations, as well as the ones obtained from these by horizontal and vertical flips:
\begin{align}\label{relssym1}
  \begin{tikzpicture}[anchorbase]
    \draw[-,thick] (0,0) \braidto (0.4,0.4) \braidto (0,0.8);
    \draw[-,thick] (0.4,0) \braidto (0,0.4) \braidto (0.4,0.8);
  \end{tikzpicture}\ &=\ \begin{tikzpicture}[anchorbase]
    \draw[-,thick] (0,0) -- (0,0.8);
    \draw[-,thick] (0.3,0) -- (0.3,0.8);
  \end{tikzpicture}\ ,
  & \begin{tikzpicture}[anchorbase]
  \draw[-,thick] (0,0)--(0.6,0.8);
  \draw[-,thick] (0.6,0)--(0,0.8);
  \draw[-,thick] (0.3,0) to[out=120,in=-120,looseness=1.5] (0.3,0.8);
  \end{tikzpicture}\ &= \begin{tikzpicture}[anchorbase]
  \draw[-,thick] (0,0)--(0.6,0.8);
  \draw[-,thick] (0.6,0)--(0,0.8);
  \draw[-,thick] (0.3,0)to[out=60,in=-60,looseness=1.5](0.3,0.8);
  \end{tikzpicture}\ ,\\\label{relssym2}
  \begin{tikzpicture}[anchorbase]
  \draw[-,thick] (0,0) -- (0.6,-0.8);
  \draw[-,thick] (0.6,0) -- (0,-0.8);
  \draw[-,thick] (0.3,0)to[out=-135,in=135,looseness=1.5](0.15,-0.6);
  \end{tikzpicture} &= \begin{tikzpicture}[anchorbase]
  \draw[-,thick] (0,0) -- (0.6,-0.8);
  \draw[-,thick] (0.6,0) -- (0,-0.8);
  \draw[-,thick] (0.3,0)--(0.45,-0.2);
  \end{tikzpicture},
  & \begin{tikzpicture}[anchorbase]
  \draw[-,thick] (0,0)--(0.6,-0.8);
  \draw[-,thick] (0.6,0)--(0.15,-0.6);
  \node at (0.11,-0.66) {$\dt$};
  \end{tikzpicture}\ &=\ \begin{tikzpicture}[anchorbase]
  \draw[-,thick] (0,0)--(0,-0.8);
  \draw[-,thick] (0.4,0)--(0.4,-0.55);
  \node at (0.4,-0.6) {$\dt$};
  \end{tikzpicture}\ ,\\ \label{relsfrob1}
\begin{tikzpicture}[anchorbase]
  \draw[-,thick] (0.3,0)--(0.3,0.3)--(-0.2,0.8);
  \draw[-,thick] (0.3,0.3)--(0.8,0.8);
  \draw[-,thick] (0.05,0.55)--(0.3,0.8);
  \end{tikzpicture}\ &=\ \begin{tikzpicture}[anchorbase]
  \draw[-,thick] (0.3,0)--(0.3,0.3)--(-0.2,0.8);
  \draw[-,thick] (0.3,0.3)--(0.8,0.8);
  \draw[-,thick] (0.55,0.55)--(0.3,0.8);
  \end{tikzpicture}\ ,&   \begin{tikzpicture}[anchorbase]
  \draw[-,thick] (0,0)--(0,0.2)to[out=45,in=down](0.2,0.4)to[out=up,in=-45](-0.3,0.8);
  \draw[-,thick] (0,0.2)to[out=135,in=down](-0.2,0.4)to[out=up,in=-135](0.3,0.8);
  \end{tikzpicture}\ &=\ \begin{tikzpicture}[anchorbase]
  \draw[-,thick] (0,0)--(0,0.3)--(-0.3,0.8);
  \draw[-,thick] (0,0.3)--(0.3,0.8);
  \end{tikzpicture}\ , \\
    \begin{tikzpicture}[anchorbase]
    \draw[-,thick] (0.3,0)--(0.3,-0.3)--(0.8,-0.8);
    \draw[-,thick] (0.05,-0.55)--(0.3,-0.3);
    \node at (0,-0.6) {$\dt$};
  \end{tikzpicture}\ &=\ \begin{tikzpicture}[anchorbase]
    \draw[-,thick] (0,0.3)--(0,-0.6);
  \end{tikzpicture}\ ,& 
  \begin{tikzpicture}[anchorbase]
    \draw[-,thick] (0,0) -- (0,1);
    \draw[-,thick] (0,.8) -- (0.6,0.2);
    \draw[-,thick] (0.6,0) -- (0.6,1);
  \end{tikzpicture}\ &=\ \begin{tikzpicture}[anchorbase]
    \draw[-,thick] (0,0) -- (0.3,0.3);
    \draw[-,thick] (0.3,0.3) -- (0.3,0.7);
    \draw[-,thick] (0.3,0.7) -- (0,1);
    \draw[-,thick] (0.6,0) -- (0.3,0.3);
    \draw[-,thick] (0.3,0.7) -- (0.6,1);
  \end{tikzpicture}\ , \label{relsfrob2}\\
\begin{tikzpicture}[anchorbase]
  \draw[-,thick] (0,0)to(0,0.15)to[out=45,in=-45,looseness=2](0,0.65)to(0,0.8);
  \draw[-,thick] (0,0.15)to[out=135,in=-135,looseness=2](0,0.65);
  \end{tikzpicture}\ &=\ \begin{tikzpicture}[anchorbase]
  \draw[-,thick] (0,0)to(0,0.8);
  \end{tikzpicture}\ ,&
  \begin{tikzpicture}[anchorbase]\label{relsparam}
  \draw[-,thick] (0,0.05)--(0,0.55);
  \node at (0,0) {$\dt$};
  \node at (0,0.6) {$\dt$};
  \end{tikzpicture}\ &=\ t \one.
\end{align}
The object set of $\Par_t$ is $\{\mid^{\star n}\:|\:n \in \N\}$. We will sometimes denote
$\mid^{\star n}$ simply by $n$, so that the object set is identified with $\N$.
\end{definition}

The relations \cref{relssym1,relssym2} imply that $\Par_t$ is a
{\em symmetric} monoidal category, \cref{relsfrob1,relsfrob2} imply that the generating object is a {\em commutative Frobenius object}, and the first relation from
\cref{relsparam} means that this object is actually a {\em special} 
Frobenius object.
The symmetric monoidal category $\Par_t$ is {\em rigid} with every object being self-dual.
To justify this, 
it is enough to specify the evaluation and coevaluation morphisms 
$\ev:\mid\star\mid\to\one$ and $\coev:\one\to\mid\star\mid$
for the generating object, which we represent graphically by
the cap and cup:
\begin{align}\label{cupcap}
\ev = \begin{tikzpicture}[anchorbase]
\draw[-,thick] (0,0)to(0,0.1)to[out=up,in=up,looseness=2](0.8,0.1)to(0.8,0);
\end{tikzpicture}\ &:=\ \begin{tikzpicture}[anchorbase]
\draw[-,thick] (0,0)to(0.4,0.4)to(0.8,0);
\draw[-,thick] (0.4,0.4)to(0.4,0.65);
\node at (0.4,0.7) {$\dt$};
\end{tikzpicture}\ ,&
\coev =\begin{tikzpicture}[anchorbase]
\draw[-,thick] (0,0.8)to(0,0.7)to[out=down,in=down,looseness=2](0.8,0.7)to(0.8,0.8);
\end{tikzpicture}\ &:=\ \begin{tikzpicture}[anchorbase]
\draw[-,thick] (0,0.8)to(0.4,0.4)to(0.8,0.8);
\draw[-,thick] (0.4,0.4)to(0.4,0.15);
\node at (0.4,0.1) {$\dt$};
\end{tikzpicture}.
\end{align}
These satisfy the zig-zag identities as in \cref{zigzag}, as may easily be checked using
\cref{relsfrob2}.
Now the relations in \cref{relsparam} imply that the
{\em categorical dimension} of the generating object is $t$.

By an {\em $m \times n$ partition diagram}, we mean a 
string diagram $f$ 
representing a morphism in $\Hom_{\Par_t}(n, m)$
obtained by horizontally and vertically composing the generating morphisms \cref{gens},
such that every connected component of $f$ has at least one endpoint, i.e., is not a ``floating bubble". In view of 
the dimension relation in \cref{relsparam},
floating bubbles can be contracted then removed,
multiplying the result by the scalar $t$ each time this occurs. It follows that
every morphism in $\Hom_{\Par_t}(n, m)$
can be written as a $\kk$-linear combination of $m \times n$ partition diagrams.
Let $\sigma:\Par_t \rightarrow (\Par_t)^\op$
be the strict $\kk$-linear monoidal functor that is the identity on
objects and sends the generating morphisms to their flips in a
horizontal axis. More generally, $\sigma$ sends an
$m \times n$ partition diagram to the $n \times m$ partition diagram that is its flip in a horizontal axis.

The above definition of $\Par_t$ by generators and relations
is not the most common definition found in the literature. 
It was first formulated in this way
by Comes in \cite[Th.~2.1]{Com}; see also \cite[Prop.~2.1]{LSR}.
In the more traditional approach (e.g., see \cite[$\S$8]{Del} and \cite[Def.~2.11]{CO}), one instead defines
the morphism space $\Hom_{\Par_t}(n, m)$ to be the vector space with basis labelled by set partitions of
$\{1,\dots,n,1', \dots, m'\}$, giving explicit combinatorial rules for the horizontal and vertical compositions in terms of these partitions.
Suppose that $f$ is an $m \times n$ partition diagram.
Labelling the endpoints of $f$ from right to left by $1,\dots,n$ on the bottom bounadry and $1',\dots,m'$ on the top boundary as in the following example, the diagram $f$ 
determines a partition of the set
$\{1,\dots,n,1',\dots,m'\}$
with parts
arising from the labels at the endpoints of the connected components in the diagram.
For example, the $9 \times 7$ partition diagram
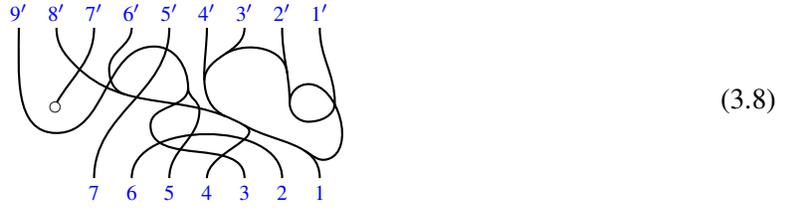
\begin{equation}\label{exmorph}
\begin{tikzpicture}[anchorbase]
\draw[-,thick] (0,0)to[out=up,in=down](1,2);
\draw[-,thick] (0.5,0)to[out=up,in=up](2.5,0);
\draw[-,thick] (-1,2)to[out=down,in=left](-0.5,0.6)to[out=right,in=left](0.8,1.75)to[out=right,in=up](1.25,1.2)to[out=down,in=up](0.75,0.7)to[out=down,in=up](2,0);
\draw[-,thick] (1.25,1.2)to[out=down,in=up](1.4,0.9)to[out=down,in=up](1,0);
\draw[-,thick] (0,2)to[out=down,in=60](-0.5,1);
\node at (-0.52,0.93) {$\dt$};
\draw[-,thick] (-0.5,2)to[out=down,in=135](-0.25,1.5)to[out=-45,in=150](2,0.7)to[out=-30,in=up](1.5,0);
\draw[-,thick] (2,0.7)to[out=-30,in=up](3,0);
\draw[-,thick] (0.5,2)to[out=down,in=up](0.2,1.45)to[out=down,in=160](0.4,1.1165);
\draw[-,thick] (3,2)to[out=down,in=up](3.2,1)to[out=down,in=right](2.85,0.75)to[out=left,in=down](2.6,1)to[out=up,in=left](2.85,1.25)to[out=right,in=15](3.1,0.25)to[out=195,in=-30](2.8,0.364);
\draw[-,thick] (1.5,2)to[out=down,in=155](1.75,0.82);
\draw[-,thick] (2,2)to[out=down,in=87](1.465,1.3);
\draw[-,thick] (2.5,2)to[out=down,in=up](2.6,1);
\draw[-,thick] (1.75,1.657)to[out=30,in=95](2.566,1.4);
\node at (0,-0.2) {$\stringnumber{7}$};\node at (0.5,-0.2) {$\stringnumber{6}$};\node at (1,-0.2) {$\stringnumber{5}$};\node at (1.5,-0.2) {$\stringnumber{4}$};\node at (2,-0.2) {$\stringnumber{3}$};\node at (2.5,-0.2) {$\stringnumber{2}$};\node at (3,-0.2) {$\stringnumber{1}$};
\node at (-1,2.2) {$\stringnumber{9'}$};\node at (-0.5,2.2) {$\stringnumber{8'}$};\node at (0,2.2) {$\stringnumber{7'}$};\node at (0.5,2.2) {$\stringnumber{6'}$};\node at (1,2.2) {$\stringnumber{5'}$};\node at (1.5,2.2) {$\stringnumber{4'}$};\node at (2,2.2) {$\stringnumber{3'}$};\node at (2.5,2.2) {$\stringnumber{2'}$};\node at (3,2.2) {$\stringnumber{1'}$};
\end{tikzpicture}
\end{equation}
determines the partition
$$
\{1,4,1',2',3',4',6',8'\}\sqcup\{2,6\}\sqcup\{3,5,9'\}\sqcup\{7,5'\}\sqcup\{7'\}.
$$
In this way, one obtains a strict $\kk$-linear monoidal functor 
from the category $\Par_t$ defined by generators and relations as above to the 
category $\Par_t$ as defined via the more traditional combinatorial approach.
Then the result of Comes just mentioned asserts that this functor is an
isomorphism.

The discussion in the previous paragraph shows that
two $m \times n$ partition diagrams represent the same morphism 
in $\Hom_{\Par_t}(n, m)$
if and only if the diagrams are equivalent in the sense that they
determine the same partition of the set 
$\{1,\dots,n,1',\dots,m'\}$ labelling their endpoints.
For example, the morphism represented by \cref{exmorph} is equal to the 
one represented by the tidier diagram
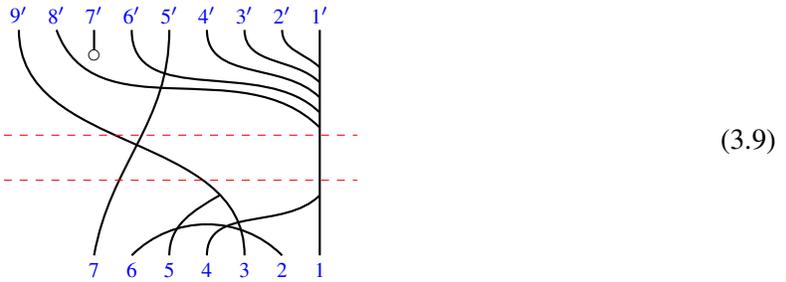
\begin{equation}\label{inmorph}
\begin{tikzpicture}[anchorbase]
\draw[-,thick] (0,3)to(0,2.72);
\draw[-,thick] (-1,3)to[out=-90,in=90](2,0);
\node at (0,2.66) {$\dt$};
\draw[-,thick] (0.5,0)to[out=45,in=135](2.5,0);
\draw[-,thick] (2.5,3)to[out=-90](3,2.5);
\draw[-,thick] (2,3)to[out=-90](3,2.3);
\draw[-,thick] (1.5,3)to[out=-90](3,2.1);
\draw[-,thick] (.5,3)to[out=-90](3,1.9);
\draw[-,thick] (-.5,3)to[out=-70](3,1.7);
\draw[-,thick] (1.5,0)to[out=90,in=-135](3,.8);
\draw[-,thick] (1,0)to[out=90,in=-150](1.67,.8);
\draw[-,thick] (0,0)to[out=80,in=-90](1,3);
\draw[-,thick] (3,3)to(3,0);
\node at (0,-0.2) {$\stringnumber{7}$};\node at (0.5,-0.2) {$\stringnumber{6}$};\node at (1,-0.2) {$\stringnumber{5}$};\node at (1.5,-0.2) {$\stringnumber{4}$};\node at (2,-0.2) {$\stringnumber{3}$};\node at (2.5,-0.2) {$\stringnumber{2}$};\node at (3,-0.2) {$\stringnumber{1}$};
\node at (-1,3.2) {$\stringnumber{9'}$};\node at (-0.5,3.2) {$\stringnumber{8'}$};\node at (0,3.2) {$\stringnumber{7'}$};\node at (0.5,3.2) {$\stringnumber{6'}$};\node at (1,3.2) {$\stringnumber{5'}$};\node at (1.5,3.2) {$\stringnumber{4'}$};\node at (2,3.2) {$\stringnumber{3'}$};\node at (2.5,3.2) {$\stringnumber{2'}$};\node at (3,3.2) {$\stringnumber{1'}$};
\draw[dashed,red] (-1.2,1) to (3.5,1);
\draw[dashed,red] (-1.2,1.6) to (3.5,1.6);
\end{tikzpicture}
\end{equation}
because this determines the same partition of the set labelling the endpoints.
In fact, Comes' result implies that 
any set of representatives for the equivalence classes
$m \times n$ partition diagrams 
give a {\em basis} for the morphism space $\Hom_{\Par_t}(n, m)$.
In particular, 
$\dim \Hom_{\Par_t}(n, m)$ is equal to the
the $(m+n)$th {\em Bell number} which counts set partitions of $m+n$.
Taking $m=n=0$, this implies that
\begin{equation}
  \End_{\Par_t}(\one) = \kk.
  \end{equation}

\subsection{Triangular decomposition}\label{stria}
Let $c$ be a connected component in some partition diagram representing a
morphism in $\Par_t$.
We call $c$ an {\em upward branch} if 
$c$ has at least two endpoints on its top boundary and no endpoints on its bottom boundary, and a {\em downward branch} if it has at least two endpoints on its bottom boundary but no endpoints at the top:
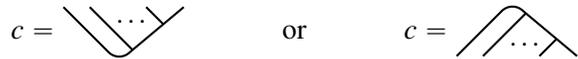
\[
c = \begin{tikzpicture}[anchorbase]
 \draw[-,thick](0,0.8) -- (0.6,0.2) arc (225:315:.2) -- (1.6,0.8);
 \draw[-,thick](0.35,0.8)--(0.92,0.23);
 \node at (0.75,0.6) {$\cdot$};
 \node at (0.9,0.6) {$\cdot$};
 \node at (1.05,0.6) {$\cdot$};
 \draw[-,thick](1.1,0.8)--(1.33,0.57);
\end{tikzpicture} \hspace{1.2cm}\text{ or }\hspace{1.2cm} 
c = \begin{tikzpicture}[anchorbase]
 \draw[-,thick](0,0) -- (0.6,0.6) arc (135:45:.2) -- (1.6,0);
 \draw[-,thick](0.35,0)--(0.92,0.57);
 \node at (0.75,0.15) {$\cdot$};
 \node at (0.9,0.15) {$\cdot$};
 \node at (1.05,0.15) {$\cdot$};
 \draw[-,thick](1.1,0)--(1.33,0.23);
\end{tikzpicture}.
\]
We call $c$ an {\em upward leaf} if it has exactly one endpoint at the top and no endpoints at the bottom, and a {\em downward leaf} if it has no endpoints at the top and exactly one at the bottom:
$$
c=\begin{tikzpicture}[anchorbase]
      \draw[-,thick](0.3,0.6) -- (0.3,0.37);
      \node at (0.3,0.3) {$\dt$};
\end{tikzpicture}
\hspace{1.2cm}\text{ or }\hspace{1.2cm}
c=\begin{tikzpicture}[anchorbase]
      \draw[-,thick](0.3,0) -- (0.3,0.25);
      \node at (0.3,0.3) {$\dt$};
\end{tikzpicture}. 
$$
We refer to $c$ as an \textit{upward tree} if it has more than one endpoint at the top and exactly one endpoint at the bottom, and a {\em downward tree} if it has exactly one endpoint at the top and more than one endpoint at the bottom: 
\[
c = \begin{tikzpicture}[anchorbase]
 \draw[-,thick](0,1.2) -- (0.8,0.4) -- (1.6,1.2);
 \draw[-,thick](0.35,1.2) -- (.97,0.58);
 \draw[-,thick](0.8,0.1) -- (0.8,0.4);
 \node at (0.75,1) {$\cdot$};
 \node at (0.9,1) {$\cdot$};
 \node at (1.05,1) {$\cdot$};
 \draw[-,thick](1.1,1.2)--(1.35,0.95);
\end{tikzpicture}\hspace{1.2cm}\text{ or }\hspace{1.2cm} c = \begin{tikzpicture}[anchorbase]
 \draw[-,thick](0,0) -- (0.8,0.8) -- (1.6,0);
 \draw[-,thick](0.35,0) -- (.97,0.62);
 \draw[-,thick](0.8,1.1) -- (0.8,0.8);
 \node at (0.75,0.15) {$\cdot$};
 \node at (0.9,0.15) {$\cdot$};
 \node at (1.05,0.15) {$\cdot$};
 \draw[-,thick](1.1,0)--(1.35,0.25);
\end{tikzpicture}
\]
We say that $c$ is a {\em double tree} if $c$ has more than one endpoint at the top and more than one endpoint at the bottom. In that case, it is equivalent to the composition of an upward tree and a downward tree; for example, the rightmost connected component in \cref{inmorph} is a double tree.
Finally we say that $c$ is a {\em trunk} if $c$ has exactly one endpoint both at the top and at the bottom: 
$$
c=\:
\begin{tikzpicture}[anchorbase]
 \draw[-,thick](0,0) to (0,.8);
\end{tikzpicture}\:. 
$$
Any connected component of a partition diagram can be represented
either as an upward branch, an upward leaf, an upward tree,
a downward branch, a downward leaf, a downward tree, a double tree, or a trunk.

Let $f$ be an $m \times n$  partition diagram. We say $f$ is
\begin{itemize}
\item a {\em permutation diagram}
if all of its connected components are trunks, in which case we must have that $m=n$;
\item an {\em upward partition diagram}
if its connected components are trunks, upward branches, upward leaves and upward trees, in which case we must have that $m \geq n$; 
\item a {\em downward partition diagram} if its connected components are trunks, downward branches, downward leaves and downward trees, in which case we must have that $m \leq n$.
\end{itemize}
Let $f$ be an upward $m \times n$ partition diagram. We say that it is 
{\em strictly upward} if $m > n$.
Let $c_1,\dots,c_k$ be the connected components of $f$ that are either trunks or upward trees, indexing them
so that their bottom endpoints are in order from right to left in $f$.
We say that $f$ is {\em normally ordered} if the rightmost of the top endpoints of each of $c_1,\dots,c_k$ are also in order from right to left in $f$.
In other words, $f$ is normally ordered if it can be drawn so that the right edges of all of the upward trees and trunks in $f$ are non-crossing.
Similarly, we define {\em strictly downward} and 
{\em normally ordered} downward partition diagrams.

Now we can define some monoidal subcategories of $\Par_t$.
Let $\Sym$ be the symmetric category as defined in \cref{ssym}.
There is a strict $\kk$-linear symmetric monoidal functor
\begin{equation}\label{canemb}
i_t^\circ:\Sym \rightarrow \Par_t
\end{equation}
sending the generating object and the generating morphism of $\Sym$ 
to the 
generating object and the generating morphism of $\Par_t$ that is represented by the crossing. Using the basis theorem for morphism spaces 
in $\Par_t$, it follows that this functor
is faithful. We use it to {\em identify} $\Sym$ with a monoidal
subcategory of $\Par_t$. In other words, $\Sym$ is
identified with the subcategory of $\Par_t$ consisting of all objects
and all the morphisms which can be written as linear combinations of permutation diagrams.

Next, let $\Par^\flat$ be the strict $\kk$-linear monoidal category generated by one object $\mid$ and the morphisms
\begin{align}\label{gens2}
\begin{tikzpicture}[anchorbase]
      \draw[-,thick](0,0) -- (0.6,0.6);
      \draw[-,thick](0.6,0) -- (0,0.6);
\end{tikzpicture}&:\mid\star\mid\to\mid\star\mid\ ,
&
\begin{tikzpicture}[anchorbase]
      \draw[-,thick](0.3,0) -- (0.3,0.3);
      \draw[-,thick](0.6,0.6) -- (0.3,0.3);
      \draw[-,thick](0.3,0.3) -- (0,0.6);
\end{tikzpicture}&:\mid\to\mid\star\mid\ ,
&
\begin{tikzpicture}[anchorbase]
      \draw[-,thick](0.3,0.6) -- (0.3,0.37);
      \node at (0.3,0.3) {$\dt$};
\end{tikzpicture}&:\one\to \mid
\end{align}
subject to the relations \cref{relssym1,relssym2,relsfrob1} and their
flips in a vertical axis.
We call this the {\em upward partition category}.
The cup can also be defined in $\Par^\flat$ as in \cref{cupcap}.
Any upward partition diagram can be interpreted as a string diagram representing a morphism in $\Par^\flat$.
Moreover, the defining relations in $\Par^\flat$ imply that
two upward $m \times n$ partition diagrams which are equivalent in the sense that
they define the same partition of the set 
$\{1,\dots,n,1',\dots,m'\}$ labelling the
endpoints
are also equal as morphisms in
$\Hom_{\Par^\flat}(n, m)$.
There is a strict $\kk$-linear monoidal functor
\begin{equation}\label{canembflat}
i_t^\flat:\Par^\flat \rightarrow \Par_t
\end{equation}
sending the generating morphisms of $\Par^\flat$ 
to the corresponding ones in $\Par_t$. As equivalence classes of upward $m \times n$ partition diagrams span $\Hom_{\Par^\flat}(n, m)$ 
and their images in $\Hom_{\Par_t}(n, m)$ are linearly independent, this functor is faithful. 
We use it to {\em identify} $\Par^\flat$ with a monoidal
subcategory of $\Par_t$. In other words, $\Par^\flat$ is
identified with the monoidal subcategory of $\Par_t$ consisting of all objects
and all of the morphisms which can be written as linear combinations of upward partition diagrams.
Also let $\Par^-$ be the monoidal subcategory of $\Par^\flat$ consisting of all objects and all of the 
morphisms which can be written as linear combinations of normally ordered upward partition diagrams.

Similarly to the previous paragraph, we define $\Par^\sharp$, the {\em
  downward partition category}, to be the strict $\kk$-linear monoidal category generated by one object $\mid$ and the morphisms that are the flips
of \cref{gens2} in a horizontal axis, subject to the relations that are the flips of the ones for $\Par^\flat$. The cap can also be defined in $\Par^\sharp$ as in \cref{cupcap}.
Evidently, $\Par^\sharp \cong (\Par^\flat)^\op$ with isomorphism being defined by the flip $\sigma$ in a horizontal axis.
There is a strict $\kk$-linear monoidal functor
\begin{equation}\label{canembsharp}
i_t^\sharp:\Par^\sharp \rightarrow \Par_t
\end{equation}
sending the generating morphisms of $\Par^\sharp$ 
to the corresponding ones in $\Par_t$. 
We have that $i_t^\sharp = \sigma \circ i_t^\flat \circ \sigma$, so we
deduce from the previous paragraph that $i_t^\sharp$ is faithful too.
We use it to {\em identify} $\Par^\sharp$ with a monoidal
subcategory of $\Par_t$. In other words, $\Par^\sharp$ is
identified with the monoidal subcategory of $\Par_t$ consisting of all objects
and all of the morphisms which can be written as linear combinations of downward partition diagrams.
Also let $\Par^+$ be the monoidal subcategory of $\Par^\sharp$ consisting of all objects and all of the morphisms which can be written as linear combinations of normally ordered downward partition diagrams.

Finally we let $Par_t$ be the path algebra of $\Par_t$. It is a locally unital algebra
with distinguished idempotents $\{1_n\:|\:n \in \N\}$ arising from the identity endomorphisms of the objects of $\Par_t$.
We also have the path algebras
$Par^\flat, Par^-, \sym, Par^+, Par^\sharp$ 
of $\Par^\flat, \Par^-, \Sym, \Par^+, \Par^\sharp$, which we may view as
locally unital subalgebras of $Par_t$ via the embeddings
\cref{canemb,canembflat,canembsharp}. The following theorem is 
the {\em triangular decomposition} of $Par_t$.

\begin{theorem}\label{td}
Let $\KK := \bigoplus_{n \geq 0} \kk 1_n$ viewed as a
locally unital subalgebra of $Par_t$. Multiplication defines a linear isomorphism
\begin{align}
Par^- \otimes_{\KK} \sym \otimes_{\KK} Par^+
  &\stackrel{\sim}{\rightarrow} Par_t.\label{td1}\\
  \intertext{Hence, we also have isomorphisms}
Par^- \otimes_{\KK} \sym 
&\stackrel{\sim}{\rightarrow} Par^\flat,\label{td2}\\
\sym \otimes_{\KK} Par^+
&\stackrel{\sim}{\rightarrow} Par^\sharp,\label{td3}\\
Par^\flat \otimes_{\sym} Par^\sharp
&\stackrel{\sim}{\rightarrow} Par_t.\label{td4}
\end{align}
\end{theorem}

\begin{proof}
Any partition diagram is equivalent to a diagram
that is the composition of a normally ordered upward partition diagram, a permutation diagram, and a normally ordered downward partition diagram; see \cref{inmorph} for an example of such a decomposition.
Moreover, equivalence classes of these sorts of 
diagrams give bases for $Par_t$, $Par^-, \sym$ and $Par^+$. This
implies that \cref{td1} is an isomorphism.
Then \cref{td2,td3,td4} follow as in \cite[Rem.~5.32]{BS}.
\end{proof}

\cref{td} is all that is needed to see that 
the locally finite-dimensional locally unital algebra $$
Par_t = \bigoplus_{m,n \in \N} 1_m Par_t 1_n
$$ 
has a {\em split triangular
decomposition} in the sense of \cite[Rem.~5.32]{BS}.
Its negative and positive Borel subalgebras are $Par^\flat$ and $Par^\sharp$,
and its Cartan subalgebra is $\sym = Par^\flat \cap Par^\sharp$.
The set $I$ in the notation of \cite{BS} is the set $\N$ 
indexing the distinguished idempotents $\{1_n\:|\:n \in \N\}$. The upper finite 
poset $(\Lambda,\leq)$ in the setup of \cite{BS} is $(\N,\geq)$; we stress
 that the ordering is reversed here, as it has to be in order to have an upper finite poset, thereby conforming to the general conventions of \cite{BS}.
 The function $\partial$ from \cite{BS} is the identity function.

 As $\sym$ is semisimple, this discussion
  shows equivalently that $\Par_t$ is a triangular category in the
  sense of  \cite[Def.~4.1]{SaS}. Its upward and downward
  subcategories $\cU$ and $\cD$ in the setup of {\em
    loc. cit.} are $\Par^\flat$
  and $\Par^\sharp$, respectively. In fact, $\Par_t$ is a {\em monoidal
    triangular category} as defined in \cite[$\S$4.11]{SaS}, as was
  established already by Sam and Snowden in \cite[Prop.~6.3]{SaS}.
  This means that induction commutes with induction product: we have that
  \begin{align}
    \ind_{\sym}^{Par_t} (V \ostar W) &\cong (\ind_{\sym}^{Par_t}
                                            V) \ostar
                                            (\ind_{\sym}^{Par_t}
                                            W),
  \end{align}
  for
    $\sym$-modules $V$ and $W$. This is
    easily seen directly from the definition of the induction product
    using
      $i_t^\circ \circ \star \cong\star \circ (i_t^\circ \boxtimes
      i_t^\circ)$. Similarly,
      \begin{align}
        \ind_{\sym}^{Par^\sharp} (V \ostar W) &\cong (\ind_{\sym}^{Par^\sharp}
                                            V) \ostar
                                            (\ind_{\sym}^{Par^\sharp}
                                            W),&
    \ind_{Par^\sharp}^{Par_t} (V \ostar W) &\cong (\ind_{Par^\sharp}^{Par_t}
    V) \ostar (\ind_{Par^\sharp}^{Par_t} W).
  \end{align}

\subsection{Classification of irreducible modules and highest weight structure}
As $Par_t$ has a triangular decomposition with Cartan subalgebra $\sym$ being semisimple, we can appeal to the general results of 
\cite[$\S$5.5]{BS} to obtain the classification of irreducible
$Par_t$-modules. Alternatively, this follows from the results in
\cite[$\S$5.5]{SaS}, but note that Sam and Snowden use the language of {\em lowest weight}
rather than {\em highest weight} categories.
Since isomorphism classes of irreducible $Par_t$-modules are in bijection with isomorphism classes of indecomposable projective $Par_t$-modules, and the latter are identified with isomorphism classes of indecomposable objects in
$\Kar(\Par_t)$, the results discussed
in this subsection are equivalent to the classification obtained originally in \cite[Th.~3.7]{CO}.

The algebra $Par_t$ is $\Z$-graded with $1_m Par_t 1_n$ being in
degree $m-n$. The induced gradings on the
subalgebras $Par^\flat$ and $Par^\sharp$ make these into  positively
and negatively graded algebras, respectively, with degree zero
components in both cases being the semisimple algebra $\sym$.
It follows that
the Jacobson radicals of
 $Par^\flat$ and $Par^\sharp$ are the direct sums of their non-zero graded
 components. Moreover, the quotients by their Jacobson radicals are
 naturally identified with $\sym$, i.e., there are
 locally
unital algebra homomorphisms
 \begin{align}\label{pies}
 \pi^\flat:  &Par^\flat \twoheadrightarrow \sym,&
  \pi^\sharp: &Par^\sharp\twoheadrightarrow \sym.
 \end{align}
 Let
$\infl^\sharp:\sym\Modfd \rightarrow Par^\sharp\Modfd$ and $\infl^\flat:\sym\Modfd\rightarrow Par^\flat\Modfd$
be the functors defined by restriction along these
homomorphisms.
The modules
\begin{align}
  &\big\{S^\flat(\lambda) := \infl^\flat S(\lambda)\:\big|\:\lambda
\in \P\big\},&
&\big\{S^\sharp(\lambda) := \infl^\sharp S(\lambda)\:\big|\:\lambda
                      \in \P\big\}
                      \end{align}
                      give full sets of pairwise inequivalent irreducible modules for
$Par^\flat$ and $Par^\sharp$, respectively.

As in \cite[(5.13)--(5.14)]{BS}, we define the {\em standardization} and {\em costandardization functors}
\begin{align}\label{stdize}
  j_! := \ind_{Par^\sharp}^{Par_t} \circ \infl^\sharp&:\sym\Modfd \rightarrow Par_t \Modlfd,\\
j_* := \coind_{Par^\flat}^{Par_t} \circ \infl^\flat&:
\sym\Modfd \rightarrow Par_t \Modlfd,
\end{align}
where $\ind_{Par^\sharp}^{Par_t} := Par_t \otimes_{Par^\sharp} $ and
$\coind_{Par^\flat}^{Par_t} :=\bigoplus_{n \in \N} \Hom_{Par^\flat}(Par_t 1_n,
?)$. From \cref{td1,td2,td3} it follows that $Par_t$ is projective both as a right
$Par^\sharp$-module and as a left $Par^\flat$-module, hence, these
functors are exact. Then we define the {\em standard}
and {\em costandard modules} for $Par_t$ by
\begin{align}\label{stdcostd}
\Delta(\lambda) &:= j_! S(\lambda) = \ind_{Par^\sharp}^{Par_t} S^\sharp(\lambda),
&
\nabla(\lambda) &= j_* S(\lambda) = \ind_{Par^\flat}^{Par_t} S^\flat(\lambda),
\end{align}
respectively. 

\begin{theorem}\label{firstthm}
The $Par_t$-modules $\{L(\lambda)\:|\:\lambda \in \P\}$ defined from
  $$
  L(\lambda) := \operatorname{hd} \Delta(\lambda) \cong
  \operatorname{soc} \nabla(\lambda)
  $$
give a complete set of pairwise inequivalent irreducible left $Par_t$-modules.
Moreover, $Par_t\Modlfd$ is an upper finite highest weight category 
in the sense of \cite[Def.~3.34]{BS}
with weight poset $(\P, \preceq)$, where $\preceq$ is the partial order on $\P$
defined by $\lambda \preceq \mu$ if and only if either $\lambda = \mu$ or 
$|\lambda| > |\mu|$. Its standard and costandard objects are the
modules $\Delta(\lambda)$ and $\nabla(\lambda)$, respectively.
\end{theorem}

\begin{proof}
This follows immediately from \cite[Cor.~5.39]{BS} using the triangular decomposition from 
\cref{td} and the semisimplicity of $\sym$;
see also \cite[$\S$5.5]{SaS}.
\end{proof}

The fact established in \cref{firstthm} that $Par_t\Modlfd$ is an upper finite highest weight category has several significant consequences.
As for any Schurian category, $L(\lambda)$ has a projective cover we denote by $P(\lambda) \in Par_t\Modlfd$. 
Let $Par_t\Moddelta$ be the exact subcategory of $Par_t\Modlfd$ consisting of all modules with a $\Delta$-flag, that is, a finite filtration whose sections are 
of the form $\Delta(\lambda)$ for $\lambda \in \P$.
For any 
$V \in Par_t\Moddelta$, the multiplicity $(V:\Delta(\mu))$ of $\Delta(\mu)$ as a section of some $\Delta$-flag in $V$ is well-defined independent of the flag, indeed, it can be calculated from
\begin{equation}\label{bggrecproof}
(V:\Delta(\mu)) = \dim \Hom_{Par_t}(V, \nabla(\mu)).
\end{equation}
This follows from the fundamental $\Ext$-vanishing property
of highest weight categories, namely, that
\begin{equation}\label{deltanabla}
\dim \Ext^i_{Par_t}(\Delta(\lambda), \nabla(\mu)) = \delta_{i,0} \delta_{\lambda,\mu}
\end{equation}
for any $\lambda,\mu \in \P$ and $i \geq 0$; see \cite[Lem.~3.48]{BS}.
The definition of highest weight category gives that $P(\lambda)$ has a $\Delta$-flag,
so that $Par_t\Proj$ is a full subcategory of $Par_t\Moddelta$.
Moreover, from \cref{bggrecproof}, one obtains the usual {\em BGG reciprocity formula}
\begin{equation}\label{bggrecip}
(P(\lambda):\Delta(\mu)) = 
[\nabla(\mu):L(\lambda)].
\end{equation}

The functor $\sigma:\Par_t\stackrel{\sim}{\rightarrow} (\Par_t)^\op$ defined by flipping diagrams in a horizontal axis
can also be viewed as a locally unital anti-involution of the 
algebra $Par_t$. It interchanges the subalgebras $Par^\flat$ and $Par^\sharp$, and restricts
to an anti-involution also denoted $\sigma$ on the subalgebra $\sym$. 
Let $?^\sigmadual$ be the duality on $\sym\Modfd$ taking a finite-dimensional
left $\sym$-module to its linear dual viewed again as a left module using the anti-automorphism $\sigma$. 
Since $\sigma(g) = g^{-1}$ for a permutation $g \in S_n \subset \sym$, 
this is the usual duality on each of the subcategories $\kk S_n\Modfd$.
It is well known that the irreducible $\kk S_n$-modules are self-dual, 
hence,
\begin{equation}\label{easydual}
S(\lambda)^\sigmadual \cong S(\lambda)
\end{equation}
for all $\lambda \in \P$.
There is also a duality $?^\sigmadual$ on $Par_t\Modlfd$ defined as in
\cref{duality}. Similarly, as $\sigma$ interchanges $Par^\flat$ and
$Par^\sharp$, we get contravariant equivalences also denoted $?^\sigmadual$ between
$Par^\sharp\Modlfd$ and $Par^\flat\Modlfd$.
Similarly to \cref{predual,almostwhatIneed}, we have that
\begin{align}
                                                 ?^\sigmadual \circ
                                                 \infl^\flat &\cong
                                                 \infl^\sharp \circ ?^\sigmadual,&
  \ind_{Par^\sharp}^{Par_t} \circ ?^\sigmadual &\cong ?^\sigmadual \circ
                                                 \coind_{Par^\flat}^{Par_t}.
\end{align}
Hence:
\begin{align}\label{harddual}
  j_! \circ ?^\sigmadual &\cong ?^\sigmadual \circ j_*,&
  j_* \circ ?^\sigmadual &\cong ?^\sigmadual \circ j_!
  \end{align}
as functors from
$\sym\Modfd$ to $Par_t\Modlfd$.
Then from \cref{easydual,harddual}, we deduce that
\begin{align}\label{harderdual}
\Delta(\lambda)^\sigmadual &\cong \nabla(\lambda),&
\nabla(\lambda)^\sigmadual &\cong \Delta(\lambda),&
L(\lambda)^\sigmadual \cong L(\lambda)
\end{align}
for $\lambda \in \P$.

\begin{remark}
In fact, the duality $?^\sigmadual$ is a {\em Chevalley duality} of $Par_t\Modlfd$ in the sense of \cite[Def.~4.49]{BS}. The general construction from \cite[Cor.~5.36,
Rem.~5.40]{BS} can be used to show that $Par_t$ admits a basis making
it into an upper finite 
symmetrically based quasi-hereditary algebra in the sense of \cite[Def.~5.1]{BS}.
Equivalently, $\Par_t$ is an object-adapted cellular category in the sense of \cite[Def.~2.1]{EL}.
\end{remark}

\subsection{The downward partition category and reduced Kronecker
  coefficients}
The results in this subsection are the analogs for the downward partition
category of the results proved in
\cite[$\S\S$7.5--7.6]{SSold} for the downward Brauer and downward
walled Brauer categories.
The methods are similar. In the first lemma, our standing assumption that
$\operatorname{char}\kk = 0$ is essential in
order to define the idempotent $1_{l,m,n}^d$.

\begin{lemma}[{cf. \cite[Prop.~6.5]{SaS}}]\label{ugly}
  The right $Par^\sharp \boxtimes Par^\sharp$-module $Par^\sharp 1_\star$
  is projective. Consequently, by \cref{indprodex}, the induction product
  $
  \ostar:Par^\sharp\Mod \boxtimes
  Par^\sharp\Mod
  \rightarrow Par^\sharp \Mod
  $
  is biexact. More precisely,
for integers $l,m,n \geq 0$ and $0 \leq d \leq \min(l+m-n,l+n-m,m+n-l)$ with $d \equiv
l+m+n\pmod{2}$, let
\begin{align}\label{abc}
a&:=(m+n-l-d)/2,&b&:=(l+n-m-d)/2,
&c&:=(l+m-n-d)/2
\end{align}
and define 
\begin{equation}\label{fdlmn}
 f^d_{l,m,n} := \begin{tikzpicture}[anchorbase,scale=1.5]
   \draw[-,thick] (-.1,-.9) to[looseness=2,out=90,in=90] (.1,-.9);
   \draw[-,thick] (-.2,-.9) to[looseness=2,out=90,in=90] (.2,-.9);
   \draw[-,thick] (-.4,-.8) to[out=90,in=-135] (-.1,-.4) to[out=-45,in=90] (.4,-.8);
   \draw[-,thick] (-.5,-.8) to[out=90,in=-135] (0,-.2) to[out=-45,in=90] (.5,-.8);
   \draw[-,thick] (-.6,-.8) to[out=90,in=-135] (.1,0) to[out=-45,in=90] (.6,-.8);
   \draw[-,thick] (-.1,-.4) to (-.1,.6);
   \draw[-,thick] (0,-.2) to (0,.6);
   \draw[-,thick] (.1,0) to (.1,.6);
   \draw[-,thick] (-.8,-1) to[out=90,in=-90] (-.3,.6);
   \draw[-,thick] (-.9,-1) to[out=90,in=-90] (-.4,.6);
   \draw[-,thick] (-1,-1) to[out=90,in=-90] (-.5,.6);
   \draw[-,thick] (.8,-1) to[out=90,in=-90] (.3,.6);
   \draw[-,thick] (.9,-1) to[out=90,in=-90] (.4,.6);
   \draw[-,thick] (1,-1) to[out=90,in=-90] (.5,.6);
   \draw[-,thick] (.6,-1) to (.6,-.8);
   \draw[-,thick] (.5,-1) to (.5,-.8);
   \draw[-,thick] (.4,-1) to (.4,-.8);
   \draw[-,thick] (.2,-1) to (.2,-.9);
   \draw[-,thick] (.1,-1) to (.1,-.9);
   \draw[-,thick] (-.1,-1) to (-.1,-.9);
   \draw[-,thick] (-.2,-1) to (-.2,-.9);
   \draw[-,thick] (-.4,-1) to (-.4,-.8);
   \draw[-,thick] (-.5,-1) to (-.5,-.8);
   \draw[-,thick] (-.6,-1) to (-.6,-.8);
\node at (0.15,-1.15) {$\scriptstyle a$}; 
\node at (-0.15,-1.15) {$\scriptstyle a$}; 
\node at (0.5,-1.15) {$\scriptstyle d$}; 
\node at (-0.5,-1.15) {$\scriptstyle d$}; 
\node at (0.9,-1.15) {$\scriptstyle b$}; 
\node at (-0.9,-1.15) {$\scriptstyle c$}; 
\node at (-0.4,.75) {$\scriptstyle c$}; 
\node at (0.4,.75) {$\scriptstyle b$}; 
\node at (0,.75) {$\scriptstyle d$}; 
\end{tikzpicture}
\in 1_l
 Par^\sharp 1_{m\star n}
\end{equation}
so that
there are $a$ nested caps at the
 bottom,
 $b$ parallel trunks on the right,
$c$ parallel trunks on the left,
and $d$ nested downward binary trees in the middle.
Also let $1_{l,m,n}^d$
be the image of the idempotent $\frac{1}{a!} \sum_{w \in S_a} w$
under the embedding $\kk S_a \hookrightarrow \kk (S_m \times S_n) <  1_m Par^\sharp
\boxtimes 1_n Par^\sharp$
sending $w \in S_a$ to the diagram representing the permutation
of
 $m+n,\dots,1+n,n,\dots,1$ defined by
 $1+n-i \mapsto 1+n-w(i)$, $i+n \mapsto w(i)+n$ for
$i=1,\dots,a$,
and fixing all other points (i.e., it arises from a permutation of the $a$
nested caps in the above picture).
Finally, let $S_l / S_c \times S_d \times S_b$ be a set of coset
representatives viewed as a subset of $\sym$ via the usual embedding
(in particular, $S_c$ and $S_b$ are permuting the leftmost $c$ and
rightmost $b$ strings, respectively).
 Then there is a right $Par^\sharp\boxtimes Par^\sharp$-module isomorphism
 $$
 \bigoplus_{l,m,n \geq 0}
 \bigoplus_{\substack{d=0\\d\equiv l+m+n\;(2)}}^{\min(l+m-n,l+n-m,m+n-l)}  \bigoplus_{g \in S_l / S_c \times S_d \times S_b}
 1_{l,m,n}^d (Par^\sharp \boxtimes Par^\sharp)
 \stackrel{\sim}{\rightarrow}  Par^\sharp 1_\star
  $$
   taking the idempotent $1_{l,m,n}^d$ in the $g$-th summand on the left
   hand side to $g \circ f^d_{l,m,n}$.
 \end{lemma}

 \begin{proof}
   This follows on considering the bases for $1_l Par^\sharp$
   and
   $1_m Par^\sharp \boxtimes 1_n Par^\sharp$ given by equivalence
   classes of downward partition diagrams.
 \end{proof}

   Recall for $\lambda,\mu,\nu \in \P_n$ that the {\em Kronecker
     coefficients} are defined to be the structure constants
   for the internal Kronecker product on representations of the symmetric group $S_n$:
   \begin{equation}\label{krondef}
     G^\lambda_{\mu,\nu} = [S(\mu)\otimes S(\nu):S(\lambda)]
 =  \dim \left(S(\lambda)\otimes S(\mu)\otimes
     S(\nu)\right)^{S_n}.
 \end{equation}
 Obviously from the second equality, $G^\lambda_{\mu,\nu}$ is invariant
 under permuting the partitions $\lambda,\mu,\nu$.
        For a partition $\lambda$ and $n \geq |\lambda|+\lambda_1$,
        let $\lambda(n)$ denote $(n-|\lambda|,\lambda_1,\lambda_2,\dots)\in \P_n$.
        By a classical result of Murnaghan, for $\lambda,\mu,\nu \in
        \P$, the value of the
        Kronecker coeffcient $G^{\lambda(n),}_{\mu(n),\nu(n)}$ stabilizes
        as $n \rightarrow \infty$; see \cite{BOR} for more background.
        The stable value is the {\em reduced Kronecker coefficient},
        denoted $\overline{G}^{\lambda}_{\mu,\nu}$. Like the Kronecker
        coefficients themselves, $\overline{G}^\lambda_{\mu,\nu}$ is
        invariant under permutations of $\lambda,\mu$ and $\nu$.

        \begin{example}\label{springs}
          For any $\lambda,\mu\in\P$, the reduced Kronecker coefficient
          $\overline{G}^{\mu}_{\boxvoid,\lambda}$ is equal to zero
          unless $\mu = \lambda + \boxed{a}$ for some $a \in \add(\lambda)$,
          $\mu = \lambda - \boxed{b}$ for some $b \in \rem(\lambda)$,
          or $\mu = (\lambda-\boxed{b})+\boxed{a}$ for some $b \in \rem(\lambda)$
          and $a = \add(\lambda-\boxed{b})$.
          In these cases, $\overline{G}^{\mu}_{\boxvoid,\lambda}$ is $1$ if $\mu \neq \lambda$ and $|\rem(\lambda)|$ if $\mu = \lambda$.
          To see this using
          the definition just given, consider
          the natural $n$-dimensional permutation module for the symmetric group
          $U_n \cong S((n-1,1))\oplus S((n))$  and note that
          the functor $U_n\otimes$ is isomorphic to $
          \ind_{S_{n-1}}^{S_n} \circ \res_{S_{n-1}}^{S_n}$.
        \end{example}
        
   \begin{lemma}\label{pugly}
     Take $l,m,n \geq 0$ and
     $0 \leq d \leq \min(l+m-n,l+n-m,m+n-l)$ with $d \equiv
l+m+n\pmod{2}$, and define $a,b,c$
         according to \cref{abc}.       Let $\Omega_{l,m,n}^d$ be the set of partitions of 
        $\{1,\dots,l\}\sqcup\{1',\dots,m'\}\sqcup\{1'',\dots,n''\}$
        such that exactly $a$ of the parts are subsets of the form
        $\{j',k''\}$, exactly $b$ of the parts are subsets of the form
        $\{i,k''\}$, exactly $c$ of the parts are subsets of the form
        $\{i,j'\}$, and the remaining $d$ parts are subsets of the
        form $\{i,j',k''\}$ for $i \in \{1,\dots,l\}$, $j' \in
        \{1',\dots,m'\}$ and $k'' \in \{1'',\dots,n''\}$ as suggested
        by the picture:
                $$
        
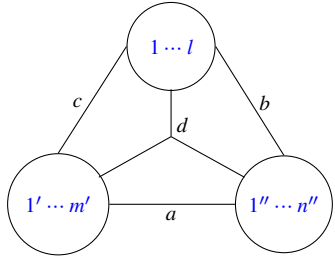
\begin{tikzpicture}[anchorbase,scale=1.5]
\node[align=center,draw,circle] at (0,.4) (a) {$\color{blue}\scriptstyle \:\:1
  \,\cdots \:l\:\:$};
\node[align=center,draw,circle] at (-1,-1) (b) {$\color{blue}\scriptstyle \:1'
  \,\cdots \:m'\:$};
\node[align=center,draw,circle] at (1,-1) (c) {$\color{blue}\scriptstyle 1''\,
  \cdots \:n''$};
\draw[-] (0,-.4) to (.65,-.76);
\draw[-] (0,-.4) to (-.64,-.76);
\draw (a.west) edge[-] (b.north);
\draw (b.east) edge[-] (c.west);
\draw (a.east) edge[-] (c.north);
\draw[-] (0,-.4) to (0,0.02);
\node at (0,-1.1){$\scriptstyle a$};
\node at (-.83,-.1){$\scriptstyle c$};
\node at (.83,-.1){$\scriptstyle b$};
\node at (0.1,-.3){$\scriptstyle d$};
\end{tikzpicture}
$$
        The group $S_l \times S_m \times S_n$ acts on the left on
        $\Omega_{l,m,n}^d$ so that $S_l$ permutes $\{1,\dots,l\}$, $S_m$
        permutes
        $\{1',\dots,m'\}$ and $S_n$ permutes $\{1'',\dots,n''\}$.
        Let $\kk \Omega_{l,m,n}^d$ be the linearization, which is a
        $\kk(S_l \times S_m \times S_n)$-module.
For $\lambda \in \P_l, \mu \in \P_m, \nu \in \P_n$ we
        have that
        $$
        [\kk \Omega_{l,m,n}^d: S(\lambda) \boxtimes S(\mu)\boxtimes
        S(\nu)]
        = \sum_{\substack{\alpha \in \P_a,\beta \in \P_b,\gamma \in \P_c\\\delta,\delta',\delta''
          \in \P_d}}
        LR^\lambda_{\beta,\gamma,\delta} LR^\mu_{\alpha,\gamma,\delta'}LR^\nu_{\alpha,\beta,\delta''}
G^\delta_{\delta',\delta''}.
$$
      \end{lemma}

      \begin{proof}
        Let $G := S_l \times S_m \times S_n$, $P$ be the parabolic
        subgroup $(S_b \times S_d \times S_c) \times (S_c\times S_d
        \times S_a) \times (S_a \times S_d \times S_b) \leq G$
        and $L$ be the subgroup $S_a \times S_b \times S_c \times S_d
        \leq P$
        embedded diagonally via the map
        $$
        (x,y,z,w) \mapsto (y,w,z;z,w,x;x,w,y).
        $$
        The action of $G$ on $\Omega_{l,m,n}^d$ is transitive and
the subgroup $L$ is a point stabilizer. Hence, the $\kk G$-module $\kk \Omega_{l,m,n}^d$
is induced from $\triv_L$, the trivial $\kk L$-module. By Frobenius reciprocity, it follows that
$$        [\kk \Omega_{l,m,n}^d: S(\lambda) \boxtimes S(\mu)\boxtimes
        S(\nu)] = \dim (S(\lambda) \boxtimes S(\mu) \boxtimes
        S(\nu))^L.
        $$
        The restriction of $S(\lambda)\boxtimes S(\mu) \boxtimes
        S(\nu)$ to the parabolic subgroup $P$ is isomorphic to
        \begin{multline*}
        \bigoplus_{\substack{\alpha,\alpha',\beta,\beta',\gamma,\gamma'\\\delta,\delta',\delta''
          \in \P_d}}\hspace{-4mm}
        \Big(\big(S(\beta) \boxtimes S(\delta) \boxtimes
S(\gamma')\big) \boxtimes
        \big(S(\gamma) \boxtimes S(\delta')\boxtimes S(\alpha')\big) \boxtimes
\big( S(\alpha) \boxtimes S(\delta'') \boxtimes S(\beta')\big)\Big)^{\oplus N}.
        \end{multline*}
        where $N:= LR^\lambda_{\beta,\delta,\gamma'} LR^\mu_{\gamma,\delta',\alpha'}LR^\nu_{\alpha,\delta''',\beta'}$.
By Schur's lemma, this contributes to the $L$-fixed points only from the summands
with $\alpha=\alpha',\beta=\beta'$ and $\gamma = \gamma'$, and for
each of those summands the contribution is $G^\delta_{\delta',\delta''}$ by \cref{krondef}.
        \end{proof}
        
   Since $Par^\sharp$ is negatively graded, any
   $Par^\sharp$-module $V$
   has the {\em degree filtration} $0 = V_{-1} < V_0 < \cdots < V_{n}
   < \cdots$ defined from $V_n :=
   \bigoplus_{m \leq n} 1_m V$. The section $V_n / V_{n-1}$ in this
   filtration is isomorphic to $\infl^\sharp (1_n V)$, viewing $1_n V$ as a
   $\sym$-module by restriction.
   It follows that
   \begin{equation}\label{timesaver}
     [V:S^\sharp(\lambda)]
     = [1_n V:S(\lambda)]
   \end{equation}
   for $\lambda \in \P_n$.
   
      \begin{theorem}\label{gcfs}
        For $\lambda \in \P_l, \mu \in \P_m$ and $\nu \in \P_n$, we
        have that
        $[S^\sharp(\mu)\ostar S^\sharp(\nu):S^\sharp(\lambda)]=\overline{G}^\lambda_{\mu,\nu}$.
      \end{theorem}

      \begin{proof}
        Call $d$ {\em admissible} if $0 \leq d \leq \min(l+m-n,l+n-m,m+n-l)$ and $d \equiv
        m+n+n\pmod{2}$. For such a $d$,
        let $M_{l,m,n}^d$ be the sub-bimodule of the $(\kk S_l, \kk S_m
        \boxtimes \kk S_n)$-bimodule
$$
M_{l,m,n} := 1_l
        \left(\big(\infl^\sharp \kk S_m\big) \ostar \big(\infl^\sharp
          \kk S_n\big)\right) = 1_l 
Par^\sharp 1_\star \otimes_{Par^\sharp \boxtimes Par^\sharp}
\left(\big(\infl^\sharp \kk S_m\big)
  \boxtimes \big(\infl^\sharp \kk S_n\big)\right)
$$
generated by the vector $f^d_{l,m,n} \otimes (1 \otimes 1)$, where
$f^d_{l,m,n}$ is as in \cref{fdlmn}.
\cref{ugly} implies that
$$
M_{l,m,n} = \bigoplus_{\text{admissible $\scriptstyle d$}} M_{l,m,n}^d.
$$
Note also that
$1_l\left(S^\sharp(\mu) \ostar S^\sharp(\nu)\right)\cong   M_{l,m,n} \otimes_{\kk S_m \boxtimes \kk S_n}
\left(S(\mu) \boxtimes S(\nu)\right)$
as $\kk S_l$-modules. In view of this and
\cref{timesaver}, the number we are trying to compute is equal to
$$
\left[
1_l \left(S^\sharp(\mu) \ostar S^\sharp(\nu)\right):S(\lambda)\right]=\sum_{\text{admissible $d$}} \left[M_{l,m,n}^d \otimes_{\kk S_m
  \boxtimes \kk S_n} (S(\mu)\boxtimes S(\nu)):S(\lambda)\right].
$$
Let $\widetilde{M}^d_{l,m,n}$ be the left $\kk S_l \boxtimes \kk S_m \boxtimes
        \kk S_n$-module obtained from the $(\kk S_l, \kk S_m\boxtimes
        \kk S_n)$-bimodule $M^d_{l,m,n}$ by twisting the right actions of $S_m$ and
        $S_n$ into left actions using $\sigma:g \mapsto g^{-1}$. 
        By the self-duality of Specht modules, we have that
        $$
        \left[M_{l,m,n}^d
        \otimes_{\kk S_m \boxtimes \kk S_n} (S(\mu)\boxtimes
        S(\nu)):S(\lambda)\right]
      =
      \left[\widetilde{M}^d_{l,m,n}:S(\lambda)\boxtimes S(\mu) \boxtimes
      S(\nu)\right].
    $$
    Now we claim that $\widetilde{M}^d_{l,m,n}$ is isomorphic to the
    module $\kk \Omega_{l,m,n}^d$ 
from \cref{pugly}. To see this, recall from the proof of that lemma
that $\kk \Omega_{l,m,n}^d$ is the permutation module induced from the
trivial representation of the subgroup $L = S_a \times S_b \times
S_c \times S_d < S_l \times S_m \times S_n$.
It is easy to see from \cref{fdlmn} that $L$ acts
trivially on the generating vector $f^d_{l,m,n} \otimes (1\otimes 1)
\in \widetilde{M}^d_{l,m,n}$. Hence, there is a surjective
homomorphism
$\kk \Omega^d_{l,m,n} \twoheadrightarrow \widetilde{M}^d_{l,m,n}$. It
is an isomorphism because both of these modules are of dimension
$(l!m!n!) / (a!b!c!d!)$.
From the claim, the previous displayed equation and \cref{pugly}, the problem is
reduced to computing
$$
\sum_{\text{admissible $d$}} [\kk \Omega^d_{l,m,n}:S(\lambda)\boxtimes
S(\mu)\boxtimes S(\nu)] =
\sum_{\alpha,\beta,\gamma,\delta,\delta',\delta'' \in \P}
        LR^\lambda_{\beta,\gamma,\delta} LR^\mu_{\alpha,\gamma,\delta'}LR^\nu_{\alpha,\beta,\delta''}
G^\delta_{\delta',\delta''}.
$$
       This expression is equal to the reduced Kronecker coefficient
       $\overline{G}^\lambda_{\mu,\nu}$ by a theorem of Littlewood
       \cite{Littlewood}.
     \end{proof}

     \subsection{Grothendieck rings}
     Next we describe the Grothendieck rings $K_0(Par^\sharp)$ and
     $K_0(Par_t)$.
     For $\lambda \in \P$, let
     \begin{equation}\label{Wed}
     P^\sharp(\lambda) := \ind_{\sym}^{Par^\sharp} S(\lambda).
     \end{equation}
     This is a finite-dimensional projective $Par^\sharp$-module. In fact, it is the
     projective cover of the irreducible $Par^\sharp$-module $S^\sharp(\lambda)$.
     This follows because $P^\sharp(\lambda) \twoheadrightarrow
     S^\sharp(\lambda)$, and
$\End_{Par^\sharp}(P^\sharp(\lambda)) \cong \End_{\sym}(S(\lambda))
\cong \kk$ so that $P^\sharp(\lambda)$ is indecomposable.
Thus, $K_0(Par^\sharp)$ is the free $\Z$-module with basis
$\big\{[P^\sharp(\lambda)]\:\big|\:\lambda \in \P\big\}$, and we see that
the monoidal functor $\ind_{\sym}^{Par^\sharp}$ induces a ring isomorphism
\begin{equation}\label{iso1}
K_0(\sym) \stackrel{\sim}{\rightarrow} K_0(Par^\sharp).
\end{equation}
Recall from \cref{ssym} that $K_0(\sym)$ is identified with the graded ring $\Lambda$ of symmetric functions.
Using \cref{iso1}, it follows that we can also identify $K_0(Par^\sharp)$
with $\Lambda$ so that $[P^\sharp(\lambda)]$
corresponds to the Schur function $s_\lambda$.
Let
\begin{equation}\label{cartanmx}
B_{\lambda,\mu} :=
\dim
\Hom_{Par^\sharp}(P^\sharp(\mu), P^\sharp(\lambda))
=[P^\sharp(\lambda):S^\sharp(\mu)],
\end{equation}
i.e., $(B_{\lambda,\mu})_{\lambda,\mu \in \P}$ is
the Cartan matrix of $Par^\sharp$.  

\begin{lemma}\label{textmessage}
We have that $B_{\lambda,\mu} = \dim e_\mu Par^\sharp
e_\lambda$ where $e_\lambda \in \sym$ denotes Young's
idempotent.
Hence:
\begin{itemize}
\item[(i)] $B_{\lambda,\lambda} = 1$ for every $\lambda \in \P$.
\item[(ii)] $B_{\lambda,\mu} = 0$ if $|\mu| >
|\lambda|$ or
if $|\mu| = |\lambda|$ and $\mu \neq \lambda$.
\item[(iii)]
  If $\lambda = (1^n)$ then
  $  B_{\lambda,\mu}
    = \left\{
      \begin{array}{ll}
        1&\text{if $\mu = (1^n)$ or $\mu = (1^{n-1})$}\\
        0&\text{for all other $\mu\in\P$.}
      \end{array}\right.
    $\end{itemize}
\end{lemma}

\begin{proof}
For $\lambda \in \P_n$, the primitive idempotent $e_\lambda \in \kk
S_n$ has the property that $S(\lambda) \cong \kk S_n
e_\lambda$. Hence,
$P^\sharp(\lambda) \cong Par^\sharp e_\lambda$. 
We deduce that
$$
B_{\lambda,\mu} \stackrel{\cref{cartanmx}}{=} 
\dim \Hom_{Par^\sharp}(P^\sharp(\mu), P^\sharp(\lambda))
=
\dim \Hom_{Par^\sharp}(Par^\sharp e_\mu, Par^\sharp e_\lambda)
= \dim e_\mu Par^\sharp e_\lambda.
$$
Parts (i) and (ii) follow easily using this formula.
For (iii), let $\lambda := (1^n)$.
Then 
$e_\lambda = \sum_{g \in S_n} (-1)^{\ell(g)} g$, so that any crossing
composed with $e_\lambda$ equals $-e_\lambda$.
The space $Par^\sharp e_\lambda$ is spanned by terms of the form $f e_\lambda$ for
downward partition diagrams $f \in Par^\sharp 1_n$. When $\lambda =
(1^n)$, it follows using the
anti-symmetry that
$f e_\lambda = 0$ if some connected component of $f$ is a downward branch
or a downward tree, or if $f$ has two components
that are downward leaves. So in this case $Par^\sharp e_\lambda$ is spanned just
by the vectors $e_\lambda$ and $(1_{n-1}\star c) e_\lambda$ where $c$ is a single downward leaf. Since these two
elements of $Par^\sharp$ lie in different weight spaces, they
are linearly independent, so $Par^\sharp e_\lambda$ is
exactly two-dimensional. It remains to observe that $\kk e_\lambda$ is
the sign representation $S(\lambda)$ of $S_n$ and $\kk (1_{n-1}\star
c) e_\lambda$ is the sign representation $S(\mu)$ of $S_{n-1}$ for $\mu = (1^{n-1})$.
\end{proof}

\cref{textmessage}(i)--(ii) shows that the Cartan matrix is unitriangular, hence, invertible.
It follows that the inclusion $Par^\sharp\Proj \rightarrow
Par^\sharp\Modfd$ induces an isomorphism 
\begin{equation}\label{iso2}
K_0(Par^\sharp)
\stackrel{\sim}{\rightarrow} K_0'(Par^\sharp),
\end{equation} 
where $K_0'(Par^\sharp)$ denotes
the Grothendieck group of the Abelian category
$Par^\sharp\Modfd$. By \cref{ugly}, we know that $\ostar$ is
biexact on $Par^\sharp\Modfd$, so it induces a multiplication making $K_0'(Par^\sharp)$
into a ring in such a way that \cref{iso2} is a ring isomorphism.
Using the isomorphisms \cref{iso1,iso2}, 
the canonical basis $\big\{[S^\sharp(\lambda)]\:\big|\:\lambda \in
\P\big\}$ of $K_0'(Par^\sharp)$ gives us
another basis $\{\tilde s_\lambda\:|\:\lambda \in \Lambda\}$ for
the ring $\Lambda$.  
We call these the {\em deformed Schur functions}.
We have that
\begin{align}
  s_\lambda &= \sum_{\mu \in \P} B_{\lambda,\mu} \tilde s_\mu,
  &\tilde s_\lambda &= \sum_{\mu \in \P} A_{\lambda,\mu} s_\mu,
\end{align}
where $(B_{\lambda,\mu})_{\lambda,\mu \in \P}$ is the Cartan matrix
from \cref{cartanmx} and $(A_{\lambda,\mu})_{\lambda,\mu \in \P}$ is the inverse matrix.
From the unitriangularity of the latter matrix,
it follows that $\tilde s_{\lambda}$ is equal to $s_\lambda$ plus a
linear combination of $s_\mu$ of strictly lower degree. In other
words, viewing the graded algebra $\Lambda$ as a filtered algebra with filtration induced
by the grading, the deformed Schur function $\tilde s_\lambda$ is in
filtered degree $n := |\lambda|$ and $\gr_n \tilde s_\lambda = s_\lambda$.
This justifies the name ``deformed Schur function''.
By \cref{gcfs} we have that
\begin{equation}\label{sconsts}
  \tilde s_\mu \tilde s_\nu = \sum_{\lambda \in \P}
  \overline{G}^{\lambda}_{\mu,\nu} \tilde s_\lambda,
\end{equation}
i.e., the reduced Kronecker coefficients are the structure constants
of $\Lambda$ in its inhomogeneous basis
arising from deformed Schur functions.
Comparing with \cref{LR}, we deduce that
$\overline{G}^\lambda_{\mu,\nu} = LR^\lambda_{\mu,\nu}$ if $|\lambda|
= |\mu|+|\nu|$ and $\overline{G}^\lambda_{\mu,\nu} = 0$ if $|\lambda|
> |\mu|+|\nu|$, both of which are well known properties of reduced Kronecker
coefficients.

\begin{remark}
  The deformed Schur function $\tilde
s_\lambda \in \Lambda$ as just defined is equal to the symmetric function $\tilde
s_\lambda$ defined in a different way \cite{OZ}. This follows from \cref{textmessage}(iii),
\cref{sconsts} and the characterization given in \cite[Th.~1(3)]{OZ}.
\end{remark}

It remains to pass from $K_0(Par^\sharp)$ to $K_0(Par_t)$.
Let $K_0''(Par_t)$ be
the Grothendieck group of the exact category $Par_t\Moddelta$.
It is the free Abelian group on basis
$\big\{[\Delta(\lambda)]\:\big|\:\lambda \in \P\big\}$. 
The following result implies that $K_0''(Par_t)$ is a ring with multiplication induced by
$\ostar$.

\begin{theorem}\label{lastfromss}
  For $V, W \in Par_t\Moddelta$, we have that
  $\Tor^{Par_t}_i(V,W) = 0$ for all $i > 0$, hence,
  $\ostar$ is biexact on $Par_t\Moddelta$.
  For $\mu,\nu \in \P$, there is a filtration
  $0 = V_{-1} < V_0 < \cdots < V_{|\mu|+|\nu|} = \Delta(\mu) \ostar \Delta(\nu)$ such that
  $$
  V_i / V_{i-1} \cong \bigoplus_{\lambda \in \P_i}
  \Delta(\lambda)^{\oplus \overline{G}_{\mu,\nu}^\lambda}
  $$
  where $\overline{G}_{\mu,\nu}^\lambda$ is the reduced Kronecker coefficient.
\end{theorem}

\begin{proof}
  The first statement is \cite[Cor.~6.6]{SaS}, which is deduced from
\cite[Props.~4.31--4.32]{SaS} using also
the exactness of $\ostar$ on $Par^\sharp\Modfd$
established in \cref{ugly}
together with exactness of the monoidal functor
$\ind_{Par^\sharp}^{Par_t}$.
The second statement follows by applying
$\ind_{Par^\sharp}^{Par_t}$ to the degree filtation of
$S^\sharp(\mu)\ostar S^\sharp(\nu)$, using \cref{gcfs} which computes
the multiplicities.
\end{proof}

Now consider the following commutative diagram of rings and ring
homomorphisms, with maps
induced by the indicated biexact monoidal functors:
\begin{equation}\label{adiag}
\begin{tikzcd}
&  K_0(Par^\sharp)\arrow[dd,"\ind_{Par^\sharp}^{Par_t}"{right}]\arrow[r,"\operatorname{inc}"]&K_0'(Par^\sharp)\arrow[dd,"\ind_{Par^\sharp}^{Par_t}"{right}]\\
\Lambda \cong K_0(\sym)\arrow[ur,"\ind_{\sym}^{Par^\sharp}"{near
  start, above}]\arrow[dr,"\ind_{\sym}^{Par_t}"{below, near start}]&&\\
&K_0(Par_t)\arrow[r,"\operatorname{inc}"{below}]&K_0''(Par_t).
\end{tikzcd}
\end{equation}
(Recall $K_0$ is the split Grothendieck group
of finitely generated projectives, $K_0'(Par^\sharp)$ is the Grothendieck
group of the Abelian category $Par^\sharp\Modfd$, $K_0''(Par_t)$
is the Grothendieck group of the exact category $Par_t\Moddelta$,
and all of these Grothendieck groups are actually rings with multiplication induced by $\ostar$.)

\begin{theorem}\label{GGthm}
  All of the arrows in \cref{adiag} are isomorphisms, so that
all of the Grothendieck rings in this diagram are identified with
$\Lambda$.
\end{theorem}

\begin{proof}
  We already established this for the top two arrows in \cref{iso1,iso2}. It is immediate for the arrow on the right since it
takes basis element $[S^\sharp(\lambda)]$ to basis element
$[\Delta(\lambda)]$. The fact that the bottom arrow is an isomorphism
is a general property of upper finite highest weight
categories. Indeed, we have that
\begin{equation}\label{cats}
  [P(\lambda)] = [\Delta(\lambda)]+\text{(a sum of $[\Delta(\mu)]$ for
    $\mu$ with $|\mu| < |\lambda|$)},
\end{equation}
equality in $K_0''(Par_t)$. Hence, the transition matrix between
the image of the canonical basis for $K_0(Par_t)$ and the standard
basis for $K_0''(Par_t)$ is
invertible, as required to see that the bottom map is an isomorphism.
We deduce that the other two arrows are isomorphisms too using the
commutativity of the diagram.
\end{proof}

From \cref{GGthm},
we see that there are {\em three} natural basis for
$K_0(Par_t)$:
\begin{itemize}
\item
  The canonical basis $\big\{[P(\lambda)]\:\big|\:\lambda
  \in \P\big\}$ arising from the indecomposable projectives.
  \item The basis
$\big\{[\Delta(\lambda)\:\big|\:\lambda \in \P\big\}$ arising from the
standard basis for $K_0''(Par_t)$ via the isomorphism that is the
bottom arrow of \cref{adiag}.
\item The basis $\big\{[Q(\lambda)]\:|\:\lambda \in \P\big\}$
where
$Q(\lambda) := \ind_{\sym}^{Par_t} S(\lambda) \cong
  \ind_{\Par^\sharp}^{Par_t} S^\sharp(\lambda)$.
\end{itemize}
Note $Q(\lambda)$ is a finitely generated projective $Par_t$-module which is usually
decomposable. In fact
\begin{equation}
  Q(\lambda) \cong P(\lambda) \oplus \text{(a finite direct sum of
    $P(\mu)$ for $\mu$ with $|\mu|< |\lambda|$)},
\end{equation}
as follows from \cref{cats} and the following lemma.

\begin{lemma}\label{bioap}
For $\lambda \in \P_n$, the $Par_t$-module $V := Q(\lambda)$ has a filtration
$0 = V_{-1} < V_0 < \cdots < V_n = V$ such that
$$
V_i / V_{i-1} \cong \bigoplus_{\mu \in \P_i} \Delta(\mu)^{\oplus
  B_{\lambda,\mu}}.
$$
In particular, $V_n / V_{n-1} \cong \Delta(\lambda)$.
\end{lemma}

\begin{proof}
 Recalling the definition \cref{cartanmx}, this follows by applying the exact functor
 $\ind_{Par^\sharp}^{Par_t}$ to the degree filtration of $P^\sharp(\lambda)$,
 also using \cref{textmessage}.
\end{proof}

Under the identification of $K_0(Par_t)$ with $\Lambda$,
the isomorphism classes $[Q(\lambda)]$ correspond
to the Schur functions $s_\lambda \in \Lambda$,
and the isomorphism classes
$[\Delta(\lambda)]$ correspond to the deformed Schur functions
$\tilde s_\lambda$. These statements are both clear from
our previous discussion of $K_0(Par^\sharp)$. 
The $Q$ and $\Delta$ bases for $K_0(Par_t)$ are independent of the value of the parameter $t$,
whereas the $P$ basis coming from indecomposable projectives
undoubtedly does depend on $t$.
For values of $t$ such that
$Par_t$ is semisimple (see \cref{sscor} below), we have that $P(\lambda) = \Delta(\lambda) =
\nabla(\lambda) = L(\lambda)$, and \cref{lastfromss} implies that
\begin{equation}\label{innas}
\Delta(\mu) \ostar \Delta(\nu) \cong \bigoplus_{\lambda \in \P}
\Delta(\lambda)^{\oplus \overline{G}^\lambda_{\mu,\nu}}.
\end{equation}
This was established before in \cite{EA};
see also \cite[Lem~5.14]{CO} and 
\cite[Cor.~3.2.2]{BDO}.

\section{Jucys-Murphy elements via the affine partition category}\label{sec4}

Next, we introduce an auxiliary monoidal category $\APar$, the {\em
  affine partition category}. We define this as a
certain monoidal subcategory of the Heisenberg category $\Heis$,
exploiting an observation of Likeng and Savage from \cite{LSR}.
We then use $\APar$ to give a new approach to the definition of
the Jucys-Murphy elements of $Par_t$.
These were first defined in the context of the partition algebra
by Halverson and Ram \cite{HR} and computed recursively by Enyang \cite{En}.
We also construct more general central elements.

\subsection{Schur-Weyl duality}
Recall the generators and relations for the partition category from \cref{Par}.
The following theorem of Deligne will play a key role in this section;
see e.g. \cite[Th.~2.3]{Com} for a proof.

\begin{theorem}\label{SWD}
  Suppose that $t \in \N$.
  Let $\Nat$ be the
  natural permutation representation of the symmetric group $S_t$
  with standard basis $u_1,\dots,u_t$.
  Viewing $\kk S_t\Modfd$ as a symmetric monoidal category via the usual Kronecker
  tensor product $\otimes$,
there is full $\kk$-linear symmetric monoidal functor 
$\psi_t:\Par_t\to\kk S_t\Modfd$
sending the generating object $\mid$ to $\Nat$
and defined on generating morphisms by
\begin{align*}
\psi_t\big(\begin{tikzpicture}[anchorbase]\draw[-,thick](0,0)--(0.4,0.4);\draw[-,thick](0,0.4)--(0.4,0);\end{tikzpicture}\big)&:
  \Nat\otimes \Nat \to \Nat\otimes \Nat\ ,
&u_i\otimes u_j&\mapsto u_j\otimes u_i\ ,\\ 
\psi_t\big(\begin{tikzpicture}[anchorbase]\draw[-,thick](0,0)--(0.2,0.2)--(0.2,0.4);\draw[-,thick](0.2,0.2)--(0.4,0);\end{tikzpicture}\big)&:
  \Nat\otimes \Nat\to \Nat, &
u_i\otimes u_j &\mapsto \delta_{i,j} u_i\ ,\\
                      \psi_t\big(\begin{tikzpicture}[anchorbase]\draw[-,thick](0,0.4)--(0.2,0.2)--(0.2,0);\draw[-,thick](0.2,0.2)--(0.4,0.4);\end{tikzpicture}\big)&: \Nat\to \Nat\otimes \Nat,&
                                                                                                                                                                         u_i&\mapsto u_i\otimes u_i\ ,\\
   \psi_t\big(\begin{tikzpicture}[anchorbase]\draw[-,thick](0,0)--(0,0.25);\node at (0,0.3) {$\dt$};\end{tikzpicture}\big)&:\Nat\to\kk\ ,&u_i&\mapsto 1\ ,\\
\psi_t\big(\begin{tikzpicture}[anchorbase]\draw[-,thick](0,0.4)--(0,0.18);\node at (0,0.1) {$\dt$};\end{tikzpicture}\big)&: \kk\to \Nat
& 1 &\mapsto u_1+\cdots+u_t\ .                                                                                                                                                                              \end{align*}
Furthermore, the linear map $\Hom_{\Par_t}(n,
m)\to\Hom_{\kk S_t}(\Nat^{\otimes n}, \Nat^{\otimes m}), f \mapsto \psi_t(f)$ is an isomorphism whenever $t\ge m+n$.
\end{theorem}

For the next corollary, we assume some basic facts
about semisimplification of monoidal categories;
e.g., see \cite[Sec.~2]{BEEO} which gives a concise summary of everything
needed here.

\begin{corollary}\label{SWDcor}
  When $t \in \N$, the
  functor $\psi_t$ induces a monoidal equivalence $\overline{\psi}_t$ between the
  semisimplification of $\Kar(\Par_t)$ and $\kk S_t\Modfd$.
  In particular, $Par_t$ is {\em not} a semisimple locally unital algebra
  in these cases.
\end{corollary}

\begin{proof}
  The functor $\psi_t$ extends canonically to a functor
  $\Kar(\Par_t) \rightarrow \kk S_t\Modfd$. It is well known that
every irreducible
$\kk S_t$-module appears as a constituent of some tensor power of $\Nat$, hence, this functor is dense.  Now
  the first statement follows from the fullness of the functor $\psi_t$ using
  \cite[Lem.~2.6]{BEEO}; see also \cite[Th.~2.18]{Del} and \cite[Th.~3.24]{CO}.
  Since $\Kar(\Par_t)$ has infinitely many isomorphism classes of irreducible objects, it is definitely not equivalent to its semisimplification
  $\kk S_t\Modfd$. This shows that $\Kar(\Par_t)$ is not a semisimple Abelian category as it
  contains non-zero negligible morphisms. Equivalently, the path algebra
  $Par_t$ is not semisimple in these cases. 
\end{proof}

\begin{remark}\label{corerem}
  Continue to assume that $t \in \N$.
  By the general theory of semisimplification,
  the irreducible objects in the semisimplification of $\Kar(\Par_t)$
correspond to the indecomposable projective $Par_t$-modules $P(\lambda)$
of non-zero categorical dimension. In \cite[Prop.~6.4]{Del}, Deligne showed
that $P(\lambda)$ has non-zero categorical dimension if and only if $t-|\lambda| > \lambda_1-1$, in which case the irreducible
object of the semisimpliciation arising from $P(\lambda)$ corresponds
under the equivalence $\overline{\psi}_t$
to the irreducible $\kk S_t$-module $S(\kappa)$
where $\kappa := (t-|\lambda|,\lambda_1,\lambda_2,\dots)$.
  \end{remark}

The {\em generic partition category} $\Par$ is the strict $\kk$-linear monoidal category with the same generating object and generating morphisms as $\Par_t$ subject to all of the same relations except for the
final relation in \cref{relsparam}, which is omitted. The morphism
\begin{equation}\label{week8}
T := 
\begin{tikzpicture}[anchorbase]
  \draw[-,thick] (0,0.05)--(0,0.55);
  \node at (0,0) {$\dt$};
  \node at (0,0.6) {$\dt$};
  \end{tikzpicture} \in \End_\Par(\one)\
\end{equation}
is strictly central in $\Par$, so that
$\Par$ can be viewed as a
$\kk[T]$-linear monoidal category.
For $t \in \kk$, let
\begin{equation}\label{specialization}
\ev_t:\Par \rightarrow \Par_t
\end{equation}
be the canonical functor taking $T$ to $t 1_\one$.
Using the basis theorem for $\Par_t$ for infinitely many values of $t$,
one obtains a basis theorem for the generic partition category:
each morphism space
$\Hom_\Par(n, m)$ is free as a $\kk[T]$-module
with basis given by a set of representatives for the equivalence classes of $m \times n$ partition diagrams.
From this, we see that $\ev_t$ induces an isomorphism
$\kk \otimes_{\kk[T]} \Par \cong \Par_t$,
where on the left hand side we are viewing $\kk$ as a $\kk[T]$-module so that $T$ acts as $t$.
This point of view is often useful since it can be used to 
prove a statement involving relations in
$\Par_t$ for all values of $t$ just by checking it for all sufficiently large positive integers, in which case \cref{SWD} can often be applied to reduce to a question about
symmetric groups.
To make a precise statement, let
\begin{equation}\label{cucumber}
  \phi_t := \psi_t \circ \ev_t : \Par \rightarrow \kk S_t\Modfd,
\end{equation}
assuming $t \in \N$.

\begin{lemma}\label{SWDapp}
  If $f \in \Hom_{\Par}(n, m)$ satifies $\phi_t(f) = 0$
  for infinitely many values of $t \in \N$
  then $f = 0$.
\end{lemma}

\begin{proof}
  We can write $f = \sum_i p_i(T) f_i$ for polynomials
  $p_i(T) \in \kk[T]$ and $f_i$ running over a set of representatives for the equivalence classes of $m\times n$ partition diagrams.
  Since $\phi_t(f) = 0$ we have that $\sum_i p_i(t) \phi_t(f_i) = 0$ for infinitely many values of $t$. By the final assertion in \cref{SWD}, this implies
  that $\sum_i p_i(t) \ev_t(f_i) = 0$ for infinitely many values of $t \geq m+n$.
  By the basis theorem in $\Par_t$, this means for each $i$ that $p_i(t) = 0$ for infinitely many values of $t$.
  Hence, $p_i(T) = 0$ for each $i$.
  \end{proof}

We note that the proof of \cref{SWDapp} depends on our standing assumption that the ground field $\kk$
is of characteristic zero.

\subsection{Heisenberg category}
Next we recall the definition of the Heisenberg category $\Heis$
which was introduced by Khovanov in \cite{K}. We follow the
  approach of \cite{Bheis}; Khovanov's category is denoted
  $\Heis_{-1}(0)$ in the more general setup developed there.
 
  \begin{definition}[{\cite[Rem.~1.5(2)]{Bheis}}] The Heisenberg category $\Heis$ is the
  strict $\kk$-linear monoidal category with two generating objects $\uparrow$ and
  $\downarrow$ and five generating morphisms
\begin{align*}
\mathord{
\begin{tikzpicture}[baseline = 0]
	\draw[->] (0.28,-.3) to (-0.28,.4);
	\draw[->] (-0.28,-.3) to (0.28,.4);
\end{tikzpicture}
}&:\uparrow\star \uparrow \rightarrow \uparrow \star \uparrow
   \:,
  &\mathord{
\begin{tikzpicture}[baseline = 1mm]
	\draw[<-] (0.4,0.4) to[out=-90, in=0] (0.1,0);
	\draw[-] (0.1,0) to[out = 180, in = -90] (-0.2,0.4);
\end{tikzpicture}
}&:\one\rightarrow \downarrow\star \uparrow
\:,
&\mathord{
\begin{tikzpicture}[baseline = 1mm]
	\draw[<-] (0.4,0) to[out=90, in=0] (0.1,0.4);
	\draw[-] (0.1,0.4) to[out = 180, in = 90] (-0.2,0);
\end{tikzpicture}
}&:\uparrow \star \downarrow\rightarrow\one\:,
&
\mathord{
\begin{tikzpicture}[baseline = 1mm]
	\draw[-] (0.4,0.4) to[out=-90, in=0] (0.1,0);
	\draw[->] (0.1,0) to[out = 180, in = -90] (-0.2,0.4);
\end{tikzpicture}
}&:\one\rightarrow \uparrow\star \downarrow
\:,
&
\mathord{
\begin{tikzpicture}[baseline = 1mm]
	\draw[-] (0.4,0) to[out=90, in=0] (0.1,0.4);
	\draw[->] (0.1,0.4) to[out = 180, in = 90] (-0.2,0);
\end{tikzpicture}
}&:\downarrow \star \uparrow\rightarrow\one,
\end{align*}
subject to the following relations:
\begin{align}\label{dAHA}
\mathord{
\begin{tikzpicture}[baseline = -1mm]
	\draw[->] (0.28,0) to[out=90,in=-90] (-0.28,.6);
	\draw[->] (-0.28,0) to[out=90,in=-90] (0.28,.6);
	\draw[-] (0.28,-.6) to[out=90,in=-90] (-0.28,0);
	\draw[-] (-0.28,-.6) to[out=90,in=-90] (0.28,0);
\end{tikzpicture}
}&=
\mathord{
\begin{tikzpicture}[baseline = -1mm]
	\draw[->] (0.18,-.6) to (0.18,.6);
	\draw[->] (-0.18,-.6) to (-0.18,.6);
\end{tikzpicture}
}\:,
&\mathord{
\begin{tikzpicture}[baseline = -1mm]
	\draw[<-] (0.45,.6) to (-0.45,-.6);
	\draw[->] (0.45,-.6) to (-0.45,.6);
        \draw[-] (0,-.6) to[out=90,in=-90] (-.45,0);
        \draw[->] (-0.45,0) to[out=90,in=-90] (0,0.6);
\end{tikzpicture}
}
&=
\mathord{
\begin{tikzpicture}[baseline = -1mm]
	\draw[<-] (0.45,.6) to (-0.45,-.6);
	\draw[->] (0.45,-.6) to (-0.45,.6);
        \draw[-] (0,-.6) to[out=90,in=-90] (.45,0);
        \draw[->] (0.45,0) to[out=90,in=-90] (0,0.6);
\end{tikzpicture}
}\:,\\\label{rightadj}
\mathord{
\begin{tikzpicture}[baseline = -.8mm]
  \draw[->] (0.3,0) to (0.3,.4);
	\draw[-] (0.3,0) to[out=-90, in=0] (0.1,-0.4);
	\draw[-] (0.1,-0.4) to[out = 180, in = -90] (-0.1,0);
	\draw[-] (-0.1,0) to[out=90, in=0] (-0.3,0.4);
	\draw[-] (-0.3,0.4) to[out = 180, in =90] (-0.5,0);
  \draw[-] (-0.5,0) to (-0.5,-.4);
\end{tikzpicture}
}
&=
\mathord{\begin{tikzpicture}[baseline=-.8mm]
  \draw[->] (0,-0.4) to (0,.4);
\end{tikzpicture}
}\:,
&\mathord{
\begin{tikzpicture}[baseline = -.8mm]
  \draw[->] (0.3,0) to (0.3,-.4);
	\draw[-] (0.3,0) to[out=90, in=0] (0.1,0.4);
	\draw[-] (0.1,0.4) to[out = 180, in = 90] (-0.1,0);
	\draw[-] (-0.1,0) to[out=-90, in=0] (-0.3,-0.4);
	\draw[-] (-0.3,-0.4) to[out = 180, in =-90] (-0.5,0);
  \draw[-] (-0.5,0) to (-0.5,.4);
\end{tikzpicture}
}
&=
\mathord{\begin{tikzpicture}[baseline=-.8mm]
  \draw[<-] (0,-0.4) to (0,.4);
\end{tikzpicture}
}\:,  \\
  \label{leftcurl}
\mathord{
\begin{tikzpicture}[baseline = -0.5mm]
	\draw[<-] (0,0.6) to (0,0.3);
	\draw[-] (0,0.3) to [out=-90,in=0] (-.3,-0.2);
	\draw[-] (-0.3,-0.2) to [out=180,in=-90](-.5,0);
	\draw[-] (-0.5,0) to [out=90,in=180](-.3,0.2);
	\draw[-] (-0.3,.2) to [out=0,in=90](0,-0.3);
	\draw[-] (0,-0.3) to (0,-0.6);
\end{tikzpicture}
}&=0,&
\!\!\!\!\mathord{
\begin{tikzpicture}[baseline = 1.25mm]
  \draw[->] (0.2,0.2) to[out=90,in=0] (0,.4);
  \draw[-] (0,0.4) to[out=180,in=90] (-.2,0.2);
\draw[-] (-.2,0.2) to[out=-90,in=180] (0,0);
  \draw[-] (0,0) to[out=0,in=-90] (0.2,0.2);
\end{tikzpicture}
}\!&= 1_\one,
\\\label{sidewaysp}
\mathord{
\begin{tikzpicture}[baseline = 0mm]
	\draw[-] (0.28,0) to[out=90,in=-90] (-0.28,.6);
	\draw[->] (-0.28,0) to[out=90,in=-90] (0.28,.6);
	\draw[-] (0.28,-.6) to[out=90,in=-90] (-0.28,0);
	\draw[<-] (-0.28,-.6) to[out=90,in=-90] (0.28,0);
\end{tikzpicture}
}
&=
\mathord{
\begin{tikzpicture}[baseline = 0]
	\draw[->] (0.08,-.6) to (0.08,.6);
	\draw[<-] (-0.28,-.6) to (-0.28,.6);
\end{tikzpicture}
}
    -
\mathord{
\begin{tikzpicture}[baseline=-.5mm]
	\draw[<-] (0.3,0.6) to[out=-90, in=0] (0,.1);
	\draw[-] (0,.1) to[out = 180, in = -90] (-0.3,0.6);
 	\draw[-] (0.3,-.6) to[out=90, in=0] (0,-0.1);
	\draw[->] (0,-0.1) to[out = 180, in = 90] (-0.3,-.6);
\end{tikzpicture}}\:,&
\!\!\!\!\mathord{
\begin{tikzpicture}[baseline = 0mm]
	\draw[->] (0.28,0) to[out=90,in=-90] (-0.28,.6);
	\draw[-] (-0.28,0) to[out=90,in=-90] (0.28,.6);
	\draw[<-] (0.28,-.6) to[out=90,in=-90] (-0.28,0);
	\draw[-] (-0.28,-.6) to[out=90,in=-90] (0.28,0);
\end{tikzpicture}
}
&=\mathord{
\begin{tikzpicture}[baseline = 0]
	\draw[<-] (0.08,-.6) to (0.08,.6);
	\draw[->] (-0.28,-.6) to (-0.28,.6);
\end{tikzpicture}
}
\:.
\end{align}
Here, we have used the
the sideways crossings which are defined from
\begin{align*}
\mathord{
\begin{tikzpicture}[baseline = 0]
	\draw[<-] (0.28,-.3) to (-0.28,.4);
	\draw[->] (-0.28,-.3) to (0.28,.4);
\end{tikzpicture}
}
&:=
\mathord{
\begin{tikzpicture}[baseline = 0]
	\draw[->] (0.3,-.5) to (-0.3,.5);
	\draw[-] (-0.2,-.2) to (0.2,.3);
        \draw[-] (0.2,.3) to[out=50,in=180] (0.5,.5);
        \draw[->] (0.5,.5) to[out=0,in=90] (0.9,-.5);
        \draw[-] (-0.2,-.2) to[out=230,in=0] (-0.6,-.5);
        \draw[-] (-0.6,-.5) to[out=180,in=-90] (-0.9,.5);
\end{tikzpicture}
}\:,&
\mathord{
\begin{tikzpicture}[baseline = 0]
	\draw[->] (0.28,-.3) to (-0.28,.4);
	\draw[<-] (-0.28,-.3) to (0.28,.4);
\end{tikzpicture}
}
&:=
\mathord{
\begin{tikzpicture}[baseline = 0]
	\draw[<-] (0.3,.5) to (-0.3,-.5);
	\draw[-] (-0.2,.2) to (0.2,-.3);
        \draw[-] (0.2,-.3) to[out=130,in=180] (0.5,-.5);
        \draw[-] (0.5,-.5) to[out=0,in=270] (0.9,.5);
        \draw[-] (-0.2,.2) to[out=130,in=0] (-0.6,.5);
        \draw[->] (-0.6,.5) to[out=180,in=-270] (-0.9,-.5);
\end{tikzpicture}
}\:.
\end{align*}
\end{definition}

It is also convenient to introduce the shorthand
\begin{equation}\label{shorthand}
\mathord{
\begin{tikzpicture}[anchorbase]
	\draw[->] (0.08,-.4) to (0.08,.5);
      \node at (0.08,0.05) {$\dt$};
\end{tikzpicture}
}:=
\mathord{
  \begin{tikzpicture}[anchorbase]
    	\draw[<-] (0,0.6) to (0,0.3);
	\draw[-] (0,0.3) to [out=-90,in=180] (.3,-0.2);
	\draw[-] (0.3,-0.2) to [out=0,in=-90](.5,0);
	\draw[-] (0.5,0) to [out=90,in=0](.3,0.2);
	\draw[-] (0.3,.2) to [out=180,in=90](0,-0.3);
	\draw[-] (0,-0.3) to (0,-0.6);
\end{tikzpicture}
}\:,
\end{equation}
which automatically satisfies the degenerate affine Hecke algebra relation as in \cref{daharel}:
\begin{align}\label{someoneelse}
\mathord{
\begin{tikzpicture}[baseline = -1mm]
	\draw[<-] (0.25,.3) to (-0.25,-.3);
	\draw[->] (0.25,-.3) to (-0.25,.3);
 \node at (-0.12,-0.145) {$\dt$};
\end{tikzpicture}
}
&=
\mathord{
\begin{tikzpicture}[baseline = -1mm]
	\draw[<-] (0.25,.3) to (-0.25,-.3);
	\draw[->] (0.25,-.3) to (-0.25,.3);
     \node at (0.12,0.135) {$\dt$};
\end{tikzpicture}}
+\:\mathord{
\begin{tikzpicture}[baseline = -1mm]
 	\draw[->] (0.08,-.3) to (0.08,.3);
	\draw[->] (-0.28,-.3) to (-0.28,.3);
\end{tikzpicture}
}\ ,&
\mathord{
\begin{tikzpicture}[baseline = -1mm]
	\draw[<-] (0.25,.3) to (-0.25,-.3);
	\draw[->] (0.25,-.3) to (-0.25,.3);
 \node at (-0.12,0.135) {$\dt$};
\end{tikzpicture}
}
&=
\mathord{
\begin{tikzpicture}[baseline = -1mm]
	\draw[<-] (0.25,.3) to (-0.25,-.3);
	\draw[->] (0.25,-.3) to (-0.25,.3);
     \node at (0.12,-0.145) {$\dt$};
\end{tikzpicture}}
+\:\mathord{
\begin{tikzpicture}[baseline = -1mm]
 	\draw[->] (0.08,-.3) to (0.08,.3);
	\draw[->] (-0.28,-.3) to (-0.28,.3);
\end{tikzpicture}
}\:.
\end{align}
Note by \cref{leftcurl} that
\begin{equation}
\mathord{\begin{tikzpicture}[baseline = -1mm]
  \draw[-] (0,0.2) to[out=180,in=90] (-.2,0);
  \draw[->] (0.2,0) to[out=90,in=0] (0,.2);
 \draw[-] (-.2,0) to[out=-90,in=180] (0,-0.2);
  \draw[-] (0,-0.2) to[out=0,in=-90] (0.2,0);
      \node at (0.2,0) {$\dt$};
\end{tikzpicture}
} = \begin{tikzpicture}[baseline = -1mm]
  \draw[->] (0,0.2) to[out=180,in=90] (-.2,0) to[out=-90,in=180] (0,-.2) to [out=0,in=180] (.5,.2) to [out=0,in=90] (.7,0) to [out=-90,in=0] (.5,-.2) to [out=180,in=0] (0,.2);
\end{tikzpicture}=0.\label{misstate}
\end{equation}
In addition, the following relations hold, so that $\Heis$ is strictly
pivotal with duality functor defined by rotating diagrams through $180^\circ$:
\begin{align}
\mathord{
\begin{tikzpicture}[baseline = -.8mm]
  \draw[-] (0.3,0) to (0.3,-.4);
	\draw[-] (0.3,0) to[out=90, in=0] (0.1,0.4);
	\draw[-] (0.1,0.4) to[out = 180, in = 90] (-0.1,0);
	\draw[-] (-0.1,0) to[out=-90, in=0] (-0.3,-0.4);
	\draw[-] (-0.3,-0.4) to[out = 180, in =-90] (-0.5,0);
  \draw[->] (-0.5,0) to (-0.5,.4);
\end{tikzpicture}
}
&=
\mathord{\begin{tikzpicture}[baseline=-.8mm]
  \draw[->] (0,-0.4) to (0,.4);
\end{tikzpicture}
}\:,
&\mathord{
\begin{tikzpicture}[baseline = -.8mm]
  \draw[-] (0.3,0) to (0.3,.4);
	\draw[-] (0.3,0) to[out=-90, in=0] (0.1,-0.4);
	\draw[-] (0.1,-0.4) to[out = 180, in = -90] (-0.1,0);
	\draw[-] (-0.1,0) to[out=90, in=0] (-0.3,0.4);
	\draw[-] (-0.3,0.4) to[out = 180, in =90] (-0.5,0);
  \draw[->] (-0.5,0) to (-0.5,-.4);
\end{tikzpicture}
}
&=
\mathord{\begin{tikzpicture}[baseline=-.8mm]
  \draw[<-] (0,-0.4) to (0,.4);
\end{tikzpicture}
}\:,\label{leftadj}\\\label{mates}
\mathord{
\begin{tikzpicture}[baseline = -.5mm]
	\draw[<-] (0.08,-.3) to (0.08,.4);
      \node at (0.08,0.05) {$\dt$};
\end{tikzpicture}
}:=
\mathord{
\begin{tikzpicture}[baseline = -.5mm]
  \draw[->] (0.3,0) to (0.3,-.4);
	\draw[-] (0.3,0) to[out=90, in=0] (0.1,0.4);
	\draw[-] (0.1,0.4) to[out = 180, in = 90] (-0.1,0);
	\draw[-] (-0.1,0) to[out=-90, in=0] (-0.3,-0.4);
	\draw[-] (-0.3,-0.4) to[out = 180, in =-90] (-0.5,0);
  \draw[-] (-0.5,0) to (-0.5,.4);
   \node at (-0.1,0) {$\dt$};
\end{tikzpicture}
}&=\:
\mathord{
\begin{tikzpicture}[baseline = -.5mm]
  \draw[-] (0.3,0) to (0.3,.4);
	\draw[-] (0.3,0) to[out=-90, in=0] (0.1,-0.4);
	\draw[-] (0.1,-0.4) to[out = 180, in = -90] (-0.1,0);
	\draw[-] (-0.1,0) to[out=90, in=0] (-0.3,0.4);
	\draw[-] (-0.3,0.4) to[out = 180, in =90] (-0.5,0);
  \draw[->] (-0.5,0) to (-0.5,-.4);
   \node at (-0.1,0) {$\dt$};
\end{tikzpicture}
}\:,
&\mathord{
\begin{tikzpicture}[baseline = 0]
	\draw[<-] (0.28,-.3) to (-0.28,.4);
	\draw[<-] (-0.28,-.3) to (0.28,.4);
\end{tikzpicture}
}
:=
\mathord{
\begin{tikzpicture}[baseline = 0]
	\draw[<-] (0.3,-.5) to (-0.3,.5);
	\draw[-] (-0.2,-.2) to (0.2,.3);
        \draw[-] (0.2,.3) to[out=50,in=180] (0.5,.5);
        \draw[->] (0.5,.5) to[out=0,in=90] (0.9,-.5);
        \draw[-] (-0.2,-.2) to[out=230,in=0] (-0.6,-.5);
        \draw[-] (-0.6,-.5) to[out=180,in=-90] (-0.9,.5);
\end{tikzpicture}
}&=
\mathord{
\begin{tikzpicture}[baseline = 0]
	\draw[->] (0.3,.5) to (-0.3,-.5);
	\draw[-] (-0.2,.2) to (0.2,-.3);
        \draw[-] (0.2,-.3) to[out=130,in=180] (0.5,-.5);
        \draw[-] (0.5,-.5) to[out=0,in=270] (0.9,.5);
        \draw[-] (-0.2,.2) to[out=130,in=0] (-0.6,.5);
        \draw[->] (-0.6,.5) to[out=180,in=-270] (-0.9,-.5);
\end{tikzpicture}
}\:.
\end{align}
Then we obtain further variations on \cref{someoneelse} by rotating through $90^\circ$ or $180^\circ$ using this strictly pivotal structure.
One more useful consequence of the defining relations is that
\begin{align}\label{altbraid}
\mathord{
\begin{tikzpicture}[baseline = 0.8mm]
	\draw[<-] (0.45,.8) to (-0.45,-.4);
	\draw[->] (0.45,-.4) to (-0.45,.8);
        \draw[<-] (0,-.4) to[out=90,in=-90] (-.45,0.2);
        \draw[-] (-0.45,0.2) to[out=90,in=-90] (0,0.8);
\end{tikzpicture}
}
&=
\mathord{
\begin{tikzpicture}[baseline = 0.8mm]
	\draw[<-] (0.45,.8) to (-0.45,-.4);
	\draw[->] (0.45,-.4) to (-0.45,.8);
        \draw[<-] (0,-.4) to[out=90,in=-90] (.45,0.2);
        \draw[-] (0.45,0.2) to[out=90,in=-90] (0,0.8);
\end{tikzpicture}
}
    \:,\end{align}
  There is also a symmetry
  $\sigma:\Heis \rightarrow \Heis^{\op}$, which is the strict
  $\kk$-linear monoidal functor that is the identity on objects and
  sends a morphism to the morphism obtained by reflecting
    in a horizontal axis and then reversing all orientations of strings.
  
Khovanov constructed a categorical action of $\Heis$ on $\sym\Modfd =
\bigoplus_{n \geq 0} \kk S_n\Modfd$, i.e., a strict $\kk$-linear monoidal
functor
\begin{align}\label{aftertheta}
  \Theta: \Heis \rightarrow \END_{\kk}(\sym\Modfd).
\end{align}
Explicitly, this takes the generating objects $\uparrow$ and $\downarrow$ to the
induction functor $E$ and the restriction functor $F$, respectively,
notation as in \cref{EFdef}, and $\Theta$ takes
generating morphisms for $\Heis$ to the natural transformations
defined on a $\kk S_n$-module $V$ as follows (where $g$ is an element
of the appropriate symmetric group):
\begin{align*}
\left(\mathord{
\begin{tikzpicture}[baseline = 0]
	\draw[->] (0.28,-.3) to (-0.28,.4);
	\draw[->] (-0.28,-.3) to (0.28,.4);
\end{tikzpicture}
  }\right)_V
: \kk S_{n+2}\otimes_{\kk S_{n+1}} \kk S_{n+1} \otimes_{\kk S_n} V
               &\rightarrow
               \kk S_{n+2}\otimes_{\kk S_{n+1}} \kk S_{n+1}
               \otimes_{\kk S_n} V,\\g \otimes 1 \otimes v &\mapsto g
                                                             (n\!+\!1\:n\!+\!2)\otimes
                                                             1 \otimes
                                                             v,\hspace{56mm}\end{align*}\begin{align*}
\left(\mathord{
\begin{tikzpicture}[baseline = 0]
	\draw[<-] (0.28,-.3) to (-0.28,.4);
	\draw[->] (-0.28,-.3) to (0.28,.4);
\end{tikzpicture}
  }\right)_V
  &:\kk S_n \otimes_{\kk S_{n-1}} V\rightarrow \kk S_{n+1}
    \otimes_{\kk S_n} V,&g \otimes v&\mapsto g (n\:n\!+\!1) \otimes v,\\
\left(\mathord{
\begin{tikzpicture}[baseline = 0]
	\draw[->] (0.28,-.3) to (-0.28,.4);
	\draw[<-] (-0.28,-.3) to (0.28,.4);
\end{tikzpicture}
  }\right)_V&:\kk S_{n+1}\otimes_{\kk S_n} V \rightarrow \kk
              S_n\otimes_{\kk S_{n-1}} V, &g \otimes v &\mapsto
                                                         \left\{\begin{array}{ll}
                g_2\otimes g_1 v&\text{if }g=g_2 (n\:n\!+\!1) g_1\text{ for }g_i\in S_n,\\
                                                       0&\text{otherwise},\end{array}\right.\\
\left(\mathord{
\begin{tikzpicture}[baseline = 0]
	\draw[<-] (0.28,-.3) to (-0.28,.4);
	\draw[<-] (-0.28,-.3) to (0.28,.4);
\end{tikzpicture}
  }\right)_V
  &:V \rightarrow V,& v &\mapsto (n\!-\!1\:n) v,\\
  \left(  \mathord{
\begin{tikzpicture}[baseline = 1mm]
	\draw[<-] (0.4,0) to[out=90, in=0] (0.1,0.4);
	\draw[-] (0.1,0.4) to[out = 180, in = 90] (-0.2,0);
\end{tikzpicture}
  }\right)_V&:\kk S_n \otimes_{\kk S_{n-1}} V \rightarrow V, &g \otimes v
                                                      &\mapsto gv,
  \\
\left(  \begin{tikzpicture}[baseline = 1mm]
	\draw[-] (0.4,0.4) to[out=-90, in=0] (0.1,0);
	\draw[->] (0.1,0) to[out = 180, in = -90] (-0.2,0.4);
\end{tikzpicture}\right)_V
&: V \rightarrow \kk S_{n}\otimes_{\kk S_{n-1}} V\ &v &\mapsto
                                                       \sum_{i=1}^n
                                                        (i\:
                                                        n)\otimes
                                                        (i\:n) v,\\
\left(  \mathord{
\begin{tikzpicture}[baseline = 1mm]
	\draw[-] (0.4,0) to[out=90, in=0] (0.1,0.4);
	\draw[->] (0.1,0.4) to[out = 180, in = 90] (-0.2,0);
\end{tikzpicture}
  }\right)_V&:\kk S_{n+1} \otimes_{\kk S_{n}} V \rightarrow V, &g \otimes v
                                                                                          &\mapsto \left\{\begin{array}{ll}
                                                                                                            gv&\text{if }g  \in S_{n},                \\            0&\text{otherwise},\end{array}\right.        \\
\left(  \begin{tikzpicture}[baseline = 1mm]
	\draw[<-] (0.4,0.4) to[out=-90, in=0] (0.1,0);
	\draw[-] (0.1,0) to[out = 180, in = -90] (-0.2,0.4);
\end{tikzpicture}\right)_V
&: V \rightarrow \kk S_{n+1}\otimes_{\kk S_{n}} V\ &v &\mapsto
                                                        1\otimes v,\\
  \left(\mathord{
\begin{tikzpicture}[anchorbase]
	\draw[->] (0.08,-.35) to (0.08,.45);
      \node at (0.08,0.05) {$\dt$};
\end{tikzpicture}
}\right)_V&:\kk S_{n+1} \otimes_{\kk S_n} V \rightarrow \kk
  S_{n+1}\otimes_{\kk S_n} V, &g \otimes v &\mapsto g x_{n+1} \otimes v,\\
  \left(\mathord{
\begin{tikzpicture}[anchorbase]
	\draw[<-] (0.08,-.35) to (0.08,.45);
      \node at (0.08,0.05) {$\dt$};
\end{tikzpicture}
}\right)_V&:V \rightarrow V, &v &\mapsto x_n v.
\end{align*}
In the last two formulae, we have used the Jucys-Murphy elements
$x_{n+1}\in \kk S_{n+1}$ and $x_n\in \kk S_n$ from
\cref{jmdef}, respectively; the natural transformations here are the endomorphisms of $E$ and $F$ denoted $x$
and $x^\vee$ just before \cref{endofterm}.
All of the other formulae displayed here can also be found in \cite[$\S$3]{LSR}.
Note in particular that the clockwise bubble $\clock$ acts as multiplication by
$n$ on any $V \in \kk S_n\Modfd$.

It is known moreover that the functor $\Theta$ is faithful.
Indeed, in \cite{K}, Khovanov uses the functor
$\Theta$ to prove a basis theorem for morphism spaces in $\Heis$, and the
argument implicitly establishes the faithfulness of $\Theta$ over fields of characteristic zero.
We will not use this here in any essential way.

\subsection{The affine partition category}\label{apc}
Now the background is in place and we can make a new definition.

\begin{definition}\label{apardef}
  The {\em affine partition category} $\APar$ is the
  monoidal subcategory of $\Heis$ generated by the object
  $\mid := \uparrow \star \downarrow$ and the following morphisms
\begin{align}\label{believe}
                    \begin{tikzpicture}[anchorbase,scale=1.8]
\draw[-,thick](0,0)--(0.4,0.5);
\draw[-,thick](0,0.5)--(0.4,0);
\end{tikzpicture} &:= \begin{tikzpicture}[anchorbase,scale=1.8]
\draw[->](0,0)--(0.5,0.5);
\draw[<-](0.25,0)--(0.75,0.5);
\draw[<-](0,0.5)--(0.5,0);
\draw[->](0.25,0.5)--(0.75,0);
\end{tikzpicture}\ +\ \begin{tikzpicture}[anchorbase,scale=1.8]
\draw[->](0,0)--(0,0.5);
\draw[<-](0.25,0)to[out=up,in=up,looseness=2.5](0.5,0);
\draw[->](0.25,0.5)to[out=down,in=down,looseness=2.5](0.5,0.5);
\draw[<-](0.75,0)--(0.75,0.5);
\end{tikzpicture}\:,\\\label{stup}
                      \begin{tikzpicture}[anchorbase,scale=1.8]
\draw[-,thick](0,0)--(0.2,0.2)--(0.4,0);
\draw[-,thick](0.2,0.2)--(0.2,0.4);
\end{tikzpicture} &:= \begin{tikzpicture}[anchorbase,scale=1.8]
\draw[->](0,0)--(0.2,0.4);
\draw[<-](0.2,0)to[out=up,in=up,looseness=2](0.5,0);
\draw[->](0.5,0.4)--(0.7,0);
\end{tikzpicture}\ ,
&\begin{tikzpicture}[anchorbase,scale=1.8]
\draw[-,thick](0,0.4)--(0.2,0.2)--(0.4,0.4);
\draw[-,thick](0.2,0.2)--(0.2,0);
\end{tikzpicture} &:= \begin{tikzpicture}[anchorbase,scale=1.8]
\draw[<-](0,0.4)--(0.2,0);
\draw[->](0.2,0.4)to[out=down,in=down,looseness=2](0.5,0.4);
\draw[<-](0.5,0)--(0.7,0.4);
                      \end{tikzpicture}\:,\\
  \begin{tikzpicture}[anchorbase,scale=1.5]
\draw[-,thick] (0,0)--(0,0.2);
\node at (0,0.24) {$\dt$};
\end{tikzpicture}&:=\ \begin{tikzpicture}[anchorbase,scale=1.5]
\draw[->] (0,0)--(0,0.1)to[out=up,in=up,looseness=2](0.3,0.1)to(0.3,0);
\end{tikzpicture}\ ,&
\begin{tikzpicture}[anchorbase,scale=1.5]
\draw[-,thick] (0,0.4)--(0,0.19);
\node at (0,0.14) {$\dt$};
\end{tikzpicture} &:= \begin{tikzpicture}[anchorbase,scale=1.5]
\draw[<-] (0,0.4)--(0,0.3)to[out=down,in=down,looseness=2](0.3,0.3)to(0.3,0.4);
\end{tikzpicture}\ ,\label{step}\\\label{thedots}  
                                    \begin{tikzpicture}[anchorbase,scale=1.8]
\draw[-,thick](0,0)--(0,0.4);
\draw[-,thick](0,0.2)--(-0.15,0.2);
\node  at (-0.15,0.2) {$\bullet$};
\end{tikzpicture} &:= \begin{tikzpicture}[anchorbase,scale=1.8]\draw[->](0,0)to(0,0.4);\draw[->](0.2,0.4)to(0.2,0); \node
    at (0,0.2) {$\dt$};\end{tikzpicture}+\begin{tikzpicture}[anchorbase,scale=1.8]\draw[->](0,0)to(0,0.4);\draw[->](0.2,0.4)to(0.2,0);\end{tikzpicture}\ ,&\begin{tikzpicture}[anchorbase,scale=1.8]
\draw[-,thick](0,0)--(0,0.4);
\draw[-,thick](0,0.2)--(0.15,0.2);
\node  at (0.15,0.2) {$\bullet$};
\end{tikzpicture}
  &:= \begin{tikzpicture}[anchorbase,scale=1.8]\draw[->](0,0)to(0,0.4);\draw[->](0.2,0.4)to(0.2,0);\node
                                      at (0.2,0.2) {$\dt$};\end{tikzpicture}+\begin{tikzpicture}[anchorbase,scale=1.8]\draw[->](0,0)to(0,0.4);\draw[->](0.2,0.4)to(0.2,0);\end{tikzpicture}\ ,\\
  \begin{tikzpicture}[anchorbase,scale=.8]
\draw[-,thick](-0.4,-0.6)--(0.4,0.6);
\draw[-,thick](0.4,-0.6)--(-0.4,0.6);
\draw[-,thick](0,0)--(-0.4,0);
\closeddot{-0.4,0};
  \end{tikzpicture} &:=
 \begin{tikzpicture}[anchorbase,scale=.8]
\draw[<-](-0.8,0.8)--(-0.8,-0.4);
	\draw[->] (0.45,.8) to (-0.45,-.4);
	\draw[<-] (0.45,-.4) to (-0.45,.8);
        \draw[-] (0,-.4) to[out=90,in=-90] (-.45,0.2);
        \draw[->] (-0.45,0.2) to[out=90,in=-90] (0,0.8);
  \end{tikzpicture}+\begin{tikzpicture}[anchorbase,scale=1.8]
\draw[->](0,0)--(0,0.6);
\draw[<-](0.25,0)to[out=up,in=up,looseness=3](0.5,0);
\draw[->](0.25,0.6)to[out=down,in=down,looseness=3](0.5,0.6);
\draw[<-](0.75,0)--(0.75,0.6);
\end{tikzpicture}
  \ ,
  &\begin{tikzpicture}[anchorbase,scale=.8]
\draw[-,thick](0.4,-0.6)--(-0.4,0.6);
\draw[-,thick](-0.4,-0.6)--(0.4,0.6);
\draw[-,thick](0,0)--(0.4,0);
\closeddot{0.4,0};
 \end{tikzpicture} 
&:=\begin{tikzpicture}[anchorbase,scale=.8]
\draw[->](0.8,0.8)--(0.8,-0.4);
	\draw[<-] (0.45,.8) to (-0.45,-.4);
	\draw[->] (0.45,-.4) to (-0.45,.8);
        \draw[<-] (0,-.4) to[out=90,in=-90] (.45,0.2);
        \draw[-] (0.45,0.2) to[out=90,in=-90] (0,0.8);
  \end{tikzpicture}+\begin{tikzpicture}[anchorbase,scale=1.8]
\draw[->](0,0)--(0,0.6);
\draw[<-](0.25,0)to[out=up,in=up,looseness=3](0.5,0);
\draw[->](0.25,0.6)to[out=down,in=down,looseness=3](0.5,0.6);
\draw[<-](0.75,0)--(0.75,0.6);
\end{tikzpicture}\ .
\label{irrelevant}\end{align}
\end{definition}

We refer to the morphisms in \cref{thedots}
as the {\em left dot} and the {\em right dot}, and the morphisms in \cref{irrelevant} as the {\em left crossing} and the {\em right crossing}, respectively.
The other shorthands for the generating morphisms of $\APar$ introduced in \cref{apardef}
are the same as the
symbols used for generators of the partition category. This is deliberate, indeed, the morphisms \cref{believe,stup,step} generate a copy of the generic partition category $\Par$ as a monoidal subcategory of $\Heis$. This important
observation is due to Likeng and Savage; see \cref{lscor} below. For now, we just need the following, which is proved in \cite{LSR} by a direct calculation using the defining relations in $\Heis$.

\begin{lemma}[{\cite[Th.~4.1]{LSR}}]\label{ideflemma}
There is a strict $\kk$-linear
monoidal functor
\begin{equation}\label{idef}
  i:\Par \rightarrow \APar
  \end{equation}
sending the generating object and generating morphisms of $\Par$ to the generating object and generating morphisms in $\APar$ denoted by the same diagrams.
\end{lemma}

Because of the symmetry of the generators of $\APar$
under rotation through $180^\circ$, the strictly pivotal structure on $\Heis$ restricts to a strictly pivotal structure on $\APar$. The left and right dots are duals,
as are the left and right crossings.
Moreover, the cap and the cup making $\mid$ into a self-dual object are given by the same formula \cref{cupcap} as we had before in $\Par$, hence, $i$ is a pivotal monoidal functor.
Note also that
\begin{align}\label{T}
T&:= 

\ .
\end{align}
Also, the symmetry $\sigma$ on $\Heis$ restricts to
$\sigma:\APar \rightarrow \APar^\op$. This just reflects affine partition
diagrams in a horizontal axis, just like the earlier anti-automorphism $\sigma$ on $\Par_t$.
Here are some further
relations, all of which are easily proved using the defining relations in $\Heis$:
\begin{align}
\label{rightslide}
%
\ .
\end{align}
Of course, the horizontal and vertical flips of all of these also hold.
The next two lemmas establish some less obvious relations.

\begin{lemma}\label{somerelations}
  The following relations hold in $\APar$:
  \begin{align}
%
\ .\label{may}
  \end{align}
\end{lemma}

\begin{proof}
  For each of \cref{rightdot,crazycrossing,june}, it suffices just to prove the first equality,
  and then all the others follow using $\sigma$ and duality to reflect in horizontal and/or vertical axes.
  For \cref{rightdot}, use \cref{rightslide} and \cref{relsfrob2}.
  To prove \cref{crazycrossing}, we expand as morphisms in $\Heis$ to see that
   \begin{align*}
  %
\ .
   \end{align*}
          For \cref{june}, we again expand the left hand side as a morphism in $\Heis$:
          $$
   %
\ .
   $$
Finally, to prove \cref{may}, the second set of relations follows from the first set of relations by composing on the bottom with a crossing and using \cref{june}.
For the first set of relations, it suffices to prove the first equality, the second then follows by duality.
Expanding both of the left crossings as morphisms in $\Heis$ produces a sum of four terms, two of which are zero, so we obtain:
\begin{align*}
%
\ .
\end{align*}
\end{proof}

\begin{corollary}\label{red1}
As a $\kk$-linear monoidal category, the subcategory $\APar$ of $\Heis$
is generated by the object $\mid\,$, the five undotted generators
\cref{believe,stup,step}, and the left dot.
\end{corollary}

\begin{proof}
  The relations \cref{rightdot,crazycrossing} together show that the right dot and the left and right crossings may be written in terms of the left dot and the other undotted generating morphisms.
  \end{proof}

\begin{lemma}\label{amazinggrace}
  The following relations hold:
  \begin{align}
    \label{A1}
%
\ .
\end{align}
\end{lemma}

\begin{proof}
  To prove \cref{A1}, we observe by composing on the bottom with the crossing and using relations from $\Par$ plus \cref{june} that
   the relation we are trying to prove is equivalent to
  \begin{equation*}
                      %
                      \ .
                      \end{equation*}
  Now we expand the left hand side in terms of morphisms in $\Heis$ using \cref{leftcurl,misstate}, then we use \cref{someoneelse} to commute
  the dot past a crossing in the first and fourth terms:
  \begin{multline*}
%
\,.
  \end{multline*}
   Similarly, the expansion of the right hand side is
  \begin{multline*}
                      %
\,,
  \end{multline*}
  where we commuted the dot past a crossing just in the first term.
  These are equal.
  To deduce \cref{A3}, first apply duality to \cref{A1}, i.e., rotate through $180^\circ$. Then compose on the top and bottom with a crossing and simplify using relations in $\Par$ together with \cref{june}.

  To prove \cref{august}, we rewrite its left hand side, replacing the right crossing with a left dot using \cref{crazycrossing}, then we apply \cref{A1} to push this left dot past the right hand string:
   \begin{equation*}
  %
\ .
\end{equation*}
Now we simplify the five terms on the right hand side of
the equation just displayed to obtain
the five terms on the right hand side of \cref{august}
(there is no need here to expand in terms of morphisms in $\Heis$).
The following treats the first term:
$$
  %
\ .
  $$
  The second and third terms are easy to handle, we omit the details.
  For the fourth and the fifth terms, it suffices by symmetry to consider the fifth term, which we rewrite as follows:
  $$
  %
  \ .
  $$
  Finally \cref{july} follows easily from \cref{august} on composing on the bottom with the crossing of the leftmost two strings and using \cref{june}.
  \end{proof}

\begin{corollary}\label{red2}
  As a $\kk$-linear category, $\APar$
  has object set $\{\mid^{\star n}\:|\:n \in \N\}$ (which we often identify simply
  with $\N$)
  and morphisms that are linear combinations of vertical
  compositions of morphisms in the image of $i:\Par \rightarrow \Heis$ together with the morphisms
  \begin{equation}\label{firstdot}
  %
  \end{equation}
  for all $n \geq 1$.
\end{corollary}

\begin{proof}
  In view of \cref{red1}, we just need to show that one can obtain the endomorphism
  of $n$ defined by the left dot on the $m$th string ($m=1,\dots,n$)
  by taking a linear combinations of compositions of morphisms in the image of $i$ and the given morphism \cref{firstdot} (in which the left dot is on the first string).
  This follows by induction on $m$ using relations \cref{A1} and \cref{crazycrossing}.
   \end{proof}

    \begin{remark}
      We have not attempted to formulate or prove a basis theorem for the morphism spaces in $\APar$.
      This is closely related to the problem of finding a complete monoidal presentation for $\APar$.
    \end{remark}
    
\subsection{\boldmath Action of $\APar$ on $\kk S_t\Modfd$}
Suppose that $t \in \N$.
The restriction of the functor $\Theta$ from \cref{aftertheta} to the subcategory $\APar$ is a strict $\kk$-linear monoidal functor
$\APar \rightarrow \END_{\kk}(\sym\Modfd)$
sending the generating object $\mid$ to the endofunctor $E\circ F$ (induction after restriction).
Since $E \circ F$ takes $\kk S_t$-modules to $\kk S_t$-modules, the restriction of $\Theta$ gives strict $\kk$-linear monoidal functors
\begin{equation}\label{PSI}
  \Theta_t:\APar \rightarrow \END_{\kk}(\kk S_t\Modfd).
\end{equation}
The functor $\Theta_t$ takes
$\mid$ to the endofunctor
$\ind_{S_{t-1}}^{S_t} \circ \res_{S_{t-1}}^{S_t} = \kk S_t
\otimes_{\kk S_{t-1}}$
of $\kk S_t\Modfd$; this should be interpreted as the zero functor in the case $t=0$.
The natural transformations arising by
applying $\Theta_t$ to the other generating morphisms of $\APar$ may be
computed using the formulae after \cref{aftertheta} (taking $n:=t$). Explicitly, one obtains the following for $V \in \kk S_t\Modfd$ and $g,h \in S_t$:
\begin{align*}
 \left(\, \begin{tikzpicture}[anchorbase,scale=1.8]
\draw[-,thick](0,0)--(0.4,0.5);
\draw[-,thick](0,0.5)--(0.4,0);
\end{tikzpicture}\,\right)_V: \kk S_t \otimes_{\kk S_{t-1}} \kk S_t
                            \otimes_{\kk S_{t-1}} V&\rightarrow \kk S_t
                            \otimes_{\kk S_{t-1}} \kk S_t \otimes_{\kk S_{t-1}} V,\hspace{56mm}\\
                                                          g \otimes
                                                          h \otimes
                                                          v&\mapsto
 gh \otimes h^{-1} \otimes hv,\\
                                                                   \left( \begin{tikzpicture}[anchorbase,scale=.8]
\draw[-,thick](-0.4,-0.5)--(0.4,0.5);
\draw[-,thick](0.4,-0.5)--(-0.4,0.5);
\draw[-,thick](0,0)--(-0.4,0);
\closeddot{-0.4,0};
\end{tikzpicture}\,\right)_V : \kk S_t \otimes_{\kk S_{t-1}} \kk S_t
                            \otimes_{\kk S_{t-1}} V &\rightarrow \kk S_t
                            \otimes_{\kk S_{t-1}} \kk S_t \otimes_{\kk S_{t-1}} V,\\
                            g \otimes h \otimes v&\mapsto
                            g \otimes h \otimes (h^{-1}(t)\:\,t) v,\\
\left(\, \begin{tikzpicture}[anchorbase,scale=.8]
\draw[-,thick](0.4,-0.5)--(-0.4,0.5);
\draw[-,thick](-0.4,-0.5)--(0.4,0.5);
\draw[-,thick](0,0)--(0.4,0);
\closeddot{0.4,0};
\end{tikzpicture} \right)_V :\  \kk S_t \otimes_{\kk S_{t-1}} \kk S_t
                            \otimes_{\kk S_{t-1}} V&\rightarrow \kk S_t
                            \otimes_{\kk S_{t-1}} \kk S_t \otimes_{\kk S_{t-1}} V,\\
                            g \otimes h \otimes v &\mapsto
                            g h \otimes (h^{-1}(t)\:\,t) \otimes  v,
                                                        \end{align*}
\begin{align*}
\left(\, \begin{tikzpicture}[anchorbase,scale=1.8]
\draw[-,thick](0,0)--(0.2,0.2)--(0.4,0);
\draw[-,thick](0.2,0.2)--(0.2,0.4);
\end{tikzpicture}\,\right)_V&:\kk S_t \otimes_{\kk S_{t-1}} \kk S_t
                            \otimes_{\kk S_{t-1}} V \rightarrow \kk
                            S_t \otimes_{\kk S_{t-1}} V,&g \otimes h
                                                          \otimes
                                                          v&\mapsto
                                                             \delta_{h(t),t} gh
                                                                      \otimes
                                                                      v,\\
  \left(\,\begin{tikzpicture}[anchorbase,scale=1.8]
\draw[-,thick](0,0.4)--(0.2,0.2)--(0.4,0.4);
\draw[-,thick](0.2,0.2)--(0.2,0);
\end{tikzpicture}\,\right)_V&:\kk S_t \otimes_{\kk S_{t-1}} V
                            \rightarrow
                            \kk S_t \otimes_{\kk S_{t-1}} \kk S_t \otimes_{\kk S_{t-1}} V,&
                                                                                            g \otimes v&\mapsto g \otimes 1 \otimes v,\\
  \left(\,\begin{tikzpicture}[anchorbase,scale=1.5]
\draw[-,thick] (0,0)--(0,0.2);
\node at (0,0.24) {$\dt$};
\end{tikzpicture}\,\right)_V & :\kk S_{t} \otimes_{\kk S_{t-1}} V \rightarrow V,&
                                                                                g\otimes v&\mapsto gv,\\
 \left(\, \begin{tikzpicture}[anchorbase,scale=1.5]
\draw[-,thick] (0,0.4)--(0,0.19);
\node at (0,0.14) {$\dt$};
\end{tikzpicture}\,\right)_V&:V \rightarrow \kk S_t \otimes_{\kk S_{t-1}}
                            V,&v&\mapsto\sum_{i=1}^t (i\:t)\otimes
                            (i\:t) v,\\
\left(\!\!\!\begin{tikzpicture}[anchorbase,scale=1.8]
\draw[-,thick](0,0)--(0,0.4);
\draw[-,thick](0,0.2)--(-0.15,0.2);
\node  at (-0.15,0.2) {$\bullet$};
\end{tikzpicture}\:\right)_V&:\kk S_t \otimes_{\kk S_{t-1}} V \rightarrow \kk S_t \otimes_{\kk S_{t-1}} V,&
g \otimes v &\mapsto \sum_{j=1}^{t} g (j\:t) \otimes v
,\\
\left(\:\begin{tikzpicture}[anchorbase,scale=1.8]
\draw[-,thick](0,0)--(0,0.4);
\draw[-,thick](0,0.2)--(0.15,0.2);
\node  at (0.15,0.2) {$\bullet$};
\end{tikzpicture}\!\!\!\right)_V&:\kk S_t \otimes_{\kk S_{t-1}} V \rightarrow \kk S_t \otimes_{\kk S_{t-1}} V,&
g \otimes v &\mapsto \sum_{j=1}^{t} g \otimes (j\:t) v.
\end{align*}
Recall the natural $\kk S_d$-module $\Nat$ from \cref{SWD}; in particular, $U_0$ is the zero module.
Using the Kronecker product, we can consider $\Nat \otimes$ as an endofunctor
of $\kk S_t\Modfd$.
Also let $\triv_{S_t}$ be the trivial module. 

\begin{lemma}\label{PHIlem}
  The functor $\Theta_t$ is monoidally isomorphic to the
  strict $\kk$-linear monoidal functor
\begin{equation}\label{PHI}
  \Phi_t:\APar \rightarrow \END_{\kk}(\kk S_t \Modfd)
\end{equation}
which sends the generating object $\mid$ to the endofunctor
$\Nat \otimes$ and taking the generating morphisms for $\APar$ to
the natural transformations defined as follows on
$V \in \kk S_t\Modfd$ and $1 \leq i,j \leq t$:
\begin{align*}
 \left( \begin{tikzpicture}[anchorbase,scale=1.8]
\draw[-,thick](0,0)--(0.4,0.5);
\draw[-,thick](0,0.5)--(0.4,0);
\end{tikzpicture}\right)_V&: \Nat \otimes \Nat \otimes V \rightarrow \Nat \otimes \Nat \otimes V,&
                                                          u_i \otimes
                                                          u_j \otimes
                                                          v&\mapsto
                                                                   u_j
                                                                   \otimes u_i
                                                                   \otimes
                                                                   v,\\
                                                                   \left( \begin{tikzpicture}[anchorbase,scale=.8]
\draw[-,thick](-0.4,-0.5)--(0.4,0.5);
\draw[-,thick](0.4,-0.5)--(-0.4,0.5);
\draw[-,thick](0,0)--(-0.4,0);
\closeddot{-0.4,0};
\end{tikzpicture}\,\right)_V &: \Nat \otimes \Nat \otimes V \rightarrow \Nat \otimes \Nat \otimes V,&
                            u_i \otimes u_j \otimes v&\mapsto
                            u_i \otimes u_j \otimes (i\:j) v,\\
\left(\, \begin{tikzpicture}[anchorbase,scale=.8]
\draw[-,thick](0.4,-0.5)--(-0.4,0.5);
\draw[-,thick](-0.4,-0.5)--(0.4,0.5);
\draw[-,thick](0,0)--(0.4,0);
\closeddot{0.4,0};
\end{tikzpicture} \right)_V &:\ \Nat \otimes \Nat \otimes V\rightarrow \Nat \otimes \Nat \otimes V,&
                            u_i \otimes u_j \otimes v &\mapsto
                            u_j\otimes u_i \otimes (i\:j) v,                                                                   
              \\
\left( \begin{tikzpicture}[anchorbase,scale=1.8]
\draw[-,thick](0,0)--(0.2,0.2)--(0.4,0);
\draw[-,thick](0.2,0.2)--(0.2,0.4);
\end{tikzpicture}\right)_V&:\Nat \otimes \Nat \otimes V \rightarrow \Nat \otimes V,&u_i \otimes u_j \otimes v&\mapsto \delta_{i,j} u_i \otimes v,\\
  \left(\begin{tikzpicture}[anchorbase,scale=1.8]
\draw[-,thick](0,0.4)--(0.2,0.2)--(0.4,0.4);
\draw[-,thick](0.2,0.2)--(0.2,0);
\end{tikzpicture}\right)_V&:\Nat \otimes V \rightarrow \Nat \otimes \Nat \otimes V,&
                                                                                            u_i \otimes v&\mapsto u_i \otimes u_i \otimes v,\\
  \left(\begin{tikzpicture}[anchorbase,scale=1.5]
\draw[-,thick] (0,0)--(0,0.2);
\node at (0,0.24) {$\dt$};
\end{tikzpicture}\right)_V & :\Nat \otimes V \rightarrow V,&
                                                                                u_i\otimes v&\mapsto v,\\
 \left( \begin{tikzpicture}[anchorbase,scale=1.5]
\draw[-,thick] (0,0.4)--(0,0.19);
\node at (0,0.14) {$\dt$};
 \end{tikzpicture}\right)_V&:V \rightarrow \Nat \otimes V,&v&\mapsto\sum_{i=1}^t u_i\otimes v, \\
\left(\!\!\!\begin{tikzpicture}[anchorbase,scale=1.8]
\draw[-,thick](0,0)--(0,0.4);
\draw[-,thick](0,0.2)--(-0.15,0.2);
\node  at (-0.15,0.2) {$\bullet$};
\end{tikzpicture}\:\right)_V&:\Nat\otimes V \rightarrow \Nat \otimes V,&
u_i \otimes v &\mapsto \sum_{j=1}^t u_j \otimes (i\:j) v,
\\
\left(\:\begin{tikzpicture}[anchorbase,scale=1.8]
\draw[-,thick](0,0)--(0,0.4);
\draw[-,thick](0,0.2)--(0.15,0.2);
\node  at (0.15,0.2) {$\bullet$};
\end{tikzpicture}\!\!\!\right)_V&:\Nat \otimes V \rightarrow \Nat \otimes V,&
u_i \otimes v &\mapsto \sum_{j=1}^{t} u_i \otimes (i\:j) v.
\end{align*}
\end{lemma}

\begin{proof}
  There is an isomorphism $\kk S_t \otimes_{\kk S_{t-1}} \triv_{S_{t-1}} \stackrel{\sim}{\rightarrow} \Nat, g \otimes 1 \mapsto g u_t$.
  Combining this with the tensor identity, we obtain a natural $\kk S_t$-module isomorphism
  \begin{align}\label{chickenhouse}
    (\alpha^{(t)}_1)_V:\kk S_t \otimes_{\kk S_{t-1}} V &\stackrel{\sim}{\rightarrow} \Nat \otimes V, &g \otimes v &\mapsto g u_t \otimes gv
  \end{align}
  for $V \in \kk S_t\Modfd$.
  This defines an  isomorphism $\alpha^{(t)}_1:\kk S_t \otimes_{\kk S_{t-1}} \stackrel{\sim}{\Rightarrow} \Nat\otimes$.
  Let $\alpha^{(t)}_n := \alpha^{(t)}_1 \cdots \alpha^{(t)}_1$ be the $n$-fold horizontal composition of $\alpha^{(t)}_1$.
  This is a natural isomorphism $\alpha^{(t)}_n:(\kk S_t\otimes_{\kk S_{t-1}} )^{\circ n}
  \stackrel{\sim}{\Rightarrow} (U\otimes)^{\circ n}$ whose value on a $\kk S_t$-module $V$ is given explicitly by the map
  $$
  \big(\alpha^{(t)}_n\big)_V :g_n \otimes \cdots \otimes g_1 \otimes v \mapsto g_n u_t \otimes g_n g_{n-1} u_t \otimes\cdots\otimes g_n g_{n-1}\cdots g_1 u_t \otimes g_n g_{n-1}\cdots g_1 v.
  $$
Now {\em define} $\Phi_t:\APar \rightarrow \END_{\kk}(\kk S_t\Modfd)$
to be the strict $\kk$-linear monoidal functor
taking the object $n$ to
$(\Nat\otimes)^{\circ n}$, and defined on a morphism
$f \in \Hom_{\APar}(n, m)$ by
$\Phi_t(f) := \alpha^{(t)}_{m} \circ \Theta_t(f) \circ \big(\alpha^{(t)}_{n}\big)^{-1}$.
It is immediate from this definition that
$\alpha^{(t)} = \big(\alpha^{(t)}_n\big)_{n \geq 0}:\Theta_t \Rightarrow \Phi_t$
is an isomorphism of strict $\kk$-linear monoidal functors.

It remains to check that $\Phi_t$ as defined in the previous paragraph is equal to the functor $\Phi_t$ defined on generating morphisms
in the statement of the lemma. So we need to check for
each generating morphism $f\in\Hom_{\APar}(n,m)$
that the formula for $\Phi_t(f)_V$ written
in the statement of the lemma is equal to
$\big(\alpha^{(t)}_m\big)_V \circ \Theta_t(f)_V \circ \big(\alpha^{(t)}_n\big)_V^{-1}$
for $V \in \kk S_t\Modfd$ and $t \in \N$. This is a routine but lengthy
calculation. We just go through a couple of the cases.

If $f$ is the crossing, we need to show that
$\left(\big(\alpha^{(t)}_2\big)_V \circ \Theta_t(f)_V \circ \big(\alpha^{(t)}_2\big)_V^{-1}\right) (u_i \otimes u_j \otimes v) = u_j \otimes u_i \otimes v$.
Now we consider four cases.
If $t \neq i \neq j \neq t$ we have that
\begin{align*}
\left(\big(\alpha^{(t)}_2\big)_V \circ \Theta_t(f)_V \circ \big(\alpha^{(t)}_2\big)_V^{-1}\right) (u_i \otimes u_j \otimes v)
&= \left(\big(\alpha^{(t)}_2\big)_V \circ \Theta_t(f)_V\right) \left((i\:t) \otimes (j\: t) \otimes
(j\:t)(i\:t) v\right)\\
  &=\big(\alpha^{(t)}_2\big)_V\left((i\:t) (j\:t) \otimes (j\:t) \otimes (i\:t) v\right)
= u_j \otimes u_i \otimes v.
\end{align*}
If $i=j$ we have that
\begin{align*}
\left(\big(\alpha^{(t)}_2\big)_V \circ \Theta_t(f)_V \circ\big(\alpha^{(t)}_2\big)_V^{-1}\right) (u_i \otimes u_j \otimes v)
&= \left(\big(\alpha^{(t)}_2\big)_V \circ \Theta_t(f)_V\right) \left((i\:t) \otimes
1\otimes (i\: t) v\right)\\
  &=\big(\alpha^{(t)}_2\big)_V\left((i\:t) \otimes 1 \otimes (i\:t) v\right)
= u_j \otimes u_i \otimes v.
\end{align*}
If $i = t \neq j$ we have that
\begin{align*}
\left(\big(\alpha^{(t)}_2\big)_V \circ \Theta_t(f)_V \circ \big(\alpha^{(t)}_2\big)_V^{-1}\right) (u_i \otimes u_j \otimes v)
&= \left(\big(\alpha^{(t)}_2\big)_V \circ \Theta_t(f)_V\right) \left(1 \otimes (j\:t) \otimes (j\: t) v\right)\\
  &=\big(\alpha^{(t)}_2\big)_V\left((j\:t) \otimes (j\:t) \otimes v\right)
= u_j \otimes u_i \otimes v.
\end{align*}
Finally if $i \neq t = j$ we have that
\begin{align*}
\left(\big(\alpha^{(t)}_2\big)_V \circ \Theta_t(f)_V \circ \big(\alpha^{(t)}_2\big)_V^{-1}\right) (u_i \otimes u_j \otimes v)
&= \left(\big(\alpha^{(t)}_2\big)_V \circ \Theta_t(f)_V\right) \left((i\:t) \otimes (i\:t) \otimes v\right)\\
  &=\big(\alpha^{(t)}_2\big)_V\left(1 \otimes (i\:t) \otimes (i\:t) v\right)
= u_j \otimes u_i \otimes v.
\end{align*}
This completes the check in this case.

If $f$ is the left dot, we have that
\begin{multline*}
  \left(\big(\alpha^{(t)}_1\big)_V \circ \Theta_t(f)_V \circ \big(\alpha^{(t)}_1\big)_V^{-1}\right)
  (u_i \otimes v) =
  \left(\big(\alpha^{(t)}_1\big)_V \circ \Theta_t(f)_V\right)
  \left((i\:t) \otimes (i\:t) v\right)\\
  =\sum_{j=1}^{t} \big(\alpha^{(t)}_1\big)_V((i\:t) (j\:t) \otimes (i\:t) v)
  = \sum_{j=1}^{t} (i\:t) (j\:t) u_t \otimes (i\:t) (j\:t) (i\:t) v.
\end{multline*}
If $i=t$ this is $\sum_{j=1}^{t} u_j \otimes (j\:t) v$ which is right.
If $i \neq t$ we pull out the $j=i$ and $j=t$ terms of the sum, simplify
the three types of terms separately, then recombine
to get the desired expression
$\sum_{j=1}^t u_j \otimes (i\:j) v$.
  \end{proof}

We now have in our hands monoidal functors
$\phi_t$ from \cref{cucumber},
$i$ from \cref{idef}, and $\Phi_t$
from \cref{PHI}. Let
\begin{equation}
  \Act:\kk S_t\Modfd \rightarrow \END_{\kk}(\kk S_t\Modfd)
\end{equation}
be the $\kk$-linear monoidal functor induced by the Kronecker product, i.e.,
$\Act(V) = V \otimes$ for a $\kk S_t$-module $V$
and $\Act(f) = f\otimes$ for a homomorphism $f:V \rightarrow V'$.

\begin{lemma}\label{lalala} For every $t \in \N$, the following diagram
   commutes up to the obvious canonical isomorphism of monoidal functors:
  \begin{equation}\label{EvDgm}\begin{tikzcd}
\APar \ar{r}{\Phi_t} & \END_\kk(\kk S_t\Modfd)\\
\Par\ar{u}{i}\ar{r}{\phi_t} & \kk S_t\Modfd
\arrow[u,"\Act"{right}].
  \end{tikzcd}\end{equation}
\end{lemma}

\begin{proof}
  The composition
  $\Phi_t \circ i$ takes the $n$th object of $\APar$ to $(\Nat \otimes)^{\circ n}$, while $\Act\circ\phi_t$ takes it to $\Nat^{\otimes n} \otimes$.
  Let $$
  \beta^{(t)}_n:(\Nat \otimes)^{\circ n} \stackrel{\sim}{\Rightarrow} \Nat^{\otimes n} \otimes
  $$
  be the canonical isomorphism between these functors defined by associativity of tensor product.
  Then
$\beta^{(t)} = \big(\beta^{(t)}_n\big)_{n \geq 0}:\Phi_t \circ i \Rightarrow
\Act\circ \phi_t$ is an isomorphism
  of monoidal functors. To see this, we need to check naturality. This follows
because the five formulae defining $\phi_t$
from \cref{SWD} tensored on the right with a vector $v$
  are exactly
  the same as the formulae defining $\Phi_t$ on these five generating morphisms
  from \cref{PHIlem}.
  \end{proof}

Now we can prove the main theorem justifying
the significance of the affine partition category. Let
\begin{equation}
  \Ev:\END_{\kk}(\kk S_t\Modfd) \rightarrow \kk S_t\Modfd
\end{equation}
be the (non-monoidal) $\kk$-linear functor defined by evaluating on
$\triv_{S_t}$.
There is an obvious isomorphism of functors
$\Ev \circ \Act \stackrel{\sim}{\Rightarrow} \Id_{\kk S_t\Modfd}$
defined on $V$
by the isomorphism $V \otimes \triv_{S_t} \rightarrow V,
v \otimes 1 \mapsto v$.

\begin{theorem}\label{nextmain} There is a unique (non-monoidal) $\kk$-linear functor
  \begin{equation}\label{movers}
    p:\APar\rightarrow \Par
  \end{equation}
  such that $p \circ i = \Id_{\Par}$ and
  \begin{align}\label{boycott}
  p\left(\begin{tikzpicture}[baseline=3mm]
\draw[-,thick] (0,0) to (0,0.8);
\node at (0.1,0.4) {$\cdot$};
\node at (0.25,0.4) {$\cdot$};
\node at (0.4,0.4) {$\cdot$};
\draw[-,thick] (0.5,0) to (0.5,0.8);
\draw[-,thick] (0.85,0) to (0.85,0.8);
\draw[-,thick] (0.65,0.4) to (0.85,0.4);
\node at (0.85,-.15) {$\stringlabel{1}$};
\node at (0.5,-.15) {$\stringlabel{2}$};
\node at (0,-.15) {$\stringlabel{n}$};
\node at (0.65,0.4) {$\bullet$};
\end{tikzpicture}\right) &= \begin{tikzpicture}[baseline=3mm]
\draw[-,thick] (0,0) to (0,0.8);
\node at (0.1,0.4) {$\cdot$};
\node at (0.25,0.4) {$\cdot$};
\node at (0.4,0.4) {$\cdot$};
\draw[-,thick] (0.5,0) to (0.5,0.8);
\draw[-,thick] (0.85,0) to (0.85,0.2);
\draw[-,thick] (0.85,.8) to (0.85,0.6);
\node at (0.85,-.15) {$\stringlabel{1}$};
\node at (0.5,-.15) {$\stringlabel{2}$};
\node at (0,-.15) {$\stringlabel{n}$};
\node at (0.85,0.25) {$\dt$};
\node at (0.85,0.53) {$\dt$};
  \end{tikzpicture}\ .
  \end{align}
  Moreover, for any $t \in \N$,
  the following diagram of functors commutes up to natural isomorphism:
  \begin{equation}\label{EvDgn}\begin{tikzcd}
      \APar \arrow[d,"p"{left}]\ar{r}{\Phi_t} & \END_\kk(\kk S_t\Modfd)\ar{d}{\Ev}\\
\Par\ar{r}{\phi_t} & \kk S_t\Modfd.
  \end{tikzcd}\end{equation}  
The functor $p$ also maps
\begin{align}
    \begin{tikzpicture}[baseline=3mm]
\draw[-,thick] (0,0) to (0,0.8);
\node at (0.1,0.4) {$\cdot$};
\node at (0.25,0.4) {$\cdot$};
\node at (0.4,0.4) {$\cdot$};
\draw[-,thick] (0.5,0) to (0.5,0.8);
\draw[-,thick] (0.85,0) to (0.85,0.8);
\draw[-,thick] (0.85,0.4) to (1.05,0.4);
\node at (0.85,-.15) {$\stringlabel{1}$};
\node at (0.5,-.15) {$\stringlabel{2}$};
\node at (0,-.15) {$\stringlabel{n}$};
\node at (1.05,0.4) {$\bullet$};
\end{tikzpicture} &\mapsto T\begin{tikzpicture}[baseline=3mm]
\draw[-,thick] (0,0) to (0,0.8);
\node at (0.1,0.4) {$\cdot$};
\node at (0.25,0.4) {$\cdot$};
\node at (0.4,0.4) {$\cdot$};
\draw[-,thick] (0.5,0) to (0.5,0.8);
\node at (0.5,-.15) {$\stringlabel{1}$};
\node at (0,-.15) {$\stringlabel{n}$};
\end{tikzpicture}
\ ,\label{babyjoey}&
    \begin{tikzpicture}[baseline=3mm]
\draw[-,thick] (0,0) to (0,0.8);
\node at (0.1,0.4) {$\cdot$};
\node at (0.25,0.4) {$\cdot$};
\node at (0.4,0.4) {$\cdot$};
\draw[-,thick] (0.5,0) to (0.5,0.8);
\draw[-,thick] (0.85,0) to (1.2,0.8);
\draw[-,thick] (1.025,0.4) to (1.3,0.4);
\draw[-,thick] (1.2,0) to (0.85,0.8);
\draw[-,thick] (1.2,0.4) to (1.05,0.4);
\node at (1.2,-.15) {$\stringlabel{1}$};
\node at (.85,-.15) {$\stringlabel{2}$};
\node at (0.5,-.15) {$\stringlabel{3}$};
\node at (0,-.15) {$\stringlabel{n}$};
\node at (1.3,0.4) {$\bullet$};
    \end{tikzpicture} &\mapsto
    \begin{tikzpicture}[baseline=3mm]
\draw[-,thick] (0,0) to (0,0.8);
\node at (0.1,0.4) {$\cdot$};
\node at (0.25,0.4) {$\cdot$};
\node at (0.4,0.4) {$\cdot$};
\draw[-,thick] (0.5,0) to (0.5,0.8);
\draw[-,thick] (0.85,0) to (1.2,0.8);
\draw[-,thick] (1.2,0) to (0.85,0.8);
\node at (1.2,-.15) {$\stringlabel{1}$};
\node at (.85,-.15) {$\stringlabel{2}$};
\node at (0.5,-.15) {$\stringlabel{3}$};
\node at (0,-.15) {$\stringlabel{n}$};
\end{tikzpicture}\ ,&
\begin{tikzpicture}[baseline=3mm]
\draw[-,thick] (0,0) to (0,0.8);
\node at (0.1,0.4) {$\cdot$};
\node at (0.25,0.4) {$\cdot$};
\node at (0.4,0.4) {$\cdot$};
\draw[-,thick] (0.5,0) to (0.5,0.8);
\draw[-,thick] (0.85,0) to (1.2,0.8);
\draw[-,thick] (1.025,0.4) to (.725,0.4);
\draw[-,thick] (1.2,0) to (0.85,0.8);
\node at (1.2,-.15) {$\stringlabel{1}$};
\node at (.85,-.15) {$\stringlabel{2}$};
\node at (0.5,-.15) {$\stringlabel{3}$};
\node at (0,-.15) {$\stringlabel{n}$};
\node at (.725,0.4) {$\bullet$};
\end{tikzpicture}&\mapsto\begin{tikzpicture}[baseline=3mm]
\draw[-,thick] (0,0) to (0,0.8);
\node at (0.1,0.4) {$\cdot$};
\node at (0.25,0.4) {$\cdot$};
\node at (0.4,0.4) {$\cdot$};
\draw[-,thick] (0.5,0) to (0.5,0.8);
\draw[-,thick] (0.85,0) to (.85,0.8);
\draw[-,thick] (1.2,0) to (1.2,0.8);
\node at (1.2,-.15) {$\stringlabel{1}$};
\node at (.85,-.15) {$\stringlabel{2}$};
\node at (0.5,-.15) {$\stringlabel{3}$};
\node at (0,-.15) {$\stringlabel{n}$};
\end{tikzpicture}\ .
  \end{align}
\end{theorem}

\begin{proof}
For $t \in \N$, let $\gamma^{(t)}_n: \Nat^{\otimes n} \otimes \triv_{S_t}
  \stackrel{\sim}{\rightarrow}  \Nat^{\otimes n}$ be the obvious isomorphism
  sending
  $u_{i_n}\otimes \cdots \otimes u_{i_1} \otimes 1 \mapsto u_{i_n} \otimes \cdots \otimes u_{i_1}$.
  We say that $f \in \Hom_{\APar}(n,m)$ is {\em good} if there exists a morphism $\bar f \in \Hom_{\Par}(n,m)$ such that
\begin{equation}\label{mastereq}
\phi_t(\bar f) = \gamma^{(t)}_m \circ \Ev(\Phi_t(f)) \circ \big(\gamma^{(t)}_n\big)^{-1}
\end{equation}
for all $t \in \N$.
If $f$ is good, there is a {\em unique} $\bar f$ such that
\cref{mastereq} holds for all $t$.
To see this, suppose that $\bar f$ and $\bar f'$ both satisfy \cref{mastereq} for all $t \in \N$.
Then $\phi_t(\bar f) = \gamma^{(t)}_m\circ \Ev(\Phi_t(f)) \circ \big(\gamma^{(t)}_n\big)^{-1} = \phi_t(\bar f')$, so that $\phi_t(\bar f-\bar f') = 0$ for all $t\in\N$.
In view of \cref{SWDapp} this implies that $\bar f = \bar f'$ as claimed.

Suppose that
$f\in \Hom_{\APar}(n,m)$ and $g \in \Hom_{\APar}(l,m)$
are both good. Then $f \circ g$ is good with $\overline{f\circ g} := \bar f \circ \bar g$.
This follows because
$$
\phi_t(\bar f \circ \bar g) = \gamma^{(t)}_m \circ \Ev(\Phi_t(f)) \circ \big(\gamma^{(t)}_m\big)^{-1}
\circ \gamma^{(t)}_m \circ \Ev(\Phi_t(f)) \circ \big(\gamma^{(t)}_l\big)^{-1}
= \gamma^{(t)}_m \circ \Ev(\Phi_t(f \circ g)) \circ \big(\gamma^{(t)}_l\big)^{-1}.
$$
Similarly, sums of good morphisms are good with $\overline{f+g} := \bar f + \bar g$.

In this paragraph, we show that every morphism in $\APar$ is good.
In view of the previous paragraph, it suffices to show that some family of generating morphisms for $\APar$ are all good. Hence, in view of \cref{red2}, 
it is enough to show that $i(f)$
is good for every morphism $f$ in $\Par$ and that the morphisms \cref{firstdot} are good
for all $n$.
For $f \in \Hom_{\Par}(n,m)$, the morphism
$i(f)$ is good with $\overline{i(f)} := f$. This follows from the following calculation using \cref{lalala}:
$$
\gamma^{(t)}_m \circ \Ev(\Phi_t(i(f))) \circ \big(\gamma^{(t)}_m\big)^{-1}
= \gamma^{(t)}_m \circ \Ev(\Act(\phi_t(f))) \circ \big(\gamma^{(t)}_m\big)^{-1}
= \phi_t(f).
$$
Also the morphism $f$ from \cref{firstdot} is good for every $n$.
To see this, let $\bar f$ be the morphism on the right hand side of \cref{boycott}. Using the definition in \cref{SWD}, $\phi_t(\bar f)$ is the map
$u_{i_n} \otimes \cdots \otimes u_{i_1} \mapsto \sum_{j=1}^t
u_{i_n}\otimes \cdots\otimes
u_{i_2} \otimes u_j$.
Also using the definition in \cref{PHIlem}, $\Ev(\Phi_t(f))$ is the map
$u_{i_n} \otimes \cdots \otimes u_{i_1} \otimes 1 \mapsto \sum_{j=1}^t
u_{i_n}\otimes \cdots\otimes
u_{i_2} \otimes u_j\otimes 1$.
On contracting the final $\otimes 1$ using $\gamma^{(t)}_n$, these are equal,
as required to prove that $f$ is good.

Now we can define a $\kk$-linear functor $p$ making \cref{EvDgn}
  commute (up to natural isomorphism) for all $t \in \N$.
  On objects, define $p$ by declaring that
  $p(n)= n$ for each $n \geq 0$.
  On a morphism $f \in \Hom_{\APar}(n,m)$, we define
  $p(f) := \bar f$. The checks made so far imply that this is a well-defined
  $\kk$-linear functor satisfying \cref{boycott}.
  The equation \cref{mastereq} shows that
  $\gamma^{(t)} = \big(\gamma^{(t)}_n\big)_{n \geq 1}:\Ev\circ\Phi_t \Rightarrow
  \phi_t \circ p$ is a natural isomorphism.
  We have also already shown that $p\circ i = \Id_{\Par}$
  and that \cref{boycott} holds.
    Thus, we have established the existence of a $\kk$-linear functor
  $p:\APar\rightarrow\Par$ satisfying all of the properties in the statement of the theorem. The uniqueness of $p$ follows from \cref{red2}.

    It remains to check the three properties \cref{babyjoey}.
    These can be checked using the commutativity of \cref{EvDgn} in the same way as we just established \cref{boycott}. Alternatively, and possibly quicker, they can be deduced
    directly from \cref{boycott} using the relations \cref{rightdot,crazycrossing,june},
    respectively. We leave the details to the reader.
\end{proof}

The faithfulness of $i$ in the
following corollary was already proved in two different ways in \cite{LSR}.
Our approach is similar in spirit to the first proof given in
{\em loc. cit.}, i.e., the
argument used to prove \cite[Th.~5.2]{LSR}.

\begin{corollary}\label{lscor}
  The functor $i:\Par\rightarrow \APar$ is faithful
  and the functor $p:\APar \rightarrow \Par$ is full.
\end{corollary}

\begin{proof} This follows because $p \circ i = \Id_{\Par}$.
\end{proof}

\begin{corollary}\label{cyclotomic}
    The functor $p$ induces an isomorphism
  $\APar / \mathcal{I} \stackrel{\sim}{\rightarrow} \Par$ where $\mathcal{I}$ is the left tensor ideal of $\APar$ generated by the morphism
  $\begin{tikzpicture}[anchorbase,scale=1.2]
\draw[-,thick](0,0)--(0,0.5);
\draw[-,thick](0,0.25)--(-0.15,0.25);
\node  at (-0.15,0.25) {$\bullet$};
  \end{tikzpicture}
  -\begin{tikzpicture}[anchorbase,scale=1.2]\draw[-,thick](0,0)to(0,0.15);\opendot{0,0.15};\draw[-,thick](0,.35)to(0,0.5);\opendot{0,0.35};\end{tikzpicture}\;$.
\end{corollary}

\begin{proof}
  The left tensor ideal $\mathcal{I}$ is the data of subspaces
  $\mathcal{I}(n, m)$ of $\Hom_{\APar}(n,m)$
  for each $m,n \geq 0$ which are closed under vertical composition on the top or bottom with any morphism
  and closed under horizontal composition on the left with any morphism.
  It is clear from \cref{boycott}
  that $p$ sends morphisms in $\mathcal{I}$ to zero, hence, $p$ induces a $\kk$-linear functor
  $\bar p:\APar / \mathcal{I} \rightarrow \Par$. This is surjective on objects
  and full. To see that it is faithful, suppose that
  $f+\mathcal{I}(n,m)\in\Hom_{\APar / \mathcal{I}}(n,m) = \Hom_{\APar}(n,m) / \mathcal I(n,m)$
  is a morphism sent to zero by $\bar p$, hence, $p(f) = 0$.
  In view of \cref{red2} and the definition of $\mathcal I$, we may assume that $f = i(\bar f)$ for some $\bar f \in \Hom_{\Par}(n,m)$.
  Then $\bar f = p(i(\bar f)) = p(f) = 0$, so that $f = i(\bar f) = 0$.
  \end{proof}

Composing the functor $p:\APar \rightarrow \Par$
      with evaluation at any $t\in \kk$ gives a full $\kk$-linear functor
      \begin{align}\label{pt}
        p_t:=\ev_t \circ p&:\APar \rightarrow \Par_t
      \end{align}
      such that
\begin{align}\label{bigjoey1}
  \begin{tikzpicture}[baseline=3mm]
\draw[-,thick] (0,0) to (0,0.8);
\node at (0.1,0.4) {$\cdot$};
\node at (0.25,0.4) {$\cdot$};
\node at (0.4,0.4) {$\cdot$};
\draw[-,thick] (0.5,0) to (0.5,0.8);
\draw[-,thick] (0.85,0) to (0.85,0.8);
\draw[-,thick] (0.65,0.4) to (0.85,0.4);
\node at (0.85,-.15) {$\stringlabel{1}$};
\node at (0.5,-.15) {$\stringlabel{2}$};
\node at (0,-.15) {$\stringlabel{n}$};
\node at (0.65,0.4) {$\bullet$};
\end{tikzpicture} &\mapsto \begin{tikzpicture}[baseline=3mm]
\draw[-,thick] (0,0) to (0,0.8);
\node at (0.1,0.4) {$\cdot$};
\node at (0.25,0.4) {$\cdot$};
\node at (0.4,0.4) {$\cdot$};
\draw[-,thick] (0.5,0) to (0.5,0.8);
\draw[-,thick] (0.85,0) to (0.85,0.2);
\draw[-,thick] (0.85,.8) to (0.85,0.6);
\node at (0.85,-.15) {$\stringlabel{1}$};
\node at (0.5,-.15) {$\stringlabel{2}$};
\node at (0,-.15) {$\stringlabel{n}$};
\node at (0.85,0.25) {$\circ$};
\node at (0.85,0.53) {$\circ$};
  \end{tikzpicture}
\ ,&
    \begin{tikzpicture}[baseline=3mm]
\draw[-,thick] (0,0) to (0,0.8);
\node at (0.1,0.4) {$\cdot$};
\node at (0.25,0.4) {$\cdot$};
\node at (0.4,0.4) {$\cdot$};
\draw[-,thick] (0.5,0) to (0.5,0.8);
\draw[-,thick] (0.85,0) to (0.85,0.8);
\draw[-,thick] (0.85,0.4) to (1.05,0.4);
\node at (0.85,-.15) {$\stringlabel{1}$};
\node at (0.5,-.15) {$\stringlabel{2}$};
\node at (0,-.15) {$\stringlabel{n}$};
\node at (1.05,0.4) {$\bullet$};
\end{tikzpicture} &\mapsto t\begin{tikzpicture}[baseline=3mm]
\draw[-,thick] (0,0) to (0,0.8);
\node at (0.1,0.4) {$\cdot$};
\node at (0.25,0.4) {$\cdot$};
\node at (0.4,0.4) {$\cdot$};
\draw[-,thick] (0.5,0) to (0.5,0.8);
\draw[-,thick] (0.85,0) to (0.85,0.8);
\node at (0.85,-.15) {$\stringlabel{1}$};
\node at (0.5,-.15) {$\stringlabel{2}$};
\node at (0,-.15) {$\stringlabel{n}$};
\end{tikzpicture}\ ,\\
    \begin{tikzpicture}[baseline=3mm]
\draw[-,thick] (0,0) to (0,0.8);
\node at (0.1,0.4) {$\cdot$};
\node at (0.25,0.4) {$\cdot$};
\node at (0.4,0.4) {$\cdot$};
\draw[-,thick] (0.5,0) to (0.5,0.8);
\draw[-,thick] (0.85,0) to (1.2,0.8);
\draw[-,thick] (1.025,0.4) to (1.3,0.4);
\draw[-,thick] (1.2,0) to (0.85,0.8);
\draw[-,thick] (1.2,0.4) to (1.05,0.4);
\node at (1.2,-.15) {$\stringlabel{1}$};
\node at (.85,-.15) {$\stringlabel{2}$};
\node at (0.5,-.15) {$\stringlabel{3}$};
\node at (0,-.15) {$\stringlabel{n}$};
\node at (1.3,0.4) {$\bullet$};
    \end{tikzpicture} &\mapsto
    \begin{tikzpicture}[baseline=3mm]
\draw[-,thick] (0,0) to (0,0.8);
\node at (0.1,0.4) {$\cdot$};
\node at (0.25,0.4) {$\cdot$};
\node at (0.4,0.4) {$\cdot$};
\draw[-,thick] (0.5,0) to (0.5,0.8);
\draw[-,thick] (0.85,0) to (1.2,0.8);
\draw[-,thick] (1.2,0) to (0.85,0.8);
\node at (1.2,-.15) {$\stringlabel{1}$};
\node at (.85,-.15) {$\stringlabel{2}$};
\node at (0.5,-.15) {$\stringlabel{3}$};
\node at (0,-.15) {$\stringlabel{n}$};
\end{tikzpicture}\ ,&
\begin{tikzpicture}[baseline=3mm]
\draw[-,thick] (0,0) to (0,0.8);
\node at (0.1,0.4) {$\cdot$};
\node at (0.25,0.4) {$\cdot$};
\node at (0.4,0.4) {$\cdot$};
\draw[-,thick] (0.5,0) to (0.5,0.8);
\draw[-,thick] (0.85,0) to (1.2,0.8);
\draw[-,thick] (1.025,0.4) to (.725,0.4);
\draw[-,thick] (1.2,0) to (0.85,0.8);
\node at (1.2,-.15) {$\stringlabel{1}$};
\node at (.85,-.15) {$\stringlabel{2}$};
\node at (0.5,-.15) {$\stringlabel{3}$};
\node at (0,-.15) {$\stringlabel{n}$};
\node at (.725,0.4) {$\bullet$};
\end{tikzpicture}&\mapsto\begin{tikzpicture}[baseline=3mm]
\draw[-,thick] (0,0) to (0,0.8);
\node at (0.1,0.4) {$\cdot$};
\node at (0.25,0.4) {$\cdot$};
\node at (0.4,0.4) {$\cdot$};
\draw[-,thick] (0.5,0) to (0.5,0.8);
\draw[-,thick] (0.85,0) to (.85,0.8);
\draw[-,thick] (1.2,0) to (1.2,0.8);
\node at (1.2,-.15) {$\stringlabel{1}$};
\node at (.85,-.15) {$\stringlabel{2}$};
\node at (0.5,-.15) {$\stringlabel{3}$};
\node at (0,-.15) {$\stringlabel{n}$};
\end{tikzpicture}\ .
\label{bigjoey2}
\end{align}
Like in \cref{cyclotomic}, the functor $p_t$ induces an isomorphism
  $\APar / \mathcal{I}_t \stackrel{\sim}{\rightarrow} \Par_t$ where $\mathcal{I}_t$ is the left tensor ideal of $\APar$ generated by $T - t 1_\one$ and
  $\begin{tikzpicture}[anchorbase,scale=1.2]
\draw[-,thick](0,0)--(0,0.5);
\draw[-,thick](0,0.25)--(-0.15,0.25);
\node  at (-0.15,0.25) {$\bullet$};
  \end{tikzpicture}
    -\begin{tikzpicture}[anchorbase,scale=1.2]\draw[-,thick](0,0)to(0,0.15);\opendot{0,0.15};\draw[-,thick](0,.35)to(0,0.5);\opendot{0,0.35};\end{tikzpicture}\;$.

    \subsection{Jucys-Murphy elements for partition algebras}
      Now we can explain how affine partition category is related to the works of
      Enyang \cite{En} and Halverson-Ram \cite{HR}.      
      These are concerned with the {\em partition algebra}, which
      is the endomorphism algebra
\begin{equation}
  P_n(t) := \End_{\Par_t}(n) = 1_n Par_t 1_n.
\end{equation}
By analogy, we define the {\em affine partition algebra} to be
\begin{equation}
  AP_n := \End_{\APar}(n) = 1_n APar 1_n.
\end{equation}
Let us denote the elements of $AP_n$ defined by the
left and right dots on the $j$th string by
$X_j^L$ and $X_j^R$, and the elements defined by the left and right crossings
of the $k$th and $(k+1)$th strings by $S_k^L$ and $S_k^R$:
\begin{align}\label{run1}
 X_j^L &:=
  \begin{tikzpicture}[baseline=3mm]
\draw[-,thick] (0,0) to (0,0.8);
\node at (0.1,0.4) {$\cdot$};
\node at (0.25,0.4) {$\cdot$};
\node at (0.4,0.4) {$\cdot$};
\draw[-,thick] (0.85,0) to (.85,0.8);
\node at (1.15,0.4) {$\cdot$};
\node at (1.3,0.4) {$\cdot$};
\node at (1.45,0.4) {$\cdot$};
\draw[-,thick] (1.55,0) to (1.55,0.8);
\draw[-,thick] (.65,.4) to (.85,0.4);
\node at (1.55,-.15) {$\stringlabel{1}$};
\node at (0.85,-.15) {$\stringlabel{j}$};
\node at (0,-.15) {$\stringlabel{n}$};
\node at (.65,0.4) {$\bullet$};
\end{tikzpicture},&
  X_j^R &:=\begin{tikzpicture}[baseline=3mm]
\draw[-,thick] (0,0) to (0,0.8);
\node at (0.1,0.4) {$\cdot$};
\node at (0.25,0.4) {$\cdot$};
\node at (0.4,0.4) {$\cdot$};
\draw[-,thick] (0.7,0) to (.7,0.8);
\node at (1.15,0.4) {$\cdot$};
\node at (1.3,0.4) {$\cdot$};
\node at (1.45,0.4) {$\cdot$};
\draw[-,thick] (1.55,0) to (1.55,0.8);
\draw[-,thick] (.7,.4) to (.9,0.4);
\node at (1.55,-.15) {$\stringlabel{1}$};
\node at (0.7,-.15) {$\stringlabel{j}$};
\node at (0,-.15) {$\stringlabel{n}$};
\node at (.9,0.4) {$\bullet$};
\end{tikzpicture},\\
  S_k^L &:=\begin{tikzpicture}[baseline=3mm]
\draw[-,thick] (-.2,0) to (-.2,0.8);
\node at (-0.1,0.4) {$\cdot$};
\node at (0.05,0.4) {$\cdot$};
\node at (0.2,0.4) {$\cdot$};
\draw[-,thick] (0.5,0) to (.9,0.8);
\draw[-,thick] (0.9,0) to (.5,0.8);
\node at (1.2,0.4) {$\cdot$};
\node at (1.35,0.4) {$\cdot$};
\node at (1.5,0.4) {$\cdot$};
\draw[-,thick] (1.6,0) to (1.6,0.8);
\draw[-,thick] (.7,.4) to (.5,0.4);
\node at (1.6,-.15) {$\stringlabel{1}$};
\node at (0.9,-.15) {$\stringlabel{k}$};
\node at (0.5,-.15) {$\stringlabel{k\!+\!1}$};
\node at (-.2,-.15) {$\stringlabel{n}$};
\node at (.5,0.4) {$\bullet$};
\end{tikzpicture},&
  S_k^R &:=\begin{tikzpicture}[baseline=3mm]
\draw[-,thick] (-.2,0) to (-.2,0.8);
\node at (-0.1,0.4) {$\cdot$};
\node at (0.05,0.4) {$\cdot$};
\node at (0.2,0.4) {$\cdot$};
\draw[-,thick] (0.5,0) to (.9,0.8);
\draw[-,thick] (0.9,0) to (.5,0.8);
\node at (1.2,0.4) {$\cdot$};
\node at (1.35,0.4) {$\cdot$};
\node at (1.5,0.4) {$\cdot$};
\draw[-,thick] (1.6,0) to (1.6,0.8);
\draw[-,thick] (.7,.4) to (.9,0.4);
\node at (1.6,-.15) {$\stringlabel{1}$};
\node at (0.9,-.15) {$\stringlabel{k}$};
\node at (0.5,-.15) {$\stringlabel{k\!+\!1}$};
\node at (-.2,-.15) {$\stringlabel{n}$};
\node at (.9,0.4) {$\bullet$};
\end{tikzpicture}\label{run2}
\end{align}
for $1 \leq j \leq n$ and $1 \leq k \leq n-1$.
We note that $\left\{X_j^L, X_j^R\:\big|\:j=1,\dots,n\right\}$ are algebraically
independent, so they generate a free polynomial algebra of rank $2n$
inside $AP_n(t)$;
his follows easily from the basis theorem for morphism spaces
$\Heis$ proved in \cite{K}.
Taking the images of the elements \cref{run1,run2} under the functor $p_t$ from \cref{pt}
gives us elements of $P_n(t)$ denoted
\begin{align}\label{walk1}
  x_j^L &:=
  p_t(X_j^L),&
  x_j^R &:=p_t(X_k^R),&
  s_k^L &:=p_t(S_k^L),&
  s_k^R &:=p_t(S_k^R).
\end{align}
The notation here depends
implicitly on the values of $n$ and $t$, which should be
clear from the context.
By \cref{bigjoey1,bigjoey2}, we have that
$x_1^L = \begin{tikzpicture}[anchorbase,scale=1.2]
  \draw[-,thick](-.5,0)--(-.5,.5);
  \node at (-.25,.25) {$\scriptstyle\cdots$};
\draw[-,thick](0,0)--(0,0.5);
\draw[-,thick](0.15,.15)--(0.15,0);
\draw[-,thick](0.15,.35)--(0.15,.5);
\node  at (0.15,0.15) {$\dt$};
\node  at (0.15,0.35) {$\dt$};
\end{tikzpicture}$,
$x_1^R = t$,
$s_1^L = 1$ and $s_1^R = (1\:2) \in S_n \subset P_n(t)$.

\begin{theorem}\label{dictionary}
  Suppose that $t \in \N$ and let $\psi_t:P_n(t) \rightarrow
\End_{\kk S_t}(\Nat^{\otimes n})$
be the homomorphism induced by the functor $\phi_t$ from \cref{SWD}.
  The elements $x_j^L, x_j^R, s_k^L, s_k^R \in P_n(t)$ satisfy
  \begin{align}\label{meow1}
    \psi_t(x_j^L) (u_{i_n}\otimes\cdots\otimes u_{i_1}) &=
    \sum_{i=1}^t
    u_{i_n}\otimes\cdots\otimes u_{i_{j+1}} \otimes (i\:\; i_j)
    \left[u_{i_j} \otimes \cdots \otimes  u_{i_2}\otimes u_{i_1}\right],\\\label{meow2}
    \psi_t(x_j^R) (u_{i_n}\otimes\cdots\otimes u_{i_1}) &=
    \sum_{i=1}^t
    u_{i_n}\otimes\cdots\otimes u_{i_j} \otimes (i\:\; i_j) \left[u_{i_{j-1}}\otimes\cdots \otimes u_{i_2}\otimes u_{i_1}\right],\\\label{meow3}
    \psi_t(s_k^L) (u_{i_n}\otimes\cdots\otimes u_{i_1}) &=
    u_{i_n}\otimes\cdots\otimes u_{i_k} \otimes (i_k\:\; i_{k+1}) \left[u_{i_{k-1}}\otimes\cdots \otimes u_{i_2}\otimes u_{i_1}\right],\\\label{meow4}
    \psi_t(s_k^R) (u_{i_n}\otimes\cdots\otimes u_{i_1}) &=
    u_{i_n}\otimes\cdots \otimes u_{i_{k+2}} \otimes (i_k\:\; i_{k+1}) \left[u_{i_{k+1}}\otimes\cdots \otimes  u_{i_2}\otimes u_{i_1} \right]
  \end{align}
  for $1 \leq i_1,\dots,i_n \leq t$, where we are using the diagonal action of
$S_t$ on tensor powers of $\Nat$.
\end{theorem}

\begin{proof}
  This follows from the commutativity of \cref{EvDgn}, \cref{walk1} and the formulae in \cref{PHIlem}.
  \end{proof}

\begin{corollary}\label{agree}
  Identifying $P_n(t)$ with the partition algebra in \cite{En}
  by reflecting diagrams through a vertical axis to account for the fact that we number vertices from right to left rather than from left to right,
  the elements \cref{walk1}
  are related to the elements $L_{\frac{1}{2}}, L_1,\dots$ and $\sigma_{\frac{3}{2}},\sigma_2,\dots$ of the partition algebra $P_n(t)$ defined 
   in \cite{En} according to the dictionary
    \begin{align}\label{ducktionary}
    x_j^L &\leftrightarrow L_j, &t-x_j^R &\leftrightarrow
    L_{j-\frac{1}{2}}, &s_k^L &\leftrightarrow \sigma_{k+\frac{1}{2}},
&    s_k^R &\leftrightarrow \sigma_{k+1}.
    \end{align}
    Hence, by \cite[Th.~5.5]{En}, the elements $x_j^L$ and $t-x_j^R$
    are identified with the Jucys-Murphy elements introduced originally by Halverson and Ram in \cite{HR}.
\end{corollary}

\begin{proof}
  Enyang's elements are defined by a recurrence relation which is independent of the value of the parameter $t$. Hence, his elements can be viewed as specializations at $T=t$ of corresponding elements of the generic partition algebra $\End_{\Par}(n)$.
To identify them with our elements, we can use \cref{SWDapp} to see that it suffices to check that they act in the same way on $\Nat^{\otimes n}$ for infinitely many values of
the parameter $t \in \N$. This follows on comparing \cref{meow1,meow2,meow3,meow4}
to the formulae in
\cite[Prop.~5.2, Prop.~5.3]{En}.
\end{proof}

\begin{remark}
  Alternatively, one can prove \cref{agree} inductively, using the recurrence relations in \cref{amazinggrace} which are equivalent to Enyang's recurrence relations \cite[(3.1)--(3.4)]{En}. In fact, all of the relations derived in {\em loc. cit.} can now be deduced easily using the relations in $\APar$ derived in the previous subsection.
\end{remark}

\begin{remark}\label{creedonrem}
  Recently, Creedon \cite{Creedon} has
  introduced a renormalization of the Jucys-Murphy elements, which he denotes
by $N_1,N_2,\dots,N_{2n} \in P_n(t)$. They are defined in terms of the Enyang-Halverson-Ram elements simply by $N_{2j-1} := L_{j-\frac{1}{2}} - {\textstyle\frac{t}{2}}$ and
$N_{2j} := L_j - {\textstyle\frac{t}{2}}$.
The dictionary between Creedon's elements and ours is
    \begin{align}\label{creedondict}
    x_j^L-\textstyle{\frac{t}{2}} &\leftrightarrow N_{2j}, &{\textstyle\frac{t}{2}}-x_j^R &\leftrightarrow
    N_{2j-1}.
    \end{align}
The motivation for such a
  renormalization will be discussed further in \cref{creedoncenter} below.
  \end{remark}

\subsection{Central elements}\label{centelts}
By the {\em center} of a $\kk$-linear category $\cA$, we mean the (unital)
commutative algebra
$Z(\cA) := \End_{\cA}(\Id_{\cA})$ of endomorphisms of the identity endofunctor of $\cA$.
Thus, an element $z \in Z(\cA)$ is a tuple $(z_X)_{X \in \ob \cA}$ such that
$z_Y \circ f = f \circ z_X$ for all morphisms $f:X \rightarrow Y$ in $\cA$.
Equivalently, in terms of the path algebra $A$,
it is the algebra
\begin{equation}\label{beaches}
Z(A) := \left\{z = \left(z_X\right)_{X \in \OO_A} \in \prod_{X \in \OO_A} 1_X A 1_X\:\Bigg|\:
z a = a z\text{ for all }a \in A\right\},
\end{equation}
interpreting the products in the obvious way.
We note that there is an algebra isomorphism
\begin{equation}\label{altcenterdef}
\End_{A \boxtimes A^{\op}}(A) \stackrel{\sim}{\rightarrow} Z(A),
\qquad
\zeta \mapsto \left(\zeta(1_X)\right)_{x \in \OO_A} \in \prod_{X \in \OO_A} 1_X A 1_X,
\end{equation}
where the algebra on the left is
the endomorphism algebra of the
$A \boxtimes A^\op$-module associated to the
$(A,A)$-bimodule $A$.
If $A$ is locally finite-dimensional, then it is a locally finite-dimensional
$A \boxtimes A^{\op}$-module, hence, by \cite[Lem.~2.10]{BS},
the endomorphism algebra
$\End_{A\boxtimes A^{\op}}(A) \cong Z(A)$ is a
{pseudo-compact topological algebra} with respect to the pro-finite topology (ideals of finite codimension form a base of neighborhoods of zero).

  In the locally finite-dimensional case, $Z(A)$ is isomorphic to the algebra $C(A)^*$ that is the linear dual of the
  {\em cocenter} $C(A)$. The cocenter is a cocommutative coalgebra isomorphic
  to $\operatorname{Coend}_{A\boxtimes A^{\op}}(A)$ in the notation of \cite[(2.15)]{BS}. To define $C(A)$ explicitly, note that the space
  $D := \bigoplus_{X,Y \in \OO_A} (1_X A 1_Y)^*$ is naturally an $(A,A)$-bimodule
  with $1_Y D 1_X = (1_X A 1_Y)^*$.
  Also each $1_X D 1_X$ is a coalgebra as it is the dual of the finite-dimensional algebra $1_X A 1_X$. Hence,
  $\bigoplus_{X \in \OO_A} 1_X D 1_X$ is a coalgebra. Then the cocenter is
  \begin{equation}\label{cocenter}
  C(A) := \Big(\bigoplus_{X\in \OO_A} 1_X D 1_X\Big) \Big / J
  \end{equation}
  where $J$ is the coideal spanned by the elements
$\left\{a f - f a\:\big|\:X,Y \in \OO_A, a \in 1_X A 1_Y, f \in 1_Y D 1_X\right\}$.
To identify $C(A)^*$ with $Z(A)$, note that the linear dual of the coalgebra
$\bigoplus_{X \in \OO_A} 1_X D 1_X$ is the algebra
$\prod_{X \in \OO_A} 1_X A 1_X$; the annihilator $J^\circ$
of the coideal $J$ defines a subalgebra of
$\prod_{X \in \OO_A} 1_X A 1_X$
which is exactly the center $Z(A)$ according
to the original definition \cref{beaches}.

In this subsection, we are going to construct a
family of elements
$(z^{(r)})_{r \geq 1}$
in the center $Z(APar_t)$ of the affine partition category $\APar_t$.
We start by introducing some convenient shorthand.
Given a monomial $x^r y^s \in \kk[x,y]$, we use the notation
\begin{equation}\label{tricky}
\begin{tikzpicture}[anchorbase,scale=1.8]
	\draw[-,thick] (0,0) to (0,.4);
      \node at (0,0.2) {$\bullet$};
      \node at (0.2,.2) {$\scriptstyle x^r y^s$};
\end{tikzpicture} :=
\left(\:\begin{tikzpicture}[anchorbase,scale=1.8]
\draw[-,thick](0,0)--(0,0.4);
\draw[-,thick](0,0.2)--(0.15,0.2);
\node  at (0.15,0.2) {$\bullet$};
\end{tikzpicture}\!\!\right)^{\circ r}
\circ
\left(\!\!
\begin{tikzpicture}[anchorbase,scale=1.8]
\draw[-,thick](0,0)--(0,0.4);
\draw[-,thick](0,0.2)--(-0.15,0.2);
\node  at (-0.15,0.2) {$\bullet$};
\end{tikzpicture}\:\right)^{\circ s}
\end{equation}
to denote the element of $\End_{\APar}(\mid)$ on the right hand side, that is, it is the $r$th power of the right dot (represented by $x$) composed with the $s$th power of the left dot (represented by $y$).
It then makes sense to label dots by polynomials $f(x) \in \kk[x,y]$, meaning
the linear combination of the morphisms
$\begin{tikzpicture}[baseline = -1.5]
	\draw[-,thick] (0.08,-.15) to (0.08,.3);
      \node at (0.08,0.05) {$\bullet$};
\node at (0.42,.05) {$\scriptstyle x^r y^s$};
\end{tikzpicture}
$
just as $f(x)$ is the linear combination of its monomials.
We are also going to use generating functions in the same way as explained in the context of $\Heis$ in \cite[$\S$3.1]{BSW}.
For these, $u$ will be a formal variable which should always be
interpreted by expanding as formal Laurent series in $\kk(\!(u^{-1})\!)$, e.g.,
$(u-x)^{-1} = u^{-1} + u^{-2} x + u^{-3} x^2 +
\cdots.$

Let
\begin{align}
\bigcirc(u) &:=
u 1_\one - \begin{tikzpicture}[anchorbase,scale=1.8]
	\draw[-,thick] (0,0) to (0,.4);
      \node at (0,0.2) {$\bullet$};
      \node at (0,0.4) {$\dt$};
      \node at (0,0) {$\dt$};
      \node at (0.35,.2) {$\scriptstyle(u-x)^{-1}$};
\end{tikzpicture}
=
u 1_\one -
\begin{tikzpicture}[anchorbase,scale=1.8]
	\draw[-,thick] (0,0) to (0,.4);
      \node at (0,0.2) {$\bullet$};
      \node at (0,0.4) {$\dt$};
      \node at (0,0) {$\dt$};
      \node at (0.35,.2) {$\scriptstyle(u-y)^{-1}$};
     \end{tikzpicture}
\in u 1_\one + u^{-1}\End_{\APar}(\one)[\![u^{-1}]\!].
             \label{igp}
\end{align}
For $r \geq 0$, the coefficient of $u^{-r-1}$ in this formal Laurent series
is $-\begin{tikzpicture}[anchorbase]
  \draw[-,thick] (0,0.1)--(0,0.5);
  \node at (0,0.1) {$\dt$};
  \node at (0,0.3) {$\bullet$};
  \node at (0.3,0.32) {$\scriptstyle{x^r}$};
\node at (0,0.5) {$\dt$};
\end{tikzpicture}$; the $x^r$ here can be replaced by $y^r$ due to
the third relation in \cref{rightslide}.
Also introduce the rational function
\begin{equation}\label{alphadef}
\alpha_x(u) := \frac{(u-(x+1))(u-(x-1))}{(u-x)^2} \in \kk(x,u).
\end{equation}
The expansion of this as a power series in $\kk[x][\![u^{-1}]\!]$
is
\begin{align}\label{attack1}
  \alpha_x(u)
  &= 1 - (u-x)^{-2} = 1 - u^{-2} - 2x u^{-3} - 3x^2 u^{-4} - 4x^3 u^{-5}
  -\cdots,\\
  \alpha_x(u)^{-1} &=1 + u^{-2} + 2x u^{-3} + (3x^2+1) u^{-4} + (4x^3+4x) u^{-5}+
  \cdots.\label{attack2}
\end{align}
The following elementary lemma will play a fundamental role in the rest of the article. It would be hard to formulate this without the aid of generating functions.

\begin{lemma} The following {bubble slide} relations hold in $\APar$:
\begin{align}\label{bubbleslide1}
\begin{tikzpicture}[anchorbase,scale=1.8]
\node at (-.3,.25) {$\bigcirc(u)$};
  \draw[-,thick] (0.02,0) to (0.02,.5);
\end{tikzpicture}
&=
\begin{tikzpicture}[anchorbase,scale=1.8]
	\draw[-,thick] (0,0) to (0,.5);
      \node at (0,0.25) {$\bullet$};
      \node at (-0.3,.25) {$\frac{\alpha_y(u)}{\alpha_x(u)}$};
  \node at (.45,.25) {$\bigcirc(u)$};
\end{tikzpicture},
&\begin{tikzpicture}[anchorbase,scale=1.8]
\node at (.3,.25) {$\bigcirc(u)$};
  \draw[-,thick] (-0.02,0) to (-0.02,.5);
\end{tikzpicture}
&=
\begin{tikzpicture}[anchorbase,scale=1.8]
	\draw[-,thick] (0,0) to (0,.5);
      \node at (0,0.25) {$\bullet$};
      \node at (0.3,.25) {$\frac{\alpha_x(u)}{\alpha_y(u)}$};
  \node at (-.45,.25) {$\bigcirc(u)$};
\end{tikzpicture}.
\end{align}
\end{lemma}

\begin{proof}
  The two equations are equivalent, so we just prove the first one.
  When working with $\Heis$, we adopt the notation of \cite[$\S$3.1]{BSW}:
  an open dot labelled by $x^r$ means the $r$th power of the open dot in $\Heis$, and
  $\clock(u)$ is the formal Laurent series from \cite[(3.13)]{BSW}.
Under the embedding of $\APar$ into $\Heis$, we have that
$$\bigcirc(u+1)
=
(u+1) 1_\one-\begin{tikzpicture}[anchorbase,scale=1.8]
	\draw[-,thick] (0,0) to (0,.4);
      \node at (0,0.2) {$\bullet$};
      \node at (0,0.4) {$\dt$};
      \node at (0,0) {$\dt$};
      \node at (0.55,.2) {$\scriptstyle(u-(y-1))^{-1}$};
     \end{tikzpicture}
=
(u+1) 1_\one-\begin{tikzpicture}[anchorbase,scale=1.8]
\draw[-] (0,-0.2) to[out=180,in=-90] (-.2,0);
\draw[->] (-0.2,0) to[out=90,in=180] (0,0.2);
\draw[-] (0,0.2) to[out=12,in=90] (0.2,0);
\draw[-] (0.2,0) to[out=-90,in=0] (0,-0.2);
      \node at (-.2,0) {$\dt$};
      \node at (-0.54,0.05) {$\scriptstyle(u-x)^{-1}$};
     \end{tikzpicture}
=
1_\one + \clock(u).
$$
The bubble slide relation for $\Heis$ from \cite[(3.18)]{BSW} gives that
$$
\begin{tikzpicture}[anchorbase,scale=1.8]
\node at (-.4,.25) {$\clock(u)$};
  \draw[->] (0.02,0) to (0.02,.5);
  \draw[<-] (0.3,0) to (0.3,.5);
\end{tikzpicture}
=
\begin{tikzpicture}[anchorbase,scale=1.8]
  \draw[->] (0,0) to (0,.5);
  \draw[<-] (0.8,0) to (0.8,.5);
      \node at (0,0.25) {$\dt$};
      \node at (-0.27,.25) {$\scriptstyle \alpha_x(u)$};
  \node at (.4,.25) {$\clock(u)$};
\end{tikzpicture}
=
\begin{tikzpicture}[anchorbase,scale=1.8]
  \draw[->] (0,0) to (0,.5);
  \draw[<-] (0.3,0) to (0.3,.5);
      \node at (0,0.25) {$\dt$};
      \node at (0.3,0.25) {$\dt$};
      \node at (-0.27,.25) {$\scriptstyle \alpha_x(u)$};
      \node at (0.65,.25) {$\scriptstyle \alpha_x(u)^{-1}$};
  \node at (1.3,.25) {$\clock(u)$};
\end{tikzpicture}\ .
$$
According to \cref{thedots},
the label $x$ on the open dot on
the $\downarrow$ string
translates into the label $x-1$ on a closed dot in $\APar$, and the
label $x$ on the open dot on
the $\uparrow$ string translates into the label $y-1$ on a closed dot in $\APar$.
So the relation just recorded can be written equivalently as
$$
\begin{tikzpicture}[anchorbase,scale=1.8]
\node at (-.5,.25) {$\bigcirc(u+1)$};
  \draw[-,thick] (0.02,0) to (0.02,.5);
\end{tikzpicture}
=
\begin{tikzpicture}[anchorbase,scale=1.8]
	\draw[-,thick] (0,0) to (0,.5);
      \node at (0,0.25) {$\bullet$};
      \node at (-0.42,.25) {$\frac{\alpha_{y-1}(u)}{\alpha_{x-1}(u)}$};
  \node at (.6,.25) {$\bigcirc(u+1)$};
\end{tikzpicture}
=
\begin{tikzpicture}[anchorbase,scale=1.8]
	\draw[-,thick] (0,0) to (0,.5);
      \node at (0,0.25) {$\bullet$};
      \node at (-0.42,.25) {$\frac{\alpha_y(u+1)}{\alpha_x(u+1)}$};
  \node at (.6,.25) {$\bigcirc(u+1)$};
\end{tikzpicture}\ .
$$
Replacing $u$ by $u-1$ everywhere gives the desired
relation.
\end{proof}

The rational function
$\alpha_y(u)/\alpha_x(u) \in \kk(x,y,u)$ will also be important later on.
The low degree terms of its expansion as a power series in $u^{-1}$
can be computed
using \cref{attack1,attack2}:
\begin{equation}\label{firstfew}
\frac{\alpha_y(u)}{\alpha_x(u)} = 1+ 2(x-y) u^{-3} + 3(x^2-y^2) u^{-4} + \left[4(x^3-y^3)+2(x-y)\right] u^{-5}+\cdots.
\end{equation}
For $n \geq 0$, let
\begin{equation}
C_n(u) = \sum_{r \geq 0}C^{(r)}_n u^{-r}
:=
\bigcirc(u) \star 1_n \star \bigcirc(u)^{-1}
=
\begin{tikzpicture}[anchorbase,scale=1.8]
	\draw[-,thick] (0,0) to (0,.5);
      \node at (0,0.25) {$\bullet$};
      \node at (-0.25,.25) {$\frac{\alpha_y(u)}{\alpha_x(u)}$};
\node at (0.125,0.25) {$\cdot$};
\node at (0.25,0.25) {$\cdot$};
 \node at (0.375,0.25) {$\cdot$};
      	\draw[-,thick] (.5,0) to (.5,.5);
      \node at (.5,0.25) {$\bullet$};
      \node at (0.75,.25) {$\frac{\alpha_y(u)}{\alpha_x(u)}$};
      	\draw[-,thick] (1.05,0) to (1.05,.5);
      \node at (1.05,0.25) {$\bullet$};
      \node at (1.35,.25) {$\frac{\alpha_y(u)}{\alpha_x(u)}$};
      \node at (0,-.1) {$\stringlabel{n}$};
      \node at (.5,-.1) {$\stringlabel{2}$};
      \node at (1.05,-.1) {$\stringlabel{1}$};
\end{tikzpicture}
\in 1_n APar 1_n [\![u^{-1}]\!],\label{bubbleslide2}
\end{equation}
where the final equality follows by applying the bubble slide relation
repeatedly.
Then we define
\begin{equation}
  C(u) = \sum_{r \geq 0} C^{(r)}u^{-r} := \left(C_n(u)\right)_{n \geq 0} \in \prod_{n \geq 0}
 1_n APar 1_n[\![u^{-1}]\!].
\end{equation}
Note by \cref{firstfew} that
$C^{(0)} = 1$ and $C^{(1)} = C^{(2)} = 0$.

\begin{theorem}\label{ourcentralelements}
$C(u) \in Z(APar)[\![u^{-1}]\!]$.
\end{theorem}

\begin{proof}
The interchange law immediately gives that
$$
  \begin{tikzpicture}[baseline=2mm]
  \draw[-,thick] (0.15,0)--(0.15,0.28);
  \draw[-,thick] (.95,0)--(.95,0.28);
  \draw[-,thick] (0.15,.72)--(0.15,1.02);
  \draw[-,thick] (.95,.72)--(.95,1.02);
  \draw[-,thick] (0,-.7)--(0,-0.4);
  \draw[-,thick] (1.1,-.7)--(1.1,-0.4);
    \node[align=center,draw,text width=1cm, text height=.4mm] at (.55,.5) (b)
    {$_{\phantom{f}}$};
         \node at (.55,.5) {$\scriptstyle C_m(u)$};
  \node at (0.57,.14){$\cdots$};
  \node at (0.57,.87){$\cdots$};
  \node at (0.57,-.6){$\cdots$};
    \node[align=center,draw,text width=1cm, text height=.4mm] at (.55,-.2) (b)
         {$_{\phantom f}$};
         \node at (.55,-.2) {$\scriptstyle f$};
  \end{tikzpicture}
=  \begin{tikzpicture}[baseline=2mm]
  \draw[-,thick] (0.15,0)--(0.15,1.02);
  \draw[-,thick] (.95,0)--(.95,1.02);
  \draw[-,thick] (0,-.7)--(0,-0.4);
  \draw[-,thick] (1.1,-.7)--(1.1,-0.4);
  \node at (-.6,.5) {$\bigcirc(u)$};
  \node at (1.9,.5) {$\bigcirc(u)^{-1}$};
  \node at (0.57,.505){$\cdots$};
  \node at (0.57,-.6){$\cdots$};
    \node[align=center,draw,text width=1cm, text height=.4mm] at (.55,-.2) (b)
    {$_{\phantom f}$};
         \node at (.55,-.2) {$\scriptstyle f$};
  \end{tikzpicture}
  =
  \begin{tikzpicture}[baseline=2mm]
  \draw[-,thick] (0.15,.72)--(0.15,1.02);
  \draw[-,thick] (.95,.72)--(.95,1.02);
  \draw[-,thick] (0,-.7)--(0,.3);
  \draw[-,thick] (1.1,-.7)--(1.1,.3);
  \node at (-.6,-.2) {$\bigcirc(u)$};
  \node at (1.9,-.2) {$\bigcirc(u)^{-1}$};
    \node[align=center,draw,text width=1cm, text height=.4mm] at (.55,.5) (b)
    {$_{\phantom{f}}$};
         \node at (.55,.5) {$\scriptstyle f$};
  \node at (0.57,.87){$\cdots$};
  \node at (0.57,-.1){$\cdots$};
  \end{tikzpicture}
  =
  \begin{tikzpicture}[baseline=2mm]
  \draw[-,thick] (0,0)--(0,0.28);
  \draw[-,thick] (1.1,0)--(1.1,0.28);
  \draw[-,thick] (0.15,.72)--(0.15,1.02);
  \draw[-,thick] (.95,.72)--(.95,1.02);
  \draw[-,thick] (0,-.7)--(0,-0.4);
  \draw[-,thick] (1.1,-.7)--(1.1,-0.4);
    \node[align=center,draw,text width=1cm, text height=.4mm] at (.55,.5) (b)
    {$_{\phantom{f}}$};
         \node at (.55,.5) {$\scriptstyle f$};
  \node at (0.57,.14){$\cdots$};
  \node at (0.57,.87){$\cdots$};
  \node at (0.57,-.6){$\cdots$};
    \node[align=center,draw,text width=1cm, text height=.4mm] at (.55,-.2) (b)
    {$_{\phantom{f}}$};
         \node at (.55,-.2) {$\scriptstyle C_n(u)$};
  \end{tikzpicture}
  $$
  for any $f \in \Hom_{\APar}(n,m)$.
\end{proof}

\begin{corollary}
  For each $r \geq 1$, the element $Z^{(r)} = (Z^{(r)}_n)_{n \geq 0}
  \in \prod_{n \geq 0} 1_n APar 1_n$
defined from
 $$
Z^{(r)}_n := \sum_{i=1}^n \left((X_i^L)^r - (X_i^R)^r\right)
=
\left(X_1^L\right)^r+\cdots
+\left(X_n^L\right)^r
-\left(X_1^R\right)^r
-\cdots-\left(X_n^R\right)^r
$$
belongs to $Z(APar)$ (notation as in \cref{run1,run2}).
Moreover, the elements $Z^{(1)}, Z^{(2)}, \dots$ generate the same subalgebra
$Z_0(APar)$
of $Z(APar)$ as the elements $C^{(3)}, C^{(4)}, \dots$.
\end{corollary}

\begin{proof}
  Let $f(u) := \alpha_y(u) / \alpha_x(u)$ for short.
Then define $g(u) := f'(u) / f(u) = \frac{d}{du}(\ln f(u))$ to be its logarithmic derivative.
We have that
  \begin{align*}
g(u) &=\left(-\frac{1}{u-(x-1)}
  +\frac{2}{u-x}-\frac{1}{u-(x+1)}\right)
-
\left(-\frac{1}{u-(y-1)}
+\frac{2}{u-y}
-\frac{1}{u-(y+1)}
\right)\\
&= 2\cdot 3 (y-x) u^{-4} + 2 \cdot 6 (y^2-x^2) u^{-5} +
2 \cdot [10(y^3-x^3) + 5(y-x)]u^{-6} +\cdots.
\end{align*}
We deduce for $r \geq 1$ that the $u^{-r-3}$-coefficient of $g(u)$ is equal to
$2 \binom{r+2}{2} (y^{r}-x^{r})$ plus a linear combination of terms
$(y^s - x^s)$ for $1 \leq s < r$ with $s\equiv r \pmod{2}$.

The coefficients of the power series $C'(u) / C(u)$ are polynomials in the coefficients of the series
$C(u)$.
Hence, by the theorem, these coefficients are all central.
To compute them, we
take logarithmic derivatives of \cref{bubbleslide2} to obtain the identity
$$
C'_n(u) / C_n(u) =
\sum_{i=1}^n 
\begin{tikzpicture}[baseline=3mm]
\draw[-,thick] (0,0) to (0,0.8);
\node at (0.1,0.4) {$\cdot$};
\node at (0.25,0.4) {$\cdot$};
\node at (0.4,0.4) {$\cdot$};
\draw[-,thick] (1.18,0) to (1.18,0.8);
\node at (1.45,0.4) {$\cdot$};
\node at (1.6,0.4) {$\cdot$};
\node at (1.75,0.4) {$\cdot$};
\draw[-,thick] (1.85,0) to (1.85,0.8);
\node at (1.18,.4) {$\bullet$};
\node at (1.85,-.15) {$\stringlabel{1}$};
\node at (1.18,-.15) {$\stringlabel{i}$};
\node at (0,-.15) {$\stringlabel{n}$};
\node at (.82,0.4) {$\scriptstyle g(u)$};
\end{tikzpicture}
\:.
$$
Using the previous paragraph and the definition of $Z^{(r)}$,
we deduce for $r \geq 1$ that the central element defined by the
$u^{-r-3}$-coefficient of
$C'(u)/C(u)$ is equal to $2 \binom{r+2}{2} Z^{(r)}$
plus a linear combination of $Z^{(s)}$ for $1 \leq s < r$
with $s\equiv r \pmod{2}$. Finally, induction on $r$
shows that each $Z^{(r)}$ is central.

The argument just given shows that each $Z^{(r)}$ lies in the subalgebra generated
by $C^{(3)}, C^{(4)},\dots$. Conversely, by exponentiating an anti-derivative of
the series $C'(u) / C(u)$, one shows that each $C^{(r)}$ can
be expressed as a polynomial in $Z^{(1)}, Z^{(2)},\dots$. Hence, the two families of elements generate the same subalgebra of $Z(APar)$.
\end{proof}

Taking the images of $C(u)$ and each $Z^{(r)}$
under the functor
$p_t$ from \cref{pt} give 
\begin{align}\label{koepka}
  c(u) = \sum_{r \geq 0} c^{(r)} u^{-r} &:= \left(c_n(u)\right)_{n \geq 0} \in Z(Par_t)[\![u^{-1}]\!]
  &\text{where }&c_n(u) = \sum_{r \geq 0} c_n^{(r)} u^{-r}:= p_t(C_n(u)),\\
  z^{(r)}&:= \big(z^{(r)}_n\big)_{n \geq 0} \in Z(Par_t)&\text{where }&
  z^{(r)}_n := p_t(Z_n^{(r)}).\label{nadal}
\end{align}
The elements $c^{(r)}_n$ and $z^{(r)}_n$
belong to the center
$Z(P_n(t))$ of the partition algebra $P_n(t)$.
In terms of the Jucys-Murphy elements \cref{walk1},
we have that
\begin{equation}\label{z1}
  z_n^{(r)} = \sum_{i=1}^n \left[\left(x_i^L\right)^r - \left(x_i^R\right)^r\right]
  = (x_1^L)^r + \cdots + (x_n^L)^r - (x_1^R)^r - \cdots - (x_n^R)^r.
\end{equation}
From \cref{agree}, it follows that $z_n^{(1)}$ equals $z_n -nt$ where $z_n$ is the central element from
\cite[Th.~3.10(2)]{En}.
In fact, $z_n^{(1)}$ is closely related to the central elements of the group
algebras $\kk S_t$ defined by sums of transpositions:

\begin{lemma}[{\cite[Prop.~5.4]{En}}]
  If $t \in \N$ then $\psi_t(z^{(1)}_n):\Nat^{\otimes n}\rightarrow \Nat^{\otimes n}$
is equal to the endomorphism defined by the action of $\sum_{1 \leq i < j \leq t}\left((i\:j) - 1\right) \in Z(\kk S_t)$.
\end{lemma}

\begin{remark}\label{creedoncenter}
  After constructing the elements $z_n^{(r)} \in Z(P_n(t))$ in the manner
  explained above,
  we came across a recent paper of Creedon which constructs similar
  central elements; see \cite[Th.~3.2.6]{Creedon}.
  To explain the connection, recall that the {\em $r$th supersymmetric power sum}
  in variables $x_1,\dots,x_n,y_1,\dots,y_m$ is
  $p_r(x_1,\dots,x_n|y_1,\dots,y_m) = x_1^r + \cdots + x_n^r - y_1^r - \cdots - y_m^r$.
  The expression on the right hand side of \cref{z1}
  is $p_r(x_1^L,\dots,x_n^L|x_1^R,\dots,x_n^R)$.
  It is easy to see that these elements belong to $Z(P_n(t))$ for all $r \geq 1$
  if and only if
  $p_r((x_1^L-t/2)^r,\dots,(x_n^L-t/2)^r|(x_1^R-t/2)^r,
  \dots,(x_n^R-t/2)) \in Z(P_n(t))$ for all $r \geq 1$. Moreover,
  $p_r((x_1^L-t/2)^r,\dots,(x_n^L-t/2)^r|(x_1^R-t/2)^r,
  \dots,(x_n^R-t/2)) \in Z(P_n(t))$
  coincides with
  the $r$th supersymmetric power sum
  $p_r(N_2,N_4,\dots,N_{2n}|-N_1,-N_3,\dots,-N_{2n-1})$ in Creedon's renormalized Jucys-Murphy elements from \cref{creedondict}.
  Creedon showed that his elements are central in $P_n(t)$ by a direct check of relations. This gives an independent way to verify
  that $z^{(r)} = (z^{(r)}_n)_{n \geq 0}$ belong to $Z(\Par_t)$: Creedon's calculations show that they commute with all crossings (a surprisingly hard calculation), and after that it is easy to see that they commute with all other generators of $\Par_t$.
\end{remark}

\begin{remark}\label{omegas}
  In \cite[Def.~4.5]{CO}, Comes and Ostrik define another family of
  central elements $\omega^r(t) = (\omega_n^r(t))_{n \geq 0}$
  which lift the central elements of the group algebras
  $\kk S_t$ defined by the sums of all
   $r$-cycles. We expect that our elements $z_n^{(r)}$ and their elements
  $\omega_n^r(t)$ are closely related, but we do not know any explicit formula.
  In particular, the Comes-Ostrik elements should generate the same subalgebra
  of $Z(Par_t)$ as our elements.
\end{remark}

\section{Classification and structure of blocks}\label{sec5}

Now we return to the study of the representation theory of $\Par_t$.
By considering images of the central elements from \cref{centelts}
under an analog of the Harish-Chandra homomorphism, we decompose
$\Par_t\Mod$ as a product
of subcategories, which turn out to be precisely the blocks.
In fact, $Par_t$ is semisimple if and only if $t \notin \N$,
while if $t \in \N$ the
non-simple blocks are in bijection with partitions of $t$.
We also determine the structure of the non-simple blocks and explicitly
show that they are all equivalent, recovering
the results of Comes and Ostrik \cite{CO}.

\subsection{Harish-Chandra homomorphism}\label{HCH}
Although we just explain in the case of $\Par_t$,
the general
development in this subsection is valid for any monoidal triangular category,
replacing $\Sym$ with the (semisimple) Cartan subcategory and replacing the set $\P$ of partitions by a set parametrizing isomorphism classes of
irreducible representations of the Cartan subcategory.

According to the general definition \cref{beaches},
the center of the partition category is a subalgebra of the unital algebra
$\prod_{n \geq 0} 1_n Par_t 1_n$.
Let $K^+$ (resp., $K^-$) be the left ideal (resp., right ideal)
of $Par_t$ generated by the strictly downward partition diagrams (resp., the strictly upward partition diagrams).
From the triangular basis, it is easy to see that
$1_n K^+ 1_n = 1_n K^- 1_n$. We denote this by $K_n$. It is a two-sided ideal of
the finite-dimensional algebra $1_n Par_t 1_n$, and we have that
\begin{equation}\label{playoff}
1_n Par_t 1_n = \kk S_n \oplus K_n.
\end{equation}
Equivalently, $K_n$ is the two-sided ideal of $1_n Par_t 1_n$ spanned by morphisms
that factor through objects $m < n$.
By analogy with Lie theory, we define the {\em Harish-Chandra homomorphism}
\begin{align}\label{feed}
  \widehat{\operatorname{HC}}:\prod_{n \geq 0} 1_n Par_t 1_n &\rightarrow \prod_{n \geq 0} \kk S_n,&
  (z_n)_{n \geq 0} &\mapsto (\HC_n z_n)_{n \geq 0},
\end{align}
where $\HC_n:1_n Par_t 1_n \twoheadrightarrow \kk S_n$ is the projection
along the direct sum decomposition \cref{playoff}.
It is obvious from \cref{playoff} that 
the restriction of $\widehat{\operatorname{HC}}$ to $Z(Par_t)$ defines an algebra homomorphism
  \begin{equation}\HC:Z(Par_t) \rightarrow
    Z(\sym) = \prod_{n \geq 0} Z(\kk S_n).\label{coast}
  \end{equation}

  As each $\kk S_n$ is semisimple with its isomorphism classes of
  irreducible representations parametrized by $\P_n$, we can identify the algebra
appearing on the right hand side of \cref{coast}
with the algebra $\kk[\P]$ of all functions from the set $\P$ to the field $\kk$ with pointwise operations.
Under this identification,
the tuple $(z_n)_{n \geq 0} \in \prod_{n \geq 0} Z(\kk S_n)$
corresponds to the function $f:\P \rightarrow \kk$
such that $f(\lambda)$ is the scalar that $z_n$ acts by on
the Specht module $S(\lambda)$ for each $\lambda \in \P_n$.
Then the Harish-Chandra homorphism becomes a homomorphism
\begin{equation}\label{louis}
\HC:Z(Par_t) \rightarrow \kk[\P].
\end{equation}
To describe $\HC$ more explicitly in these terms,
let $\lambda \in \P_n$ be a partition.
As we have that
$\End_{Par_t}(\Delta(\lambda)) \cong \End_{\sym}(S(\lambda)) \cong \kk$,
an element $z = (z_n)_{n \geq 0} \in Z(Par_t)$ acts on the standard module
$\Delta(\lambda)$ as multiplication by a scalar denoted $\chi_\lambda(z)$.
This defines an algebra homomorphism
\begin{equation}\label{louis2}
\chi_\lambda:Z(Par_t) \rightarrow \kk.
\end{equation}
To compute $\chi_\lambda(z)$, note that it is
the scalar by which  $z_n$ acts on the highest
weight space $1_n \Delta(\lambda)$, which
is the scalar arising from the
action of $\HC_n(z_n) \in Z(\kk S_n)$ on $S(\lambda)$.
It follows that
\begin{equation}\label{louis3}
  \chi_\lambda(z) = \HC_n(z_n)(\lambda) =
  \HC(z)(\lambda).
\end{equation}
Recall that $Z(Par_t)$ is a commutative
pseudo-compact topological algebra with respect to the profinite topology.
Let $\Specm(Z(Par_t))$ be its set of open ($=$ finite-codimensional)
maximal ideals.

\begin{lemma}\label{spectrum}
$\Specm(Z(Par_t)) = \{\ker \chi_\lambda\:|\:\lambda \in \P\}$.
\end{lemma}

\begin{proof}
  Points in $\Specm(Z(Par_t))$ parametrize isomorphism classes of finite-dimensional irreducible modules for $Z(Par_t)$
  Let $L_\lambda$ be the irreducible $Z(Par_t)$-module associated to
  $\chi_\lambda:Z(Par_t) \rightarrow \kk$.
  Then we need to show that any finite-dimensional irreducible $Z(Par_t)$-module $L$ is isomorphic to $L_\lambda$ for some $\lambda \in \P$.
  To see this, we find it easiest to work equivalently
  in terms of irreducible comodules over the
  cocenter $C := C(Par_t)$ defined in \cref{cocenter}.
  So let $L$ be an irreducible $C$-comodule and $L^*$ be the dual comodule, there being no need to distinguish between left or right since $C$ is cocommutative.
  By definition, $C$
  is a quotient of the coalgebra $D$ that is the direct sum of the coalgebras
  $(1_n Par_t 1_n)^*$ for all $n \geq 0$.
  Since $L^*$ is isomorphic to a subcomodule of the regular $C$-comodule,
  it follows that $L^*$ is isomorphic to a subquotient of the restriction of
  the regular $D$-comodule to $C$.
  Hence, $L^*$ is isomorphic to a subquotient of $(1_n Par_t 1_n)^*$
  for some $n$. So $L$ is isomorphic to a subquotient of $1_n Par_t 1_n$.
  Now recall that the left $Par_t$-module $Par_t 1_n$ has a $\Delta$-flag, and
  $z \in Z(Par_t)$ acts on $\Delta(\lambda)$ 
  as multiplication by the scalar $\chi_\lambda(z)$. 
  Hence, all composition factors of
  the finite-dimensional $Z(Par_t)$-module $1_n Par_t 1_n$ are of the form
  $L_\lambda$ for $\lambda \in \P$.
\end{proof}

Let $\approx_t$ be the equivalence relation on $\P$ defined by
\begin{equation}\label{approxdef}
  \lambda \approx_t \mu \Leftrightarrow \chi_\lambda = \chi_\mu.
  \end{equation}
From \cref{spectrum}, we see that the
equivalence classes $\P / \approx_t$ parametrize
the points in $\Specm(Z(Par_t))$.

\begin{lemma}\label{centlem}
  The image of $\HC:Z(Par_t) \rightarrow \kk[\P]$ consists of
   of all functions
  $f \in \kk[\P]$ which are constant on $\approx_t$-equivalence classes.
  Moreover, for each subset $S$ of $\P$ that is a union of $\approx_t$-equivalence classes, there is a unique central idempotent $1_S \in Z(Par_t)$
  such that
  \begin{equation}\label{bourbon}
    \HC(1_S)(\lambda) = \left\{\begin{array}{ll}
    1&\text{if $\lambda \in S$},\\
    0&\text{otherwise.}
    \end{array}\right.
  \end{equation}
  If $S$ is a single equivalence class
  then $1_S$ is a primitive idempotent,
  and
   $Z(Par_t) = \prod_{S \in \P / \approx_t} 1_S Z(Par_t).$  
\end{lemma}

\begin{proof}
  It is clear from \cref{louis3} that any function in the image of $\HC$
  is constant on $\approx_t$-equivalence classes.
  Conversely, take a function $f \in \kk[\P]$ which is constant on
  equivalence classes.
  For an equivalence class $S \in \P / \approx_t$, let
  $L_S$ be the irreducible $Z(Par_t)$-module associated to
 the central character $\chi_\lambda\:(\lambda \in S)$.
 The previous lemma shows that these give a full set of pairwise inequivalent
  irreducible finite-dimensional
$Z(Par_t)$-modules.
  It follows that the cocommutative coalgebra
  $C(Par_t)$ decomposes as a direct sum of indecomposable coideals
  $$
  C(Par_t) = \bigoplus_{S \in \P / \approx_t} C_S,
  $$
  where $C_S$ is the injective hull
  of $L_S$.  
  Then we consider the linear map
  $\theta:C(Par_t) \rightarrow C(Par_t)$
defined by multiplication by the scalar
$f(\lambda)\:(\lambda \in S)$ on the summand $C_S$.
This is a comodule homomorphism.
Now we use that
$$
  Z(Par_t) = C(\Par_t)^* \cong \End_{C(Par_t)}(C(Par_t))^{\op}
$$
as holds for any coalgebra, e.g., see \cite[Lem.~2.2]{BS}.
It implies that $\theta$ defines an element of $Z(Par_t)$. The image of this element under $\HC$ is the function $f \in \kk[\P]$.

To prove the existence of the idempotent $1_S$ for any $S$ that is a union of $\approx_t$-equivalence classes, we apply the construction in the previous paragraph
to obtain $1_S \in Z(Par_t)$ such that $1_S$ acts as the identity on the indecomposable summands
$C_{S'}$ of $C(Par_t)$ for all $\approx_t$-equivalence classes $S' \subseteq S$
and as zero on all other summands.
This is an idempotent satisfying \cref{bourbon}, and it is a primitive
idempotent
if and only if $S$ is a single equivalence class.
We then have that $$
Z(Par_t) = \prod_{S \in \P / \approx_t} 1_S Z(Par_t)
$$
as this is the algebra decomposition
that is dual to the decomposition of $C(Par_t)$ as the direct sum of its indecomposable coideals.
\end{proof}

For $S \in \P / \approx_t$,
the primitive central idempotent $1_S \in Z(Par_t)$ from \cref{centlem}
is not an element of $Par_t$, but we have that
$1_S = (1_{S,n})_{n \geq 0}$ for idempotents
$1_{S,n} = 1_S 1_n = 1_n 1_S \in 1_n Par_t 1_n$.
Moreover, for a fixed $n$ the idempotent
$1_{S,n}$ is zero for all but finitely many $S$,
so that $1_n = \sum_{S \in \P / \approx_t} 1_{S,n}$.
  The locally unital algebras
  $1_{S} Par_t = \bigoplus_{m,n \geq 0} 1_{S,m} Par_t 1_{S,n}$
  are the {\em blocks} of the partition algebra $Par_t$,
  and we have the block decompositions
   \begin{align}\label{block1}
  Par_t &= \bigoplus_{S \in \P / \approx_t} 1_{S} Par_t,&
  Par_t\Mod &= \prod_{S \in \P / \approx_t} 1_{S} Par_t\Mod.
  \end{align}
  Representatives for the isomorphism classes of irreducible $1_{S} Par_t$-modules are given by the modules
  $L(\lambda)$ for all $\lambda \in S$.
  
\begin{lemma}\label{semisimplicitycriterion}
  The following properties are equivalent:
  \begin{itemize}
  \item[(i)] $Par_t$ is semisimple.
  \item[(ii)] All of the $\approx_t$-equivalence classes are singletons.
  \item[(iii)] $\HC:Z(Par_t) \rightarrow \kk[\P]$ is surjective.
  \item[(iv)] $\HC:Z(Par_t) \rightarrow \kk[\P]$ is an isomorphism.
  \end{itemize}
  \end{lemma}

\begin{proof}
 If (i) holds, then $Par_t$ is a direct sum of locally unital matrix algebras indexed by
  the set $\P$ that labels its irreducible representations. Hence, its center
  is the direct product $\prod_{\lambda \in \P} \kk_\lambda$. It follows easily that $\HC$ is an isomorphism, i.e., (iv) holds.

  Obviously, (iv) implies (iii).

  The equivalence of (ii) and (iii)
  follows from \cref{centlem}.

  It remains to show that (ii) implies (i). Assuming (ii),
 \cref{centlem} shows for any $\lambda \in \P$ that
 there is a primitive central idempotent in $Z(Par_t)$ which acts as the identity on $\Delta(\lambda)$ and as zero on $L(\mu)$ for all $\mu \neq \lambda$.
 We deduce that all composition factors of $\Delta(\lambda)$ are isomorphic to $L(\lambda)$.
  Since this is a highest weight module we have that $[\Delta(\lambda) :L(\lambda)] = 1$, so actually $\Delta(\lambda)$ is irreducible.
 This is the case for all $\lambda \in \P$, so by BGG reciprocity we deduce that $P(\lambda) = \Delta(\lambda) = L(\lambda)$ for all $\lambda$, and (i) holds.
\end{proof}

\begin{remark}
  When $Par_t$ is semisimple,
  the standardization functor $j_!:\sym\Modfd \rightarrow \Par_t\Modlfd$
  sends the irreducible $\sym$-modules $S(\lambda)$
  to the irreducible $Par_t$-modules $\Delta(\lambda) =L(\lambda)$ for all
  $\lambda \in \P$.
  It follows easily that $j_!$ is an equivalence of categories in the semisimple case (although it is not a monoidal equivalence). Since the center is a Morita invariant, it follows that $Z(\sym)\cong Z(Par_t)$ in the semisimple case.
  Recalling that $Z(\sym) \cong
  \kk[\P]$, this gives another way to understand the equivalence
  (i)$\Rightarrow$(iv) of \cref{semisimplicitycriterion}.
  \end{remark}

\begin{remark}\label{minorder}
  As $Par_t\Modlfd$ is an upper finite highest weight category,
  there is also a canonical partial order on $\P$, called the {\em minimal order} in \cite[Rem.~3.68]{BS},
  which we denote here by $\succeq_t$. By definition, this
  is the partial order generated by the relation $\lambda \succeq_t \mu$ if $[\Delta(\lambda):L(\mu)] \neq 0$. As always for highest weight categories,
  the equivalence relation $\approx_t$ defining the blocks of $Par_t$ is the
  transitive closure of the minimal order $\succeq_t$.
  We will describe $\succeq_t$ explicitly in \cref{minordercor} below.
\end{remark}
  
  \subsection{``Blocks''}
  In the previous subsection, we introduced an equivalence relation $\approx_t$
  on $\P$ whose equivalence classes parametrize the blocks of $Par_t$. The relation $\approx_t$ was defined
  in terms of the central characters $\chi_\lambda:Z(Par_t) \rightarrow \kk$
  arising from the irreducible $Par_t$-modules $L(\lambda)$; see \cref{approxdef}.
  On the other hand, in \cref{koepka,nadal},
  we constructed some explicit central
  elements of $Par_t$. Let $\sim_t$ be the equivalence relation on $\P$
  defined from
\begin{equation}\label{mysim}
    \lambda \sim_t \mu\Leftrightarrow\chi_\lambda|_{Z_0(Par_t)}
    = \chi_\mu|_{Z_0(Par_t)}
\end{equation}
where $Z_0(Par_t)$ is the subalgebra of $Z(Par_t)$ generated
by the elements $\left\{c^{(r)}\:\big|\:r \geq 3\right\}$ (equivalently,
by the elements $\left\{z^{(r)}\:\big|\:r \geq 1\right\}$).
  We refer to the $\sim_t$-equivalence classes as ``blocks''.
  We obviously have that \begin{equation}\label{approxsim}
    \lambda \approx_t \mu \Rightarrow \lambda \sim_t \mu,
    \end{equation}i.e., ``blocks'' are unions of blocks.
   Defining $1_S$ as in \cref{centlem}, there are induced ``block'' decompositions
  \begin{align}\label{BLOCKS}
    Par_t &= \bigoplus_{S \in \P / \sim_t} 1_S Par_t,
    &
    Par_t\Mod &= \prod_{S \in \P / \sim_t} 1_S Par_t\Mod.
  \end{align}
In this subsection, we are going to describe the relation $\sim_t$ in explicit combinatorial terms.

\begin{lemma}\label{c}
  The images of the elements $x_j^L, x_j^R, s_k^L, s_k^R \in 1_n Par_t 1_n$ from
  \cref{walk1} under the Harish-Chandra homomorphism $\widehat{\operatorname{HC}}$
  from \cref{feed} are
  \begin{align}\label{interest1}
    \HC_n(x_j^L)
    &= x_j,
    &\HC_n(x_j^R) &= t-j+1,\\
    \HC_n(s_k^L)
    &= 1,
       &\HC_n(s_k^R) &= (k\:k\!+\!1),\label{interest2}
  \end{align}
  where $x_j \in \kk S_n$ is the Jucys-Murphy element from \cref{jmdef}.
\end{lemma}

\begin{proof}
  Applying $\HC_n$ to the relations \cref{august,july} (on the
  $k$th, $(k+1)$th and $(k+2)$th strings) we deduce that
$\HC_n(s_{k+1}^L) = (k\:\,k\!+\!1\:\,k\!+\!2) \HC_n(s_k^L)(k\!+\!2\:\,k\!+\!1\:\,k)$ and $\HC_n(s_{k+1}^R) = (k\:\,k\!+\!1\:\,k\!+\!2) \HC_n(s_{k}^R) (k\!+\!2\:\,k\!+\!1\:\,k)$.
  Now \cref{interest2} follows by induction on $k$, the base case $k=1$ being immediate from \cref{bigjoey2}.
  Note for this that
  $(k\:\,k\!+\!1\:\,k\!+\!2) (k\:\,k\!+\!1) (k\!+\!2\:\,k\!+\!1\:\,k)
  =(k\!+\!1\:\,k\!+\!2)$.

Applying $\HC_n$ to the relations \cref{A1,A3} (on the $j$th and $(j+1)$th strings), using also \cref{crazycrossing},
we deduce that $\HC_n(x_{j+1}^L) = (j\:\,j\!+\!1)\HC_n(x_j^L)(j\:\,j\!+\!1) + \HC_n(s_j^R)$
and
$\HC_n(x_{j+1}^R) = (j\:\,j\!+\!1)\HC_n(x_j^R)(j\:\,j\!+\!1) - \HC_n(s_j^L)$.
Now \cref{interest1} follows using \cref{interest2} and induction on $j$,
the base case $j=1$ being immediate from \cref{bigjoey1}.
Note for this that $(j\:\,j\!+\!1) x_j (j\:\,j\!+\!1) + (j\:\,j\!+\!1)= x_{j+1}$.
\end{proof}

\begin{lemma}\label{d}
For $\lambda \in \P_n$, we have that
  \begin{align}\label{310}
    \chi_\lambda(c(u)) &=
\prod_{i=1}^n \frac{\alpha_{\cont_i(\TT)}(u)}{\alpha_{t-i+1}(u)}
  \end{align}
  where $\TT$ is some fixed standard $\lambda$-tableau
  and $\alpha_x(u)$ is as in \cref{alphadef}.
\end{lemma}

\begin{proof}
  Note
 by \cref{louis3} that $\chi_\lambda(c(u)) = \HC_n(c_n(u))(\lambda)
  \in \kk[\![u^{-1}]\!]$. To compute this, we use \cref{c}
  and the explicit formula for $c_n(u) = p_t(C_n(u))$
 arising from \cref{bubbleslide2} to deduce that
  $$
  \HC_n(c_n(u)) =
  \prod_{i=1}^n
\frac{\alpha_{x_i}(u)}{\alpha_{t-i+1}(u)}\in Z(\kk S_n)[\![u^{-1}]\!].
  $$
  To evaluate this at $\lambda$, we act on the basis vector $v_\TT$
 from Young's orthonormal basis for $S(\lambda)$, remembering that
 $x_i v_\TT = \cont_i(\TT) v_\TT$.
  \end{proof}

\cref{c} suggests
some combinatorics of weights.
Let $P$ be the free Abelian group on basis $\{\Lambda_c\:|\:c \in \kk\}$.
Let
$\eps_c := \Lambda_c - \Lambda_{c+1}$ and
$\alpha_c := \eps_c - \eps_{c-1}$.
We define the {\em weight} of a rational function
$f(u) \in \kk(u)$ to be
\begin{equation}
  \wt f(u) :=
\sum_{c \in \kk} \left[\left(\begin{array}{l}
    \text{Multiplicity of $c$}\\\text{as a pole of $f(u)$}
  \end{array}
  \right)
  - \left(\begin{array}{l}
    \text{Multiplicity of $c$}\\\text{as a zero of $f(u)$}
  \end{array}
  \right)\right]
\Lambda_c \in P.
\end{equation}
For example,
$\wt \alpha_c(u) = -\Lambda_{c-1}+2\Lambda_c - \Lambda_{c+1} = \alpha_c$.
For $\lambda \in \P_n$, let $\wt_t(\lambda)$ be the weight of the rational function appearing on the right hand side of \cref{310}.
As the coefficients of the power series $c(u)$ generate the subalgebra
$Z_0(Par_t)$, the equivalence relation $\sim_t$ defined by \cref{mysim}
satisfies
\begin{equation}\label{sofar}
\lambda \sim_t \mu\Leftrightarrow \wt_t(\lambda) = \wt_t(\mu).
\end{equation}
This suggests using elements of $P$ rather than $\sim_t$-equivalence classes to index the ``blocks'' from \cref{BLOCKS}:
for any $\gamma \in P$, let
\begin{equation}\label{Sgamma}
  S(\gamma) := \left\{\lambda \in \P\:\big|\:\wt_t(\lambda)=\gamma\right\}.
\end{equation}
Then define
\begin{equation}\label{PRgamma}
\PR_\gamma:Par_t\Mod \rightarrow Par_t\Mod
\end{equation}
to be the projection functor defined by multiplication by the
central idempotent $1_{S(\gamma)}$ from \cref{centlem}.
In other words, $\PR_\gamma$ projects a $Par_t$-module $V$ to its largest submodule
all of whose irreducible subquotients are of the form $L(\lambda)$ for $\lambda \in \P$ with $\wt_t(\lambda) = \gamma$.
The admissible $\gamma \in P$
which parametrize ``blocks'' are the ones with $S(\gamma)\neq\varnothing$;
if $S(\gamma) = \varnothing$ then $\PR_\gamma$ is the zero functor.

\begin{lemma}\label{altchar}
  For $\lambda \in \P_n$ and any standard $\lambda$-tableau $T$, we have that
\begin{align}\label{wtt}
  \wt_t(\lambda)
  = 
  \sum_{i=1}^{n} (\alpha_{\cont_i(\TT)}-\alpha_{t-i+1})= 
  (\eps_{t-|\lambda|}-\eps_{t})+
(\eps_{\lambda_1-1}-\eps_{-1}) + \cdots+(\eps_{\lambda_k-k}-\eps_{-k})
\end{align}
for any $k \geq \ell(\lambda)$.
Moreover, given another partition $\mu \in \P$,
we have that $\wt_t(\lambda) = \wt_t(\mu)$
if and only if
the infinite sequences $(t-|\lambda|,\lambda_1-1,\lambda_2-2,\dots)$
and $(t-|\mu|,\mu_1-1,\mu_2-2,\dots)$ are rearrangements of each other.
\end{lemma}

\begin{proof}
  The first equality in \cref{wtt} follows immediately from \cref{d}.
  To deduce the second equality,
  take $k \geq \ell(\lambda)$. For $1 \leq r \leq k$
  the contents of the nodes in the $r$th row of the Young diagram of $\lambda$ are
$1-r,2-r,\dots,\lambda_r-r$, and we have that
  $\alpha_{1-r}+\cdots + \alpha_{\lambda_r-r} = \eps_{\lambda_r-r} - \eps_{-r}$.
  Also $\alpha_{t} + \alpha_{t-1} + \cdots + \alpha_{t-n+1} = \eps_{t}-\eps_{t-n}$.
Now the desired formula follows easily.

Rearranging the right hand side of \cref{wtt} gives that
$\eps_{t-|\lambda|} + \eps_{\lambda_1-1}+\eps_{\lambda_2-2}+\cdots+\eps_{\lambda_k-k}=
\wt_t(\lambda)+\eps_{-1}+\eps_{-2}+\cdots+\eps_{-k} + \eps_t
$
for any $k \geq \ell(\lambda)$.
Hence, we have that $\wt_t(\lambda) = \wt_t(\mu)$ if and only if
$$
\eps_{t-|\lambda|}+\eps_{\lambda_1-1}+\eps_{\lambda_2-2} + \cdots
+ \eps_{\lambda_k - k}=
\eps_{t-|\mu|} + \eps_{\mu_1-1}+\eps_{\mu_2-2} + \cdots +
\eps_{\mu_k-k}
$$
for all $k \gg 0$.
This is clearly equivalent to saying that
the infinite sequences
$(t-|\lambda|,\lambda_1-1,\lambda_2-2,\dots)$
and $(t-|\mu|,\mu_1-1,\mu_2-2,\dots)$ may be obtained from each other by
permuting the entries.
\end{proof}

The final assertion from \cref{altchar} shows that $\sim_t$
is exactly the same as the equivalence relation on partitions defined
in \cite[Def.~5.1]{CO}. The equivalence classes of this relation were
investigated in detail in \cite[$\S$5.3]{CO}.
The following summarizes the results obtained there. For the statement, we say that $\lambda \in \P$ is {\em typical} if it is the only partition in its $\sim_t$-equivalence class; otherwise we say that $\lambda$ is {\em atypical}. Of course, these notions depend on the fixed value of the parameter $t$.

\begin{theorem}[Comes-Ostrik]\label{cocomb}
If $t \notin \N$ then all partitions are typical.
If $t \in \N$ then
there is a bijection
    $\P_t \stackrel{\sim}{\rightarrow} \{\text{atypical $\sim_t$-equivalence classes}\}$
    taking $\kappa \in \P_t$ to the
    $\sim_t$-equivalence class $\{\kappa^{(0)}, \kappa^{(1)}, \kappa^{(2)},\dots\}$
    where
    \begin{equation}\label{kappadef}
    \kappa^{(n)} := (\kappa_1+1,\dots,\kappa_n+1,\kappa_{n+2}, \kappa_{n+3},\dots) \in \P_{t+n-\kappa_{n+1}},
    \end{equation}
   i.e., it is the partition obtained from $\kappa$ by
   adding a node to the first $n$ rows of its Young diagram
   then removing its $(n+1)$th row.
Moreover, still assuming $t \in \N$, a partition $\lambda \in \P$ is typical if and only if
$t-|\lambda| = \lambda_i-i$ for some $i \geq 1$.
\end{theorem}

\begin{example}\label{generalblock}
  For any $t \in \N$,
  the $\sim_t$-equivalence class associated to $\kappa = (t) \in \P_t$ is
  $$
  S = \left\{\varnothing, (t\!+\!1), (t\!+\!1, 1), (t\!+\!1,1^2), \cdots\right\}.
  $$
  For $t \in \N - \{0,1\}$, the $\sim_t$-equivalence class associated to
  $\kappa = (1^t) \in \P_t$ is
  $$
  S = \left\{(1^{t-1}), (2, 1^{t-2}), (2^2, 1^{t-3}), \cdots, (2^{t-1}),(2^t), (2^t, 1), (2^t, 1^2), \cdots\right\}.
  $$
\end{example}

As noted in \cite[Cor.~5.23]{CO}
(using a different argument for the forward implication),
the first assertion of \cref{cocomb}
allows us to recover the following well known result of Deligne \cite[Th.~2.18]{Del}: $\REP(S_t)$ is semisimple if and only if $t \notin \N$. In terms of the path
algebra $Par_t$, Deligne's result can be stated as follows.

\begin{corollary}[Deligne]\label{sscor}
  $Par_t$ is semisimple if and only if $t \notin \N$.
\end{corollary}

\begin{proof}
  We already know that $Par_t$ is not semisimple if $t \in \N$ by \cref{SWDcor}.
  Conversely, if $t \notin \N$, we apply the criterion from \cref{semisimplicitycriterion}, noting that
  all $\approx_t$-equivalence classes are singletons thanks to \cref{approxsim}
  and the first part of \cref{cocomb}.
  \end{proof}

\begin{remark}
  When $t \notin \N$, the above arguments show for $\lambda, \mu \in \P$
  with $\lambda \neq \mu$ that there is a central element in the
  subalgebra $Z_0(Par_t)$ of $Z(Par_t)$ which acts by different scalars
  on the irreducible modules $L(\lambda)$ and $L(\mu)$.
  It follows in these cases that $Z_0(Par_t)$ is a dense subalgebra of
the pseudo-compact topological algebra  $Z(Par_t)$\footnote{
  These algebras are certainly not equal since $Z(Par_t) \cong \prod_{\lambda \in \P} \kk_\lambda$ is of uncountable dimension.}.
We do not expect that this is the case when $t \in \N$, but nevertheless
 $Z_0(Par_t)$ is still sufficiently large to separate blocks.
  This will be established in \cref{blocksareblocks} below, which shows
  for any value of $t$
that  the relations $\sim_t$ and $\approx_t$ coincide, so that ``blocks'' are blocks, and \cref{BLOCKS} is always the same decomposition as \cref{block1}; see also
\cite[Th.~5.3]{CO}.
   \end{remark}

\subsection{Special projective functors}
From now on, we will primarily be
interested in parameter values $t \in \N$, so that
$Par_t$ is not semisimple.
Consider the atypical block
$\{\kappa^{(0)}, \kappa^{(1)}, \kappa^{(2)},\dots\}$ associated to
$\kappa \in \P_t$.
From \cref{kappadef}, it follows that
$\kappa^{(n)}$ is obtained from $\kappa^{(n-1)}$ by adding
$\kappa_n - \kappa_{n+1}+1$ nodes to the $n$th row of its Young diagram, leaving all other rows unchanged.
The partition $\kappa^{(0)}$ is the smallest
of all of the $\kappa^{(n)}$, hence, it is maximal in the highest weight ordering from \cref{firstthm}.
It follows that
\begin{align}\label{topinblock}
  P(\kappa^{(0)}) &\cong \Delta(\kappa^{(0)}).
\end{align}
The indecomposable projectives
$\Delta(\kappa^{(0)})$ are exactly the ones of non-zero
categorical dimension mentioned already in \cref{corerem}, with
the irreducible $\kk S_t$-module associated to the image of $\Delta(\kappa^{(0)})$
under the equivalence $\overline{\psi}_t$ between the semisimplification of $\Kar(\Par_t)$ and
$\kk S_t\Modfd$
being the Specht module $S(\kappa)$.
It is also useful to note for $t \in \N$ and $\kappa \in \P_t$ that
the associated block $\{\kappa^{(0)}, \kappa^{(1)},\dots\}$ is
the set $S(\gamma)$ from \cref{Sgamma} for
\begin{equation}\label{kappatogamma}
  \gamma := (\eps_{\kappa_1} - \eps_t) + (\eps_{\kappa_2-1}-\eps_{-1})
  +\cdots + (\eps_{\kappa_t-t+1} - \eps_{-t}) \in P.
  \end{equation}
This is follows easily using \cref{wtt,kappadef}.

In order to understand the structure of the atypical blocks more fully, 
 we are going to use the endofunctor
$\mid\,\star:\Par_t\rightarrow \Par_t$.
Let
\begin{align}
  D &:= \res_{\mid\,\star}= 1_{\mid\,\star}Par_t  \otimes_{Par_t}:Par_t\Mod\rightarrow Par_t\Mod
\end{align}
be the corresponding restriction functor from
\cref{if2}. This obviously preserves locally finite-dimensional modules.
The object $\mid$ is self-dual so, by \cref{dualitylem},
the restriction functor $D$ is isomorphic to the induction functor
$\ind_{\mid\,\star}$.
By \cref{dualitycor}, $D$ is a
self-adjoint projective functor,
so it preserves finitely generated projectives (and finitely cogenerated injectives).
To make the canonical adjunction as explicit as possible, we note that
its unit and counit
are induced by the bimodule homomorphisms
\begin{align}\label{etaadj}
\eta&:Par_t\to 1_{\mid\,\star} Par_t\otimes_{Par_t} 1_{\mid\,\star} Par_t,
&\begin{tikzpicture}[scale=1.2,anchorbase]
\draw[-,thick] (-0.5,-0.2)--(-0.5,0.2)--(0.5,0.2)--(0.5,-0.2)--(-0.5,-0.2);
\node at (-0,-0) {$\scriptstyle{f}$};
\draw[-,thick] (-0.4,-0.2)--(-0.4,-0.5);
\draw[-,thick] (0.4,-0.2)--(0.4,-0.5);
\draw[-,thick] (-0.4,0.2)--(-0.4,0.5);
\draw[-,thick] (0.4,0.2)--(0.4,0.5);
\node at (-0.2,0.35) {$\cdot$};
\node at (-0,0.35) {$\cdot$};
\node at (0.2,0.35) {$\cdot$};
\node at (-0.2,-0.4) {$\cdot$};
\node at (-0,-0.4) {$\cdot$};
\node at (0.2,-0.4) {$\cdot$};
\node at (.4,-.6) {$\stringlabel{1}$};
\node at (-.4,-.6) {$\stringlabel{n}$};
\node at (.4,.6) {$\stringlabel{1}$};
\node at (-.4,.6) {$\stringlabel{m}$};
\end{tikzpicture}\mapsto\begin{tikzpicture}[scale=1.2,anchorbase]
\draw[-,thick] (-0.7,-0.5)to(-0.7,.5);
\draw[-,thick] (-0.5,-0.2)--(-0.5,0.2)--(0.5,0.2)--(0.5,-0.2)--(-0.5,-0.2);
\node at (-0,-0) {$\scriptstyle{f}$};
\draw[-,thick] (-0.4,-0.2)--(-0.4,-0.5);
\draw[-,thick] (0.4,-0.2)--(0.4,-0.5);
\draw[-,thick] (-0.4,0.2)--(-0.4,0.5);
\draw[-,thick] (0.4,0.2)--(0.4,0.5);
\node at (-0.2,0.35) {$\cdot$};
\node at (-0,0.35) {$\cdot$};
\node at (0.2,0.35) {$\cdot$};
\node at (-0.2,-0.4) {$\cdot$};
\node at (-0,-0.4) {$\cdot$};
\node at (0.2,-0.4) {$\cdot$};
\node at (.4,-.6) {$\stringlabel{1}$};
\node at (-.4,-.6) {$\stringlabel{n}$};
\node at (-.7,-.6) {$\stringlabel{n\!+\!1}$};
\node at (.4,.6) {$\stringlabel{1}$};
\node at (-.4,.6) {$\stringlabel{m}$};
\node at (-.7,.6) {$\stringlabel{m\!+\!1}$};
\end{tikzpicture}\ \otimes\ \begin{tikzpicture}[anchorbase,scale=1.2]
\draw[-,thick] (-1.1,0.5)to(-1.1,0)to[out=down,in=left](-0.9,-0.2)to[out=right,in=down](-0.7,0)to(-0.7,0.5);
\draw[-,thick] (-0.4,0.5)--(-0.4,-0.5);
\draw[-,thick] (0.4,0.5)--(0.4,-0.5);
\node at (-0.2,0.05) {$\cdot$};
\node at (-0,0.05) {$\cdot$};
\node at (0.2,0.05) {$\cdot$};
\node at (.4,.6) {$\stringlabel{1}$};
\node at (-.4,.6) {$\stringlabel{n}$};
\node at (.4,-.6) {$\stringlabel{1}$};
\node at (-.4,-.6) {$\stringlabel{n}$};
\node at (-.72,.6) {$\stringlabel{n\!+\!1}$};
\node at (-1.12,.6) {$\stringlabel{n\!+\!2}$};
\end{tikzpicture}\ ,\\\label{epsadj}
  \varepsilon&:1_{\mid\,\star}Par_t \otimes_{Par_t} 1_{\mid\,\star} Par_t\rightarrow Par_t, &
  \begin{tikzpicture}[anchorbase,scale=1.2]
\draw[-,thick] (-0.5,-0.2)--(-0.5,0.2)--(0.5,0.2)--(0.5,-0.2)--(-0.5,-0.2);
\node at (-0,-0) {$\scriptstyle{f}$};
\draw[-,thick] (-0.4,-0.2)--(-0.4,-0.5);
\draw[-,thick] (0.4,-0.2)--(0.4,-0.5);
\draw[-,thick] (-0.45,0.2)--(-0.45,0.5);
\draw[-,thick] (-0.25,0.2)--(-0.25,0.5);
\draw[-,thick] (0.45,0.2)--(0.45,0.5);
\node at (-0.05,0.4) {$\cdot$};
\node at (.1,0.4) {$\cdot$};
\node at (0.25,0.4) {$\cdot$};
\node at (-0.2,-0.4) {$\cdot$};
\node at (-0,-0.4) {$\cdot$};
\node at (0.2,-0.4) {$\cdot$}; 
\node at (.4,-.6) {$\stringlabel{1}$};
\node at (-.4,-.6) {$\stringlabel{n}$};
\node at (.45,.6) {$\stringlabel{1}$};
\node at (-.25,.6) {$\stringlabel{m}$};
\node at (-.5,.6) {$\stringlabel{m\!+\!1}$};
\end{tikzpicture}\ \otimes\ \begin{tikzpicture}[anchorbase,scale=1.2]
\draw[-,thick] (-0.5,-0.2)--(-0.5,0.2)--(0.5,0.2)--(0.5,-0.2)--(-0.5,-0.2);
\node at (-0,-0) {$\scriptstyle{g}$};
\draw[-,thick] (-0.4,-0.2)--(-0.4,-0.5);
\draw[-,thick] (0.4,-0.2)--(0.4,-0.5);
\draw[-,thick] (-0.4,0.2)--(-0.4,0.5);
\draw[-,thick] (-0.2,0.2)--(-0.2,0.5);
\draw[-,thick] (0.4,0.2)--(0.4,0.5);
\node at (-0,0.35) {$\cdot$};
\node at (0.1,0.35) {$\cdot$};
\node at (0.2,0.35) {$\cdot$};
\node at (-0.2,-0.4) {$\cdot$};
\node at (-0,-0.4) {$\cdot$};
\node at (0.2,-0.4) {$\cdot$}; 
\node at (.4,.6) {$\stringlabel{1}$};
\node at (-.2,.6) {$\stringlabel{n}$};
\node at (.4,-.6) {$\stringlabel{1}$};
\node at (-.4,-.6) {$\stringlabel{n\!+\!1}$};
\node at (-.45,.6) {$\stringlabel{n\!+\!1}$};
  \end{tikzpicture}
\mapsto\begin{tikzpicture}[anchorbase,scale=1.2]
\draw[-,thick] (-0.5,-0.5)--(-0.5,-0.1)--(0.5,-0.1)--(0.5,-0.5)--(-0.5,-0.5);
\draw[-,thick] (-0.5,0.5)--(-0.5,0.1)--(0.5,0.1)--(0.5,0.5)--(-0.5,0.5);
\draw[-,thick] (-0.4,-0.5)--(-0.4,-0.7);
\draw[-,thick] (0.4,-0.7)--(0.4,-0.5);
\draw[-,thick] (-0.4,0.1)--(-0.2,-0.1);
\draw[-,thick] (-0.4,-.1) to[out=120,in=down] (-0.7,.3) to (-.7,.5) to[out=up,in=up,looseness=2] (-.45,.5);
\draw[-,thick] (0.4,0.1)--(0.4,-0.1);
\draw[-,thick] (0.45,0.7)--(0.45,0.5);
\draw[-,thick] (-0.25,0.7)--(-0.25,0.5);
\node at (0,0.3) {$\scriptstyle{f}$};
\node at (0,-0.3) {$\scriptstyle{g}$};
\node at (-0.1,-0.03) {$\cdot$};
\node at (0.05,-0.03) {$\cdot$};
\node at (0.2,-0.03) {$\cdot$};
\node at (-0.2,-0.62) {$\cdot$};
\node at (0,-0.62) {$\cdot$};
\node at (0.2,-0.62) {$\cdot$}; 
\node at (-0.05,0.62) {$\cdot$};
\node at (.1,0.62) {$\cdot$};
\node at (0.25,0.62) {$\cdot$};
\node at (.4,-.8) {$\stringlabel{1}$};
\node at (-.4,-.8) {$\stringlabel{n\!+\!1}$};
\node at (.4,.8) {$\stringlabel{1}$};
\node at (-.25,.8) {$\stringlabel{m}$};
\end{tikzpicture}\ .
\end{align}
Using \cref{predual}, it follows that
$D$ commutes with the duality
$?^\sigmadual$ on $Par_t\Modlfd$.

\begin{lemma}\label{hood}
  For $\lambda \in \P$, there is
  a filtration
  $0=V_0 \subseteq V_1 \subseteq V_2 \subseteq V_3 = D \Delta(\lambda)$
  such that
  \begin{align*}
    V_3/V_2 &\cong \bigoplus_{a \in \add(\lambda)} \Delta\left(\lambda+\boxed{a}\right),\\
    V_2 / V_1 &\cong \Delta(\lambda)\oplus
    \bigoplus_{b \in \rem(\lambda)}
    \bigoplus_{a \in \add(\lambda - \boxed{b})}
    \Delta\left((\lambda - \boxed{b})+\boxed{a}\right),\\
    V_1 / V_0 &\cong \bigoplus_{b \in \rem(\lambda)} \Delta\left(\lambda-\boxed{b}\right).
  \end{align*}
\end{lemma}

\begin{proof}
By \cref{dillylilly}, $D$ is isomorphic to the functor
$Q(\boxempty) \ostar?$ defined by taking the induction product with
the projective module $Q(\boxempty)$.
By \cref{textmessage}(iii) and \cref{bioap},
$Q(\boxempty)$ has a $\Delta$-flag of length two with sections
$\Delta(\boxempty)$ at the top and $\Delta(\varnothing)$ at the bottom.
Applying \cref{lastfromss}, we deduce that
$D \Delta(\lambda) \cong Q(\boxempty) \ostar \Delta(\lambda)$
has a $\Delta$-flag with one section
$\Delta(\varnothing) \ostar \Delta(\lambda) \cong \Delta(\lambda)$
and other sections arising from the $\Delta$-flag
of $\Delta(\boxempty)\ostar \Delta(\lambda)$ described in that theorem.
For this, one just needs to know the values of the
reduced Kronecker coefficients $\overline{G}_{\boxvoid,\lambda}^{\mu}$ which were worked out in \cref{springs}.
The $\Delta$-flag can be ordered in the way described
since $\Ext^1(\Delta(\mu),\Delta(\nu)) = 0$ if $|\mu| \leq |\nu|$.
  \end{proof}

\begin{remark}
  \cref{hood} also follows from \cref{filtration} below, which
  constructs the filtration explicitly. The proof of \cref{filtration} is also valid over fields of positive characteristic.
    \end{remark}

Now we are going to use the affine partition category $\APar$
to decompose the endofunctor $D$ as a direct sum of {\em special projective functors}
$D_{b|a}$. The approach here is analogous to the way the affine symmetric category
$\ASym$ was used to decompose
$E$ and $F$ as direct sums of
$E_a$ and $F_b$ in \cref{endofterm}.
As noted at the end of \cref{apc},
$\Par_t$ is isomorphic to the quotient of $\APar$ by a left tensor ideal.
Hence, $\Par_t$ is a strict $\APar$-module category.
The self-adjoint functor $D$ is also the restriction functor $\res_{\mid\,\star}$
arising from this categorical action of $\APar$ on $\Par_t$.
Now the left and right dots give us natural transformations
\begin{align*}
  \alpha &:= \begin{tikzpicture}[anchorbase,scale=1.2]
\draw[-,thick](0,0)--(0,0.4);
\draw[-,thick](0,0.2)--(0.15,0.2);
\node  at (0.15,0.2) {$\bullet$};
\node at (.4,.2) {$\star$};
  \end{tikzpicture}:\mid\,\star \Rightarrow \mid\,\star,&
  \beta &:= \begin{tikzpicture}[anchorbase,scale=1.2]
\draw[-,thick](0,0)--(0,0.4);
\draw[-,thick](0,0.2)--(-0.15,0.2);
\node  at (-0.15,0.2) {$\bullet$};
\node at (.2,.2) {$\star$};
  \end{tikzpicture}:\mid\,\star\Rightarrow\mid\,\star.
  \end{align*}
Applying the general construction from \cref{parmesan} to these, 
we obtain commuting endomorphisms
\begin{align}
x &:= \res_\alpha:D \Rightarrow D,
&
y &:= \res_\beta:D \Rightarrow D.
\end{align}
Let $D_{b|a}$ be the summand of $D$ that is the
simultaneous generalized eigenspace of $x$ and $y$ of eigenvalues $a$ and $b$, respectively.
Explicitly, $D = \res_{\mid\,\star}$ is defined by tensoring
with the bimodule
$1_{\mid\,\star} Par_t$, and
the endomorphisms $x$ and $y$ of $D$
are induced by the bimodule endomorphisms $\rho$ and $\lambda$ of $1_{\mid\,\star} Par_t$
given by left multiplication by $x^R_{m+1}$ and $x^L_{m+1}$, respectively,
on the summand $1_{m+1}  Par_t$ of $1_{\mid\,\star} Par_t$ for each $m \geq 0$.
Then, $D_{b|a}$ is the functor defined by tensoring with the summand of
$1_{\mid\,\star} Par_t$ that is the
simultaneous generalized eigenspaces
of $\rho$ and $\lambda$ for
the eigenvalues $a$ and $b$, respectively.
As $1_{m+1} Par_t = \bigoplus_{n \geq 0} 1_{m+1} Par_t 1_n$
with each $1_{m+1} Par_t 1_n$ being finite-dimensional,
these endomorphisms are locally finite, so we have that
\begin{equation}\label{summertime}
  D = \bigoplus_{a,b \in \kk} D_{b|a}.
\end{equation}

\begin{lemma}\label{selfadj}
  For $a,b \in \kk$, the endofunctor $D_{b|a}$ commutes with the duality $?^\sigmadual$, i.e., $D_{b|a} \circ ?^\sigmadual \cong ?^\sigmadual \circ D_{b|a}$.
\end{lemma}

\begin{proof}
  This follows
  from the fact that $D$ commutes with the duality $?^\sigmadual$, and $\sigma$ fixes both the left dot and the right dot.
  \end{proof}

\begin{lemma}\label{biadj}
For $a,b \in \kk$, the endofunctors $D_{b|a}$ and $D_{a|b}$ are biadjoint.
\end{lemma}

\begin{proof}
  The adjunction $(D_{a|b}, D_{b|a})$ is induced by the self-adjunction
  of $D$. The unit $\bar \eta$ of adjunction comes from the bimodule homomorphism
  that is the composition of the unit $\eta$ from \cref{etaadj} with the projection onto the generalized $a$ and $b$ eigenspaces of $\rho$ and $\lambda$ on the left tensor factor and the generalized $b$ and $a$ eigenspaces of $\rho$ and $\lambda$ on the right tensor factor.
  The counit $\bar\eps$ of adjunction comes from the composition of the counit $\eps$ from \cref{epsadj} with the inclusion
  of the generalized $b$ and $a$ eigenspaces of $\rho$ and $\lambda$ on the left tensor factor and the generalized $a$ and $b$ eigenspaces of $\rho$ and $\lambda$ on the right tensor factor. To check the zig-zag identities, one just needs to use the relations
\begin{align*}
\begin{tikzpicture}[anchorbase]
\draw[-,thick] (0,0.5) to (0,0.35) to[out=down,in=left] (0.35,0) to[out=right,in=down] (0.7,0.35) to (0.7,0.5);
\draw[-,thick] (0.35,0.35)--(0.7,0.35);
\node at (0.35,0.35) {$\bullet$};
\end{tikzpicture}\ &=\ \begin{tikzpicture}[anchorbase]
\draw[-,thick] (0,0.5) to (0,0.35) to[out=down,in=left] (0.35,0) to[out=right,in=down] (0.7,0.35) to (0.7,0.5);
\draw[-,thick] (0.35,0.35)--(0,0.35);
\node at (0.35,0.35) {$\bullet$};
\end{tikzpicture}\ ,&\begin{tikzpicture}[anchorbase]
\draw[-,thick] (0,0) to (0,0.15) to[out=up,in=left] (0.35,0.5) to[out=right,in=up] (0.7,0.15) to (0.7,0);
\draw[-,thick] (0.35,0.15)--(0.7,0.15);
\node at (0.35,0.15) {$\bullet$};
\end{tikzpicture}\ &=\ \begin{tikzpicture}[anchorbase]
\draw[-,thick] (0,0) to (0,0.15) to[out=up,in=left] (0.35,0.5) to[out=right,in=up] (0.7,0.15) to (0.7,0);
\draw[-,thick] (0.35,0.15)--(0,0.15);
\node at (0.35,0.15) {$\bullet$};
\end{tikzpicture}\ ,&\begin{tikzpicture}[anchorbase]
\draw[-,thick] (0,0.5) to (0,0.35) to[out=down,in=left] (0.35,0) to[out=right,in=down] (0.7,0.35) to (0.7,0.5);
\draw[-,thick] (0.95,0.35)--(0.7,0.35);
\node at (0.95,0.35) {$\bullet$};
\end{tikzpicture}\ &=\ \begin{tikzpicture}[anchorbase]
\draw[-,thick] (0,0.5) to (0,0.35) to[out=down,in=left] (0.35,0) to[out=right,in=down] (0.7,0.35) to (0.7,0.5);
\draw[-,thick] (-0.25,0.35)--(0,0.35);
\node at (-0.25,0.35) {$\bullet$};
\end{tikzpicture}\ ,&\begin{tikzpicture}[anchorbase]
\draw[-,thick] (0,0) to (0,0.15) to[out=up,in=left] (0.35,0.5) to[out=right,in=up] (0.7,0.15) to (0.7,0);
\draw[-,thick] (0.95,0.15)--(0.7,0.15);
\node at (0.95,0.15) {$\bullet$};
\end{tikzpicture}\ &=\ \begin{tikzpicture}[anchorbase]
\draw[-,thick] (0,0) to (0,0.15) to[out=up,in=left] (0.35,0.5) to[out=right,in=up] (0.7,0.15) to (0.7,0);
\draw[-,thick] (-0.25,0.15)--(0,0.15);
\node at (-0.25,0.15) {$\bullet$};
\end{tikzpicture}\ ,
\end{align*}
i.e., the fact that the left and right dots are duals.
   \end{proof}

When $a \neq b$, \cref{biadj} can also be proved
a bit more easily using the description of $D_{b|a}$ given in the following lemma,
since the projection functor $\PR_\gamma$ commutes with $?^\sigmadual$ thanks to \cref{harderdual}.

\begin{lemma}\label{altcenterdef2}
 Let $\PR_\gamma$ be the projection functor defined by \cref{PRgamma}.
  If $a \neq b$ then
   $$
  D_{b|a} \cong \bigoplus_{\gamma\in P} \PR_{\gamma + \alpha_a-\alpha_b}
  \circ D \circ \PR_\gamma.
  $$
  Also $\bigoplus_{\gamma \in P} \PR_\gamma \circ D \circ \PR_\gamma\cong\bigoplus_{a \in \kk} D_{a|a}$.
\end{lemma}

\begin{proof}
  Take a module $V$ in the ``block'' parametrized by $\gamma \in P$,
  so that $\wt_t(\lambda) = \gamma$ for all irreducible subquotients of $V$.
  We need to show that $D_{b|a} V$ is in the ``block'' parametrized
  by $\gamma + \alpha_a - \alpha_b$.
  Since $D_{b|a}$ is exact, we may assume that $V$ is irreducible, so $V = L(\lambda)$ for $\lambda \in \P$ with $\wt_t(\lambda) = \gamma$.
The module $D V = 1_{\mid\,\star} Par_t \otimes_{Par_t} V \cong
 1_{\mid\,\star} V$ is generated by the finite-dimensional vector spaces
 $1_{m+1} V$ for all $m \geq 0$.
 Hence, $D_{b|a} V$ is generated by the simultaneous generalized eigenspaces of
 $x_{m+1}^R$ and $x_{m+1}^L$ on $1_{m+1} V$
of eigenvalues $a$ and $b$, respectively. Consequently, if $L(\mu)$ is an irreducible subquotient of
 $D_{b|a} V$, then $c(u)$ must act on $L(\mu)$ in the same way as
 $\mid\star c_{m}(u)$ acts on a simultaneous eigenvector $v \in 1_{m+1} V$ for
 $x_{m+1}^R$ and $x_{m+1}^L$ of eigenvalues $a$ and $b$.
Also $c_{m+1}(u)$ acts on $v \in V$ as multiplication by $\chi_\lambda(c(u))$,
the rational function displayed on the right hand side of \cref{310}.
Using \cref{bubbleslide1}, we deduce that
 $$
\chi_\mu(c(u))= \frac{\alpha_a(u)}{\alpha_b(u)} \times \chi_\lambda(c(u)).
 $$
Hence, $\wt_t(\mu) = \wt_t(\lambda)+\alpha_a - \alpha_b$.
\end{proof}

Our main combinatorial result about the functors $D_{b|a}$ is as follows.

\begin{theorem}\label{keyresult}
  For $\lambda \in \P$ and $a,b \in \kk$, there is
  a filtration
  $0=V_0 \subseteq V_1 \subseteq V_2 \subseteq V_3 = D_{b|a} \Delta(\lambda)$
  such that
  \begin{align*}
    V_3/V_2 &\cong \left\{\begin{array}{l}\Delta(\lambda+\boxed{a})\\
    0\end{array}\right.
    &&\begin{array}{l}
       \text{if $a \in \add(\lambda)$ and $b=t-|\lambda|$}\\
    \text{otherwise,}\end{array}\\
    V_2 / V_1 &\cong
    \left\{\begin{array}{ll}
    \Delta(\lambda)\oplus \Delta(\lambda)\\
    \Delta(\lambda)\\
     \Delta\left((\lambda - \boxed{b})+\boxed{a}\right)\\
    0\end{array}\right.&&
    \begin{array}{l}
    \text{if $t-|\lambda| = a=b \in \rem(\lambda)$}\\
    \text{if $t-|\lambda| \neq a=b \in \rem(\lambda)$ or $t-|\lambda|=a=b\notin\rem(\lambda)$}\\
    \text{if $a \neq b \in \rem(\lambda)$ and $a \in \add(\lambda-\boxed{b})$}\\
    \text{otherwise},
    \end{array}\\
    V_1/V_0 &\cong \left\{
    \begin{array}{l}\Delta(\lambda-\boxed{b})\\
      0\end{array}\right.
 &&\begin{array}{l}\text{if $a=t-|\lambda|+1$ and $b \in \rem(\lambda)$}\\
      \text{otherwise.}\end{array}
  \end{align*}
In particular, when $t \in \Z$, the functor $D_{b|a}$ is zero unless both
  $a$ and $b$ are integers.
\end{theorem}

\begin{proof}
See \cref{postponed} below.
\end{proof}

The following corollary
is an immediate consequence of
the theorem, but actually it has a much easier proof which we include below.

\begin{corollary}\label{only}
  For $\lambda \in \P$ and $a,b \in \kk$ with $a \neq b$, there is
  a filtration
  $0=V_0 \subseteq V_1 \subseteq V_2 \subseteq V_3 = D_{b|a} \Delta(\lambda)$
  such that
  \begin{align*}
    V_3/V_2 &\cong
    \left\{\begin{array}{l}\Delta(\lambda+\boxed{a})\\
    0\end{array}\right.&&
    \begin{array}{l}\text{if $a \in \add(\lambda)$ and $b=t-|\lambda|$}\\
      \text{otherwise,}\end{array}\\
    V_2 / V_1 &\cong
     \left\{\begin{array}{l}
    \Delta\left((\lambda - \boxed{b})+\boxed{a}\right)\\
    0\end{array}\right.
     &&\begin{array}{l}
    \text{if $b \in \rem(\lambda)$ and $a \in \add(\lambda-\boxed{b})$}\\
    \text{otherwise},
    \end{array}\\
V_1/V_0&\cong \left\{\begin{array}{l}    \Delta(\lambda-\boxed{b}) \\
      0\end{array}\right.
&&    \begin{array}{l}\text{if $a=t-|\lambda|+1$ and $b \in \rem(\lambda)$}\\
      \text{otherwise.}\end{array}
  \end{align*}
\end{corollary}

\begin{proof}[Direct proof avoiding \cref{keyresult}]
  Let $\gamma := \wt_t(\lambda)$.
By \cref{altcenterdef2}, we can compute $D_{b|a} \Delta(\lambda)$
  by applying $\PR_{\gamma+\alpha_a-\alpha_b}$ to the $\Delta$-flag for $D \Delta(\lambda)$ from \cref{hood}. This produces a module with a $\Delta$-flag consisting of all $\Delta(\mu)$ in the original $\Delta$-flag such that
  $\wt_t(\mu) -\wt_t(\lambda)=\alpha_a - \alpha_b$.
  It just remains to compute $\wt_t(\mu)-\wt_t(\lambda)$ for the various possible
  $\mu$.
  If $\mu = \lambda + \boxed{c}$ for $c \in \add(\lambda)$
  then, by a computation using the first equality from \cref{wtt},
  we have that $\wt_t(\mu)-\wt_t(\lambda) =
  \alpha_c-\alpha_{t-|\lambda|}$; for this to equal $\alpha_a-\alpha_b$
  we must have $b=t-|\lambda|$ and $c=a$.
  If $\mu = \lambda - \boxed{d}$ for $d \in \rem(\lambda)$
  then, by a similar computation, $\wt_t(\mu)-\wt_t(\lambda)
  = \alpha_{t-|\lambda|+1}-\alpha_d$; for this to equal $\alpha_a-\alpha_b$
  we must have $d=b$ and $a=t-|\lambda|+1$.
  Finally if $\mu=(\lambda - \boxed{d})+\boxed{c}$ for $d \in \rem(\lambda)$
  and $c \in \add(\lambda-\boxed{d})$
  then $\wt_t(\mu)-\wt_t(\lambda) = \alpha_c-\alpha_d$;
  for this to equal $\alpha_a-\alpha_b$ we must have $c=a$ and $d=b$.
\end{proof}

\subsection{Blocks}\label{mainresults}
We assume throughout the subsection that $t \in \N$. We are going to
describe the structure of the atypical ``blocks,'' revealing in particular that
they are indecomposable, hence, they are actually blocks.
Recall from \cref{cocomb} that the atypical ``blocks'' are parametrized by partitions $\kappa \in \P_t$, with the irreducible modules in the ``block'' being
the ones labelled by the partitions
$\{\kappa^{(0)}, \kappa^{(1)}, \dots\}$. This is the set
$S(\gamma)$ from \cref{Sgamma} where $\gamma \in P$
is obtained from $\kappa$ according to \cref{kappatogamma}.

The first step is to show that all of the atypical ``blocks'' are equivalent to each other. The proof of this uses the special projective functors
$D_{b|a}$ with $a \neq b$. These are the ones which can be defined just using
information about central characters rather than requiring
the Jucys-Murphy elements;
cf. \cref{altcenterdef2,only}.
In view of \cref{omegas},
this sort of information was already available to Comes and Ostrik
in an equivalent form, and indeed they were also able to prove
a similar result by an analogous argument; see
\cite[Lem.~5.18(2)]{CO} and \cite[Prop.~6.6]{CO}.

\begin{lemma}\label{111degrees}
  Let $\kappa$ and $\tilde\kappa$ be partitions of $t$ such that $\tilde\kappa$ is obtained from $\kappa$ by moving a node from the first row of its Young diagram
  to its $(r+1)$th row for some $r \geq 1$.
  Let $a := \kappa_{r+1}-r+1$ and $b := \kappa_1$.
  Then for all $n \geq 0$ we have that $D_{b|a} \Delta(\kappa^{(n)}) \cong \Delta(\tilde\kappa^{(n)})$ and $D_{a|b} \Delta(\tilde\kappa^{(n)}) \cong \Delta(\kappa^{(n)})$.
\end{lemma}

\begin{proof}
  Let $\gamma, \tilde\gamma \in P$ be defined from $\kappa$ and $\tilde\kappa$
  according to \cref{kappatogamma}. From this formula it follows that
  $\tilde\gamma = \gamma + \alpha_a - \alpha_b$ where $a = \kappa_{r+1}-r+1$ and $b = \kappa_1$ as in the statement of the lemma.
  Note that $a \neq b$.
  So we can apply \cref{altcenterdef2} to see that
  $D_{b|a} \Delta(\kappa^{(n)}) = \PR_{\gamma + \alpha_a - \alpha_b} (D \Delta(\kappa^{(n)}))$
  and $D_{a|b} \Delta(\tilde\kappa^{(n)}) = \PR_{\gamma - \alpha_a+\alpha_b}
  (D \Delta(\tilde\kappa^{(n)}))$.
  Now we use this description to show that
  $D_{b|a} \Delta(\kappa^{(n)}) \cong \Delta(\tilde\kappa^{(n)})$. The
  proof that $D_{a|b} \Delta(\tilde\kappa^{(n)}) \cong \Delta(\kappa^{(n)})$ is similar and we leave this to the reader.

  Fix $n \geq 0$ and
  let $B_n$ be the set of $\mu \in \P$ which are obtained from $\kappa^{(n)}$
  by removing a node, removing a node then adding a different
  node, or adding a node.
  Bearing in mind that $a \neq b$,
  the standard modules $\Delta(\mu)$ for $\mu \in B_n$ include all
  of the ones which are sections of the $\Delta$-flag from \cref{hood}
  which could possibly be in the same block as $\Delta(\tilde\kappa^{(n)})$.
  Now it suffices to show for $m \geq 0$
  that $\tilde\kappa^{(m)} \in B_n$ if and only if $m=n$.
   There are four cases to consider.

  \vspace{1mm}
  \noindent{\em Case one: $n=0$.} We have that $\kappa^{(0)} = (\kappa_2,\kappa_3,\dots,\kappa_{r+1},\dots)$
  and $\tilde\kappa^{(0)} = (\kappa_2,\kappa_3,\dots,\kappa_{r+1}+1,\dots)$, which
  is $\kappa^{(0)}$ with one node added to the $r$th row of
  its Young diagram.
  We definitely have that $\tilde\kappa^{(0)} \in B_0$.
  All other $\mu \in B_0$ satisfy
  $|\mu| \leq |\tilde\kappa^{(0)}|$.
  Since all $\tilde\kappa^{(m)}$ with $m > 0$
  have $|\tilde\kappa^{(m)}| > |\tilde\kappa^{(0)}|$,
  none of these belong to $B_0$.

  \vspace{1mm}
  \noindent{\em Case two: $1 \leq n < r$.}
  We have that $\kappa^{(n)} = (\kappa_1+1,\kappa_2+1,\dots,\kappa_n+1,\dots,\kappa_{r+1},\dots)$ and $\tilde\kappa^{(n)} = (\kappa_1,\kappa_2+1,\dots,\kappa_n+1,\dots,\kappa_{r+1}+1,\dots)$, which is $\kappa^{(n)}$ with a node removed from the first row and a node added to the $r$th row of its Young diagram.
  We definitely have that $\tilde\kappa^{(n)} \in B_n$.
  For $m < n$,  $\tilde\kappa^{(m)}$ is of smaller size than
  $\kappa^{(n)}$ and its $r$th row is of length $\kappa_{r+1}+1$. This cannot be obtained
  from $\kappa^{(n)}$ by removing a node since $\kappa^{(n)}$ has $r$th row of length $\kappa_{r+1}$. So it does not belong to $B_n$.
  For $m > n$, $\tilde\kappa^{(m)}$ is of greater size than
  $\kappa^{(n)}$ and its first row is of length $\kappa_1$. This cannot be obtained from $\kappa^{(n)}$ by adding a node since $\kappa^{(n)}$ has first row of length $\kappa_1+1$. So again it does not belong to $B_n$.
  
  \vspace{1mm}
  \noindent{\em Case three: $n=r$.}
  We have that $\kappa^{(n)} = (\kappa_1+1,\kappa_2+1,\dots,\kappa_{r}+1,\kappa_{r+2},\dots)$ and $\tilde\kappa^{(n)} = (\kappa_1,\kappa_2+1\dots,\kappa_{r}+1,\kappa_{r+2},\dots)$, which is $\kappa^{(n)}$ with a node removed from the first row of its Young diagram.
  We definitely have that $\tilde\kappa^{(n)} \in B_n$.
  The $\tilde\kappa^{(m)}$ with $m < n$ have $|\tilde\kappa^{(m)}|\leq|\tilde\kappa^{(n)}| -1= |\kappa^{(n)}|-2$ so are not elements of $B_n$.
  The $\tilde\kappa^{(m)}$ with $m > n$ have
  $(r+1)$th row of length $\kappa_{r+1}+2$,
  so these are not elements of $B_n$ either since this is at least two more than the length of the $(r+1)$th row of $\kappa^{(n)}$.
  
  \vspace{1mm}
  \noindent{\em Case four: $n > r$.}
  We have that $\kappa^{(n)} = (\kappa_1+1,\kappa_2+1,\dots,\kappa_{r+1}+1,\dots)$ and $\tilde\kappa^{(n)} = (\kappa_1,\kappa_2+1,\dots,\kappa_{r+1}+2,\dots)$, which is $\kappa^{(n)}$ with
  a node removed from its first row and a node added to its $(r+1)$th row.
  We definitely have that $\tilde\kappa^{(n)} \in B_n$.
  The $\tilde\kappa^{(m)}$ with $m > n$ are of greater size than $\kappa^{(n)}$
  and have first row of length $\kappa_1$; these cannot be obtained by adding a node to $\kappa^{(n)}$.
  The $\tilde\kappa^{(m)}$ with $r+1\leq m < n$ are of smaller size than $\kappa^{(n)}$
  and have $(r+1)$th row of length $\kappa_{r+1}+2$; these cannot be obtained by removing a node from $\kappa^{(n)}$.
  The $\tilde\kappa^{(m)}$ with $ m \leq r$
  have first row of length $\leq \kappa_1$ and $(r+1)$th row of length
  $\kappa_{r+2}$, whereas these two rows of
  $\kappa^{(n)}$ are of lengths $\kappa_1+1$ and $\kappa_{r+1}+1 > \kappa_{r+2}$,
  so these are not elements of $B_n$.
\end{proof}

\begin{theorem}[Comes-Ostrik]\label{blockreduction}
Let $\kappa$ and $\tilde\kappa$ be partitions of $t$,
  denoting the associated $\sim_t$-equivalence classes by $S := \{\kappa^{(0)}, \kappa^{(1)}, \dots\}$ and
  $\widetilde{S} := \{\tilde\kappa^{(0)}, \tilde\kappa^{(1)}, \dots\}$. There is an equivalence of categories
  $$
  \Sigma:1_S Par_t\Mod \rightarrow 1_{\widetilde{S}} Par_t\Mod
  $$
  between the corresponding ``blocks''
  such that $\Sigma L(\kappa^{(n)}) \cong L(\tilde\kappa^{(n)})$
  for all $n \geq 0$.
  The functor $\Sigma$ is
a composition of the special projective functors $D_{b|a}\:(a \neq b)$, hence, it is
 a projective functor.
\end{theorem}

\begin{proof}
  We may assume that $\tilde\kappa$ is obtained from $\kappa$ by moving a node
  from the first row of its Young diagram to its $(r+1)$th row for some $r \geq 1$. Thus, we are in the situation of \cref{111degrees}. The lemma gives us functors
  $D_{b|a}:1_S Par_t\Mod \rightarrow 1_{\widetilde{S}}Par_t\Mod$
  and $D_{a|b}:1_{\widetilde{S}}Par_t\Mod\rightarrow 1_S Par_t\Mod$
  such that $D_{b|a} \Delta(\kappa^{(n)}) \cong \Delta(\tilde\kappa^{(n)})$
  and
  $D_{a|b} \Delta(\tilde\kappa^{(n)}) \cong \Delta(\kappa^{(n)})$.
  These functors are also biadjoint thanks to \cref{biadj}.
  It follows easily that they are quasi-inverse equivalences of categories as claimed in the theorem.
  In more detail, the unit and counit of one of the adjunctions
  gives natural transformations $D_{a|b} \circ D_{b|a} \Rightarrow \Id$
  and $\Id \Rightarrow D_{b|a}\circ D_{a|b}$. We claim that these natural transformations are isomorphisms. They are non-zero, hence, they are isomorphisms
  on all standard modules. The functors are exact and indecomposable projectives have finite $\Delta$-flags,
  so it follows that
  the natural transformations
  are isomorphisms on all indecomposable projectives.
  Then we get that they are isomorphisms on an arbitrary module
  by considering a two step projective resolution and applying the Five Lemma.
\end{proof}

The next lemma does use the functors $D_{b|a}$ in the case $a=b$, i.e., it definitely
requires the full strength of \cref{keyresult} rather than merely \cref{only}.

\begin{lemma}\label{within}
  Let $\kappa \in \P_t$ and $S := \{\kappa^{(0)},\kappa^{(1)},\dots\}$
  be the corresponding $\sim_t$-equivalence class.
  For each $n \geq 0$, there is an endofunctor
  $\Pi_n:Par_t\Mod \rightarrow Par_t\Mod$
  such that
$\Pi_n \Delta(\kappa^{(m)}) = 0$ for $m \neq n,n+1$, and 
moreover  there exist short exact sequences
$0\to \Delta(\kappa^{(n)}) \rightarrow \Pi_n \Delta(\kappa^{(n)})
  \rightarrow \Delta(\kappa^{(n+1)})\rightarrow 0$
  and
$0\to \Delta(\kappa^{(n)}) \rightarrow \Pi_n \Delta(\kappa^{(n+1)})
  \rightarrow \Delta(\kappa^{(n+1)})\rightarrow 0$. The functor
  $\Pi_n$ is a
composition of the special projective functors $D_{b|a}\:(a,b \in \Z)$, hence, it
  is a projective functor.
\end{lemma}

\begin{proof}
  In view of \cref{blockreduction}, it suffices to prove
  the lemma in the special case that $\kappa =(t)$,
  when $S = \{\varnothing, (t+1), (t+1,1),(t+1,1^2),\dots\}$ as in \cref{generalblock}.
  Then we take $\Pi_0 := D_{0|t}\circ \cdots \circ D_{t-1|1} \circ D_{t|0}$
  and $\Pi_n := D_{-n|-n}$ for $n > 0$.
  Now it is just a matter of applying \cref{keyresult} to see that these functors have the stated properties.

  The situation for $\Pi_0$ is the most interesting.
To understand this, 
let $u := \lceil\frac{t}{2}\rceil$
and $v := \lfloor\frac{t}{2}\rfloor$.
Then one checks
that $D_{v+1|u-1} \circ \cdots \circ D_{t-1|1}\circ D_{t|0}
(\Delta(\varnothing)) \cong \Delta((u))$; each of these functors adds a single
node to the first row of the Young diagram.
After that we apply $D_{v|u}$ to get a
module with a two step $\Delta$-flag, with
a copy of $\Delta((u+1))$ at the top
and a copy of $\Delta((v))$ at the bottom.
Note this is obtained from \cref{keyresult} in a slightly different way according to whether $u=v$ (i.e., $t$ is even) or $u  =v+1$ (i.e., $t$ is odd).
Also, this is now a module in an atypical block.
Finally we apply $D_{0|t} \circ D_{1|t-1}\circ\cdots D_{v-1|u+1}$ to
end up with the desired two step $\Delta$-flag with a copy of
$\Delta(\kappa^{(1)}) =
\Delta((t+1))$ at the top and $\Delta(\kappa^{(0)}) =
\Delta(\varnothing)$ at the bottom; each of these functors adds a single node to the first row of the Young diagram labelling the module at the top and removes a node from the Young diagram labelling the module at the bottom.
This is what $\Pi_0$ is meant to do to $\Delta(\varnothing)$.
A similar argument shows that
$\Pi_0 \Delta((t+1))$ has a $\Delta$-flag with the same two sections.
It is also easy to check that $\Pi_0 \Delta(\kappa^{(m)}) = 0$ for $m > 1$, indeed, $D_{t|0}$ already annihilates these standard modules.

The functors $\Pi_n = D_{-n|-n}$ for $n > 0$ are easier to analyze.
Noting that $\kappa^{(n)} = \Delta((t+1,1^{n-1}))$,
the module $\Pi_n \Delta(\kappa^{(n)})$ has a two step $\Delta$-flag with
$\Delta(\kappa^{(n+1)}) = \Delta((t+1,1^{n}))$ at the top and
$\Delta(\kappa^{(n)})$ at the bottom; this uses
the $t-|\lambda|=a=b\notin\rem(\lambda)$ case from \cref{keyresult}.
Similarly, $\Pi_n \Delta(\kappa^{(n+1)})$ has a $\Delta$-flag
with the same two sections.
Finally, one checks that $\Pi_{n} \Delta(\kappa^{(m)}) = 0$ for $m \neq n,n+1$.
\end{proof}

\begin{remark}
  In the proof of the next theorem, we will show that the functor $\Pi_n$
  from \cref{within} satisfies
  $\Pi_n \Delta(\kappa^{(n)}) \cong \Pi_n \Delta(\kappa^{(n+1)})
  \cong \Pi_n L(\kappa^{(n+1)}) \cong P(\kappa^{(n+1)})$
  for all $n \geq 0$.
  \end{remark}

Now we can prove the main result about blocks.
This can also be deduced from
\cite[Th.~6.10]{CO}, but the proof of that
appealed to results of Martin \cite{Martin} in order to obtain
the precise submodule structure of the indecomposable projectives,
whereas we are able to establish this by exploiting the
highest weight structure and the Chevalley duality $?^\sigmadual$.

\begin{theorem}\label{MAIN}
  Let $\kappa \in \Par_t$ and $S := \{\kappa^{(0)}, \kappa^{(1)},\dots\}$ be the corresponding $\sim_t$-equivalence class.
  \begin{itemize}
    \item[(i)]
      For each $n \geq 0$, the standard module $\Delta(\kappa^{(n)})$ is of length
      two with head $L(\kappa^{(n)})$ and socle $L(\kappa^{(n+1)})$.
    \item[(ii)]
      The indecomposable projective module $P(\kappa^{(0)})$
      is isomorphic to $\Delta(\kappa^{(0)})$, while for $n \geq 1$
      the  module $P(\kappa^{(n)})$
      has a two step $\Delta$-flag with top section $\Delta(\kappa^{(n)})$ and bottom section $\Delta(\kappa^{(n-1)})$.
    \item[(iii)]
      For each $n \geq 1$, $P(\kappa^{(n)})$ is self-dual with irreducible
      head and socle isomorphic to
      $L(\kappa^{(n)})$ and completely reducible heart $\rad P(\kappa^{(n)}) / \soc P(\kappa^{(n)})\cong L(\kappa^{(n-1)})\oplus L(\kappa^{(n+1)})$.
  \end{itemize}
  \end{theorem}

\begin{proof}
  To improve the readability, we  write simply $P(n), \Delta(n)$ and $L(n)$ in place of $P(\kappa^{(n)}), \Delta(\kappa^{(n)})$ and $L(\kappa^{(n)})$.
For $n \geq 0$, \cref{within} shows that the module
$P_n := \Pi_{n-1} \circ \cdots \Pi_1 \circ \Pi_0 (\Delta(0))$
  has a two step $\Delta$-flag with top section $\Delta(n)$
  and bottom section $\Delta(n-1)$.
  Since $\Delta(0)$ is projective by the minimality observed in \cref{topinblock} and each $\Pi_i$ is a projective functor,
  $P_n$ is projective. Since $P_n$ has $L(n)$ in its head, it must contain the indecomposable projective $P(n)$ as a summand, so
  we either have that $P(n) \cong P_n$ if $P_n$ is indecomposable, or $P(n) \cong \Delta(n)$ otherwise.
  In the former case, $(P(n):\Delta(m)) = \delta_{m,n}+\delta_{m,n-1}$, while $(P(n):\Delta(m)) = \delta_{m,n}$ in the latter situation.
Now we apply BGG reciprocity to deduce for any $m \geq 0$  that
  $[\Delta(m):L(n)] = \delta_{n,m}+\delta_{n,m+1}$ if $P_n$
is indecomposable and $[\Delta(m):L(n)] = \delta_{n,m}$ otherwise. Hence, for each $m \geq 0$, we either have that $\Delta(m)
\cong L(m)$, or $\Delta(m)$ is of composition
  length two with composition factors $L(m)$ and
  $L(m+1)$.

  We claim for any $n \geq 0$ that $\Delta(n) \cong L(n)$ if and only if $\Delta(n+1)\cong L(n+1)$.
  Suppose first that $\Delta(n) \cong L(n)$.
  Since $\Pi_n$ commutes with duality by \cref{selfadj}, this implies that
  $\Pi_n \Delta(n)$ is self-dual. But
  this module has a two step
  $\Delta$-flag with top section $\Delta(n+1)$
  and bottom section $\Delta(n)\cong L(n)$. The only way such a module can be self-dual is if $\Delta(n+1) \cong L(n+1)$ (and the module must be completely reducible).
  Conversely, suppose for a contradiction that $\Delta(n+1) \cong L(n+1)$
  but $\Delta(n) \not\cong L(n)$.
Then $\Delta(n)$ is of length two with composition factors
  $L(n)$ and $L(n+1)$, so that $P(n+1)$ has a two step $\Delta$-flag
  with top section $\Delta(n+1)\cong L(n+1)$
  and bottom section
  $\Delta(n)$. Since $\Pi_{n+1} \Delta(n) = 0$ according to \cref{within}
  and $\Pi_{n+1}$ is exact,
  we must have that $\Pi_{n+1} L(n+1) = 0$. Since $\Delta(n+1) \cong L(n+1)$, this implies that $\Pi_{n+1} \Delta(n+1) = 0$, which contradicts \cref{within}.

  From the claim, we see that if $\Delta(n)$ is irreducible for any one $n \geq 0$, then it is irreducible for all $n \geq 0$.
  Since all atypical ``blocks'' are equivalent by \cref{blockreduction},
  it follows in that case that
 the standard modules $\Delta(\lambda)$ for {\em all} $\lambda \in \P$
 are irreducible. This implies that the minimal ordering $\succeq_t$ from
 \cref{minorder} is trivial, hence, the blocks are trivial and $Par_t$ is
 semisimple, which contradicts \cref{sscor}.
  Thus, we have proved
  that $\Delta(n)$ must be of length two
  for every $n \geq 0$, and (i) is proved.

  Property (ii) follows immediately from (i) and BGG reciprocity as noted earlier.

  It remains to prove (iii).
  Take $n \geq 1$. By \cref{within},
  we have that $\Pi_{n-1} \Delta(n+1) = 0$. Since $\Pi_{n-1}$
  is exact and $L(n+1)$ is a composition factor of $\Delta(n+1)$, it follows that
  $\Pi_{n-1} L(n+1) = 0$ too. From this, we deduce that
  $\Pi_{n-1} \Delta(n) \cong \Pi_{n-1} L(n)$. By \cref{within} again,
  $\Pi_{n-1} \Delta(n-1)$ has the same
  composition length as $\Pi_{n-1} \Delta(n) \cong \Pi_{n-1} L(n)$.
  Also $\Delta(n-1)$ has $L(n)$ as a constituent. Using the exactness of $\Pi_{n-1}$ again, we must therefore have that
  $\Pi_{n-1} \Delta(n-1) \cong \Pi_{n-1} L(n)$.
  As observed earlier in the proof, this module is isomorphic to $P(n)$,
  so using that $L(n)$ is self-dual and $\Pi_{n-1}$ commutes with duality, we now see that $P(n)$ is self-dual.
  We also know that it has length four with irreducible head
  $L(n)$,
  $[P(n):L(n)] = 2$ and
  $[P(n):L(n-1)] = 
  [P(n):L(n+1)] = 1$.
  The only possible structure is the one claimed.
\end{proof}

\begin{corollary}[Comes-Ostrik]\label{blocksareblocks}
All ``blocks'' of $Par_t\Mod$ are indecomposable, hence, they coincide with the blocks.
\end{corollary}

\begin{corollary}\label{minordercor}
  The minimal ordering $\succeq_t$ from \cref{minorder} is
  the partial order such that $\kappa^{(m)} \succeq_t \kappa^{(n)}$
  for each $\kappa \in \P_t$ and $m  \leq n$,
  with all other pairs of partitions being incomparable.
\end{corollary}

 In general, in an upper finite highest weight category, the standard objects can have infinite length. Our final corollary, which is also noted in \cite[Rem.~6.4]{SaS},
 shows that this is not the
 case in $Par_t\Modlfd$.
Consequently, the full subcategory consisting of all modules of
finite length has enough projectives and injectives, indeed, this subcategory is
an essentially finite highest weight category in the sense of \cite[Def.~3.7]{BS}.

  \begin{corollary}\label{locartinian}
  The locally unital
  algebra $Par_t$ is locally Artinian, i.e., the left ideals $Par_t 1_n$ and the right ideals $1_n Par_t$ are of finite length for all $n \geq 0$.
\end{corollary}

\begin{proof}
 \cref{MAIN} shows that all indecomposable projective left $Par_t$-modules
  are of finite length, hence, all finitely generated projectives are of finite length too. This includes all of the left ideals $Par_t 1_n$.
  Since there is a duality $?^\sigmadual$, it also follows that all
  fintely cogenerated injective left $Par_t$-modules are of finite length.
This includes all of the duals $(1_n Par_t)^\circledast$, hence, each $1_n Par_t$ is of finite length as a right module.
\end{proof}

\subsection{\boldmath Proof of \cref{keyresult}}\label{postponed}
It just remains to prove \cref{keyresult}.
In fact, we will prove the following slightly stronger result, from which \cref{keyresult} follows easily on applying the functors involved
to the Specht module $S(\lambda)$.
To state this stronger result,
let $j_!:\sym\Modfd \rightarrow Par_t \Modlfd$ be the standardization functor from \cref{stdize},
$E_a$ and $F_b$ be the refined induction and restriction functors from
\cref{endofterm},
$D_{b|a}$ be the special projective functor from \cref{summertime},
and $\PR_c:\sym\Modfd \rightarrow \sym\Modfd$ be the functor defined by
multiplication by the identity element of the symmetric group
$S_c$ if $c \in \N$, i.e., it is the projection onto $\kk S_c \Modfd$
followed followed by the inclusion of $\kk S_c \Modfd$ into 
$\sym\Modfd$, or the zero functor if $c \in \kk - \N$.

 \begin{theorem}\label{refinedfilt}
    For $a,b \in \kk$,
    there is a filtration of the functor
    $D_{b|a} \circ j_!:\sym\Modfd \rightarrow Par_t\Modlfd$ by subfunctors
    $0 = S_0 \subseteq S_1 \subseteq S_2 \subseteq S_3 \subseteq S_4 = D_{b|a} \circ j_!$ such that
    \begin{align*}
      S_4 / S_3 &\cong j_! \circ E_a \circ \PR_{t-b},\\
      S_3 / S_2 &\cong j_! \circ \PR_{t-a} \circ \PR_{t-b},\\
      S_2 / S_1 &\cong j_! \circ E_a \circ F_b,\\
      S_1 / S_0 &\cong j_! \circ \PR_{t-a} \circ F_b.
    \end{align*}
(Recall that a {\em subfunctor} $S$ of a functor
$T:\sym\Modfd \rightarrow Par_t\Modlfd$ is a functor $S:\sym\Modfd \rightarrow Par_t \Modlfd$ such that $S V$ is a submodule of $T V$ for all $V \in \sym\Modfd$
and $S f = Tf |_{SV}$ for all $f\in\Hom_{\sym}(V,V')$; then the quotient  $T / S$ is the obvious functor with $(T / S)(V) := T V / SV$.)
\end{theorem}

The proof will take up the rest of the subsection.
We begin by constructing a filtration of the functor
$D\circ j_!:\sym\Modfd\rightarrow Par_t\Modlfd$.
Note that
$D \circ j_! \cong M\otimes_{\sym}$ where $M$ is the
$(Par_t,\sym)$-bimodule
\begin{equation}\label{M}
M := 1_{\mid\,\star} Par_t \otimes_{Par^\sharp} \infl^\sharp \sym.
\end{equation}
  We also have the $(Par_t,\sym)$-bimodules
  \begin{align}
    N_4 &= Par_t\otimes_{Par^\sharp} \infl^\sharp (\sym 1_{\mid\,\star}),\label{N4}\\
    N_3 &:= Par_t\otimes_{Par^\sharp} \infl^\sharp \sym,\label{N3}\\
    N_2 &:= Par_t\otimes_{Par^\sharp} \infl^\sharp (\sym 1_{\mid\,\star}\otimes_{\sym} 1_{\mid\,\star}\sym)\label{N2}\\
    N_1 &:= Par_t\otimes_{Par^\sharp} \infl^\sharp (1_{\mid\,\star} \sym).\label{N1}
  \end{align}
   The functors $\sym\Modfd\rightarrow Par_t\Modlfd$
    defined by tensoring with $N_4, N_3, N_2$ and $N_1$ are isomorphic to
    $j_! \circ E$, $j_!$, $j_! \circ E \circ F$
    and  $j_! \circ F$, respectively.

    For $m \geq n \geq 0$, let $B_{m,n}$ be the basis for $1_m Par^- 1_n$
  defined by representatives for the equivalence classes
  of normally ordered upward partition diagrams.
  By \cref{td}, the vector space $M$
  is isomorphic to $1_{\mid\,\star} Par^- \otimes_{\KK} \sym$, hence, it
    has basis
    \begin{equation}\label{loud}
    \left\{f \otimes g\:\big|\:m\geq 0,n \geq 0, m+1 \geq n, f \in B_{m+1,n}, g \in S_n\right\}.
    \end{equation}
    For any $f\in B_{m+1,n}$,
  let $c(f)$ be the connected component of the diagram containing the top
  left vertex.
  In the language from \cref{stria}, this component
  could be a trunk, an upward tree, an upward leaf, or an upward branch.
  Then we introduce the following subspaces of $M$:
  \begin{itemize}
   \item Let  
    $M_1$ be the subspace of $M$ spanned by all
     $f \otimes g$ in this basis such that $c(f)$ is a trunk.
     \item Let
     $M_2$ be the subspace spanned by all $f \otimes g$ such that
     $c(f)$ is either a trunk or an upward tree.
   \item Let
    $M_3$ be the subspace spanned by all $f \otimes g$ such that
     $c(f)$ is either a trunk, an upward tree, or an upward leaf.
   \item
     Let $M_0 := 0$ and $M_4 := M$.
  \end{itemize}
  The following is a generalization of \cref{hood}.
  
    \begin{lemma}\label{filtration}
      The subspaces $0 = M_0 \subset M_1 \subset M_2 \subset M_3 \subset M_4 = M$
      are sub-bimodules of the $(Par_t,\sym)$-bimodule $M$. Moreover,
      there are bimodule isomorphisms $\theta_i:N_i \stackrel{\sim}{\rightarrow} M_i / M_{i-1}$
       for each $i=1,\dots,4$.
\end{lemma}

\begin{proof}
    The fact that each $M_i$ is a sub-bimodule of $M$ is easily checked by
    vertically composing a basis vector $f \otimes g$
    with an arbitrary partition diagram on the top and with any permutation diagram on the bottom. One just needs to note
    that the action on top involves
     $\res_{\mid\,\star}$, so that the top left vertex is untouched. This implies that
    the type $c(f)$ does not change if it is a trunk or an upward leaf, while if it is an upward tree it can only be changed to another
 upward tree or to a trunk.

 We show in this paragraph that there is a bimodule isomorphism
  \begin{align}\label{theta1}
    \theta_1:N_1 &\rightarrow M_1,&
\begin{tikzpicture}[anchorbase,scale=1.1]
  \draw[-] (-0.5,-0.2)--(-0.5,0.2)--(0.5,0.2)--(0.5,-0.2)--(-0.5,-0.2);
    \node at (-0,-0) {$\scriptstyle{f}$};
    \draw[-,thick] (-0.4,-0.2)--(-0.4,-0.5);
    \draw[-,thick] (0.4,-0.2)--(0.4,-0.5);
    \draw[-,thick] (-0.4,0.2)--(-0.4,0.5);
    \draw[-,thick] (0.4,0.2)--(0.4,0.5);
    \node at (-0.2,0.35) {$\cdot$};
    \node at (-0,0.35) {$\cdot$};
    \node at (0.2,0.35) {$\cdot$};
    \node at (-0.2,-0.4) {$\cdot$};
    \node at (-0,-0.4) {$\cdot$};
    \node at (0.2,-0.4) {$\cdot$};
  \end{tikzpicture}\otimes\begin{tikzpicture}[anchorbase,scale=1.1]\draw[-] (-0.5,-0.2)--(-0.5,0.2)--(0.5,0.2)--(0.5,-0.2)--(-0.5,-0.2);
    \node at (-0,-0) {$\scriptstyle{g}$};
    \draw[-,thick] (-0.4,-0.2)--(-0.4,-0.5);
    \draw[-,thick] (0.4,-0.2)--(0.4,-0.5);
    \draw[-,thick] (-0.4,0.2)--(-0.4,0.5);
    \draw[-,thick] (0.4,0.2)--(0.4,0.5);
    \node at (-0.2,0.35) {$\cdot$};
    \node at (-0,0.35) {$\cdot$};
    \node at (0.2,0.35) {$\cdot$};
    \node at (-0.2,-0.4) {$\cdot$};
    \node at (-0,-0.4) {$\cdot$};
    \node at (0.2,-0.4) {$\cdot$};
    \end{tikzpicture}
&\mapsto     \begin{tikzpicture}[anchorbase,scale=1.1]
  \draw[-] (-0.5,-0.2)--(-0.5,0.2)--(0.5,0.2)--(0.5,-0.2)--(-0.5,-0.2);
    \node at (-0,-0) {$\scriptstyle{f}$};
    \draw[-,thick] (-0.4,-0.2)--(-0.4,-0.5);
    \draw[-,thick] (0.4,-0.2)--(0.4,-0.5);
    \draw[-,thick] (-0.4,0.2)--(-0.4,0.5);
    \draw[-,thick] (0.4,0.2)--(0.4,0.5);
    \draw[-,thick] (-0.6,0.5)--(-0.6,-0.5);
    \node at (-0.2,0.35) {$\cdot$};
    \node at (-0,0.35) {$\cdot$};
    \node at (0.2,0.35) {$\cdot$};
    \node at (-0.2,-0.4) {$\cdot$};
    \node at (-0,-0.4) {$\cdot$};
    \node at (0.2,-0.4) {$\cdot$};
  \end{tikzpicture}\otimes\begin{tikzpicture}[anchorbase,scale=1.1]\draw[-] (-0.5,-0.2)--(-0.5,0.2)--(0.5,0.2)--(0.5,-0.2)--(-0.5,-0.2);
    \node at (-0,-0) {$\scriptstyle{g}$};
    \draw[-,thick] (-0.4,-0.2)--(-0.4,-0.5);
    \draw[-,thick] (0.4,-0.2)--(0.4,-0.5);
    \draw[-,thick] (-0.4,0.2)--(-0.4,0.5);
    \draw[-,thick] (0.4,0.2)--(0.4,0.5);
    \node at (-0.2,0.35) {$\cdot$};
    \node at (-0,0.35) {$\cdot$};
    \node at (0.2,0.35) {$\cdot$};
    \node at (-0.2,-0.4) {$\cdot$};
    \node at (-0,-0.4) {$\cdot$};
    \node at (0.2,-0.4) {$\cdot$};
  \end{tikzpicture}  \end{align}
  for any $m\geq 0,n > 0$,
   $f \in 1_{m} Par_t 1_{n-1}$ and $g \in S_n$.
   This is a well-defined bimodule homomorphism.
By \cref{td}, $N_1$ is isomorphic
as a vector space to $Par^-\otimes_{\KK} 1_{\mid\,\star} \sym$, hence, it has basis
\begin{equation}\label{N1basis}
\left\{f \otimes g\:\big|\:m \geq n-1 \geq 0, f \in B_{m,n-1}, g \in S_n\right\}.
\end{equation}
The vector space $M_1$ has basis given by all $f_1 \otimes g$ for
$m+1\geq n > 0,
f_1 \in B_{m+1,n}$ and $g \in S_n$ such that $c(f_1)$ is a trunk.
As it is normally ordered, any such $f_1$ is of the form
  \begin{align*}
f_1= \begin{tikzpicture}[anchorbase,scale=1.1]
  \draw[-] (-0.5,-0.2)--(-0.5,0.2)--(0.5,0.2)--(0.5,-0.2)--(-0.5,-0.2);
    \node at (-0,-0) {$\scriptstyle{f}$};
    \draw[-,thick] (-0.4,-0.2)--(-0.4,-0.5);
    \draw[-,thick] (0.4,-0.2)--(0.4,-0.5);
    \draw[-,thick] (-0.4,0.2)--(-0.4,0.5);
    \draw[-,thick] (0.4,0.2)--(0.4,0.5);
    \draw[-,thick] (-0.6,0.5)--(-0.6,-0.5);
    \node at (-0.2,0.35) {$\cdot$};
    \node at (-0,0.35) {$\cdot$};
    \node at (0.2,0.35) {$\cdot$};
    \node at (-0.2,-0.4) {$\cdot$};
    \node at (-0,-0.4) {$\cdot$};
    \node at (0.2,-0.4) {$\cdot$};
  \end{tikzpicture}
  \end{align*}
  for a unique $f \in B_{m,n-1}$.
  Moreover, $f_1 \otimes g = \theta_1(f \otimes g)$ for every $g \in S_n$.
  It follows that $\theta_1$ takes a basis for $N_1$ to a basis for $M_1$,
  so it is an isomorphism.
  
  Next we show that there is a bimodule isomorphism
  \begin{align}\label{theta2}
  \theta_2:N_2&\rightarrow M_2 / M_1,&
  \begin{tikzpicture}[anchorbase,scale=1.1]\draw[-] (-0.5,-0.2)--(-0.5,0.2)--(0.5,0.2)--(0.5,-0.2)--(-0.5,-0.2);
    \node at (-0,-0) {$\scriptstyle{f}$};
    \draw[-,thick] (-0.4,-0.2)--(-0.4,-0.5);
    \draw[-,thick] (0.4,-0.2)--(0.4,-0.5);
    \draw[-,thick] (-0.4,0.2)--(-0.4,0.5);
    \draw[-,thick] (0.4,0.2)--(0.4,0.5);
    \node at (-0.2,0.35) {$\cdot$};
    \node at (-0,0.35) {$\cdot$};
    \node at (0.2,0.35) {$\cdot$};
    \node at (-0.2,-0.4) {$\cdot$};
    \node at (-0,-0.4) {$\cdot$};
    \node at (0.2,-0.4) {$\cdot$};
  \end{tikzpicture}\otimes\begin{tikzpicture}[anchorbase,scale=1.1]\draw[-] (-0.5,-0.2)--(-0.5,0.2)--(0.5,0.2)--(0.5,-0.2)--(-0.5,-0.2);
    \node at (-0,-0) {$\scriptstyle{g}$};
    \draw[-,thick] (-0.4,-0.2)--(-0.4,-0.5);
    \draw[-,thick] (0.4,-0.2)--(0.4,-0.5);
    \draw[-,thick] (-0.4,0.2)--(-0.4,0.5);
    \draw[-,thick] (0.4,0.2)--(0.4,0.5);
    \node at (-0.2,0.35) {$\cdot$};
    \node at (-0,0.35) {$\cdot$};
    \node at (0.2,0.35) {$\cdot$};
    \node at (-0.2,-0.4) {$\cdot$};
    \node at (-0,-0.4) {$\cdot$};
    \node at (0.2,-0.4) {$\cdot$};
  \end{tikzpicture}\otimes\begin{tikzpicture}[anchorbase,scale=1.1]\draw[-] (-0.5,-0.2)--(-0.5,0.2)--(0.5,0.2)--(0.5,-0.2)--(-0.5,-0.2);
    \node at (-0,-0) {$\scriptstyle{h}$};
    \draw[-,thick] (-0.4,-0.2)--(-0.4,-0.5);
    \draw[-,thick] (0.4,-0.2)--(0.4,-0.5);
    \draw[-,thick] (-0.4,0.2)--(-0.4,0.5);
    \draw[-,thick] (0.4,0.2)--(0.4,0.5);
    \node at (-0.2,0.35) {$\cdot$};
    \node at (-0,0.35) {$\cdot$};
    \node at (0.2,0.35) {$\cdot$};
    \node at (-0.2,-0.4) {$\cdot$};
    \node at (-0,-0.4) {$\cdot$};
    \node at (0.2,-0.4) {$\cdot$};
  \end{tikzpicture} &\mapsto \begin{tikzpicture}[anchorbase,scale=1.1]\draw[-] (-0.5,-0.2)--(-0.5,0.4)--(0.5,0.4)--(0.5,-0.2)--(-0.5,-0.2);
    \draw[-,thick] (-0.4,-0.2)--(-0.4,-0.5);
    \draw[-,thick] (0.4,-0.2)--(0.4,-0.5);
    \draw[-,thick] (-0.4,0.4)--(-0.4,0.6);
    \draw[-,thick] (0.4,0.4)--(0.4,0.6);
    \draw[-,thick] (-0.5,0.1)--(0.5,0.1);
    \draw[-,thick] (-0.6,0.6) to[out=down,in=160,looseness=1] (-0.4,-0.35);
    \node at (-0.2,0.45) {$\cdot$};
    \node at (-0,0.45) {$\cdot$};
    \node at (0.2,0.45) {$\cdot$};
    \node at (-0.2,-0.4) {$\cdot$};
    \node at (-0,-0.4) {$\cdot$};
    \node at (0.2,-0.4) {$\cdot$};
    \node at (0,0.25) {$\scriptstyle{f}$};
    \node at (0,-0.08) {$\scriptstyle{g}$};
  \end{tikzpicture}\otimes\begin{tikzpicture}[anchorbase,scale=1.1]\draw[-] (-0.5,-0.2)--(-0.5,0.2)--(0.5,0.2)--(0.5,-0.2)--(-0.5,-0.2);
    \node at (-0,-0) {$\scriptstyle{h}$};
    \draw[-,thick] (-0.4,-0.2)--(-0.4,-0.5);
    \draw[-,thick] (0.4,-0.2)--(0.4,-0.5);
    \draw[-,thick] (-0.4,0.2)--(-0.4,0.5);
    \draw[-,thick] (0.4,0.2)--(0.4,0.5);
    \node at (-0.2,0.35) {$\cdot$};
    \node at (-0,0.35) {$\cdot$};
    \node at (0.2,0.35) {$\cdot$};
    \node at (-0.2,-0.4) {$\cdot$};
    \node at (-0,-0.4) {$\cdot$};
    \node at (0.2,-0.4) {$\cdot$};
  \end{tikzpicture} + M_1
\end{align}
for $m\geq 0, n > 0, f \in 1_{m} Par_t 1_{n}$ and $g, h \in S_n$.
Again, this is a well-defined bimodule homomorphism.
By \cref{td}, $N_2$
is isomorphic to $Par^- \otimes_{\KK} \sym 1_{\mid\,\star}\otimes_{\sym}
1_{\mid\,\star}\sym$. Also $\kk S_n$ is free as a right $\kk S_{n-1}$-module
with basis given by
$\left\{(i\:\,i\!+\!1\:\,\cdots\:\,n)\:|\:1 \leq i \leq n\right\}$, which is a set of $S_n / S_{n-1}$-cosets.
It follows that
$N_2$ has basis
\begin{equation}\label{N2basis}
\left\{f \otimes (i\:\,i\!+\!1\:\,\cdots\:\,n) \otimes g\:\big|\:
m \geq n > 0, f \in B_{m,n}, 1 \leq i \leq n, g \in S_n\right\}.
\end{equation}
The vector space $M_2 / M_1$ has a basis given by all $f_2 \otimes g + M_1$
for $m +1\geq n > 0, f_2 \in B_{m+1,n}$ and
$g \in S_n$ such that $c(f_2)$ is an upward tree.
Any such $f_2$ is equal to
 \begin{align*}
f_2=  \begin{tikzpicture}[scale=1.1,anchorbase]
  \draw[-] (-0.5,-0.2)--(-0.5,0.2)--(0.5,0.2)--(0.5,-0.2)--(-0.5,-0.2);
  \node at (-0,-0) {$\scriptstyle{f}$};
  \draw[-,thick] (-0.7,-0.8)--(-0.7,0.5);
  \draw[-,thick] (-0.7,-0.7) to [out=20,in=down,looseness=1] (0,-.2);
  \draw[-,thick] (-0.4,0.2)--(-0.4,0.5);
  \draw[-,thick] (0.4,0.2)--(0.4,0.5);
  \draw[-,thick] (-0.4,-0.2)--(-0.4,-0.8);
  \draw[-,thick] (0.4,-0.2)--(0.4,-0.8);
  \node at (-0.3,-0.3) {$\cdot$};
  \node at (-.2,-0.3) {$\cdot$};
  \node at (.2,-0.3) {$\cdot$};
  \node at (0.3,-0.3) {$\cdot$};
  \node at (-.2,-0.7) {$\cdot$};
  \node at (0,-0.7) {$\cdot$};
  \node at (0.2,-0.7) {$\cdot$};
   \node at (-0.2,0.35) {$\cdot$};
  \node at (-0,0.35) {$\cdot$};
  \node at (0.2,0.35) {$\cdot$};
  \node at (0.05,-.35) {$\stringlabel{i}$};
  \node at (0.5,-.35) {$\stringlabel{1}$};
  \node at (-0.5,-.35) {$\stringlabel{n}$};
  \end{tikzpicture}
 \end{align*}
 for a
 unique $f \in B_{m,n}$ and
 a unique $1 \leq i \leq n$ (the index of the string at which
 the component $c(f_2)$ meets the bottom of $f$).
 Moreover, $f_2 \otimes g = \theta_2(f \otimes (i\:\,i\!+\!1\:\,\cdots\:\,n) \otimes g)$ for each $g \in S_n$. It follows that $\theta_2$ takes a basis for $N_2$
 to a basis for $M_2 / M_1$, so it is an isomorphism.
 
 The isomorphism $\theta_3$ is defined by
\begin{align}\label{theta3}
  \theta_3:N_3 &\rightarrow M_3/M_2,&
\begin{tikzpicture}[anchorbase,scale=1.1]
    \draw[-] (-0.5,-0.2)--(-0.5,0.2)--(0.5,0.2)--(0.5,-0.2)--(-0.5,-0.2);
    \node at (-0,-0) {$\scriptstyle{f}$};
    \draw[-,thick] (-0.4,-0.2)--(-0.4,-0.5);
    \draw[-,thick] (0.4,-0.2)--(0.4,-0.5);
    \draw[-,thick] (-0.4,0.2)--(-0.4,0.5);
    \draw[-,thick] (0.4,0.2)--(0.4,0.5);
    \node at (-0.2,0.35) {$\cdot$};
    \node at (-0,0.35) {$\cdot$};
    \node at (0.2,0.35) {$\cdot$};
    \node at (-0.2,-0.4) {$\cdot$};
    \node at (-0,-0.4) {$\cdot$};
    \node at (0.2,-0.4) {$\cdot$};
  \end{tikzpicture}\otimes \begin{tikzpicture}[anchorbase,scale=1.1]
    \draw[-] (-0.5,-0.2)--(-0.5,0.2)--(0.5,0.2)--(0.5,-0.2)--(-0.5,-0.2);
    \node at (-0,-0) {$\scriptstyle{g}$};
    \draw[-,thick] (-0.4,-0.2)--(-0.4,-0.5);
    \draw[-,thick] (0.4,-0.2)--(0.4,-0.5);
    \draw[-,thick] (-0.4,0.2)--(-0.4,0.5);
    \draw[-,thick] (0.4,0.2)--(0.4,0.5);
    \node at (-0.2,0.35) {$\cdot$};
    \node at (-0,0.35) {$\cdot$};
    \node at (0.2,0.35) {$\cdot$};
    \node at (-0.2,-0.4) {$\cdot$};
    \node at (-0,-0.4) {$\cdot$};
    \node at (0.2,-0.4) {$\cdot$};
  \end{tikzpicture}
&\mapsto
  \begin{tikzpicture}[anchorbase,scale=1.1]
    \draw[-] (-0.5,-0.2)--(-0.5,0.2)--(0.5,0.2)--(0.5,-0.2)--(-0.5,-0.2);
    \node at (-0,-0) {$\scriptstyle{f}$};
    \draw[-,thick] (-0.4,-0.2)--(-0.4,-0.5);
    \draw[-,thick] (0.4,-0.2)--(0.4,-0.5);
    \draw[-,thick] (-0.4,0.2)--(-0.4,0.5);
    \draw[-,thick] (0.4,0.2)--(0.4,0.5);
    \draw[-,thick] (-0.7,0.5)--(-0.7,0.1);
    \node at (-0.2,0.35) {$\cdot$};
    \node at (-0,0.35) {$\cdot$};
    \node at (0.2,0.35) {$\cdot$};
    \node at (-0.2,-0.4) {$\cdot$};
    \node at (-0,-0.4) {$\cdot$};
    \node at (0.2,-0.4) {$\cdot$};
    \node at (-0.7,0.05) {$\dt$};
  \end{tikzpicture}\otimes \begin{tikzpicture}[anchorbase,scale=1.1]
    \draw[-] (-0.5,-0.2)--(-0.5,0.2)--(0.5,0.2)--(0.5,-0.2)--(-0.5,-0.2);
    \node at (-0,-0) {$\scriptstyle{g}$};
    \draw[-,thick] (-0.4,-0.2)--(-0.4,-0.5);
    \draw[-,thick] (0.4,-0.2)--(0.4,-0.5);
    \draw[-,thick] (-0.4,0.2)--(-0.4,0.5);
    \draw[-,thick] (0.4,0.2)--(0.4,0.5);
    \node at (-0.2,0.35) {$\cdot$};
    \node at (-0,0.35) {$\cdot$};
    \node at (0.2,0.35) {$\cdot$};
    \node at (-0.2,-0.4) {$\cdot$};
    \node at (-0,-0.4) {$\cdot$};
    \node at (0.2,-0.4) {$\cdot$};
  \end{tikzpicture}+M_2
\end{align}
for $m \geq 0$, $n \geq 0$, $f \in 1_{m} Par_t 1_n$ and $g \in S_n$.
This is obviously a well-defined bimodule homomorphism.
It is an isomorphism because it takes the
basis \begin{equation}\label{N3basis}
\{f \otimes g\:|\:m \geq n \geq 0,
f \in B_{m,n}, g \in S_n\}
\end{equation}
 for $N_3$ to the basis
 for $M_3/M_2$ consisting of all $f_3 \otimes g + M_2$ for $m+1> n \geq 0,
 f_3 \in B_{m+1,n}$ and $g \in S_n$ such that $c(f_3)$ is an upward leaf.

 Finally, we construct the isomorphism $\theta_4$.
 The vector space $N_4$ has basis
 \begin{equation}\label{N4basis}
 \{f \otimes g\:|\:m-1\geq n \geq 0, f \in B_{m,n+1}, g \in S_{n+1}\}.
 \end{equation}
 We define the linear map
 \begin{align}\label{theta4}
   \theta_4:N_4 &\rightarrow M_4 / M_3,&
\begin{tikzpicture}[anchorbase,scale=1.1]
    \draw[-] (-0.5,-0.2)--(-0.5,0.2)--(0.5,0.2)--(0.5,-0.2)--(-0.5,-0.2);
    \node at (-0,-0) {$\scriptstyle{f}$};
    \draw[-,thick] (-0.4,-0.2)--(-0.4,-0.5);
    \draw[-,thick] (0.4,-0.2)--(0.4,-0.5);
    \draw[-,thick] (-0.4,0.2)--(-0.4,0.5);
    \draw[-,thick] (0.4,0.2)--(0.4,0.5);
    \node at (-0.2,0.35) {$\cdot$};
    \node at (-0,0.35) {$\cdot$};
    \node at (0.2,0.35) {$\cdot$};
    \node at (-0.2,-0.4) {$\cdot$};
    \node at (-0,-0.4) {$\cdot$};
    \node at (0.2,-0.4) {$\cdot$};
  \end{tikzpicture}\otimes \begin{tikzpicture}[anchorbase,scale=1.1]
    \draw[-] (-0.5,-0.2)--(-0.5,0.2)--(0.5,0.2)--(0.5,-0.2)--(-0.5,-0.2);
    \node at (-0,-0) {$\scriptstyle{g}$};
    \draw[-,thick] (-0.4,-0.2)--(-0.4,-0.5);
    \draw[-,thick] (0.4,-0.2)--(0.4,-0.5);
    \draw[-,thick] (-0.4,0.2)--(-0.4,0.5);
    \draw[-,thick] (0.4,0.2)--(0.4,0.5);
    \node at (-0.2,0.35) {$\cdot$};
    \node at (-0,0.35) {$\cdot$};
    \node at (0.2,0.35) {$\cdot$};
    \node at (-0.2,-0.4) {$\cdot$};
    \node at (-0,-0.4) {$\cdot$};
    \node at (0.2,-0.4) {$\cdot$};
  \end{tikzpicture}
&\mapsto
\begin{tikzpicture}[anchorbase,scale=1.1]
  \draw[-] (-0.5,-0.2)--(-0.5,0.2)--(0.5,0.2)--(0.5,-0.2)--(-0.5,-0.2);
  \node at (-0,-0) {$\scriptstyle{f}$};
  \draw[-,thick] (-.9,.5) to [out=down,in=180,looseness=1] (-0.45,-0.6) to [out=0,in=down,looseness=1] (0,-.2);
  \draw[-,thick] (-0.4,0.2)--(-0.4,0.5);
  \draw[-,thick] (0.4,0.2)--(0.4,0.5);
  \draw[-,thick] (-0.4,-0.2)--(-0.4,-0.8);
  \draw[-,thick] (0.4,-0.2)--(0.4,-0.8);
  \node at (-0.25,-0.3) {$\cdot$};
  \node at (-.15,-0.3) {$\cdot$};
  \node at (.2,-0.3) {$\cdot$};
  \node at (0.3,-0.3) {$\cdot$};
  \node at (-.2,-0.7) {$\cdot$};
  \node at (0,-0.7) {$\cdot$};
  \node at (0.2,-0.7) {$\cdot$};
   \node at (-0.2,0.35) {$\cdot$};
  \node at (-0,0.35) {$\cdot$};
  \node at (0.2,0.35) {$\cdot$};
  \node at (0.05,-.35) {$\stringlabel{i}$};
  \node at (0.5,-.35) {$\stringlabel{1}$};
  \node at (-0.6,-.35) {$\stringlabel{n\!+\!1}$};
  \end{tikzpicture}\otimes \begin{tikzpicture}[anchorbase,scale=1.1]
    \draw[-] (-0.5,-0.2)--(-0.5,0.2)--(0.5,0.2)--(0.5,-0.2)--(-0.5,-0.2);
    \node at (-0,-0) {$\scriptstyle{g'}$};
    \draw[-,thick] (-0.4,-0.2)--(-0.4,-0.5);
    \draw[-,thick] (0.4,-0.2)--(0.4,-0.5);
    \draw[-,thick] (-0.4,0.2)--(-0.4,0.5);
    \draw[-,thick] (0.4,0.2)--(0.4,0.5);
    \node at (-0.2,0.35) {$\cdot$};
    \node at (-0,0.35) {$\cdot$};
    \node at (0.2,0.35) {$\cdot$};
    \node at (-0.2,-0.4) {$\cdot$};
    \node at (-0,-0.4) {$\cdot$};
    \node at (0.2,-0.4) {$\cdot$};
\end{tikzpicture} + M_3,
 \end{align}
 where $f \otimes g$ is a vector from the basis for $N_4$ just displayed,
 and $g' \in S_n$ and $1 \leq i \leq n+1$ are defined from the equation
 $g = (i\:\,i\!+\!1\:\,\cdots\:\,n\!+\!1) g'$.
To see that this linear map is actually a bimodule isomorphism, we construct a
bimodule homorphism in the other direction and show that it is a two-sided inverse of $\theta_4$. Consider the map
 \begin{align}\label{phi}
   \phi: M &\rightarrow N_4,
   &
    \begin{tikzpicture}[anchorbase,scale=1.1]
    \draw[-] (-0.5,-0.2)--(-0.5,0.2)--(0.5,0.2)--(0.5,-0.2)--(-0.5,-0.2);
    \node at (-0,-0) {$\scriptstyle{f}$};
    \draw[-,thick] (-0.4,-0.2)--(-0.4,-0.5);
    \draw[-,thick] (0.4,-0.2)--(0.4,-0.5);
    \draw[-,thick] (-0.4,0.2)--(-0.4,0.5);
    \draw[-,thick] (0.4,0.2)--(0.4,0.5);
    \node at (-0.2,0.35) {$\cdot$};
    \node at (-0,0.35) {$\cdot$};
    \node at (0.2,0.35) {$\cdot$};
    \node at (-0.2,-0.4) {$\cdot$};
    \node at (-0,-0.4) {$\cdot$};
    \node at (0.2,-0.4) {$\cdot$};
  \end{tikzpicture}\otimes \begin{tikzpicture}[anchorbase,scale=1.1]
    \draw[-] (-0.5,-0.2)--(-0.5,0.2)--(0.5,0.2)--(0.5,-0.2)--(-0.5,-0.2);
    \node at (-0,-0) {$\scriptstyle{g}$};
    \draw[-,thick] (-0.4,-0.2)--(-0.4,-0.5);
    \draw[-,thick] (0.4,-0.2)--(0.4,-0.5);
    \draw[-,thick] (-0.4,0.2)--(-0.4,0.5);
    \draw[-,thick] (0.4,0.2)--(0.4,0.5);
    \node at (-0.2,0.35) {$\cdot$};
    \node at (-0,0.35) {$\cdot$};
    \node at (0.2,0.35) {$\cdot$};
    \node at (-0.2,-0.4) {$\cdot$};
    \node at (-0,-0.4) {$\cdot$};
    \node at (0.2,-0.4) {$\cdot$};
  \end{tikzpicture}&\mapsto \begin{tikzpicture}[anchorbase,scale=1.1]
    \draw[-] (-0.5,-0.2)--(-0.5,0.2)--(0.5,0.2)--(0.5,-0.2)--(-0.5,-0.2);
    \node at (-0,-0) {$\scriptstyle{f}$};
    \draw[-,thick] (-0.4,-0.2)--(-0.4,-0.5);
    \draw[-,thick] (0.4,-0.2)--(0.4,-0.5);
    \draw[-,thick] (-0.4,0.2)to(-0.4,0.3)to[out=up,in=up,looseness=2](-0.7,0.3)to(-0.7,-0.5);
    \draw[-,thick] (0.4,0.2)--(0.4,0.5);
    \node at (-0.2,0.35) {$\cdot$};
    \node at (-0,0.35) {$\cdot$};
    \node at (0.2,0.35) {$\cdot$};
    \node at (-0.2,-0.4) {$\cdot$};
    \node at (-0,-0.4) {$\cdot$};
    \node at (0.2,-0.4) {$\cdot$};
  \end{tikzpicture}\otimes \begin{tikzpicture}[anchorbase,scale=1.1]
    \draw[-] (-0.5,-0.2)--(-0.5,0.2)--(0.5,0.2)--(0.5,-0.2)--(-0.5,-0.2);
    \node at (-0,-0) {$\scriptstyle{g}$};
    \draw[-,thick] (-0.4,-0.2)--(-0.4,-0.5);
    \draw[-,thick] (0.4,-0.2)--(0.4,-0.5);
    \draw[-,thick] (-0.4,0.2)--(-0.4,0.5);
    \draw[-,thick] (0.4,0.2)--(0.4,0.5);
    \draw[-,thick] (-0.7,-0.5)--(-0.7,0.5);
    \node at (-0.2,0.35) {$\cdot$};
    \node at (-0,0.35) {$\cdot$};
    \node at (0.2,0.35) {$\cdot$};
    \node at (-0.2,-0.4) {$\cdot$};
    \node at (-0,-0.4) {$\cdot$};
    \node at (0.2,-0.4) {$\cdot$};
  \end{tikzpicture}
 \end{align}
 for $m \geq 0, n \geq 0, f \in 1_{m+1} Par_t 1_n$ and $g \in S_n$.
 It is easy to show that this is a well-defined bimodule homomorphism.
 Moreover, $M_3\subseteq \ker\phi$
 since applying $\phi$ to any basis vector
 $f \otimes g \in M_3$ produces a downward leaf, a
 cap or a downward tree which can be pushed across the tensor to act
 as zero on $\infl^\sharp \sym$.
 Hence, $\phi$ induces a homomorphism $\bar\phi:M_4 / M_3\to N_4$.
 It remains to check that $\bar \phi \circ \theta_4$ and $\theta_4 \circ \bar \phi$ are both identity morphisms, which is straightforward.
\end{proof}

In the next two lemmas, we
finally need to make some explicit calculations with the relations involving the left and right dots in the affine partition category. However, we are working now with $\Par_t$, not with $\APar$, so all
string diagrams from now on
should be interpreted as the canonical images of these morphisms in $\APar$
under the functor $p_t:\APar \rightarrow Par_t$ from \cref{pt}.
We will also use the notation from \cref{jmdef} for an open dot on the interior of a string, meaning the canonical image of this morphism in $\ASym$ under the functor
$p:\ASym \rightarrow \Sym$ from \cref{symp}. This is quite different from an open dot at the end of a string!

\begin{lemma}\label{technical}
  Suppose that $m \geq 0, n \geq 0, f \in 1_{m} Par_t 1_n$ and $g \in S_n$.
  \begin{itemize}
  \item[(i)] The following holds in the bimodule $M = 1_{\mid\,\star}Par_t\otimes_{Par^\sharp} \infl^\sharp \sym $ for $i=1,\dots,n$:
\begin{equation*}
  \begin{tikzpicture}[anchorbase]\draw[-] (-.9,0)--(-.9,0.4)--(.9,0.4)--(.9,0)--(-.9,0);
    \draw[-,thick] (-0.8,0)--(-0.8,-1.1);
    \draw[-,thick] (-0.05,0) to[out=down,in=up] (-.2,-.65) to[out=down,in=left] (-.05,-.8) to[out=right,in=down] (.1,-.65) to[out=up,in=down,looseness=1] (-1.2,.2) to (-1.2,.6);
      \draw[-,thick] (-0.8,0.4)--(-0.8,0.6);
      \draw[-,thick] (0.8,0)--(0.8,-1.1);
    \draw[-,thick] (0.8,0.4)--(0.8,0.6);
    \draw[-,thick] (.1,-0.63)--(-.08,-0.63);
    \closeddot{-0.08,-.63};
    \draw[-,thick] (-0.05,-.8)to(-.05,-1.1);
    \node at (-0.3,0.5) {$\cdot$};
    \node at (-0,0.5) {$\cdot$};
    \node at (0.3,0.5) {$\cdot$};
    \node at (-0.4,-0.7) {$\cdot$};
    \node at (-.55,-0.7) {$\cdot$};
    \node at (-0.7,-0.7) {$\cdot$};
    \node at (0.3,-0.7) {$\cdot$};
    \node at (.5,-0.7) {$\cdot$};
    \node at (0.7,-0.7) {$\cdot$};
    \node at (0,0.2) {$\scriptstyle{f}$};
    \node at (-.05,-1.2) {$\stringlabel{i}$};
    \node at (.8,-1.2) {$\stringlabel{1}$};
    \node at (-.8,-1.2) {$\stringlabel{n}$};
  \end{tikzpicture}\otimes\ \begin{tikzpicture}[anchorbase]\draw[-] (-0.5,-0.2)--(-0.5,0.2)--(0.5,0.2)--(0.5,-0.2)--(-0.5,-0.2);
    \node at (-0,-0) {$\scriptstyle{g}$};
    \draw[-,thick] (-0.4,-0.2)--(-0.4,-0.5);
    \draw[-,thick] (0.4,-0.2)--(0.4,-0.5);
    \draw[-,thick] (-0.4,0.2)--(-0.4,0.5);
    \draw[-,thick] (0.4,0.2)--(0.4,0.5);
    \node at (-0.2,0.35) {$\cdot$};
    \node at (-0,0.35) {$\cdot$};
    \node at (0.2,0.35) {$\cdot$};
    \node at (-0.2,-0.4) {$\cdot$};
    \node at (-0,-0.4) {$\cdot$};
    \node at (0.2,-0.4) {$\cdot$};
  \end{tikzpicture} \equiv
 \begin{tikzpicture}[anchorbase]\draw[-] (-0.5,-0.2)--(-0.5,0.2)--(0.5,0.2)--(0.5,-0.2)--(-0.5,-0.2);
    \node at (-0,-0) {$\scriptstyle{f}$};
    \draw[-,thick] (-0.4,-0.2)--(-0.4,-0.5);
    \draw[-,thick] (0.4,-0.2)--(0.4,-0.5);
    \draw[-,thick] (-0.4,0.2)--(-0.4,0.5);
    \draw[-,thick] (0.4,0.2)--(0.4,0.5);
    \draw[-,thick] (-0.7,0.1)--(-0.7,0.5);
    \opendot{-.7,.1};
      \node at (-0.2,0.35) {$\cdot$};
    \node at (-0,0.35) {$\cdot$};
    \node at (0.2,0.35) {$\cdot$};
    \node at (-0.2,-0.4) {$\cdot$};
    \node at (-0,-0.4) {$\cdot$};
    \node at (0.2,-0.4) {$\cdot$};
 \end{tikzpicture}
 \otimes \begin{tikzpicture}[anchorbase]\draw[-] (-0.5,-0.2)--(-0.5,0.2)--(0.5,0.2)--(0.5,-0.2)--(-0.5,-0.2);
    \node at (-0,-0) {$\scriptstyle{g}$};
    \draw[-,thick] (-0.4,-0.2)--(-0.4,-0.5);
    \draw[-,thick] (0.4,-0.2)--(0.4,-0.5);
    \draw[-,thick] (-0.4,0.2)--(-0.4,0.5);
    \draw[-,thick] (0.4,0.2)--(0.4,0.5);
    \node at (-0.2,0.35) {$\cdot$};
    \node at (-0,0.35) {$\cdot$};
    \node at (0.2,0.35) {$\cdot$};
    \node at (-0.2,-0.4) {$\cdot$};
    \node at (-0,-0.4) {$\cdot$};
    \node at (0.2,-0.4) {$\cdot$};
  \end{tikzpicture} \pmod{M_2}.
\end{equation*}
\item[(ii)] The following holds in the bimodule $M$ for $i=0,1,\dots,n$
  (the case $i=0$ is when there are no strings to the right of the dangling dots):
\begin{equation*}
  \begin{tikzpicture}[anchorbase]\draw[-] (-.9,0)--(-.9,0.4)--(.9,0.4)--(.9,0)--(-.9,0);
    \draw[-,thick] (-1.2,.6) to (-1.2,.2) to[out=down,in=up,looseness=1.7] (-.25,-.9);
    \draw[-,thick] (-0.8,0)--(-0.8,-1.1);
    \draw[-,thick] (0.8,0)--(0.8,-1.1);
    \draw[-,thick] (-0.8,0.4)--(-0.8,0.6);
    \draw[-,thick] (0.8,0.4)--(0.8,0.6);
    \opendot{-.25,-.9};
    \draw[-,thick] (-.07,-0.7)--(-.27,-0.7);
    \closeddot{-0.07,-.7};
    \draw[-,thick] (0.1,0)to(0.1,-1.1);
    \node at (-0.3,0.5) {$\cdot$};
    \node at (-0,0.5) {$\cdot$};
    \node at (0.3,0.5) {$\cdot$};
    \node at (-0.4,-0.7) {$\cdot$};
    \node at (-.55,-0.7) {$\cdot$};
    \node at (-0.7,-0.7) {$\cdot$};
    \node at (0.3,-0.7) {$\cdot$};
    \node at (.5,-0.7) {$\cdot$};
    \node at (0.7,-0.7) {$\cdot$};
    \node at (0,0.2) {$\scriptstyle{f}$};
    \node at (.1,-1.2) {$\stringlabel{i}$};
    \node at (.8,-1.2) {$\stringlabel{1}$};
    \node at (-.8,-1.2) {$\stringlabel{n}$};
  \end{tikzpicture}\otimes\ \begin{tikzpicture}[anchorbase]\draw[-] (-0.5,-0.2)--(-0.5,0.2)--(0.5,0.2)--(0.5,-0.2)--(-0.5,-0.2);
    \node at (-0,-0) {$\scriptstyle{g}$};
    \draw[-,thick] (-0.4,-0.2)--(-0.4,-0.5);
    \draw[-,thick] (0.4,-0.2)--(0.4,-0.5);
    \draw[-,thick] (-0.4,0.2)--(-0.4,0.5);
    \draw[-,thick] (0.4,0.2)--(0.4,0.5);
    \node at (-0.2,0.35) {$\cdot$};
    \node at (-0,0.35) {$\cdot$};
    \node at (0.2,0.35) {$\cdot$};
    \node at (-0.2,-0.4) {$\cdot$};
    \node at (-0,-0.4) {$\cdot$};
    \node at (0.2,-0.4) {$\cdot$};
  \end{tikzpicture} \equiv
  (t-i)\ 
 \begin{tikzpicture}[anchorbase]\draw[-] (-0.5,-0.2)--(-0.5,0.2)--(0.5,0.2)--(0.5,-0.2)--(-0.5,-0.2);
    \node at (-0,-0) {$\scriptstyle{f}$};
    \draw[-,thick] (-0.4,-0.2)--(-0.4,-0.5);
    \draw[-,thick] (0.4,-0.2)--(0.4,-0.5);
    \draw[-,thick] (-0.4,0.2)--(-0.4,0.5);
    \draw[-,thick] (0.4,0.2)--(0.4,0.5);
    \draw[-,thick] (-0.7,0.1)--(-0.7,0.5);
    \opendot{-.7,.1};
      \node at (-0.2,0.35) {$\cdot$};
    \node at (-0,0.35) {$\cdot$};
    \node at (0.2,0.35) {$\cdot$};
    \node at (-0.2,-0.4) {$\cdot$};
    \node at (-0,-0.4) {$\cdot$};
    \node at (0.2,-0.4) {$\cdot$};
 \end{tikzpicture}
 \otimes \begin{tikzpicture}[anchorbase]\draw[-] (-0.5,-0.2)--(-0.5,0.2)--(0.5,0.2)--(0.5,-0.2)--(-0.5,-0.2);
    \node at (-0,-0) {$\scriptstyle{g}$};
    \draw[-,thick] (-0.4,-0.2)--(-0.4,-0.5);
    \draw[-,thick] (0.4,-0.2)--(0.4,-0.5);
    \draw[-,thick] (-0.4,0.2)--(-0.4,0.5);
    \draw[-,thick] (0.4,0.2)--(0.4,0.5);
    \node at (-0.2,0.35) {$\cdot$};
    \node at (-0,0.35) {$\cdot$};
    \node at (0.2,0.35) {$\cdot$};
    \node at (-0.2,-0.4) {$\cdot$};
    \node at (-0,-0.4) {$\cdot$};
    \node at (0.2,-0.4) {$\cdot$};
  \end{tikzpicture} \pmod{M_2}.
\end{equation*}
  \end{itemize}
\end{lemma}

\begin{proof}
  (i) We proceed by induction on $i=1,\dots,n$. The base case $i=1$ follows
  from \cref{bigjoey1}.
  For the induction step, we take $i > 1$ and assume the result has been proved for $i-1$. Then we apply \cref{A1} to commute the left dot past the string to its right.
  This produces a sum of five terms.
  Ordering these terms in the same way as they appear on the right hand side of \cref{A1}, the induction hypothesis can be applied to the first term, to produce the right hand side that we are after. It remains to show that the other four terms lies in $M_2$.
These terms are as follows:
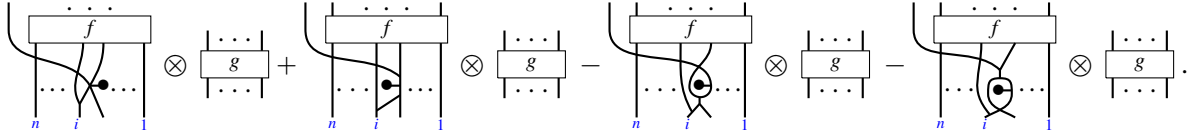
\begin{align*}
  \begin{tikzpicture}[anchorbase,scale=.9]\draw[-] (-.9,0)--(-.9,0.4)--(.9,0.4)--(.9,0)--(-.9,0);
\draw[-,thick] (-1.2,.6) to (-1.2,.2) to [out=down,in=120] (0,-.63) to (.2,-1.1);
\draw[-,thick] (-.15,-1.1) to (-.15,-.9) to (0,-.63) to[out=60,in=down] (.2,0);
\draw[-,thick] (-0.8,0)--(-0.8,-1.1);
\draw[-,thick] (-.15,-.9) to [out=120,in=down] (-.1,0);
  \draw[-,thick] (0,-0.63)--(.2,-0.63);
    \closeddot{0.2,-.63};
      \draw[-,thick] (0.8,0)--(0.8,-1.1);
    \node at (-0.3,0.5) {$\cdot$};
    \node at (-0,0.5) {$\cdot$};
    \node at (0.3,0.5) {$\cdot$};
    \node at (-0.4,-0.7) {$\cdot$};
    \node at (-.55,-0.7) {$\cdot$};
    \node at (-0.7,-0.7) {$\cdot$};
    \node at (0.35,-0.7) {$\cdot$};
    \node at (.5,-0.7) {$\cdot$};
    \node at (0.65,-0.7) {$\cdot$};
    \node at (0,0.2) {$\scriptstyle{f}$};
    \node at (-.2,-1.2) {$\stringlabel{i}$};
    \node at (.8,-1.2) {$\stringlabel{1}$};
    \node at (-.8,-1.2) {$\stringlabel{n}$};
  \end{tikzpicture}\otimes\ \begin{tikzpicture}[anchorbase,scale=.9]\draw[-] (-0.5,-0.2)--(-0.5,0.2)--(0.5,0.2)--(0.5,-0.2)--(-0.5,-0.2);
    \node at (-0,-0) {$\scriptstyle{g}$};
    \draw[-,thick] (-0.4,-0.2)--(-0.4,-0.5);
    \draw[-,thick] (0.4,-0.2)--(0.4,-0.5);
    \draw[-,thick] (-0.4,0.2)--(-0.4,0.5);
    \draw[-,thick] (0.4,0.2)--(0.4,0.5);
    \node at (-0.2,0.35) {$\cdot$};
    \node at (-0,0.35) {$\cdot$};
    \node at (0.2,0.35) {$\cdot$};
    \node at (-0.2,-0.4) {$\cdot$};
    \node at (-0,-0.4) {$\cdot$};
    \node at (0.2,-0.4) {$\cdot$};
  \end{tikzpicture}
+  \begin{tikzpicture}[anchorbase,scale=.9]\draw[-] (-.9,0)--(-.9,0.4)--(.9,0.4)--(.9,0)--(-.9,0);
  \draw[-,thick] (-0.8,0)--(-0.8,-1.1);
  \draw[-,thick] (-.15,0) to (-.15,-1.1);
  \draw[-,thick] (.2,0) to (.2,-1.1);
  \draw[-,thick] (-1.2,.6) to (-1.2,.2) to [out=down,in=140] (0.2,-.5);
  \draw[-,thick] (.2,-.77) to (-.15,-1);
  \draw[-,thick] (0,-0.63)--(.2,-0.63);
    \closeddot{0,-.63};
      \draw[-,thick] (-0.8,0.4)--(-0.8,0.6);
      \draw[-,thick] (0.8,0)--(0.8,-1.1);
    \draw[-,thick] (0.8,0.4)--(0.8,0.6);
    \node at (-0.3,0.5) {$\cdot$};
    \node at (-0,0.5) {$\cdot$};
    \node at (0.3,0.5) {$\cdot$};
    \node at (-0.3,-0.7) {$\cdot$};
    \node at (-.5,-0.7) {$\cdot$};
    \node at (-0.7,-0.7) {$\cdot$};
    \node at (0.35,-0.7) {$\cdot$};
    \node at (.5,-0.7) {$\cdot$};
    \node at (0.65,-0.7) {$\cdot$};
    \node at (0,0.2) {$\scriptstyle{f}$};
    \node at (-.15,-1.2) {$\stringlabel{i}$};
    \node at (.8,-1.2) {$\stringlabel{1}$};
    \node at (-.8,-1.2) {$\stringlabel{n}$};
  \end{tikzpicture}\otimes\ \begin{tikzpicture}[anchorbase,scale=.9]\draw[-] (-0.5,-0.2)--(-0.5,0.2)--(0.5,0.2)--(0.5,-0.2)--(-0.5,-0.2);
    \node at (-0,-0) {$\scriptstyle{g}$};
    \draw[-,thick] (-0.4,-0.2)--(-0.4,-0.5);
    \draw[-,thick] (0.4,-0.2)--(0.4,-0.5);
    \draw[-,thick] (-0.4,0.2)--(-0.4,0.5);
    \draw[-,thick] (0.4,0.2)--(0.4,0.5);
    \node at (-0.2,0.35) {$\cdot$};
    \node at (-0,0.35) {$\cdot$};
    \node at (0.2,0.35) {$\cdot$};
    \node at (-0.2,-0.4) {$\cdot$};
    \node at (-0,-0.4) {$\cdot$};
    \node at (0.2,-0.4) {$\cdot$};
  \end{tikzpicture}
\ -  \begin{tikzpicture}[anchorbase,scale=.9]\draw[-] (-.9,0)--(-.9,0.4)--(.9,0.4)--(.9,0)--(-.9,0);
    \draw[-,thick] (-0.8,0)--(-0.8,-1.1);
    \draw[-,thick] (0.3,0) to[out=down,in=up] (-.02,-.65) to[out=down,in=left] (.145,-.8) to[out=right,in=down] (.3,-.65) to[out=up,in=down,looseness=1] (-1.2,.2) to (-1.2,.6);
      \draw[-,thick] (-0.8,0.4)--(-0.8,0.6);
      \draw[-,thick] (0.8,0)--(0.8,-1.1);
    \draw[-,thick] (0.8,0.4)--(0.8,0.6);
    \draw[-,thick] (.3,-0.63)--(.12,-0.63);
    \closeddot{0.12,-.63};
    \draw[-,thick] (0.13,-.8)to(.13,-.9) to (.3,-1.1);
    \draw[-,thick] (.13,-.9) to (-.05,-1.1);
\draw[-,thick] (-.15,0) to [out=-90,in=120,looseness=1] (0.02,-1.02);
    \node at (-0.3,0.5) {$\cdot$};
    \node at (-0,0.5) {$\cdot$};
    \node at (0.3,0.5) {$\cdot$};
    \node at (-0.25,-0.7) {$\cdot$};
    \node at (-.45,-0.7) {$\cdot$};
    \node at (-0.65,-0.7) {$\cdot$};
    \node at (0.4,-0.7) {$\cdot$};
    \node at (.55,-0.7) {$\cdot$};
    \node at (0.7,-0.7) {$\cdot$};
    \node at (0,0.2) {$\scriptstyle{f}$};
    \node at (-.05,-1.2) {$\stringlabel{i}$};
    \node at (.8,-1.2) {$\stringlabel{1}$};
    \node at (-.8,-1.2) {$\stringlabel{n}$};
  \end{tikzpicture}\otimes\ \begin{tikzpicture}[anchorbase,scale=.9]\draw[-] (-0.5,-0.2)--(-0.5,0.2)--(0.5,0.2)--(0.5,-0.2)--(-0.5,-0.2);
    \node at (-0,-0) {$\scriptstyle{g}$};
    \draw[-,thick] (-0.4,-0.2)--(-0.4,-0.5);
    \draw[-,thick] (0.4,-0.2)--(0.4,-0.5);
    \draw[-,thick] (-0.4,0.2)--(-0.4,0.5);
    \draw[-,thick] (0.4,0.2)--(0.4,0.5);
    \node at (-0.2,0.35) {$\cdot$};
    \node at (-0,0.35) {$\cdot$};
    \node at (0.2,0.35) {$\cdot$};
    \node at (-0.2,-0.4) {$\cdot$};
    \node at (-0,-0.4) {$\cdot$};
    \node at (0.2,-0.4) {$\cdot$};
  \end{tikzpicture}
\ -  \begin{tikzpicture}[anchorbase,scale=.9]\draw[-] (-.9,0)--(-.9,0.4)--(.9,0.4)--(.9,0)--(-.9,0);
    \draw[-,thick] (-0.8,0)--(-0.8,-1.1);
      \draw[-,thick] (-0.8,0.4)--(-0.8,0.6);
      \draw[-,thick] (0.8,0)--(0.8,-1.1);
      \draw[-,thick] (-0.2,0) to [out=down,in=130,looseness=1.2] (-0.08,-1.05);
    \draw[-,thick] (0.8,0.4)--(0.8,0.6);
    \draw[-,thick] (.25,-0.7)--(.05,-0.7);
    \closeddot{0.05,-.7};
    \draw[-,thick] (-1.2,.6) to (-1.2,.2) to [out=down,in=135] (.075,-.4) to (.3,0);
    \draw[-,thick] (-.15,-1.1) to[out=30,in=down,looseness=1] (.25,-.75) to [out=up,in=0,looseness=1.5] (.075,-.52) to [out=left,in=up,looseness=1.5] (-.1,-.75) to [out=down,in=150,looseness=1] (.3,-1.1);
\draw[-,thick] (.075,-.52) to (.075,-.4);
    \node at (-0.3,0.5) {$\cdot$};
    \node at (-0,0.5) {$\cdot$};
    \node at (0.3,0.5) {$\cdot$};
    \node at (-0.35,-0.7) {$\cdot$};
    \node at (-.525,-0.7) {$\cdot$};
    \node at (-0.7,-0.7) {$\cdot$};
    \node at (0.4,-0.7) {$\cdot$};
    \node at (.55,-0.7) {$\cdot$};
    \node at (0.7,-0.7) {$\cdot$};
    \node at (0,0.2) {$\scriptstyle{f}$};
    \node at (-.15,-1.2) {$\stringlabel{i}$};
    \node at (.8,-1.2) {$\stringlabel{1}$};
    \node at (-.8,-1.2) {$\stringlabel{n}$};
  \end{tikzpicture}\otimes\ \begin{tikzpicture}[anchorbase,scale=.9]\draw[-] (-0.5,-0.2)--(-0.5,0.2)--(0.5,0.2)--(0.5,-0.2)--(-0.5,-0.2);
    \node at (-0,-0) {$\scriptstyle{g}$};
    \draw[-,thick] (-0.4,-0.2)--(-0.4,-0.5);
    \draw[-,thick] (0.4,-0.2)--(0.4,-0.5);
    \draw[-,thick] (-0.4,0.2)--(-0.4,0.5);
    \draw[-,thick] (0.4,0.2)--(0.4,0.5);
    \node at (-0.2,0.35) {$\cdot$};
    \node at (-0,0.35) {$\cdot$};
    \node at (0.2,0.35) {$\cdot$};
    \node at (-0.2,-0.4) {$\cdot$};
    \node at (-0,-0.4) {$\cdot$};
    \node at (0.2,-0.4) {$\cdot$};
  \end{tikzpicture}\ .
\end{align*}
The second and third terms here are zero already in $M$ because,
in both of them, the diagram to the left of the tensor is
equivalent to a diagram with a downward tree at the bottom.
It remains to show that the first and the fourth terms lie in $M_2$.
For the fourth term, we note that
$$
 
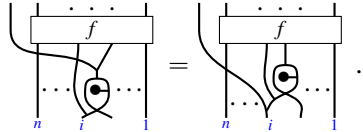
\begin{tikzpicture}[anchorbase,scale=.9]\draw[-] (-.9,0)--(-.9,0.4)--(.9,0.4)--(.9,0)--(-.9,0);
    \draw[-,thick] (-0.8,0)--(-0.8,-1.1);
      \draw[-,thick] (-0.8,0.4)--(-0.8,0.6);
      \draw[-,thick] (0.8,0)--(0.8,-1.1);
      \draw[-,thick] (-0.2,0) to [out=down,in=130,looseness=1.2] (-0.08,-1.05);
    \draw[-,thick] (0.8,0.4)--(0.8,0.6);
    \draw[-,thick] (.25,-0.7)--(.05,-0.7);
    \closeddot{0.05,-.7};
    \draw[-,thick] (-1.2,.6) to (-1.2,.2) to [out=down,in=135] (.075,-.4) to (.3,0);
    \draw[-,thick] (-.15,-1.1) to[out=30,in=down,looseness=1] (.25,-.75) to [out=up,in=0,looseness=1.5] (.075,-.52) to [out=left,in=up,looseness=1.5] (-.1,-.75) to [out=down,in=150,looseness=1] (.3,-1.1);
\draw[-,thick] (.075,-.52) to (.075,-.4);
    \node at (-0.3,0.5) {$\cdot$};
    \node at (-0,0.5) {$\cdot$};
    \node at (0.3,0.5) {$\cdot$};
    \node at (-0.35,-0.7) {$\cdot$};
    \node at (-.525,-0.7) {$\cdot$};
    \node at (-0.7,-0.7) {$\cdot$};
    \node at (0.4,-0.7) {$\cdot$};
    \node at (.55,-0.7) {$\cdot$};
    \node at (0.7,-0.7) {$\cdot$};
    \node at (0,0.2) {$\scriptstyle{f}$};
    \node at (-.15,-1.2) {$\stringlabel{i}$};
    \node at (.8,-1.2) {$\stringlabel{1}$};
    \node at (-.8,-1.2) {$\stringlabel{n}$};
 \end{tikzpicture}
 =
  \begin{tikzpicture}[anchorbase,scale=.9]\draw[-] (-.9,0)--(-.9,0.4)--(.9,0.4)--(.9,0)--(-.9,0);
    \draw[-,thick] (-0.8,0)--(-0.8,-1.1);
      \draw[-,thick] (-0.8,0.4)--(-0.8,0.6);
      \draw[-,thick] (0.8,0)--(0.8,-1.1);
      \draw[-,thick] (-0.2,0) to [out=down,in=130,looseness=1.2] (-0.08,-.83);
    \draw[-,thick] (0.8,0.4)--(0.8,0.6);
    \draw[-,thick] (.25,-0.5)--(.05,-0.5);
    \closeddot{0.05,-.5};
    \draw[-,thick] (-1.2,.6) to (-1.2,.2) to [out=down,in=up] (-.2,-1.1);
    \draw[-,thick] (-.21,-1) to [out=50,in=-120] (-.15,-.9) to[out=60,in=down,looseness=1] (.25,-.55) to [out=up,in=0,looseness=1.5] (.075,-.32) to [out=left,in=up,looseness=1.5] (-.1,-.55) to [out=down,in=120,looseness=1] (.3,-.9) to [out=-60,in=up,looseness=1] (.3,-1.1);
\draw[-,thick] (.075,-.32) to (.075,0);
    \node at (-0.3,0.5) {$\cdot$};
    \node at (-0,0.5) {$\cdot$};
    \node at (0.3,0.5) {$\cdot$};
    \node at (-0.35,-0.9) {$\cdot$};
    \node at (-.525,-0.9) {$\cdot$};
    \node at (-0.7,-0.9) {$\cdot$};
    \node at (0.4,-0.7) {$\cdot$};
    \node at (.55,-0.7) {$\cdot$};
    \node at (0.7,-0.7) {$\cdot$};
    \node at (0,0.2) {$\scriptstyle{f}$};
    \node at (-.15,-1.2) {$\stringlabel{i}$};
    \node at (.8,-1.2) {$\stringlabel{1}$};
    \node at (-.8,-1.2) {$\stringlabel{n}$};
  \end{tikzpicture}\ .
  $$
  The left dot can now be absorbed into the morphism $f$, changing it to some other morphism $f'$. The result is a linear combination of morphisms in all of which the top left vertex is connected to the bottom edge, so that the connected component containing this vertex is either a tree or a trunk, and it belongs to the sub-bimodule $M_2$.
  The reason the first term lies in $M_2$ is very similar, one just needs to rewrite
  the right crossing using \cref{crazycrossing}, and then it is easy to see that the
top left vertex is again connected to the bottom edge.

\vspace{1mm}
  \noindent(ii)
  Again we proceed by induction.
  The  base case $i=0$ follows
  from \cref{bigjoey1} using that $T = t1_\one$.
  For the induction step, we consider some $i > 0$. Then we apply \cref{A3} to commute the right dot past the string to its right. This produces a sum of five terms.
  This time, the induction hypothesis can be applied to the first term, to
  produce the vector that we are after but scaled by $(t-i+1)$ rather than
  the desired $(t-i)$. The remaining four terms are as follows:
  
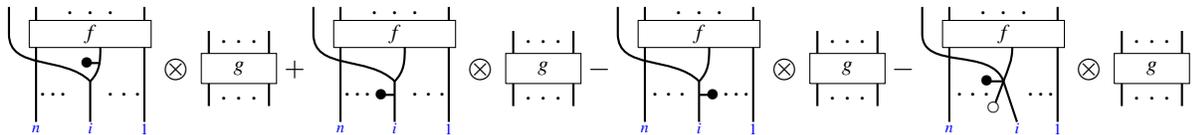
\begin{align*}
  \begin{tikzpicture}[anchorbase,scale=.9]\draw[-] (-.9,0)--(-.9,0.4)--(.9,0.4)--(.9,0)--(-.9,0);
    \draw[-,thick] (-0.8,0)--(-0.8,-1.1);
    \draw[-,thick] (0,-1.1) to (0,-.5) to[out=135,in=down,looseness=1] (-1.2,.2) to (-1.2,.6);
    \draw[-,thick] (0,-.5) to [out=45,in=down,looseness=1] (0.15,0);
      \draw[-,thick] (-0.8,0.4)--(-0.8,0.6);
      \draw[-,thick] (0.8,0)--(0.8,-1.1);
    \draw[-,thick] (0.8,0.4)--(0.8,0.6);
    \draw[-,thick] (0.15,-0.23)--(-.05,-0.23);
    \closeddot{-0.05,-.23};
    \node at (-0.3,0.5) {$\cdot$};
    \node at (-0,0.5) {$\cdot$};
    \node at (0.3,0.5) {$\cdot$};
    \node at (-0.4,-0.7) {$\cdot$};
    \node at (-.55,-0.7) {$\cdot$};
    \node at (-0.7,-0.7) {$\cdot$};
    \node at (0.3,-0.7) {$\cdot$};
    \node at (.5,-0.7) {$\cdot$};
    \node at (0.7,-0.7) {$\cdot$};
    \node at (0,0.2) {$\scriptstyle{f}$};
    \node at (0,-1.2) {$\stringlabel{i}$};
    \node at (.8,-1.2) {$\stringlabel{1}$};
    \node at (-.8,-1.2) {$\stringlabel{n}$};
  \end{tikzpicture}\otimes\ \begin{tikzpicture}[anchorbase]\draw[-] (-0.5,-0.2)--(-0.5,0.2)--(0.5,0.2)--(0.5,-0.2)--(-0.5,-0.2);
    \node at (-0,-0) {$\scriptstyle{g}$};
    \draw[-,thick] (-0.4,-0.2)--(-0.4,-0.5);
    \draw[-,thick] (0.4,-0.2)--(0.4,-0.5);
    \draw[-,thick] (-0.4,0.2)--(-0.4,0.5);
    \draw[-,thick] (0.4,0.2)--(0.4,0.5);
    \node at (-0.2,0.35) {$\cdot$};
    \node at (-0,0.35) {$\cdot$};
    \node at (0.2,0.35) {$\cdot$};
    \node at (-0.2,-0.4) {$\cdot$};
    \node at (-0,-0.4) {$\cdot$};
    \node at (0.2,-0.4) {$\cdot$};
  \end{tikzpicture}+
    \begin{tikzpicture}[anchorbase,scale=.9]\draw[-] (-.9,0)--(-.9,0.4)--(.9,0.4)--(.9,0)--(-.9,0);
    \draw[-,thick] (-0.8,0)--(-0.8,-1.1);
    \draw[-,thick] (0,-1.1) to (0,-.5) to[out=135,in=down,looseness=1] (-1.2,.2) to (-1.2,.6);
    \draw[-,thick] (0,-.5) to [out=45,in=down,looseness=1] (0.15,0);
      \draw[-,thick] (-0.8,0.4)--(-0.8,0.6);
      \draw[-,thick] (0.8,0)--(0.8,-1.1);
    \draw[-,thick] (0.8,0.4)--(0.8,0.6);
    \draw[-,thick] (0,-0.7)--(-.2,-0.7);
    \closeddot{-0.2,-.7};
    \node at (-0.3,0.5) {$\cdot$};
    \node at (-0,0.5) {$\cdot$};
    \node at (0.3,0.5) {$\cdot$};
    \node at (-0.4,-0.7) {$\cdot$};
    \node at (-.55,-0.7) {$\cdot$};
    \node at (-0.7,-0.7) {$\cdot$};
    \node at (0.3,-0.7) {$\cdot$};
    \node at (.5,-0.7) {$\cdot$};
    \node at (0.7,-0.7) {$\cdot$};
    \node at (0,0.2) {$\scriptstyle{f}$};
    \node at (0,-1.2) {$\stringlabel{i}$};
    \node at (.8,-1.2) {$\stringlabel{1}$};
    \node at (-.8,-1.2) {$\stringlabel{n}$};
  \end{tikzpicture}\otimes\ \begin{tikzpicture}[anchorbase]\draw[-] (-0.5,-0.2)--(-0.5,0.2)--(0.5,0.2)--(0.5,-0.2)--(-0.5,-0.2);
    \node at (-0,-0) {$\scriptstyle{g}$};
    \draw[-,thick] (-0.4,-0.2)--(-0.4,-0.5);
    \draw[-,thick] (0.4,-0.2)--(0.4,-0.5);
    \draw[-,thick] (-0.4,0.2)--(-0.4,0.5);
    \draw[-,thick] (0.4,0.2)--(0.4,0.5);
    \node at (-0.2,0.35) {$\cdot$};
    \node at (-0,0.35) {$\cdot$};
    \node at (0.2,0.35) {$\cdot$};
    \node at (-0.2,-0.4) {$\cdot$};
    \node at (-0,-0.4) {$\cdot$};
    \node at (0.2,-0.4) {$\cdot$};
    \end{tikzpicture}-
        \begin{tikzpicture}[anchorbase,scale=.9]\draw[-] (-.9,0)--(-.9,0.4)--(.9,0.4)--(.9,0)--(-.9,0);
    \draw[-,thick] (-0.8,0)--(-0.8,-1.1);
    \draw[-,thick] (0,-1.1) to (0,-.5) to[out=135,in=down,looseness=1] (-1.2,.2) to (-1.2,.6);
    \draw[-,thick] (0,-.5) to [out=45,in=down,looseness=1] (0.15,0);
      \draw[-,thick] (-0.8,0.4)--(-0.8,0.6);
      \draw[-,thick] (0.8,0)--(0.8,-1.1);
    \draw[-,thick] (0.8,0.4)--(0.8,0.6);
    \draw[-,thick] (0,-0.7)--(.2,-0.7);
    \closeddot{0.2,-.7};
    \node at (-0.3,0.5) {$\cdot$};
    \node at (-0,0.5) {$\cdot$};
    \node at (0.3,0.5) {$\cdot$};
    \node at (-0.25,-0.7) {$\cdot$};
    \node at (-.45,-0.7) {$\cdot$};
    \node at (-0.65,-0.7) {$\cdot$};
    \node at (0.4,-0.7) {$\cdot$};
    \node at (.55,-0.7) {$\cdot$};
    \node at (0.7,-0.7) {$\cdot$};
    \node at (0,0.2) {$\scriptstyle{f}$};
    \node at (0,-1.2) {$\stringlabel{i}$};
    \node at (.8,-1.2) {$\stringlabel{1}$};
    \node at (-.8,-1.2) {$\stringlabel{n}$};
  \end{tikzpicture}\otimes\ \begin{tikzpicture}[anchorbase]\draw[-] (-0.5,-0.2)--(-0.5,0.2)--(0.5,0.2)--(0.5,-0.2)--(-0.5,-0.2);
    \node at (-0,-0) {$\scriptstyle{g}$};
    \draw[-,thick] (-0.4,-0.2)--(-0.4,-0.5);
    \draw[-,thick] (0.4,-0.2)--(0.4,-0.5);
    \draw[-,thick] (-0.4,0.2)--(-0.4,0.5);
    \draw[-,thick] (0.4,0.2)--(0.4,0.5);
    \node at (-0.2,0.35) {$\cdot$};
    \node at (-0,0.35) {$\cdot$};
    \node at (0.2,0.35) {$\cdot$};
    \node at (-0.2,-0.4) {$\cdot$};
    \node at (-0,-0.4) {$\cdot$};
    \node at (0.2,-0.4) {$\cdot$};
  \end{tikzpicture}
-        \begin{tikzpicture}[anchorbase,scale=.9]\draw[-] (-.9,0)--(-.9,0.4)--(.9,0.4)--(.9,0)--(-.9,0);
    \draw[-,thick] (-0.8,0)--(-0.8,-1.1);
    \draw[-,thick] (.2,-1.1) to (0,-.5) to[out=110,in=down,looseness=1] (-1.2,.2) to [out=up,in=down] (-1.2,.6);
    \draw[-,thick] (-.15,-.9) to[out=70,in=-130,looseness=0] (0,-.5) to [out=70,in=down,looseness=1] (0.13,0);
\opendot{-.15,-.9};
    \draw[-,thick] (-0.8,0.4)--(-0.8,0.6);
      \draw[-,thick] (0.8,0)--(0.8,-1.1);
    \draw[-,thick] (0.8,0.4)--(0.8,0.6);
    \draw[-,thick] (0,-0.5)--(-.25,-0.5);
    \closeddot{-0.25,-.5};
    \node at (-0.3,0.5) {$\cdot$};
    \node at (-0,0.5) {$\cdot$};
    \node at (0.3,0.5) {$\cdot$};
    \node at (-0.25,-0.7) {$\cdot$};
    \node at (-.45,-0.7) {$\cdot$};
    \node at (-0.65,-0.7) {$\cdot$};
    \node at (0.4,-0.7) {$\cdot$};
    \node at (.55,-0.7) {$\cdot$};
    \node at (0.7,-0.7) {$\cdot$};
    \node at (0,0.2) {$\scriptstyle{f}$};
    \node at (0.2,-1.2) {$\stringlabel{i}$};
    \node at (.8,-1.2) {$\stringlabel{1}$};
    \node at (-.8,-1.2) {$\stringlabel{n}$};
  \end{tikzpicture}\otimes\ \begin{tikzpicture}[anchorbase]\draw[-] (-0.5,-0.2)--(-0.5,0.2)--(0.5,0.2)--(0.5,-0.2)--(-0.5,-0.2);
    \node at (-0,-0) {$\scriptstyle{g}$};
    \draw[-,thick] (-0.4,-0.2)--(-0.4,-0.5);
    \draw[-,thick] (0.4,-0.2)--(0.4,-0.5);
    \draw[-,thick] (-0.4,0.2)--(-0.4,0.5);
    \draw[-,thick] (0.4,0.2)--(0.4,0.5);
    \node at (-0.2,0.35) {$\cdot$};
    \node at (-0,0.35) {$\cdot$};
    \node at (0.2,0.35) {$\cdot$};
    \node at (-0.2,-0.4) {$\cdot$};
    \node at (-0,-0.4) {$\cdot$};
    \node at (0.2,-0.4) {$\cdot$};
  \end{tikzpicture}\ .
  \end{align*}
  In the first term here, the left dot is some morphism in $Par_t$, which has the effect of changing $f$ to some other morphism $f'$.
  After doing that, it is clear that the top left vertex is still connected to the bottom edge, so the first term lies in $M_2$. For the second and third terms, the left and right dots can be commuted across the tensor using \cref{interest1}, then again we see that these morphisms lie in $M_2$ since the top left vertex
  is connected to the bottom edge again.
  For the final term, we note
   that
 $$
 \begin{tikzpicture}[anchorbase]
\draw[-,thick](-0.125,-0.25)--(0.25,0.5);
\draw[-,thick](0.25,-0.5)--(-0.25,0.5);
\draw[-,thick](0,0)--(-0.25,0);
\opendot{-.125,-.25};
\closeddot{-0.25,0};
 \end{tikzpicture}
 \stackrel{\cref{june}}{=} \begin{tikzpicture}[anchorbase]
   \draw[-,thick] (-.25,.5) to [out=-90,in=90] (.25,0.15) to [out=-90,in=60] (-.175,-.35);
   \draw[-,thick] (.25,.5) to [out=-90,in=90] (-.25,0.15) to [out=-90,in=90] (.3,-.5);
\draw[-,thick] (0.02,-.165) to (.25,-.165);
   \opendot{-.175,-.35};
\closeddot{0.25,-.165};
  \end{tikzpicture}
 \stackrel{\cref{crazycrossing}}{=}
 \begin{tikzpicture}[anchorbase]
   \draw[-,thick] (-.25,.5) to [out=-90,in=90] (.25,0.15) to [out=-90,in=60] (-.175,-.35);
   \draw[-,thick] (.25,.5) to [out=-90,in=90] (-.25,0.15) to [out=-90,in=90] (.3,-.5);
\draw[-,thick] (0.43,-.165) to (.2,-.165);
\draw[-,thick] (0.24,0.05) to[out=-30,in=30,looseness=1.75] (0.26,-.38);
\opendot{-.175,-.35};
\closeddot{0.2,-.165};
  \end{tikzpicture}
\,\: \substack{\cref{relssym2}\\{\displaystyle =}\\\cref{relsfrob2}}\:\,
 \begin{tikzpicture}[anchorbase]
\draw[-,thick] (-.25,0.5) to[out=-80,in=up] (.3,0) to [out=down,in=0,looseness=1] (0,-0.25) to [out=180,in=down,looseness=1] (-.3,0) to [out=up,in=-100] (.25,.5);
\draw[-,thick] (.3,0)--(0.03,0);
\draw[-,thick] (0,-.5)--(0,-.25);
\closeddot{0.03,0};
\end{tikzpicture}\ .
 $$
 Making this substitution in the middle of the picture reveals that the final term is exactly the expression studied in (i). On applying the conclusion of (i), we deduce that it  contributes exactly the needed correction to complete the proof.
  \end{proof}

\begin{lemma}\label{xyz}
  Consider the bimodule endomorphisms $\rho$ and $\lambda$
  of $M$ defined on
   $1_{m+1} Par_t 1_n \otimes \kk S_n$ by the left action of
   $x_{m+1}^R \otimes 1_n$ and $x_{m+1}^L\otimes 1_n$, respectively,
  for each $m,n \geq 0$.
  These endomorphisms preserve each of the sub-bimodules $M_i\:(i=1,2,3,4)$,
  hence, $\rho$ and $\lambda$ induce endomorphisms also denoted $\rho$ and $\lambda$
  of each of the subquotients
  $M_i / M_{i-1}$.
  Moreover, for each $i$, the isomorphism $\theta_i$
  from \cref{filtration} satisfies
  \begin{align}\label{maincommutationfact}
    \theta_i\circ \rho_i &= \rho \circ \theta_i,&
    \theta_i\circ \lambda_i &= \lambda \circ \theta_i,
  \end{align}
  where $\rho_i, \lambda_i:N_i \rightarrow N_i$ are defined as follows:
  \begin{itemize}
  \item[(i)]
    $\rho_1$ and $\lambda_1$ are the bimodule endomorphisms of $N_1$
    defined on the subspace $1_{m} Par_t 1_{n-1} \otimes \kk S_n$
    by the left actions of $(t-n+1) 1_{m} \otimes 1_n$
  and  $1_{m} \otimes x_n$,
    respectively, for each $m\geq 0, n > 0$.
  \item[(ii)] $\rho_2$ and $\lambda_2$ are the bimodule endomorphisms of $N_2$
    defined on $1_{m} Par_t 1_n \otimes \kk S_n \otimes \kk S_n$
    by the right action of
    $1_{n} \otimes x_n\otimes 1_n$ and the left action of
    $1_{m} \otimes 1_n \otimes x_n$, respectively.
  \item[(iii)] $\rho_3$ and $\lambda_3$ are both equal to the bimodule endomorphism
    of $N_3$
    defined on $1_{m} Par_t 1_n \otimes \kk S_n$
    by multiplication by $(t-n)$.
     \item[(iv)] $\rho_4$ and $\lambda_4$ are the bimodule endomorphisms of $N_4$
    defined on $1_{m} Par_t 1_{n+1} \otimes \kk S_{n+1}$
    by the right actions of  $1_{n+1} \otimes x_{n+1}$ and $(t-n) 1_{n+1} \otimes 1_{n+1}$,
    respectively.
    \end{itemize}
   \end{lemma}

\begin{proof}
(i)  Recall the definition of $\theta_1$ from \cref{theta1}.
  Take a vector $f \otimes g$ in the basis for $N_1$ from \cref{N1basis}.
  By \cref{interest1},
  we have that $x_n^L \equiv x_n\pmod{K_n}$ and $x_n^R \equiv (t-n+1) 1_n
  \pmod{K_n}$
  where $K_n$ is the two sided ideal of $1_n Par_t 1_n$ from \cref{playoff}.
  Since the strictly downward partition diagrams which generate $K^+$
  are zero on $\infl^\sharp \sym$, it follows that
  \begin{align*}
\rho(\theta_1(f\otimes g)) &= 
 = \theta_1(\lambda_1(f \otimes g)).
  \end{align*}
  This shows as the same time that $\rho$ and $\lambda$ both leave $M_1$ invariant.

\vspace{1mm}
\noindent
(ii)
Recall the definition of $\theta_2$ from \cref{theta2}.
The argument for $\lambda$ is similar to in (i). It follows from the calculation
\begin{align*}
\lambda(\theta_2(f\otimes g \otimes h))
&=\!\!\!
%
\right)
\\&=\theta_2(\lambda_2(f\otimes g \otimes h)),
\end{align*}
where $f \otimes g \otimes h$ is one of the basis vectors for $N_2$ from \cref{N2basis}.
For $\rho$, we instead have that
\begin{align*}
  \rho(\theta_2(f\otimes g \otimes h)) &=
%
\right)
=
\theta_2(\rho_2(f\otimes g\otimes h)).
\end{align*}

\vspace{1mm}
\noindent
(iii)
Recall the definition of $\theta_3$ from \cref{theta3}.
Note that $\rho = \lambda$ by the third relation from \cref{rightslide}. For $\rho$,
we need to show that $\rho(\theta_3(f \otimes g)) = (t-n)\theta_3(f \otimes g)$
for any
basis vector $f \otimes g \in 1_m Par_t 1_n \otimes \kk S_n \subset N_3$
from \cref{N3basis}. This follows from \cref{technical}(ii) taking $i=n$.

\vspace{1mm}
\noindent
(iv)
Instead of working with $\theta_4$ from \cref{theta4},
it is easier to use the inverse map $\bar\phi$ induced by the homomorphism $\phi:M\rightarrow N_4$ from \cref{phi}.
We need to show that $\phi \circ \rho =\rho_4 \circ \phi$.
This follows from the following calculations for $f \otimes g \in 1_{m+1} Par_t 1_n \otimes \kk S_n$ and $m, n\geq 0$:
\begin{align*}
\phi(\rho(f \otimes g)) &= \phi\left(\begin{tikzpicture}[scale=1.1,anchorbase]
  \draw[-] (-0.5,-0.2)--(-0.5,0.2)--(0.5,0.2)--(0.5,-0.2)--(-0.5,-0.2);
    \node at (-0,-0) {$\scriptstyle{f}$};
    \draw[-,thick] (-0.4,-0.2)--(-0.4,-0.5);
    \draw[-,thick] (0.4,-0.2)--(0.4,-0.5);
    \draw[-,thick] (-0.4,0.2)--(-0.4,0.5);
    \draw[-,thick] (0.4,0.2)--(0.4,0.5);
    \node at (-0,0.35) {$\cdot$};
    \node at (0.2,0.35) {$\cdot$};
    \node at (-0.2,-0.35) {$\cdot$};
    \node at (-0,-0.35) {$\cdot$};
    \node at (0.2,-0.35) {$\cdot$};
    \draw[-,thick] (-0.4,0.35)--(-0.25,0.35);
    \closeddot{-0.25,0.35};
  \end{tikzpicture}\otimes\begin{tikzpicture}[scale=1.1,anchorbase]
  \draw[-] (-0.5,-0.2)--(-0.5,0.2)--(0.5,0.2)--(0.5,-0.2)--(-0.5,-0.2);
    \node at (-0,-0) {$\scriptstyle{g}$};
    \draw[-,thick] (-0.4,-0.2)--(-0.4,-0.5);
    \draw[-,thick] (0.4,-0.2)--(0.4,-0.5);
    \draw[-,thick] (-0.4,0.2)--(-0.4,0.5);
    \draw[-,thick] (0.4,0.2)--(0.4,0.5);
    \node at (-0.2,0.35) {$\cdot$};
    \node at (-0,0.35) {$\cdot$};
    \node at (0.2,0.35) {$\cdot$};
    \node at (-0.2,-0.35) {$\cdot$};
    \node at (-0,-0.35) {$\cdot$};
    \node at (0.2,-0.35) {$\cdot$};
  \end{tikzpicture}\right)\ =\ \begin{tikzpicture}[scale=1.1,anchorbase]
  \draw[-] (-0.5,-0.2)--(-0.5,0.2)--(0.5,0.2)--(0.5,-0.2)--(-0.5,-0.2);
    \node at (-0,-0) {$\scriptstyle{f}$};
    \draw[-,thick] (-0.4,-0.2)--(-0.4,-0.5);
    \draw[-,thick] (0.4,-0.2)--(0.4,-0.5);
    \draw[-,thick] (-0.4,0.2)--(-0.4,0.3)to[out=up,in=up,looseness=2](-0.7,0.3)to(-0.7,-0.5);
    \draw[-,thick] (0.4,0.2)--(0.4,0.5);
    \node at (-0,0.35) {$\cdot$};
    \node at (0.2,0.35) {$\cdot$};
    \node at (-0.2,-0.35) {$\cdot$};
    \node at (-0,-0.35) {$\cdot$};
    \node at (0.2,-0.35) {$\cdot$};
    \draw[-,thick] (-0.4,0.3)--(-0.25,0.3);
    \closeddot{-0.25,0.3};
  \end{tikzpicture}\otimes\!\begin{tikzpicture}[scale=1.1,anchorbase]
  \draw[-] (-0.5,-0.2)--(-0.5,0.2)--(0.5,0.2)--(0.5,-0.2)--(-0.5,-0.2);
    \node at (-0,-0) {$\scriptstyle{g}$};
    \draw[-,thick] (-0.4,-0.2)--(-0.4,-0.5);
    \draw[-,thick] (0.4,-0.2)--(0.4,-0.5);
    \draw[-,thick] (-0.4,0.2)--(-0.4,0.5);
    \draw[-,thick] (0.4,0.2)--(0.4,0.5);
    \draw[-,thick] (-0.7,-0.5)--(-0.7,0.5);
    \node at (-0.2,0.35) {$\cdot$};
    \node at (-0,0.35) {$\cdot$};
    \node at (0.2,0.35) {$\cdot$};
    \node at (-0.2,-0.35) {$\cdot$};
    \node at (-0,-0.35) {$\cdot$};
    \node at (0.2,-0.35) {$\cdot$};
  \end{tikzpicture}= \begin{tikzpicture}[scale=1.1,anchorbase]
  \draw[-] (-0.5,-0.2)--(-0.5,0.2)--(0.5,0.2)--(0.5,-0.2)--(-0.5,-0.2);
    \node at (-0,-0) {$\scriptstyle{f}$};
    \draw[-,thick] (-0.4,-0.2)--(-0.4,-0.5);
    \draw[-,thick] (0.4,-0.2)--(0.4,-0.5);
    \draw[-,thick] (-0.4,0.2)--(-0.4,0.3)to[out=up,in=up,looseness=2](-0.7,0.3)to(-0.7,-0.5);
    \draw[-,thick] (0.4,0.2)--(0.4,0.5);
    \node at (-0,0.35) {$\cdot$};
    \node at (-0.2,0.35) {$\cdot$};
    \node at (0.2,0.35) {$\cdot$};
    \node at (-0.2,-0.35) {$\cdot$};
    \node at (-0,-0.35) {$\cdot$};
    \node at (0.2,-0.35) {$\cdot$};
    \draw[-,thick] (-0.7,-0.3)--(-0.85,-0.3);
    \closeddot{-0.85,-0.3};
  \end{tikzpicture}\otimes\!\begin{tikzpicture}[scale=1.1,anchorbase]
  \draw[-] (-0.5,-0.2)--(-0.5,0.2)--(0.5,0.2)--(0.5,-0.2)--(-0.5,-0.2);
    \node at (-0,-0) {$\scriptstyle{g}$};
    \draw[-,thick] (-0.4,-0.2)--(-0.4,-0.5);
    \draw[-,thick] (0.4,-0.2)--(0.4,-0.5);
    \draw[-,thick] (-0.4,0.2)--(-0.4,0.5);
    \draw[-,thick] (0.4,0.2)--(0.4,0.5);
    \draw[-,thick] (-0.7,-0.5)--(-0.7,0.5);
    \node at (-0.2,0.35) {$\cdot$};
    \node at (-0,0.35) {$\cdot$};
    \node at (0.2,0.35) {$\cdot$};
    \node at (-0.2,-0.35) {$\cdot$};
    \node at (-0,-0.35) {$\cdot$};
    \node at (0.2,-0.35) {$\cdot$};
    \end{tikzpicture} \\&= \begin{tikzpicture}[scale=1.1,anchorbase]
  \draw[-] (-0.5,-0.2)--(-0.5,0.2)--(0.5,0.2)--(0.5,-0.2)--(-0.5,-0.2);
    \node at (-0,-0) {$\scriptstyle{f}$};
    \draw[-,thick] (-0.4,-0.2)--(-0.4,-0.5);
    \draw[-,thick] (0.4,-0.2)--(0.4,-0.5);
    \draw[-,thick] (-0.4,0.2)--(-0.4,0.3)to[out=up,in=up,looseness=2](-0.7,0.3)to(-0.7,-0.5);
    \draw[-,thick] (0.4,0.2)--(0.4,0.5);
    \node at (-0,0.35) {$\cdot$};
    \node at (-0.2,0.35) {$\cdot$};
    \node at (0.2,0.35) {$\cdot$};
    \node at (-0.2,-0.35) {$\cdot$};
    \node at (-0,-0.35) {$\cdot$};
    \node at (0.2,-0.35) {$\cdot$};
  \end{tikzpicture}\otimes\begin{tikzpicture}[scale=1.1,anchorbase]
  \draw[-] (-0.5,-0.2)--(-0.5,0.2)--(0.5,0.2)--(0.5,-0.2)--(-0.5,-0.2);
    \node at (-0,-0) {$\scriptstyle{g}$};
    \draw[-,thick] (-0.4,-0.2)--(-0.4,-0.5);
    \draw[-,thick] (0.4,-0.2)--(0.4,-0.5);
    \draw[-,thick] (-0.4,0.2)--(-0.4,0.5);
    \draw[-,thick] (0.4,0.2)--(0.4,0.5);
    \draw[-,thick] (-0.7,-0.5)--(-0.7,0.5);
    \node at (-0.2,0.35) {$\cdot$};
    \node at (-0,0.35) {$\cdot$};
    \node at (0.2,0.35) {$\cdot$};
    \node at (-0.2,-0.35) {$\cdot$};
    \node at (-0,-0.35) {$\cdot$};
    \node at (0.2,-0.35) {$\cdot$};
    \opendot{-0.7,.3};
    \end{tikzpicture}= \begin{tikzpicture}[scale=1.1,anchorbase]
  \draw[-] (-0.5,-0.2)--(-0.5,0.2)--(0.5,0.2)--(0.5,-0.2)--(-0.5,-0.2);
    \node at (-0,-0) {$\scriptstyle{f}$};
    \draw[-,thick] (-0.4,-0.2)--(-0.4,-0.5);
    \draw[-,thick] (0.4,-0.2)--(0.4,-0.5);
    \draw[-,thick] (-0.4,0.2)--(-0.4,0.3)to[out=up,in=up,looseness=2](-0.7,0.3)to(-0.7,-0.5);
    \draw[-,thick] (0.4,0.2)--(0.4,0.5);
    \node at (-0,0.35) {$\cdot$};
    \node at (-0.2,0.35) {$\cdot$};
    \node at (0.2,0.35) {$\cdot$};
    \node at (-0.2,-0.35) {$\cdot$};
    \node at (-0,-0.35) {$\cdot$};
    \node at (0.2,-0.35) {$\cdot$};
  \end{tikzpicture}\otimes\begin{tikzpicture}[scale=1.1,anchorbase]
  \draw[-] (-0.5,-0.2)--(-0.5,0.2)--(0.5,0.2)--(0.5,-0.2)--(-0.5,-0.2);
    \node at (-0,-0) {$\scriptstyle{g}$};
    \draw[-,thick] (-0.4,-0.2)--(-0.4,-0.5);
    \draw[-,thick] (0.4,-0.2)--(0.4,-0.5);
    \draw[-,thick] (-0.4,0.2)--(-0.4,0.5);
    \draw[-,thick] (0.4,0.2)--(0.4,0.5);
    \draw[-,thick] (-0.7,-0.5)--(-0.7,0.5);
    \node at (-0.2,0.35) {$\cdot$};
    \node at (-0,0.35) {$\cdot$};
    \node at (0.2,0.35) {$\cdot$};
    \node at (-0.2,-0.35) {$\cdot$};
    \node at (-0,-0.35) {$\cdot$};
    \node at (0.2,-0.35) {$\cdot$};
    \opendot{-0.7,-.35};
    \end{tikzpicture} = \rho_4(\phi(f\otimes g)),\\
\phi(\lambda(f\otimes g))&=\phi\left(\!\!\begin{tikzpicture}[scale=1.1,anchorbase]
  \draw[-] (-0.5,-0.2)--(-0.5,0.2)--(0.5,0.2)--(0.5,-0.2)--(-0.5,-0.2);
    \node at (-0,-0) {$\scriptstyle{f}$};
    \draw[-,thick] (-0.4,-0.2)--(-0.4,-0.5);
    \draw[-,thick] (0.4,-0.2)--(0.4,-0.5);
    \draw[-,thick] (-0.4,0.2)--(-0.4,0.5);
    \draw[-,thick] (0.4,0.2)--(0.4,0.5);
    \node at (-0,0.35) {$\cdot$};
    \node at (0.2,0.35) {$\cdot$};
    \node at (-0.2,0.35) {$\cdot$};
    \node at (-0.2,-0.4) {$\cdot$};
    \node at (-0,-0.4) {$\cdot$};
    \node at (0.2,-0.4) {$\cdot$};
    \draw[-,thick] (-0.4,0.35)--(-0.6,0.35);
    \closeddot{-0.6,0.35};
  \end{tikzpicture}\otimes\begin{tikzpicture}[scale=1.1,anchorbase]
  \draw[-] (-0.5,-0.2)--(-0.5,0.2)--(0.5,0.2)--(0.5,-0.2)--(-0.5,-0.2);
    \node at (-0,-0) {$\scriptstyle{g}$};
    \draw[-,thick] (-0.4,-0.2)--(-0.4,-0.5);
    \draw[-,thick] (0.4,-0.2)--(0.4,-0.5);
    \draw[-,thick] (-0.4,0.2)--(-0.4,0.5);
    \draw[-,thick] (0.4,0.2)--(0.4,0.5);
    \node at (-0.2,0.35) {$\cdot$};
    \node at (-0,0.35) {$\cdot$};
    \node at (0.2,0.35) {$\cdot$};
    \node at (-0.2,-0.4) {$\cdot$};
    \node at (-0,-0.4) {$\cdot$};
    \node at (0.2,-0.4) {$\cdot$};
  \end{tikzpicture}\right)\ =\ \begin{tikzpicture}[scale=1.1,anchorbase]
  \draw[-] (-0.5,-0.2)--(-0.5,0.2)--(0.5,0.2)--(0.5,-0.2)--(-0.5,-0.2);
    \node at (-0,-0) {$\scriptstyle{f}$};
    \draw[-,thick] (-0.4,-0.2)--(-0.4,-0.5);
    \draw[-,thick] (0.4,-0.2)--(0.4,-0.5);
    \draw[-,thick] (-0.4,0.2)--(-0.4,0.3)to[out=up,in=up,looseness=2](-0.8,0.3)to(-0.8,-0.5);
    \draw[-,thick] (0.4,0.2)--(0.4,0.5);
    \node at (-0,0.35) {$\cdot$};
    \node at (0.2,0.35) {$\cdot$};
    \node at (-0.2,0.35) {$\cdot$};
    \node at (-0.2,-0.4) {$\cdot$};
    \node at (-0,-0.4) {$\cdot$};
    \node at (0.2,-0.4) {$\cdot$};
    \draw[-,thick] (-0.4,0.3)--(-0.6,0.3);
    \closeddot{-0.6,0.3};
  \end{tikzpicture}\otimes\ \begin{tikzpicture}[scale=1.1,anchorbase]
  \draw[-] (-0.5,-0.2)--(-0.5,0.2)--(0.5,0.2)--(0.5,-0.2)--(-0.5,-0.2);
    \node at (-0,-0) {$\scriptstyle{g}$};
    \draw[-,thick] (-0.4,-0.2)--(-0.4,-0.5);
    \draw[-,thick] (0.4,-0.2)--(0.4,-0.5);
    \draw[-,thick] (-0.4,0.2)--(-0.4,0.5);
    \draw[-,thick] (0.4,0.2)--(0.4,0.5);
    \draw[-,thick] (-0.8,-0.5)--(-0.8,0.5);
    \node at (-0.2,0.35) {$\cdot$};
    \node at (-0,0.35) {$\cdot$};
    \node at (0.2,0.35) {$\cdot$};
    \node at (-0.2,-0.4) {$\cdot$};
    \node at (-0,-0.4) {$\cdot$};
    \node at (0.2,-0.4) {$\cdot$};
  \end{tikzpicture}\ =\ \begin{tikzpicture}[scale=1.1,anchorbase]
  \draw[-] (-0.5,-0.2)--(-0.5,0.2)--(0.5,0.2)--(0.5,-0.2)--(-0.5,-0.2);
    \node at (-0,-0) {$\scriptstyle{f}$};
    \draw[-,thick] (-0.4,-0.2)--(-0.4,-0.5);
    \draw[-,thick] (0.4,-0.2)--(0.4,-0.5);
    \draw[-,thick] (-0.4,0.2)--(-0.4,0.3)to[out=up,in=up,looseness=2](-0.8,0.3)to(-0.8,-0.5);
    \draw[-,thick] (0.4,0.2)--(0.4,0.5);
    \node at (-0,0.35) {$\cdot$};
    \node at (-0.2,0.35) {$\cdot$};
    \node at (0.2,0.35) {$\cdot$};
    \node at (-0.2,-0.4) {$\cdot$};
    \node at (-0,-0.4) {$\cdot$};
    \node at (0.2,-0.4) {$\cdot$};
    \draw[-,thick] (-0.8,-0.35)--(-0.6,-0.35);
    \closeddot{-0.6,-0.35};
  \end{tikzpicture}\otimes\ \begin{tikzpicture}[scale=1.1,anchorbase]
  \draw[-] (-0.5,-0.2)--(-0.5,0.2)--(0.5,0.2)--(0.5,-0.2)--(-0.5,-0.2);
    \node at (-0,-0) {$\scriptstyle{g}$};
    \draw[-,thick] (-0.4,-0.2)--(-0.4,-0.5);
    \draw[-,thick] (0.4,-0.2)--(0.4,-0.5);
    \draw[-,thick] (-0.4,0.2)--(-0.4,0.5);
    \draw[-,thick] (0.4,0.2)--(0.4,0.5);
    \draw[-,thick] (-0.8,-0.5)--(-0.8,0.5);
    \node at (-0.2,0.35) {$\cdot$};
    \node at (-0,0.35) {$\cdot$};
    \node at (0.2,0.35) {$\cdot$};
    \node at (-0.2,-0.4) {$\cdot$};
    \node at (-0,-0.4) {$\cdot$};
    \node at (0.2,-0.4) {$\cdot$};
\end{tikzpicture}\\
&=(t-n)\ \begin{tikzpicture}[scale=1.1,anchorbase]
  \draw[-] (-0.5,-0.2)--(-0.5,0.2)--(0.5,0.2)--(0.5,-0.2)--(-0.5,-0.2);
    \node at (-0,-0) {$\scriptstyle{f}$};
    \draw[-,thick] (-0.4,-0.2)--(-0.4,-0.5);
    \draw[-,thick] (0.4,-0.2)--(0.4,-0.5);
    \draw[-,thick] (-0.4,0.2)--(-0.4,0.3)to[out=up,in=up,looseness=2](-0.8,0.3)to(-0.8,-0.5);
    \draw[-,thick] (0.4,0.2)--(0.4,0.5);
    \node at (-0,0.35) {$\cdot$};
    \node at (-0.2,0.35) {$\cdot$};
    \node at (0.2,0.35) {$\cdot$};
    \node at (-0.2,-0.4) {$\cdot$};
    \node at (-0,-0.4) {$\cdot$};
    \node at (0.2,-0.4) {$\cdot$};
   \end{tikzpicture}\otimes\ \begin{tikzpicture}[scale=1.1,anchorbase]
  \draw[-] (-0.5,-0.2)--(-0.5,0.2)--(0.5,0.2)--(0.5,-0.2)--(-0.5,-0.2);
    \node at (-0,-0) {$\scriptstyle{g}$};
    \draw[-,thick] (-0.4,-0.2)--(-0.4,-0.5);
    \draw[-,thick] (0.4,-0.2)--(0.4,-0.5);
    \draw[-,thick] (-0.4,0.2)--(-0.4,0.5);
    \draw[-,thick] (0.4,0.2)--(0.4,0.5);
    \draw[-,thick] (-0.8,-0.5)--(-0.8,0.5);
    \node at (-0.2,0.35) {$\cdot$};
    \node at (-0,0.35) {$\cdot$};
    \node at (0.2,0.35) {$\cdot$};
    \node at (-0.2,-0.4) {$\cdot$};
    \node at (-0,-0.4) {$\cdot$};
    \node at (0.2,-0.4) {$\cdot$};
  \end{tikzpicture}= \lambda_4(\phi(f \otimes g)).
\end{align*}
\end{proof}

\begin{proof}[Proof of \cref{refinedfilt}]
  The functor $D_{b|a} \circ j_!$ is defined by tensoring with the bimodule
  $\overline{M}$ that is the simultaneous generalized $a$ eigenspace of
  the endomorphism $\rho$ and generalized $b$ eigenspace of the endomorphism $\lambda$
  defined in \cref{xyz}.
  \cref{filtration} defines a filtration of $M$ with sections
  $M_i / M_{i-1} \cong N_i$ for $i=1,\dots,4$.
  Then \cref{xyz} shows that the endomorphisms
  $\rho$ and $\lambda$ preserve this filtration, hence, the filtration of $M$
  induces a filtration of the summand $\overline{M}$. Moreover, for each $i$,
  $\overline{M}_i / \overline{M}_{i-1}$ is isomorphic to
  the summand $\overline{N}_i$ of $N$ defined by the simultaneous generalized $a$-eigenspace of the endomorphism $\rho_i$ and generalized $b$ eigenspace of the endomorphism $\lambda_i$. By the descriptions of $\rho_i$ and $\lambda_i$, it follows that
  $\overline{N}_i \otimes_{\sym}$ is isomorphic to the functor
  $j_! \circ E_a \circ \PR_{t-b}$, $j_! \circ \PR_{t-a} \circ \PR_{t-b}$,
  $j_! \circ E_a \circ F_b$ or $j_! \circ \PR_{t-a} \circ F_b$
  for $i=4,3,2,1$, respectively.
  It remains to observe that $\sym$ is semisimple, so every $\sym$-module is flat. This means that
  the filtration of $\overline{M}$ induces a filtration $0 = S_0 \subseteq S_1 \subseteq S_2 \subseteq S_3 \subseteq S_4 = D_{b|a} \circ j_!$ such that
  $S_i \cong \overline{M}_i / \overline{M}_{i-1} \otimes_{\sym}
  \cong \overline{N}_i \otimes_{\sym}$.
\end{proof}

\end{document}